\documentclass[10pt,reqno]{amsart}
\usepackage[affil-it]{authblk}
\usepackage{titling}
\usepackage[latin1]{inputenc}
\usepackage[english]{babel}
\usepackage{calligra}
\usepackage[OT1]{fontenc}
\usepackage{amsfonts}
\usepackage{amsmath}
\usepackage{amssymb}
\usepackage{amsthm}
\usepackage{faktor}
\usepackage{float}
\usepackage{enumerate}
\usepackage{color}
\usepackage{pdfpages}
\usepackage{esint}
\usepackage{hyperref}
\usepackage{cite}
\usepackage{fancyhdr}
\usepackage{ulem}
\usepackage{mathtools}
\usepackage{mathrsfs}

\usepackage{etoolbox}
\patchcmd{\section}{\scshape}{\bfseries}{}{}
\makeatletter
\renewcommand{\@secnumfont}{\bfseries}
\makeatother
\newcommand\testname{\textbf{ Abstract}}
\makeatletter
\newenvironment{abs}{%
    \small
    \begin{center}%
        {\textsc \testname\vspace{-.2em}\vspace{\z@}}%
    \end{center}%
    \quote
    }
   {\endquote}
\makeatother

\allowdisplaybreaks[4]

\usepackage{verbatim}
\usepackage[letterspace=150]{microtype}
\usepackage{multicol}
\usepackage[affil-it]{authblk}
\usepackage{titling}
\usepackage[latin1]{inputenc}
\usepackage{calligra}
\usepackage[OT1]{fontenc}
\usepackage[english]{babel}
\usepackage{amsfonts}
\usepackage{amsmath}
\usepackage{amssymb}
\usepackage{amsthm}
\usepackage{faktor}
\usepackage{float}
\usepackage{enumerate}
\usepackage{color}
\usepackage{pdfpages}
\usepackage{esint}
\usepackage{cite}
\usepackage{fancyhdr}
\usepackage{mathtools}
\usepackage{mathrsfs}
\usepackage{relsize}
\usepackage{bm}

\DeclareMathOperator*{\Id}{Id}

\newcommand{\Div}{\mathrm{div}}

\newcommand{\ee}{\varepsilon}

\newcommand{\dd}{\mathrm{d}}

\newcommand{\pre}{ \mathrm{p}}
\newcommand{\uu}{ \mathbf{u}}
\newcommand{\ddd}{ \mathbf{d}}

\newcommand{\TT}{\mathbb{T}}

\newcommand{\J}{\mathcal{J}}
\newcommand{\I}{\mathcal{I}}

\newcommand{\RR}{\mathbb{R}}
\newcommand{\ZZ}{\mathbb{Z}}
\newcommand{\R}{\mathbb{R}}
\newcommand{\NN}{\mathbb{N}}

\newcommand{\Dd}{\Delta}
\newcommand{\Sd}{S}

\newcommand{\sdm}[1]{\Sd_{#1\rule[1pt]{3pt}{0.4pt}1}}

\newtheorem{theorem}{Theorem}[section]

\newtheorem{prop}[theorem]{Proposition}
\newtheorem{lemma}[theorem]{Lemma}
\newtheorem{definition}[theorem]{Definition}
\newtheorem{remark}[theorem]{Remark}

\DeclareFontFamily{OT1}{rsfs}{}
\DeclareFontShape{OT1}{rsfs}{m}{n}{ <-7> rsfs5 <7-10> rsfs7 <10-> rsfs10}{}
\DeclareMathAlphabet{\mycal}{OT1}{rsfs}{m}{n}

\definecolor{grey}{rgb}{0.85,0.85,0.85}
\date{}

\title{\Large 
	\textbf{
	\uppercase{Uniqueness of global weak solutions for the general Ericksen--Leslie system with Ginzburg--Landau penalization in $\mathbb{T}^2$
	}}}

\author{
Francesco De Anna$\,^1$, \quad Hao Wu$\,^2$
}

\affil{	\small{$\,^1$Institute of Mathematics, University of W\"urzburg, 97074 W\"{u}rzburg, Germany}\\
		\small\textnormal{francesco.deanna@mathematik.uni-wuerzburg.de}
		\vspace{0.2cm}}
		
\affil{	\small{$\,^2$School of Mathematical Sciences, Fudan University, 200433 Shanghai, China}\\
        \small{Key Laboratory of Mathematics for Nonlinear Sciences (Fudan University), Ministry of Education, China}\\
		\small\textnormal{haowufd@fudan.edu.cn}
		\vspace{0.2cm}}

\date{\today}		
 
\keywords {}

\begin{document}

\maketitle 
 
\begin{abs} 
 The Ericksen-Leslie system is a fundamental hydrodynamic model that describes the evolution of incompressible liquid crystal flows of nematic type. In this paper, we prove the uniqueness of global weak solutions to the general Ericksen-Leslie system with a Ginzburg-Landau type approximation in a two dimensional periodic domain. The proof is based on some delicate energy estimates for the difference of two weak solutions within a suitable functional framework that is less regular than the usual one at the natural energy level, combined with the Osgood lemma involving a specific double-logarithmic type modulus of continuity. We overcome the essential mathematical difficulties arising from those highly nonlinear terms in the Leslie stress tensor and in particular, the lack of maximum principle for the director equation due to the stretching effect of the fluid on the director field. Our argument makes full use of the coupling structure as well as the dissipative nature of the system, and relies on some techniques from harmonic analysis and paradifferential calculus in the periodic setting.
 \end{abs}
 
\medskip 
\noindent{\bf Keywords.}  
Ericksen-Leslie system, global weak solution, uniqueness, Osgood lemma, Littlewood-Paley Theory.

\medskip
\noindent{\bf AMS subject classification.} 
35A02, 35D30, 35Q35, 76A15, 76D05

\makeatother
\allowdisplaybreaks{}

\section{Introduction}
\setcounter{equation}{0}
 
Liquid crystals are materials that present a state of matter between crystalline solids and isotropic liquids. They can flow like viscous fluids, while possessing characteristics that are typical for solid crystals, for instance, regular periodic arrangement of molecules. The hydrodynamic theory of liquid crystal flows of nematic type was first established around 1960's by Ericksen and Leslie \cite{Ericksen1,Ericksen2,Les68}. Developments on mathematical analysis for the  associated system of partial differential equations were initiated in a series of works by F.-H. Lin and C. Liu in the 1990's \cite{MR1329830,linliu96,MR2449088}. Later on, this topic has attracted a lot of attentions from mathematicians, see for instance, \cite{MR2947543,MR3713532,hong,MR3238531,linlinwang,MR4163316,MR3460623} and the references therein. 

\subsection{The PDE system} 
The proposed model addresses a macroscopic description for the time evolution of liquid crystal molecules together with the underlying incompressible flow, through a parabolic-type system of partial differential equations, coupling the Navier-Stokes equations for the fluid velocity with a nonlinear convection-diffusion system for the molecular orientation. In this paper, we are concerned with a Ginzburg-Landau approximation of the general Ericksen-Leslie system, which was proposed in \cite{MR2449088} by F.-H. Lin and C. Liu and can be expressed by the following system of partial differential equations:
\begin{equation}\label{main_system_intro}
\left\{\hspace{0.2cm}
	\begin{alignedat}{2}
		&\,\partial_t\uu   + \uu\cdot \nabla \uu\, +\,\nabla \pre  =
			-\,\Div\, 
			\big(\nabla \ddd\odot\nabla \ddd\big)
			\,+\,
			\Div\,\tilde{\bm{\sigma}}
		\hspace{2.5cm}
																		&& \text{in}\  (0,T)\times\mathbb{T}^3, \vspace{0.1cm}\\
		&\,\Div\,\uu\,=\,0		
																		&& \text{in}\ (0,T)\times \mathbb{T}^3, \vspace{0.1cm}\\
		&\,	
		\partial_t \ddd	\,+\,\uu\cdot \nabla\ddd -\omega \ddd + \frac{\lambda_2}{\lambda_1}A \ddd
		\,=\,
		\,-\frac{1}{\lambda_1}\big(\Delta\ddd\,-\,\nabla_\ddd W(\ddd)\big) 
		&&\text{in}\ (0,T)\times \mathbb{T}^3, \vspace{0.1cm}\\			
		&\,	
		(\uu,\,\ddd)_{|t=0}
		\,=\,
		(\uu_0,\,\ddd_0)
		&&\hspace{1.30cm}\text{in}\ \mathbb{T}^3. 
	\end{alignedat}
	\right.
\end{equation}
Here for the sake of simplicity, we neglect possible boundary effects and impose the problem in a periodic domain $\mathbb{T}^3=[-\pi,\pi]^3$. Besides, $T\in (0,\infty]$ denotes a general final time. For any $(t,x)\in (0,T)\times  \mathbb{T}^3$, the velocity field of the flow is described by a vector $\uu = \uu(t,x) \in \mathbb{R}^3$, while the macroscopic average of nematic liquid crystal molecular orientations is described through the director field $\ddd = \ddd(t,x)\in \mathbb{R}^3$. The fluid is assumed to be homogeneous (e.g., the fluid density equals to one) and incompressible as depicted by the divergence free condition on $\uu$, which yields the hydrodynamic pressure $\pre$ as a Lagrange multiplier in the equations of momentum balance. The tensor $\nabla \ddd\odot \nabla \ddd$ is referred to as the (elastic) Ericksen stress tensor that denotes a $3\times 3$ symmetric matrix whose $(i,j)$-th entry is given by $\partial_i\ddd \cdot \partial_j \ddd$, for $1 \leq  i,j \leq 3$. On the other hand, the viscous Leslie stress tensor $\tilde{\bm{\sigma}}$ takes the following form:
\begin{equation*}
	\tilde{\bm{\sigma}}\,=\,
	\mu_1 (\ddd\otimes\ddd)\,(\ddd\cdot (A \ddd))\,+\mu_2 \,\mathcal{N}\otimes\ddd+
	\mu_3 \,\ddd\otimes \mathcal{N} \,+\,\mu_4 A +  \mu_5 (A\ddd)\otimes\ddd\,+\, \mu_6
	\ddd\otimes(A\ddd),
\end{equation*}
where $A=(\nabla \uu+ \nabla^{\mathrm{tr}} \uu)/2$ is the rate of strain tensor, $\omega = (\nabla \uu - \nabla^{\mathrm{tr}} \uu)/2$ is the skew-symmetric part of the strain rate, and the vector $\mathcal{N} = \partial_t \ddd+\uu\cdot \nabla\ddd - \omega \ddd$ stands for the corotational time derivative of the director  (i.e., the rigid rotation part of the changing rate of $\ddd$ by the fluid). The symbol $\otimes$ stands for the standard Kronecker product, that is $(\mathbf{a}\otimes\mathbf{b})_{ij}=\mathbf{a}_i\mathbf{b}_j$, for $1\leq i,j\leq 3$. 
The notation $\nabla_\ddd$ means the gradient with respect to the director $\ddd$, while $W(\ddd)$ denotes a double-well potential that penalises the deviation of the length $|\ddd|$ from $1$:
\begin{equation}\label{Ginzbourg-Landau-appx}
	W(\ddd)\,=
	\,\frac{1}{4\ee^2}
	\left(	|\ddd|^2-1 	\right)^2,
\end{equation}
for some small parameter $\ee>0$. The six parameters $\mu_1, \dots , \mu_6$ appearing in $\tilde{\bm{\sigma}}$ are called Leslie coefficients that may depend on  characteristics of the nematic material as well as the temperature. The other two parameters $\lambda_1$, $\lambda_2$ are material coefficients reflecting the molecular shape and the slippery property between the fluid and the particle. A typical example is determined by the following relations
\begin{align*}
    & \lambda_1= -1,\quad \lambda_2 = 2\alpha-1,\quad 
    \mu_1 = 0,\quad \mu_2 = -\alpha,\quad \mu_3 = 1-\alpha,\\ 
    & \mu_4 = \nu>0,\quad 
    \mu_5 = \alpha(2\alpha-1),\quad \,\mu_6 = (\alpha-1)(2\alpha-1),
\end{align*}
for some parameter $\alpha\in [0,1]$ that reflects the shape of the constitutive molecules of the liquid crystal: $\alpha = 0$ corresponds to molecules with the disk-like shape, $\alpha = 1$ addresses the rod-like shape, and for $\alpha\in (0,1)$ the molecules are modelled by means of ellipsoids (see \cite{SunLiu}).  For a general setting of parameters, the following relations between those coefficients are frequently introduced in the literature:
\begin{equation*}
    \lambda_1 = \mu_2 -\mu_3,\qquad \lambda_2 = \mu_5-\mu_6,\qquad \mu_2+\mu_3 = \mu_6-\mu_5.
\end{equation*}
Among them, the first two relations are  compatibility conditions in order to satisfy the equation of motion identically (cf. \cite{Les68}), while the third one, often referred to as Parodi's relation, was proposed by Parodi \cite{Pa} according to Onsanger's reciprocal relation between flows and forces in irreversible thermodynamic systems. Obviously, Parodi's relation reduces one degree of freedom for the choices of $\mu_i$, its role in mathematical analysis for the Ericksen-Leslie system has been discussed by several researchers (see e.g., \cite{MR3045633,MR3713532,wuxuliu2013}). 
When the general Ericksen-Leslie system with Ginzburg-Landau approximation \eqref{main_system_intro} is concerned, in \cite{wuxuliu2013} the authors illustrated the connection between Parodi's relation and the linear as well as nonlinear stability properties of the system. Later, for the same system, it was indicated in \cite{MR3045633} that Parodi's relation is a sufficient but not necessary condition for the existence of global weak solutions.

\subsection{A brief overview of related literature}
The mathematical treatment of the Ericksen-Leslie system has been extensively addressed in the literature during the past decades. To simplify the presentation, we shall place the overview here mainly about the system \eqref{main_system_intro}, that is, the Ginzburg-Landau penalised version of the original Ericksen-Leslie system. For further information, one may refer to the recent review papers \cite{HP18,linwang2014}. 

It is well known that the (static) configuration of liquid crystal molecular orientations is determined by the Oseen-Frank energy density (with unitary constraint $|\ddd|=1$):
\begin{align*}
W_{\text{OF}}(\ddd) 
=\frac{k_1}{2}(\nabla\cdot \ddd)^2 +\frac{k_2}{2}|\ddd\cdot (\nabla
\times \ddd)|^2 +\frac{k_3}{2}|\ddd\times(\nabla\times \ddd)|^2  + \frac{(k_2+k_4)}{2}\big[ {\rm
tr} (\nabla \ddd)^2-(\nabla \cdot \ddd)^2\big],
\end{align*}
which represents the splaying, twisting and bending distortions of the liquid crystal molecules. In the simplest setting such that $k_1=k_2=k_3=1$ and $k_4=0$ (i.e., the elastically isotropic
case), the energy density function $W_{\text{OF}}$ reduces to  $W_{\text{D}}(\ddd)=\frac12|\nabla \ddd|^2$
that is the classical Dirichet energy density. Motivated by previous works on the heat flow of harmonic map, a natural way to handle the strong nonlinearity due to the constraint $|\ddd|=1$ is to introduce a relaxation of Ginzbug-Landau type in the elastic energy, i.e., \eqref{Ginzbourg-Landau-appx} and then let $\ee\to 0^+$ (cf. e.g., \cite{hong,linlinwang}). On the other hand, the approximate system with penalty also has natural physical interpretations such that it can be attributed to the extensibility of liquid crystal molecules (see \cite{Les79}).

Since the structure of the general Ericksen-Leslie system \eqref{main_system_intro} is quite challenging, efforts have been made first in the literature for its simplified versions. The simplest case, that is, only keeping the term related to the fluid viscosity in $\tilde{\bm{\sigma}}$ and neglecting the rotation as well as stretching effects of the fluid, has been extensively analysed for its well-posedness, regularity of solutions, long-time behavior and the associated optimal control problems, see for instance, \cite{CRW17,MR2947543,dai12,grawu13,huwu,MR1329830,linliu96,wu10} and the references therein. Although this system seems to be over simplified, it somehow maintains the most important structure of the Ericksen-Leslie system such that the total energy given by 
\begin{align}
\mathcal{E}_{\text{total}}(t)=\mathcal{E}_{\text{total}}(\uu(t,x),\ddd(t,x)) =\underbrace{\int_\Omega \frac12|\uu(t,x)|^2\,\dd x}_{\text{kinetic energy}} + \underbrace{\int_\Omega \frac12|\nabla \ddd(t,x)|^2 + W(\ddd(t,x)) \,\dd x}_{\text{elastic energy}} \label{total}
\end{align}
is dissipative along time evolution. Here, $\Omega$ denotes a generic bounded domain in $\mathbb{R}^d$, $d=2,3$. More precisely, when the general system \eqref{main_system_intro} is concerned, its dissipative structure is presented by the following basic energy law (written here in a formal way for smooth solutions, see e.g.,  \cite{MR2449088,wuxuliu2013}):
\begin{align}
\frac{\dd}{\dd t}\mathcal{E}_{\text{total}}(t)
&= -\int_{\Omega} \Big(\mu_1|\ddd\cdot A\ddd|^2 +\frac{\mu_4}{2}|\nabla \uu|^2 +(\mu_5+\mu_6)|A\ddd|^2 - \lambda_1 |\mathcal{N}|^2 - (\lambda_2-\mu_2-\mu_3)\mathcal{N}\cdot(A\ddd) \Big)
\,\dd x.
\label{BELnoPa}
\end{align}
The first analysis result for system \eqref{main_system_intro} on a three dimensional smooth bounded domain subject to Dirichlet boundary conditions (i.e., $\uu=\mathbf{0}$ and $\ddd=\ddd_0$ on $\partial \Omega$) was obtained in \cite{MR2449088}. There the authors proved that any initial data with finite total energy generates a global-in-time weak solution, as long as the Leslie coefficients satisfy certain restrictions that ensure the dissipative structure of the system, i.e., $\frac{\dd }{\dd t}\mathcal{E}_{\text{total}}\leq -\mathcal{D}$ with a suitable energy dissipation functional $\mathcal{D}\geq 0$ (cf. \eqref{BELnoPa}). In particular, the result therein was achieved under the so-called corotational framework corresponding to the specific assumption $\lambda_2 = 0$. Physically, this assumption neglects the stretching effect of the macroscopic fluid on the director field $\ddd$ (i.e., the relative slip between the fluid and molecules), a scenario that is more feasible when treating liquid crystals with ``small'' constitutive molecules. From the mathematical point of view, this assumption yields a weak maximum principle such that if the initial and boundary datum $\ddd_0$ satisfies $|\ddd_0|\leq 1$, then $|\ddd|\leq 1$ holds in whole domain for all time  (see \cite[Theorem 3.1]{MR2449088}). This property implies a uniform bound of $\|\ddd\|_{L^\infty}$ that is crucial when dealing the highly nonlinear Leslie stress tensor \cite{MR2449088}. 

The general case $\lambda_2\neq 0$ turns out to be more involved, since the above mentioned maximum principle for $\ddd$ is no longer valid. In this aspect, the authors of \cite{MR1873808} made a first attempt by showing the existence and uniqueness of a local-in-time classical solution within the non-corotational setting $\lambda_1 = -1$, $\lambda_2 = 1$, $\mu_4=2\nu>0$, $\mu_i = 0$, $i=1,2,3,5,6$. However, there the total energy of the system is not dissipative and thus only a small-data global existence result could be obtained even in two dimensions. This problem was partially solved later in \cite{SunLiu}, where the following alternative setting of Leslie coefficients was taken:   
$\lambda_1 = -1$, $\lambda_2 = 1$, $\mu_1 = 0$, $\mu_2 = -1$, $\mu_3 = 0$, $\mu_4 =2\nu>0$, $\mu_5 = 1$ and $\mu_6 = 0$ so that the energy dissipation property can be guaranteed. In this thermodynamically consistent framework, the authors obtained the existence of global strong solutions with regularity properties $\uu \in L^\infty(0,T;H^1)\cap L^2(0,T;H^2)$, $\ddd \in L^\infty(0,T;H^2)\cap L^2(0,T;H^3)$, provided that the spatial dimension is two or the viscosity $\nu>0$ is sufficiently large with respect to the initial data in three dimensions. We refer to \cite{FR13,grawu11,wuxuliu2012} for further analysis results on global well-posedness and long-time behavior of the system in a more general setting involving the parameter $\alpha\in [0,1]$ on molecule shape described above. Under the same choice of coefficients, the authors of \cite{MR3023076} applied a suitable regularization procedure to prove the existence of global weak solutions with finite energy in three dimensions (under some physically meaningful boundary conditions for $\uu$ and $\ddd$), overcoming the difficulty due to the lack of maximum principle for the director equation. Concerning the full Ericksen-Leslie system \eqref{main_system_intro} with general choices of coefficients, in \cite{wuxuliu2013}, the authors proved global well-posedness (i.e., global strong solutions) and long-time behavior of the system under the assumption that the viscosity $\mu_4$ is sufficiently large; besides, with Parodi's relation, they showed the global well-posedness and Lyapunov stability for the system near local energy minimizers, both results were obtained in three dimensional periodic setting. Besides, local well-posedness of classical solutions under Parodi's relation and a blow-up criterion were obtained in \cite{MR3045633}. We also refer to \cite{dai15} for strong well-posedness of the general system with non-constant fluid density. It is worth mentioning that the possible loss of maximum principle for $\ddd$ did not prevent people obtaining results on strong solutions, because of the Sobolev embedding $H^2\hookrightarrow L^\infty$ that holds when the dimension of space is less or equal than three. However, it is essential when we deal with global weak solutions, even for its definition. To this end, in \cite{MR3045633}, the authors considered the initial boundary value problem for system \eqref{main_system_intro} in a three dimensional bounded domain, under some natural boundary conditions and suitable requirements on the Leslie coefficients that ensure the energy dissipation. In order to overcome the difficulty due to lack of maximum principle when $\lambda_2\neq 0$, they introduced a new formulation of weak solutions that took into account the low regularity of those highly nonlinear stress terms in $\tilde{\bm{\sigma}}$. The existence of global-in-time weak solutions was then achieved via a suitable double-level Faedo-Galerkin approximation scheme (cf. \cite{MR3023076,MR2947543}). The same result is also valid in the three dimensional periodic setting as in \eqref{main_system_intro}.  

Above we have just focused on mathematical results that are closely related to the Ericksen-Leslie system with Ginzburg-Landau approximation \eqref{main_system_intro}. Nevertheless, it should be pointed out that important progresses have been achieved for the original Ericksen-Leslie model (as well as its simplified versions), in which the penalisation energy density $W(\ddd)$ is removed and instead, a Lagrangian multiplier corresponding to the unitary constraint $|\ddd(t,x)|=1$ has to be introduced in the director equation. For instance, we refer the interested readers to \cite{chenyu,hong,honglixin,huang,MR4163316,linlinwang,linwang,WangWangZhang,xuzhang} for the simplified liquid crystal system involving the transported heat flow of harmonic map, to \cite{MR3713532,MR3238531,MR3460623,WWZ} for the general Ericksen-Leslie system as well as its non-isothermal extension \cite{deannaliu}, and to \cite{caiwang,jiang} for the case with inertial term in the director equation (i.e., the so-called hyperbolic Ericksen-Leslie system).

\subsection{Goal of this paper} 
For the general system \eqref{main_system_intro}, a fundamental  problem that remains open so far is the uniqueness of global weak solutions. Indeed, the uniqueness of weak solutions is still unknown for the simplified system studied in \cite{MR3023076,grawu11,SunLiu,wuxuliu2012} even when the spatial dimension is two. Since \eqref{main_system_intro} includes the Navier-Stokes equations as a subsystem, uniqueness of weak solutions in three dimensions seems to be out of reach for the moment. Hence, our aim in this paper is to partially answer the question by establishing the uniqueness of global weak solutions to the reduced problem of  \eqref{main_system_intro} in a two dimensional periodic domain $\mathbb{T}^2$ (see Theorem \ref{thm:uniqueness} below). 

To be more precise, in the two dimensional setting, the fluid velocity $\uu$ becomes a planar vector field in $\mathbb{R}^2$, and thus both $A$ and $\omega$ are horizontal $2\times 2$-matrices. On the other hand, the molecular director $\ddd$ can still take its value in $\mathbb{R}^3$. In order to get a closed and consistent system, we denote $\hat{\ddd}=(d_1,d_2)^{\mathrm{tr}}$ for any given  $\ddd=(d_1,d_2,d_3)^{\mathrm{tr}}$ and the system \eqref{main_system_intro} can be reformulated as follows (using the idea in \cite{MR3238531})
\begin{equation}\label{main_system_intro_2d}
\left\{\hspace{0.2cm}
	\begin{alignedat}{2}
		&\,\partial_t\uu   + \uu\cdot \nabla \uu\, +\,\nabla \pre  =
			-\,\Div\, 
			\big(\nabla \ddd\odot\nabla \ddd\big)
			\,+\,
			\Div\,\hat{\bm{\sigma}}
		\hspace{2.5cm}
																		&& \text{in}\  (0,T)\times\mathbb{T}^2, \vspace{0.1cm}\\
		&\,\Div\,\uu\,=\,0		
																		&& \text{in}\ (0,T)\times \mathbb{T}^2, \vspace{0.1cm}\\
		&\,	
		\partial_t \hat{\ddd}	\,+\,\uu \cdot \nabla \hat{\ddd} -\omega \hat{\ddd} + \frac{\lambda_2}{\lambda_1}A \hat{\ddd}
		\,=\,
		\,-\frac{1}{\lambda_1}\big(\Delta\hat{\ddd}\,-\,\frac{1}{\ee^2}
	(|\ddd|^2-1)\hat{\ddd}\big) 
		&&\text{in}\ (0,T)\times \mathbb{T}^2, \vspace{0.1cm}\\	
		&\,	
		\partial_t d_3	\,+\,\uu\cdot \nabla d_3 
		\,=\,
		\,-\frac{1}{\lambda_1}\big(\Delta d_3\,-\,\,\frac{1}{\ee^2}
	(|\ddd|^2-1)d_3\big) 
		&&\text{in}\ (0,T)\times \mathbb{T}^2, \vspace{0.1cm}\\	
		&\,	
		(\uu,\,\ddd)_{|t=0}
		\,=\,
		(\uu_0,\,\ddd_0)
		&&\hspace{1.30cm}\text{in}\ \mathbb{T}^2, 
	\end{alignedat}
	\right.
\end{equation}
where the reduced Leslie stress tensor is given by 
\begin{equation*}
	\hat{\bm{\sigma}}\,=\,
	\mu_1 (\hat{\ddd}\otimes\hat{\ddd})\,(\hat{\ddd}\cdot (A \hat{\ddd}))\,+\mu_2 \,\hat{\mathcal{N}}\otimes\hat{\ddd} +
	\mu_3 \,\hat{\ddd}\otimes \hat{\mathcal{N}} \,+\,\mu_4 A +  \mu_5 (A\hat{\ddd})\otimes\hat{\ddd}\,+\, \mu_6
	\hat{\ddd}\otimes(A\hat{\ddd}),
\end{equation*}
with $\hat{\mathcal{N}} = \partial_t \hat{\ddd} + \uu\cdot \nabla\hat{\ddd} - \omega \hat{\ddd}$.
Then it is straightforward to check that, by using the same approximating scheme as in \cite{MR3045633}, we are able to prove existence of global weak solutions to problem \eqref{main_system_intro_2d} subject to initial data $(\uu_0,\,\ddd_0)\in L^2(\mathbb{T}^2) \times H^1(\mathbb{T}^2)$ with $\mathrm{div} \uu_0=0$ under suitable assumptions on the coefficients 
\begin{align}
\begin{cases}
\ \lambda_1=\mu_2-\mu_3,\quad \lambda_2=\mu_5-\mu_6,\\
\ \lambda_1<0,\quad \mu_1\geq 0,\quad \mu_4>0,\quad \mu_5+\mu_6\geq 0,\\
\ \text{either assuming}\quad  \mu_2+\mu_3=\mu_6-\mu_5\ \ \text{with}\ \  \displaystyle{\frac{(\lambda_2)^2}{-\lambda_1} <
 \mu_5+\mu_6}\ \text{when}\ \lambda_2\neq 0,
 \smallskip \\
\ \text{or, removing Parodi's relation but taking} \quad |\lambda_2-\mu_2-\mu_3| <
2\sqrt{-\lambda_1}\sqrt{\mu_5+\mu_6},
\end{cases}\label{coeass}
\end{align}
which guarantee sufficient dissipation of the total energy (see \cite[Section 2]{wuxuliu2013}). 

A standard approach to prove the uniqueness of weak solutions is to evaluate the difference of two solutions through its estimates in the natural energy space. This method works well for the two dimensional Navier-Stokes equations for incompressible viscous fluids, and in particular, it can be applied to a simplified Ericksen-Leslie system for incompressible liquid crystal flows subject to unitary constraint $|\ddd(t,x)|=1$ in two dimensions (with some nontrivial modifications in the argument due to the heat flow of harmonic map, see \cite{LinWang2010}). Unfortunately, for problem \eqref{main_system_intro_2d} a similar argument would eventually lead to several challenging nonlinear terms that cannot be controlled when performing the usual energy estimates. The essential mathematical difficulty in achieving our goal on the uniqueness of global weak solutions to problem \eqref{main_system_intro_2d}, is exactly due to the lack of maximum principle for $\ddd$ when the stretching effect from the fluid is present (i.e., $\lambda_2\neq 0$). In particular, those nonlinear stress terms associated with $\mu_1$, $\mu_2$, $\mu_3$, $\mu_5$, $\mu_6$ involve the director field $\ddd$ with higher powers, which are very difficult to handle. Our basic strategy is to evaluate the difference between two weak solutions $(\uu_1-\uu_2,\,\ddd_1-\ddd_2)$ in proper  function spaces that are less regular than the natural energy space $L^2(\TT^2)\times H^1(\TT^2)$ (cf. the total energy \eqref{total}). Then we overcome the above mentioned difficulties by combining several techniques in analysis, among which the Osgood lemma with a specific double-logarithmic type modulus of continuity as well as some useful tools from harmonic analysis and paradifferential calculus plays a crucial role (see Section \ref{main_result} for details). We believe that our approach will be useful for the study of other problems with highly nonlinear structure.

\subsection{Plan of the paper}
The remaining part of this paper is organized  as follows. In Section \ref{main_result}, we state the main result Theorem \ref{thm:uniqueness} and sketch its proof. In Section \ref{sec:Littlewood-Paley}, we present some basic aspects and useful tools related to the Littlewood-Paley theory in the periodic setting. From Section \ref{sec:estT1} to Section \ref{sec:double-log-est}, we provide detailed proofs for those key estimates summarized in Proposition \ref{thm:ineq}. The hardest part of the proof is contained in Section \ref{sec:double-log-est}, where we handle the most challenging nonlinear term associated with $\mu_1$ in the Leslie stress tensor that involves the highest power of the director $\ddd$.

\section{Main Result}\label{main_result}
\setcounter{equation}{0}

In this section, we first state the main result of this paper on the uniqueness of global weak solutions in a two dimensional periodic setting (see Theorem \ref{thm:uniqueness} below) and then illustrate the key ideas of its proof. For the sake of a compact presentation, in what follows, we consider only a planar director field still denoted by $\mathbf{d}$ and take a specific ansatz for the Leslie coefficients such that  
\begin{equation}
    \mu_1 = 1,\quad \mu_2 = -1,\quad \mu_3 = 0,\quad \mu_4 = 2\nu>0,\quad \mu_5 = 3,\quad \mu_6 = 
    1, \label{choicea}
\end{equation}
which yields $\lambda_1=-1$ and $\lambda_2=2$. 
Besides, since the value of the penalizing parameter $\varepsilon>0$ has no influence on the subsequent analysis, we simply set it to be one. As a consequence, the resulting PDE system that we are going to study can be written as (cf. \eqref{main_system_intro_2d})
\begin{equation}\label{main_system}
\left\{\hspace{0.2cm}
	\begin{alignedat}{2}
		&\,\partial_t\uu   + \uu\cdot \nabla \uu  -\nu \Delta \uu  +\nabla \pre =
			-\,\Div\, 
			\big(\nabla \ddd\odot\nabla \ddd\big)
			\,+\,
			\Div\,\bm{\sigma}
		\hspace{2.5cm}
																		&& \text{in}\ (0,T)\times\mathbb{T}^2, \vspace{0.1cm}\\
		&\,\Div\,\uu\,=\,0		
																		&& \text{in}\ (0,T)\times \mathbb{T}^2, \vspace{0.1cm}\\
		&\,	
		\partial_t \ddd	\,+\,\uu\cdot \nabla\ddd 
		-\frac32(\nabla \uu)\ddd- \frac12(\nabla^{\mathrm{tr}} \uu) \ddd
		\,=\,
		\,\Delta\ddd\,-\,\nabla_\ddd W(\ddd) 
		&& \text{in}\ (0,T)\times \mathbb{T}^2, \vspace{0.1cm}\\			
		&\,	
		(\uu,\,\ddd)_{|t=0}
		\,=\,
		(\uu_0,\,\ddd_0)
		&&\hspace{1.30cm}\text{in}\ \mathbb{T}^2, 
	\end{alignedat}
	\right.
\end{equation}
where the stress tensor $\bm{\sigma}$ is now given by 
\begin{equation*}
	\bm{\sigma}\,=\,
	\ddd\otimes\ddd\,(\ddd\cdot (A \ddd))\,- \,\mathcal{N}\otimes\ddd\,+\, 3(A\ddd)\otimes\ddd\,+\, \ddd\otimes(A\ddd).
\end{equation*}
It follows from the director equation that  
$\mathcal{N} = 2A\ddd + \Delta \ddd - \nabla_\ddd W(\ddd)$. Using this relation, we can recast $\bm{\sigma}$ by means of 
\begin{equation}\label{sigma_intro}
	\bm{\sigma} = \ddd\otimes \ddd (\ddd\cdot (A \ddd)) + \ddd\otimes(A\ddd)
	\,+\, (A\ddd)\otimes\ddd
	\,- (\Delta \ddd - \nabla_\ddd W(\ddd))\otimes \ddd.
\end{equation}
Here, we choose to eliminate the co-rotational time flux $\mathcal{N}$ in the original form of $\bm{\sigma}$ and deal with \eqref{sigma_intro} instead, since it will bring us some convenience when deriving energy estimates for the difference of solutions (cf. \eqref{energy-identity}). This strategy also works for the general case \eqref{main_system_intro_2d}.

It is easy to verify that the particular choice of Leslie's coefficients in \eqref{choicea} fulfills the general assumptions given by \eqref{coeass}. In particular, it does not satisfy Parodi's relation $\mu_2+\mu_3 = \mu_6-\mu_5$, and the nonlinear term of the highest order associated to $\mu_1$ is kept. Thus, our main result (Theorem \ref{thm:uniqueness}) indicates that Parodi's relation is not a necessary condition for the uniqueness of global weak solutions. We also note that, to some extent, working with a planar director field $\mathbf{d}$ can reduce the complexity of subsequent arguments, since we do not need to treat the extra equation for the third component $d_3$ that indeed has a simpler structure. Nevertheless, this simplification does not avoid any essential difficulty in the subsequent  mathematical analysis.

\subsection{Statement of the main result}
We now introduce the definition of global weak solutions to the two dimensional Ericksen-Leslie system \eqref{main_system}, as an analogue of the three dimensional case studied in \cite{MR3045633}.
\begin{definition}[Global weak solutions]\label{def:weak-sol} 
Let $T>0$. A pair $(\uu,\,\ddd)$ is called a weak solution to the Ericksen-Leslie system \eqref{main_system} with initial data $(\uu_0,\,\ddd_0)\in L^2(\TT^2)\times H^1(\TT^2)$ satisfying $\mathrm{div}\uu_0=0$, if it satisfies
\begin{equation}\label{def:weak-sol-energy-space}
\begin{alignedat}{8}
	\uu 	&\in L^\infty(0,T; L^2(\TT^2))		&&\cap 	L^2(0,T; H^1(\TT^2))  	\quad\quad\text{with}\quad\quad  &&&&\partial_t \uu \in L^2(0,T;W^{-1,\frac{2}{1+\zeta}}(\TT^2)),\\
	\ddd	&\in L^\infty(0,T; H^1(\TT^2))		&&\cap L^2(0,T;  H^2(\TT^2))	\quad\quad\text{with}\quad\quad  &&&&\partial_t \ddd \in L^2(0,T;L^{\frac{2}{1+\zeta'}}(\TT^2)),
\end{alignedat}
\end{equation}
for any $\zeta, \zeta'\in (0,1)$,
as well as 
\begin{equation*}
	\ddd\cdot A\ddd \in L^2((0,T)\times \TT^2) \quad \text{and}\quad  A\ddd \in L^2((0,T)\times\TT^2).
\end{equation*}
Furthermore, the following equalities are satisfied 
\begin{equation*}
	\int_{\TT^2} \uu(t,x)\cdot \nabla \varphi(x)\,\dd x \,=\,0 \quad\quad\text{for almost all}\  t\in (0,T)\  \text{and any}\ \varphi \in H^1(\TT^2)
\end{equation*}
and
\begin{align*}
	& \left\langle \,\partial_t \uu(t, x),\, {\rm \bf v}(x) \right\rangle_{W^{-1, \frac{2}{1+\zeta}},\, W^{1,\frac{2}{1-\zeta}}} -\int_{\TT^2} \big[\uu(t, x) \otimes \uu(t, x)\big] :\nabla {\rm \bf v}(x)\,\dd x \\
	&\qquad =\,
	\int_{\TT^2} \big[ \nabla \ddd(t, x)\odot\nabla \ddd(t,x) \big] : \nabla  {\rm \bf v}(x)\,\dd x -\int_{\TT^2} \bm{\sigma}(t, x) :\nabla  {\rm \bf v}(x)\,\dd x,
\end{align*}
for almost all $t\in(0,T)$ and any ${\rm \bf v}\in W^{1,\frac{2}{1-\zeta}}(\TT^2)$ with $\Div\, {\rm \bf v} \,=\,0$,  
$\bm{\sigma}$ satisfies \eqref{sigma_intro} in $L^2(0,T;L^\frac{2}{1+\zeta}(\mathbb{T}^2))$,
and the equation for the director field $\ddd(t,x)\in \RR^2$ is satisfied in a stronger sense such that 
\begin{equation*}
	\partial_t \ddd	\,+\,\uu\cdot \nabla\ddd \,
		-\frac32(\nabla \uu)\ddd- \frac12(\nabla^{\mathrm{tr}} \uu) \ddd
	\,=\,\Delta\ddd\,-\,\nabla_\ddd W(\ddd) 
	\quad\quad\text{a.e. in}\quad (0,T)\times\TT^2.
\end{equation*}
Besides, the initial conditions are satisfied such that $\uu(0,\cdot ) = \uu_0$ and $\ddd(0,\cdot) = \ddd_0$ a.e. in $\TT^2$.
Finally, the following energy inequality holds
	\begin{equation}\label{def:energy-inequality-tot}
	\begin{aligned}
		\mathcal{E}_{\mathrm{total}}(t) +
		\int_0^t \mathcal{D}_{\mathrm{total}}(s) \,\dd s
		\,\leq\,
		 \mathcal{E}_{\mathrm{total}}(0),
	\end{aligned}
	\end{equation}
for any $t\in [0,T]$, where $\mathcal{E}_{\mathrm{total}}(t)$ is given by \eqref{total} with $\Omega=\mathbb{T}^2$, and $$ 
\mathcal{D}_{\mathrm{total}}(t)
=\int_{\TT^2}
		\Big(
		 \nu |\nabla \uu|^2+ |\ddd\cdot A\ddd |^2 +
		\frac32 |A\ddd|^2 	+ \frac12 |\,\Delta \ddd \, -\, \nabla_\ddd W(\ddd)\,|^2
		+ \frac12 |A\ddd + (\Delta \ddd \, -\, \nabla_\ddd W(\ddd))|^2
		\Big)(t,x)\,
		\dd x.
$$
\end{definition}
\noindent 
We note that the inequality \eqref{def:energy-inequality-tot} coincides with the basic energy law \eqref{BELnoPa} (with now the less than or equal sign for weak solutions), by recasting the co-rotational time flux such that $\mathcal{N} = 2A\ddd + \Delta \ddd - \nabla_{\ddd} W(\ddd)$.

Next, we present a theorem on the existence of global weak solutions to system \eqref{main_system} in $\mathbb{T}^2$. Its proof will be omitted here, since it essentially follows a similar argument as the one used in \cite{MR3045633}, in which the authors dealt with the more challenging case with a general three dimensional smooth bounded domain. 
\begin{theorem}[Existence of weak solutions]
\label{thm:existence}
	Assume that the initial data satisfy the following conditions
	\begin{equation}\label{initial-condition}
		(\uu_0,\ddd_0)\in L^2(\mathbb{T}^2)\times H^1(\mathbb{T}^2)\quad\text{with}\quad\Div\,\uu_0\,=\,0.
	\end{equation}	
	Then for any $T>0$, there exists a global weak solution $(\uu, \ddd)$ to the Ericksen-Leslie system \eqref{main_system} on $[0,T]$ in the sense of Definition \ref{def:weak-sol}. 
\end{theorem} 

Now we are in a position to state the main result of this paper, that is, the global weak solutions built in Theorem \ref{thm:existence} are indeed unique. 
\begin{theorem}[Uniqueness of weak solutions]\label{thm:uniqueness}
	Assume that the assumptions in Theorem \ref{thm:existence} are satisfied. Let $(\uu_1,\ddd_1)$ and $(\uu_2,\ddd_2)$ be two global weak solutions to the Ericksen-Leslie system \eqref{main_system} determined by Theorem \ref{thm:existence}, subject to the same initial data $(\uu_0,\ddd_0)$ and defined on $[0,T]$. Then we have 
	$(\uu_1(t),\ddd_1(t))=(\uu_2(t),\ddd_2(t))$  for all $t\in [0,T]$.
\end{theorem}
\begin{remark}
(1) The existence and uniqueness results stated in Theorems \ref{thm:existence},  \ref{thm:uniqueness} can be easily extended to the more general setting as in \eqref{main_system_intro_2d} (subject to the structural conditions \eqref{coeass}) with only minor modifications in the corresponding proofs. (2) Several interesting issues for the general system \eqref{main_system_intro_2d} still remain open, for instance, the regularity of global weak solutions, and the existence of global strong solutions with arbitrary large initial data $(\uu_0,\ddd_0)$ in certain more regular spaces such as $H^1(\mathbb{T}^2)\times H^2(\mathbb{T}^2)$.
\end{remark}

Before ending this subsection, we make a  comment on the uniqueness result obtained in Theorem \ref{thm:uniqueness}. Let us focus on the velocity field $\uu$ and recall that a typical approach to address the uniqueness of weak solutions (\`a la Leray) for the Navier-Stokes equations is given by the so-called Serrin's criterion. More precisely, in any dimension $d\geq 2$, under the assumption that the velocity field $\uu(t,x)$ belongs to $
    \uu \in L^p(0,T; L^q(\mathbb{T}^d))$ with $1/p = 1/2- d/(2q)$, $p\in (2,\infty)$, $q\in (d, \infty)$, one then recovers the following continuity in time for the velocity field $\uu\in C([0,T]; L^2(\mathbb{T}^d))$. This relation develops further into the strong energy equality
    \begin{equation*}
        \| \uu(t) \|_{L^2}^2 + 2\nu\int_\tau^t \| \nabla \uu(s) \|_{L^2}^2\dd s = \| \uu(\tau) \|_{L^2}^2,
        \qquad 0\leq \tau \leq t <T,
    \end{equation*}
    and eventually leads to the uniqueness of these weak solutions. In particular when $d=2$, since Leray's solutions fulfill $\uu\in L^4((0,T)\times \mathbb{T}^2))$, then they satisfy the continuity property $\uu\in C([0,T]; L^2(\mathbb{T}^2))$ and thus are unique. However, the above mentioned relation is no longer valid when we deal with the Ericksen-Leslie system \eqref{main_system} (or \eqref{main_system_intro_2d}). In general, we are not able to obtain the continuity property of $\uu$ in $C([0,T]; L^2(\mathbb{T}^2))$, which makes our problem somehow non trivial. Indeed some energy of the system can be lost during the evolution of the flow, being transferred to high modes of the velocity field $\uu$. 
    The main reason for this lack of regularity should be related to some components of the Leslie stress tensor $\bm{\sigma}$ in \eqref{sigma_intro}, for instance, the stress tensor $\mu_1 \ddd\otimes \ddd (\ddd\cdot (A\ddd))$. This highly nonlinear term does not belong to $L^2((0,T)\times \mathbb{T}^2)$, because of the lack of a maximum principle for the director $\ddd$ under stretching effects induced by the fluid. Indeed, we can deduce that $\mu_1 \ddd\otimes \ddd (\ddd\cdot (A\ddd))\in L^2(0,T; L^p(\mathbb{T}^2))$ for any $p\in[1,2)$, which eventually yields $\uu \in C([0,T]; L^p(\mathbb{T}^2))$, but never reaching the threshold $p=2$. Therefore, in the two dimensional case, the energy inequality \eqref{def:energy-inequality-tot} for weak solutions can not be improved into a strong energy identity.
    
    On the other hand, we point out that the analysis of whether $\uu$ belongs $C([0,T]; L^2(\mathbb{T}^2))$ or not, relates more to the regularity property of weak solutions, which does not automatically affect its uniqueness. An analogous situation was encountered for instance in \cite{hong,linlinwang}, where the authors constructed global weak solutions for a simplified liquid crystal system that consists of the Navier-Stokes equations and the transported heat flow of harmonic map into sphere, i.e., with the unitary constraint $|\ddd(t,x)|=1$. In that system, certain different mechanism, i.e., a lack of regularity on the director $\ddd(t,x)$ leads to possible creations of singularity on the weak solutions for a finite amount of times. One can check that in the results therein, the velocity $\uu$ does not belong a-priori to  $C([0,T]; L^2(\Omega))$ (with either $\Omega\subset \mathbb{R}^2$ being  a smooth bounded domain or $\Omega =\mathbb{R}^2$), nevertheless, the uniqueness could be obtained, see \cite{LinWang2010} (see also \cite{xuzhang} for a different proof based on the Littlewood-Paley analysis). We refer to \cite{MR3460623,WangWangZhang} for further results on the uniqueness of global weak solutions for the general {E}ricksen-{L}eslie system with the (anisotropic) Oseen-Frank energy and unitary constraint $|\ddd(t,x)|=1$ in the whole space $\mathbb{R}^2$. 

\subsection{Proof of Theorem \ref{thm:uniqueness}: outlines and key idea.} 
In what follows, we present the main ingredients of the proof for Theorem \ref{thm:uniqueness} and postpone the detailed estimates that are rather involved to subsequent sections. 

Let $(\uu_1(t,x),\,\ddd_1(t,x))$ and $(\uu_2(t,x),\,\ddd_2(t,x))$ be two global weak solutions to problem \eqref{main_system} on $[0,T]$ determined by Theorem \ref{thm:existence} subject to the same initial data $(\uu_0(x),\ddd_0(x))$. For convenience, we introduce some notations related to the differences of these two solutions
\begin{align*}
 & \delta \uu =\uu_1\,-\,\uu_2,\quad \delta \ddd \,=\,\ddd_1\,-\,\ddd_2,\\ &\delta A=A_1-A_2=\frac{1}{2}(\nabla \uu_1+\nabla^{\mathrm{tr}} \uu_1)-\frac{1}{2}(\nabla \uu_2+\nabla^{\mathrm{tr}} \uu_2),\\
 &\delta \omega= \omega_1-\omega_2= \frac{1}{2}(\nabla \uu_1-\nabla^{\mathrm{tr}} \uu_1)-\frac{1}{2}(\nabla \uu_2-\nabla^{\mathrm{tr}} \uu_2).
\end{align*}
%
%
Instead of working in the natural energy spaces for problem \eqref{main_system} (recall \eqref{def:weak-sol-energy-space}), we shall perform estimates of  $(\delta\uu,\delta\ddd)$ in the following lower-order function spaces
\begin{equation}
\label{space-where-deltauanddeltad-are-defined}
	\begin{cases}
		\delta \uu\,\in\,C_b([0,T];\,H^{-\frac{1}{2}}(\TT^2)),\ \  
		\nabla \delta \uu\in 
		L^2(0,T;\,  H^{-\frac{1}{2}}(\TT^2)),\\ 
		\delta \ddd\,\in\,C_b([0,T];\,H^{\frac{1}{2}}(\TT^2)),\ \ 
		\nabla \delta \ddd \in 
		 L^2(0,T;\,H^{\frac{1}{2}}(\TT^2)),
	\end{cases}
\end{equation} 
where $H^{s}(\TT^2)$ stands for the usual Sobolev space with index of regularity $s\in\mathbb{R}$ (cf. Remark \ref{equinorm}). The above regularity properties and especially the continuity in time are well defined thanks to \eqref{def:weak-sol-energy-space}, since these function spaces arise as interpolations of $\delta \uu \in H^1([0,T]; W^{-1,p}(\mathbb{T}^2))\cap L^\infty(0,T;L^2(\mathbb{T}^2))$ and $\delta \ddd\in 
H^1([0,T];L^{p}(\mathbb{T}^2))\cap L^\infty(0,T;H^{1}(\mathbb{T}^2))$, with $1\leq p< 2$.

As a key ingredient to control the difference of weak solutions $(\delta \uu,\,\delta \ddd)$ within the framework given by \eqref{space-where-deltauanddeltad-are-defined}, we introduce a function reflecting its norm and another one related to the rate of energy dissipation:
\begin{equation}\label{def-Phi-D}
\begin{cases}
\begin{alignedat}{4}
	\,\Phi(t)
	&:=
	\frac{1}{2}
	\left(
	\|\delta \uu(t)\|_{H^{-\frac{1}{2}}}^2+\|\delta \ddd(t)\|_{H^{\frac{1}{2}}}^2
	\right)
	\ \in C([0,T])\quad \text{with}\quad \Phi(0)=0, \\
	\,\mathfrak{D}(t) 
	&:= 
	\nu \| \nabla \delta \uu (t) \|_{H^{-\frac{1}{2}}}^2 	+
	 \| \nabla \delta \ddd(t) \|_{H^\frac{1}{2}}^2		
	+ 2
	\sum_{q=-1}^\infty
	\!2^{-q}\! 
	\int_{\TT^2}\!
	|\Dd_q \delta A S_{q-1} \ddd_1 |^2(t,x)\,\dd x	\\
	&\quad\ 
	+
	\sum_{q=-1}^\infty
	\!2^{-q}\!
	\int_{\TT^2}\!
	|\Dd_q \delta A:S_{q-1} (\ddd_1\otimes \ddd_1) |^2(t,x)\,\dd x\geq 0.
\end{alignedat}
\end{cases}
\end{equation}
In the definition of $\mathfrak{D}(t)$, we have used some notations of the Littlewood-Paley theory. More precisely, $\Dd_q$ and $\Sd_{q-1}$ stand for the dyadic block operator and the low-frequency cut-off operator, respectively. A brief introduction of these operators will be given in Section  \ref{sec:Littlewood-Paley}. 

The proof of Theorem \ref{thm:uniqueness} relies on a suitable control of the function $\Phi(t)$ through an Osgood-type inequality.
To this end, we recall the so-called Osgood condition (see \cite[Chapter 3, Definition 3.1]{BCD}): 
\begin{definition}[Osgood condition] 
   Let $\mu:[0,\infty)\to [0,\infty)$ be a nondecreasing nonzero continuous  function such that $\mu(0) = 0$. We say that $\mu$ is an Osgood modulus of continuity, if 
   \begin{equation*}
       \int_0^1 \frac{\dd r}{\mu(r)} = \infty.
   \end{equation*}
\end{definition}
\noindent 
Then we state the following version of the Osgood lemma (see, e.g., \cite[Chapter 3, Lemma 3.4]{BCD}), which is a generalisation of the standard Gronwall's inequality.
\begin{lemma}[Osgood lemma]\label{lemma:osgood}
  Let $\Phi:[0,T]\to [0,\infty)$ be a measurable function and $F$ be a locally integrable function from $[0,T]$ to $[0,\infty)$. Assume that there exists an Osgood modulus of continuity $\mu : [0,\infty)\to [0,\infty)$ such that
  \begin{equation*}
      \Phi(t) \leq \int_0^t F(s)\mu(\Phi(s))\,\dd s,\qquad 
      \text{for almost all }t\in [0,T],
  \end{equation*}
  then $\Phi=0$ a.e. in $[0,T]$.
\end{lemma}

Returning to our uniqueness problem, we aim to show that the continuous function $\Phi(t)$ defined in \eqref{def-Phi-D} satisfies the following integral inequality: 
\begin{equation}\label{main-ineq}
	\Phi(t) + \gamma \int_0^t \mathfrak{D}(s)\,\dd s 
	\leq \int_0^t F(s)\mu(\Phi(s))\,\dd s,
\end{equation}
for some constant $\gamma\in (0,1)$, where the function $\mu:[0,\infty)\to [0,\infty)$ stands for the following Osgood modulus of continuity with a double logarithmic structure
\begin{equation}\label{muaa}
	\mu(r):=r \Big(1+\ln\Big(1+\frac{1}{r}\Big)\Big)\Big(1+\ln\Big(1+\ln\Big(1+\frac{1}{r}\Big)\Big)\Big),\qquad r\geq 0,
\end{equation}
and $F \in L^1(0,T)$ is a suitable non-negative integrable function that depends on both weak solutions $(\uu_1,\,\ddd_2)$ and $(\uu_2,\,\ddd_2)$. 
Once the inequality \eqref{main-ineq} can be achieved, as an immediate consequence of Lemma \ref{lemma:osgood}, we can deduce that the continuous function $\Phi(t)$ satisfies 
$$
\Phi(t)\equiv 0,\quad \forall\, t\in [0,T],
$$ 
 which yields that the two given global weak solutions $(\uu_1,\ddd_1)$,  $(\uu_2,\ddd_2)$ must coincide in $[0,T]\times \mathbb{T}^2$, and thus completes the proof of Theorem \ref{thm:uniqueness}.

Therefore, the remaining task is to derive the inequality \eqref{main-ineq}. To achieve this goal, we first observe that the difference between two weak solutions $(\delta \uu,\,\delta \ddd)$ satisfies the following system (understood in a weak sense, cf. Definition \ref{def:weak-sol}):
\begin{equation}\label{main_system2}
\left\{\hspace{0.2cm}
	\begin{alignedat}{2}
		&\,\partial_t\delta \uu + \delta \uu\cdot \nabla \uu_1+\uu_2\cdot \nabla \delta \uu -\nu\Delta\delta \uu+\nabla \delta  \pre   \\
		&\quad =
			-\Div\, 
			\big(\nabla \ddd_1\odot\nabla \delta  \ddd\big)
			-\,\Div\, 
			\big(\nabla \delta \ddd\odot\nabla \ddd_2\big)
			+
			\Div\,\delta \bm{\sigma}
		\hspace{0.4cm}
																		&& (0,T)\times\mathbb{T}^2, \vspace{0.1cm}\\
		&\,\Div\,\delta \uu\,=\,0		
																		&&(0,T)\times \mathbb{T}^2, \vspace{0.1cm}\\
		&\,	
		\partial_t \delta \ddd	\,+\,\uu_1\cdot \nabla\delta \ddd+\,\delta \uu\cdot \nabla\ddd_2 
		-\frac32(\nabla \delta \uu)\ddd_1
		-\frac32(\nabla \uu_2)\delta\ddd
		-\frac12(\nabla^{\mathrm{tr}} \delta\uu) \ddd_1
		-\frac12(\nabla^{\mathrm{tr}} \uu_2) \delta \ddd
		\\
		& \quad =\,
		\Delta\delta \ddd\,-\,\delta  \nabla_\ddd W(\ddd)\quad 
		&&(0,T)\times \mathbb{T}^2, \vspace{0.1cm}\\			
		&\,	
		(\delta \uu,\,\delta \ddd)_{|t=0}
		\,=\,
		(\mathbf{0},\,\mathbf{0})
		&&\hspace{1.27cm}\mathbb{T}^2, \vspace{0.1cm}
	\end{alignedat}
	\right.
\end{equation}
where the notation $\delta \nabla_\ddd W(\ddd)$ stands for the difference of potentials 
\begin{equation*}
	\delta \nabla_\ddd W(\ddd):=
	\nabla_\ddd W(\ddd_1)- \nabla_\ddd W(\ddd_2)
	=
	(|\ddd_1|^2-1)\delta \ddd
	\,+\,
	(\delta\ddd\cdot (\ddd_1+\ddd_2)) \ddd_2,
\end{equation*}
while the difference of stress tensors denoted by  $\delta \bm{\sigma} :=\bm{\sigma}(\uu_1,\,\ddd_1) - \bm{\sigma}(\uu_2,\,\ddd_2)$ satisfies
\begin{equation*}
\begin{aligned}
	\delta \bm{\sigma} = &\, 
	\delta \bm{\sigma}_1 
	+ \ddd_1 \otimes (\delta A \ddd_1) 
	+ \ddd_1 \otimes (A_2 \delta \ddd) 
	+ \delta \ddd\otimes (A_2 \ddd_2)  
	+ (\delta A\ddd_1)\otimes\ddd_1
	+ (A_2\delta \ddd)\otimes\ddd_1
	+ (A_2\ddd_2)\otimes\delta \ddd
	\\ 
	&-(\Delta \delta \ddd - \delta  \nabla_\ddd W(\ddd))\otimes \ddd_1 -
	(\Delta \ddd_2 - \nabla_\ddd W(\ddd_2))\otimes \delta \ddd 
\end{aligned}
\end{equation*}
such that $\delta \bm{\sigma}_1$ denotes the most challenging term to handle: 
\begin{equation}\label{delta-sigma1}
\begin{aligned}
	\delta \bm{\sigma}_1 = & \, (\ddd_1 \cdot (A_1\ddd_1))\ddd_1\otimes\ddd_1
	\,-\,(\ddd_2 \cdot (A\ddd_2))\ddd_2\otimes\ddd_2\\
	=&\,
	(\ddd_1 \cdot (\delta A\ddd_1))\ddd_1\otimes\ddd_1\,+\,
	(\ddd_1 \cdot (A_2\delta \ddd))\ddd_1\otimes\ddd_1\,+\,
	(\delta \ddd \cdot (A_2\ddd_2))\ddd_1\otimes\ddd_1\\
	&\,+ 
	(\ddd_2 \cdot (A_2\ddd_2))\delta \ddd\otimes\ddd_1\,+\,
	(\ddd_2 \cdot (A_2\ddd_2))\ddd_2\otimes \delta \ddd.
\end{aligned}
\end{equation}
In view of \eqref{def-Phi-D} and \eqref{main-ineq}, we shall estimate $(\delta \uu,\,\delta \ddd)$ within the framework  $H^{-1/2}(\TT^2)\times H^{1/2}(\TT^2)$. Under the regularity properties of weak solutions given by Definition \ref{def:weak-sol}, the following integral identity for $(\delta \uu,\,\delta \ddd)$ can be derived 
\begin{equation}\label{energy-identity}
\begin{aligned}
	&\frac{1}{2}
	\bigg(
	\|\delta \uu(t)	\|_{H^{-\frac{1}{2}}}^2
	+ \|\delta \ddd(t)\|_{L^2}^2+ 
	\| \nabla \delta \ddd(t)\|_{H^{-\frac{1}{2}}}^2
	\bigg)
	+
	\nu
	\int_0^t
	\|\nabla \delta \uu(s)\|_{H^{-\frac{1}{2}}}^2
	\, \dd s
	+
	\int_0^t
	\big(\| \nabla \delta \ddd(s)\|_{L^2}^2 +\|\Delta \delta \ddd(s)\|_{H^{-\frac{1}{2}}}^2\big)
	\,\dd s
	 \\
	&\qquad =
	\int_0^t
	\big(
	\mathcal{T}_1(s)
	+
	\mathcal{T}_2(s)
	+
	\mathcal{T}_3(s)
	+
	\mathcal{T}_4(s)
	\big)
	\,\dd s,\quad \text{for any}\ t\in (0,T).
\end{aligned}
\end{equation}
For the sake of convenience, here we group the reminder terms on the right-hand side of \eqref{energy-identity} into four terms given by 
\begin{align}
	\mathcal{T}_1	: =\ 
	& -
	  \langle 	 \delta \uu\cdot \nabla \uu_1,\, \delta \uu 	\rangle_{H^{-\frac{1}{2}}}
	-
	  \langle 	 \uu_2\cdot \nabla \delta  \uu,\, \delta \uu 	\rangle_{H^{-\frac{1}{2}}}		
	+
	  \langle \nabla \ddd_1\odot \nabla\delta \ddd ,\,\nabla \delta \uu 	\rangle_{H^{-\frac{1}{2}}}
	+ \langle \nabla \delta \ddd\odot \nabla \ddd_2 ,\,\nabla \delta \uu 	\rangle_{H^{-\frac{1}{2}}}	\notag
	\\
	& - 
	  \langle	\delta \ddd\otimes (A_2 \ddd_2),\,
	  \nabla \delta \uu \rangle_{H^{-\frac{1}{2}}} 
	  - \langle	 (A_2 \ddd_2)\otimes \delta \ddd,\,
		\nabla \delta \uu \rangle_{H^{-\frac{1}{2}}} 
	  - \langle \delta  \nabla_\ddd W(\ddd))\otimes \ddd_1,\,
		\nabla \delta \uu \rangle_{H^{-\frac{1}{2}}}\notag\\
	& + \langle
		(\Delta \ddd_2 - \nabla_\ddd W(\ddd_2))\otimes \delta \ddd,\,
		\nabla \delta \uu
	    \rangle_{H^{-\frac{1}{2}}}
	  - \int_{\TT^2} 
	  (\uu_1\cdot \nabla\delta \ddd)\cdot \delta \ddd\,\dd x
	  +
	  \langle \uu_1\cdot \nabla\delta \ddd ,\,\Delta \delta \ddd 	\rangle_{H^{-\frac{1}{2}}}	\notag		
	\\
	& -
	  \int_{\TT^2} 
	  (\delta \uu\cdot \nabla  \ddd_2)\cdot \delta \ddd\,\dd x
	  +
	  \langle \delta \uu\cdot \nabla  \ddd_2 ,\,\Delta 
	  \delta \ddd 	\rangle_{H^{-\frac{1}{2}}}	
	  - 
	  \int_{\TT^2} 
	  \delta \nabla_\ddd W(\ddd)\cdot \delta \ddd 	\,\dd x	
	  +
	  \langle \delta \nabla_\ddd W(\ddd),\,\Delta \delta \ddd 	\rangle_{H^{-\frac{1}{2}}},  
	    \label{Xi1} \smallskip \\
	\mathcal{T}_2 := \  
	& \frac32  \int_{\TT^2}
			((\nabla \delta \uu)\ddd_1)\cdot \delta\ddd  \,\dd x
			- 
			\frac32 \langle (\nabla \delta \uu)\ddd_1,\, \Delta \delta \ddd\rangle_{H^{-\frac{1}{2}}} 
			+ 
			\frac32 \int_{\TT^2}( (\nabla  \uu_2)\delta \ddd) \cdot \delta \ddd\,\dd x
			-
			\frac32 \langle  (\nabla  \uu_2)\delta \ddd,\, \Delta \delta \ddd\rangle_{H^{-\frac{1}{2}}} \notag \\
	& +\frac12  \int_{\TT^2}
		((\nabla^{\mathrm{tr}} \delta \uu)\ddd_1)\cdot \delta\ddd  \,\dd x
			- 
			\frac12 \langle (\nabla^{\mathrm{tr}} \delta \uu)\ddd_1,\, \Delta \delta \ddd\rangle_{H^{-\frac{1}{2}}} 
			+ 
			\frac12 \int_{\TT^2}( (\nabla^{\mathrm{tr}}  \uu_2)\delta \ddd) \cdot \delta \ddd\,\dd x
			 \notag
			\\
			& - 
			\frac12 \langle  (\nabla^{\mathrm{tr}}  \uu_2)\delta \ddd,\, \Delta \delta \ddd\rangle_{H^{-\frac{1}{2}}}
			+
			\langle  \Delta \delta \ddd \otimes \ddd_1,\,\nabla\delta \uu 	\rangle_{H^{-\frac{1}{2}}},    
			\label{Xi2}  \smallskip \\
	\mathcal{T}_3 := \ 
	& -	\langle
		\ddd_1 \otimes (\delta A \ddd_1),\,
		\nabla \delta \uu \rangle_{H^{-\frac{1}{2}}} 
	  - \langle 
		\ddd_1 \otimes (A_2 \delta \ddd),\, \nabla \delta \uu
    	\rangle_{H^{-\frac{1}{2}}}
      - \langle (\delta A \ddd_1)\otimes\ddd_1,\, \nabla \delta \uu
	    \rangle_{H^{-\frac{1}{2}}} 
	 \notag \\
	& - \langle 
		(A_2 \delta \ddd)\otimes \ddd_1,\, 	\nabla \delta \uu
	    \rangle_{H^{-\frac{1}{2}}},
	    \label{Xi3}  \smallskip \\
   \mathcal{T}_4 := \ 
   & - \langle 	\delta \bm{\sigma}_1,\,\nabla\delta \uu 
        \rangle_{H^{-\frac{1}{2}}},
   \label{Xi4}
\end{align}
where $\langle\cdot, \cdot\rangle_{H^{-1/2}}$ denotes the inner product in $H^{-1/2}(\mathbb{T}^2)$. 
At least formally, the identity \eqref{energy-identity} can be derived by summing the $H^{-1/2}$-inner product between $\delta \uu$ and the momentum equation for $\delta\uu$ in \eqref{main_system2}, the $L^2$-inner product between $\delta \ddd$ and the director equation for $\delta \ddd$, together with the $H^{-1/2}$-inner product between $-\Delta \delta \ddd$ and the director equation for $\delta \ddd$, and then integrating in time. Detailed calculations that lead to \eqref{energy-identity} will be performed in Appendix \ref{sec:integral-identity}. \smallskip

The following result provides estimates for each term on the right-hand side of \eqref{energy-identity}. 
\begin{prop}\label{thm:ineq}
   Denote $\widehat{\mathfrak{D}}(t)=\nu \| \nabla \delta \uu (t) \|_{H^{-\frac{1}{2}}}^2	+
	 \| \nabla \delta \ddd(t)\|_{L^2}^2 +\|\Delta \delta \ddd(t)\|_{H^{-\frac{1}{2}}}^2$.
    Then there exist four nonnegative functions $f_1$, $f_2$, $f_3$ and $f_4$ belonging to $L^1(0,T)$ that depend on the weak solutions $(\uu_1,\,\ddd_1)$ and $(\uu_2,\,\ddd_2)$ such that for any $\eta\in (0,1)$, we have 
		\begin{align*}
					& |\mathcal{T}_1(t)| \leq 
					C_\eta f_1(t) \mu(\Phi(t)) +15 \eta \widehat{\mathfrak{D}}(t),
					\\
					&|\mathcal{T}_2(t)|
					\leq C_\eta  f_2(t) \mu(\Phi(t))
					+
					5 \eta  \widehat{\mathfrak{D}}(t)+
					 \sum_{q=-1}^\infty 2^{-q} \int_{\TT^2}|\Dd_q \delta A\sdm{q}\ddd_1|^2(t,x)\,\dd x
		           +\frac14\|\Delta \delta \ddd(t)\|_{H^{-\frac12}}^2,
	             	\\
					&\Big|
 						\mathcal{T}_3(t)+
 					 2 \sum_{q=-1}^\infty 2^{-q}\int_{\TT^2}|\Dd_q\delta A \sdm{q}\ddd_1|^2(t,x)\, \dd x	
 					\Big|
					\leq C_\eta f_3(t) \mu(\Phi(t))
					+  30\eta \widehat{\mathfrak{D}}(t) + 
	10 \eta \sum_{q=-1}^\infty
	2^{-q}
	\int_{\TT^2}
	|\Dd_q \delta A\sdm{q}\ddd_1|^2\,\dd x,
	\\
					& \Big|
					\mathcal{T}_4(t) +
					\sum_{q=-1}^\infty 2^{-q}\int_{\TT^2}|\Dd_q\delta A : \sdm{q}(\ddd_1\otimes\ddd_1)|^2
					(t,x)\,\dd x 
					\Big|\\
					&\qquad 
					\leq  C_\eta  f_4(t)\mu( \Phi(t) )+  15\eta \widehat{\mathfrak{D}}(t) + 
	15 \eta \sum_{q=-1}^\infty 2^{-q}\int_{\TT^2}|\Dd_q\delta A : \sdm{q}(\ddd_1\otimes\ddd_1)|^2
					(t,x)\,\dd x,
		\end{align*}
	for almost all $t\in (0,T)$, where $\Phi(t)$ is given by \eqref{def-Phi-D}, $\mu$ takes the form of \eqref{muaa}, and $C_\eta$ is a positive constant depending on $\eta$, $L^2\times H^1$-norm of the initial data $(\uu_0,\ddd_0)$ and coefficients of the system.
\end{prop}
\smallskip \noindent 
Since the proof of Proposition \ref{thm:ineq} is rather involved, we postpone the details to Section \ref{sec:estT1} -- Section \ref{sec:double-log-est}.

Now we are in a position to finish the proof of Theorem \ref{thm:uniqueness}. Keeping in mind the equivalent relations 
$\|\delta \ddd\|_{L^2}^2+ 
	\| \nabla \delta \ddd\|_{H^{-1/2}}^2 \approx
	\|\delta \ddd\|_{H^{1/2}}^2
	$, $\| \nabla \delta \ddd\|_{L^2}^2 +\|\Delta \delta \ddd\|_{H^{-1/2}}^2\approx \| \nabla \delta \ddd\|_{H^{1/2}}^2$ (see Remark \ref{equinorm} below), 
	from the estimates for $\mathcal{T}_i$, $i=1,2,3,4$ obtained in Proposition \ref{thm:ineq} together with the integral identity \eqref{energy-identity} and the definition \eqref{def-Phi-D}, we deduce that for any given $\eta\in (0,1)$, it holds
\begin{equation}\label{double-log}
	\Phi(t) + \frac13 \int_0^t \mathfrak{D}(s)\,\dd s \leq  C_\eta \int_0^t  F(s)\mu(\Phi(s))\,\dd s + 
	100 \eta \int_0^t \mathfrak{D}(s)\,\dd s,
\end{equation}
for all $t\in [0,T]$, with $F(t) = f_1(t) + f_2(t)+f_3(t)+f_4(t) \in L^1(0,T)$. 
Choosing $\eta$ sufficiently small, e.g., $\eta=1/600$, we are able to absorb the last term  on the right-hand side of \eqref{double-log} that involves the dissipation function $\mathfrak{D}$ by the analogous term on the left-hand side. This immediately yields the desired inequality \eqref{main-ineq} with $\gamma=1/6$. As a consequence, we arrive at the conclusion of Theorem \ref{thm:uniqueness}.  \hfill $\square$
\medskip 

We conclude this section with some further comments on our uniqueness result and its proof.

\begin{remark}
    The argument described above does not imply the uniqueness of 
    weak solutions that exist in the lower order class $\uu\in L^\infty(0,T; H^{-1/2}(\mathbb{T}^2))$ and $\ddd \in L^\infty (0,T; H^{1/2}(\mathbb{T}^2))$. Indeed, the function $F:[0,T]\to \mathbb{R}$ that appears in 
    \eqref{double-log} depends on the norms of weak solutions in the function spaces given by Definition \ref{def:weak-sol}, namely, $\|\uu_i\|_{L^\infty(0,T;L^2(\mathbb{T}^2))}$, $\|\nabla  \uu_i \|_{L^2((0,T)\times \mathbb{T}^2)}$, $\|\ddd_i\|_{L^\infty(0,T;H^1(\mathbb{T}^2))}$ and $\|\Delta  \ddd_i \|_{L^2((0,T)\times \mathbb{T}^2)}$, for $i=1,2$. Hence, our uniqueness result is valid for those global weak solutions ``with finite energy'' as intended in Definition \ref{def:weak-sol}. 
\end{remark}

\begin{remark}
The idea to control the difference of weak solutions via energy estimates in certain function spaces of lower order than the natural energy level has been used in \cite{MR3460623} for the general Ericksen-Leslie system with the unitary constraint $|\ddd|=1$. For the same system, a different approach involving a suitable weaker metric on the difference of weak solutions was presented in \cite{WangWangZhang}, based on the Littlewood-Paley analysis. After exploring the algebraic structure of the system and making full use of the constraint $|\ddd|=1$, the authors of the above mentioned works were able to derive inequalities of Gronwall's type that eventually yield the uniqueness result. However, it is worth mentioning that the situation in our current case is quite different, since the essential mathematical difficulty is now due to the lack of maximum principle for $\ddd$ under the stretching effect from the fluid  (i.e., when $\lambda_2\neq 0$). Without the uniform a-priori estimate of $\|\ddd\|_{L^\infty}$, a standard Gronwall's type inequality for $\Phi(t)$  is no longer available. Indeed, the complexity expressed by the tensor $\delta \bm{\sigma}$ should be recognized as a consequence of those high order  nonlinearities due to the presence of the gradient of fluid velocities $(\nabla \uu_1,\,\nabla \uu_2)$ (here in terms of $(A_1,\,A_2)$) coupled with polynomials up to degree $4$ depending on the directors $(\ddd_1,\,\ddd_2)$. With the unitary constraints $|\ddd_1|=|\ddd_2|=1$, those tensors can be estimated in an $L^2$-setting. However, the Ginzburg-Landau relaxation \eqref{Ginzbourg-Landau-appx} does not allow such uniform $L^\infty$ bounds when the stretching effect on the director field is present.
\end{remark}

\begin{remark} \label{comm_proof}
 Finally, we point out some features of our proof for Theorem \ref{thm:uniqueness}. 
 
 (1) Construction of the inequality \eqref{double-log} and in particular, the Osgood type modulus of continuity $\mu$ as in \eqref{muaa}. We note that logarithmic type inequalities are widely used in the analysis of partial differential equations (cf. e.g., \cite{CRWX,MR2837493, MR3576270, MR3599422} for liquid crystal systems). A well known example related to our problem is given by the Brezis-Gallou\"et inequality (which we apply here to the difference of directors $\delta \ddd(t,\cdot)$ at a fixed time $t\in (0,T)$)
\begin{equation}\label{Brezis-Gall}
    \| \delta \ddd(t,\cdot) \|_{L^\infty(\mathbb{T}^2)}\leq C\| \delta \ddd(t,\cdot) \|_{H^1(\mathbb{T}^2)}
    \bigg( 
    1 + \sqrt{\ln \Big( 1+ \frac{\| \delta \ddd(t,\cdot) \|_{H^s(\mathbb{T}^2)}}{\| \delta \ddd(t,\cdot) \|_{H^1(\mathbb{T}^2)}}\Big)}
    \bigg),
\end{equation}
for some $s>1$. Being sharp for the Sobolev embedding $H^s(\mathbb{T}^2)\hookrightarrow L^\infty(\mathbb{T}^2)$, the Brezis-Gallou\"et inequality does not provide significant benefits to our uniqueness problem, since its right-hand side is not known to be uniformly bounded in $(0,T)$ for weak solutions given by Definition \ref{def:weak-sol}. Nevertheless, what is meaningful for us is the idea  behind the proof of \eqref{Brezis-Gall}: one decomposes the function $\delta \ddd$ into two components $\delta \ddd = \delta \ddd^l+\delta \ddd^h$, where the first term $\delta \ddd^l$ localises the low frequencies of $\delta \ddd$  within a radius $N>0$ (i.e., $\delta \ddd^l=S_N\delta \ddd$ with the notation of Section \ref{sec:Littlewood-Paley}), while the second one $\delta \ddd^h$ controls those high modes. A further development leads to
\begin{equation*}
    \| \delta \ddd^l \|_{L^\infty} \leq C \| \delta \ddd \|_{H^1}\sqrt{N},\quad \text{and}\quad
    \| \delta \ddd^h \|_{L^\infty}\leq 
    C_s \| \delta \ddd \|_{H^s} 2^{-N(s-1)},
\end{equation*}
which eventually imply \eqref{Brezis-Gall} by proportionally defining the radius of localisation $N \propto \ln(1+\| \delta \ddd \|_{H^s}/\| \delta \ddd \|_{H^1})$. Motivated by the above observation, in the forthcoming analysis we shall perform a similar argument when controlling and decomposing (in frequencies) those difficult nonlinearities of \eqref{energy-identity}, relating this time the radius of localisation $N$ with the value of the function $\Phi(t)$ (see for instance, \eqref{def-N-T3}, \eqref{to-cite-N-1} as well as \eqref{def-N-prop-deltasigma1-first} below). This approach, together with a rather involved analysis, finally leads to the required inequality \eqref{double-log} with $\mu$ of a double logarithmic structure.

(2) Importance of the structure:  intrinsic dissipation and specific coupling  (between the fluid velocity and the molecular director). Both of them play an essential role in the proof of uniqueness (for the general system \eqref{main_system_intro_2d}, these properties are presented by the conditions in \eqref{coeass}). We find that in the rate of energy dissipation $\mathfrak{D}(t)$ (cf. \eqref{def-Phi-D}), there are two higher order nonlinear dissipative terms involving the frequency cut-off operator $S_{q-1}$. This hidden dissipative mechanism is crucial for the derivation of \eqref{double-log} (cf. Proposition \ref{thm:ineq}) and can be revealed by a deep decomposition of those nonlinear tensors with symmetric structure in $\delta \bm{\sigma}$. Roughly speaking, the function $\mathfrak{D}(t)$ reflects an analogous energy dissipation like in \eqref{def:energy-inequality-tot}, within lower order Sobolev spaces for the difference of weak solutions. On the other hand, the compatibility conditions $\lambda_1=\mu_2-\mu_3$, $\lambda_2=\mu_5-\mu_6$ guarantee the validity of some crucial estimates for $\mathcal{T}_2$ and $\mathcal{T}_3$, which involve the low frequency part $S_{q-1}\ddd_1$ (see Section \ref{esti-T_2} and Section \ref{sec:single-log-est}).

\end{remark}

\section{Preliminaries: Littlewood-Paley Theory on a Periodic Domain}\label{sec:Littlewood-Paley}
\setcounter{equation}{0}

In this section we recall some basic facts about the Littlewood-Paley theory that will be frequently used in the forthcoming sections. Fourier analysis methods have been of great interest in the analysis of nonlinear partial differential equations. In particular, the Littlewood-Paley theory turns out to be a valuable tool to split rough solutions into sequences of spectrally well-localised smooth functions. We refer the interested reader to \cite{BCD} for a complete and exhaustive reference concerning the main aspects of this theory, when treating functions in the whole space $\mathbb{R}^d$. Below we shall address some results in a periodic setting, in order to deal with solutions defined on $(0,T)\times \mathbb{T}^d$ (in our specific case, $d=2$, we however treat a general dimension $d\geq 1$ in this section). We refer the reader to \cite[Chapter 3]{MR3445609} for some preliminary results about the Littlewood-Paley theory and Besov spaces on the torus $\mathbb{T}^d$. Nevertheless, to the best of our knowledge, there seems no complete investigation available in the literature, which transposes all the results for the Littlewood-Paley in the whole space $\mathbb{R}^d$ to the corresponding ones in the torus $\mathbb{T}^d$. Therefore, for the sake of completeness, we recall here some main aspects in the periodic domain and prove some useful lemmas that will play an essential role in the subsequent analysis.

\subsection{Dyadic decomposition and Besov spaces}
The Littlewood-Paley theory develops around the concept of a dyadic partition of unity. To present this framework, we first recall that every periodic distribution $f\in \mathcal{D}'(\TT^d)$ can be written in terms of its Fourier series:
\begin{equation*}
	f(x) := \sum_{n\in\mathbb{Z}^d} \hat{f}_n e^{\mathrm{i} n\cdot x},\quad \text{where}\quad 
	\hat{f}_n = \frac{1}{(2\pi)^d}\langle f,\,e^{-\mathrm{i} n\cdot x}\rangle_{\mathcal{D}',\, 
	\mathcal{D}}.
\end{equation*}
If $f$ is an integrable function on $\mathbb{T}^d$, then its Fourier coefficients coincide with $\hat{f}_n = \frac{1}{(2\pi)^d}\int_{\TT^d} f(x)\,e^{-\mathrm{i} n\cdot x}\,\dd x$. To set up the dyadic partition, 
we first fix a smooth radial function $\chi$ supported on the ball $B(0,4/3)\subset \RR^d $, equal to $1$ in $B(0,3/4)$ and such that $r\mapsto\chi(r\mathbf{e})$ is nonincreasing over $\R_+$ for all unitary vectors $\mathbf{e}\in\R^d$. Set
$\varphi\left(\xi\right)=\chi\left(\xi/2\right)-\chi\left(\xi\right)$ and
$\varphi_q(\xi):=\varphi(2^{-q}\xi)$ for all $q\geq 0$, $\xi \in \mathbb{R}^d$.
%
Denoting by $\tilde{\NN}= \NN\cup\{-1\}$, the (nonhomogeneous) dyadic blocks $(\Dd_q)_{q\in\tilde \NN}$ are defined by
$$
	\Dd_{\rule[1pt]{3pt}{0.4pt}1} f(x) := \sum_{n\in\mathbb{Z}^d} \chi(n)\hat{f}_n e^{\mathrm{i}n\cdot x}
	\quad \text{and}\quad  
	\Dd_q f(x) \,:=\,\sum_{n\in \ZZ^d} \varphi_q\left(n\right) \hat{f}_n e^{\mathrm{i} n\cdot x}\qquad\text{for}\quad 
	q\in\mathbb{N}.
$$ 
Each operator $\Dd_q$, for $q\in \mathbb{N}$ (resp. $q=-1$) localizes the frequencies of a distribution $f$ within a well defined annulus (resp. a ball), whose radii grow proportionally to $2^q$. Indeed, we observe that in the above definition, there is only a finite contribution to the sum for $|n|=\sqrt{n_1^2+...+n_d^2}\in [3\cdot 2^{q-2},\,2^{q+3}/3]\cap \mathbb{N}$. The dyadic blocks are regularizing operators, since they cut off the high frequencies of any distribution. Furthermore, the specific definition of the functions $\varphi_q(n)$ lead to the identity
\begin{equation*}
     f = \sum_{q=-1}^\infty \Dd_q f\qquad \text{in}\quad  \mathcal{D}'(\TT^d),
\end{equation*}
which translates the concept of splitting a general distribution into a sequence of spectrally well-localised smooth functions. For a more regular function $f$, the convergence of the series holds also in the corresponding function space. Besides, the dyadic blocks satisfy a quasi-orthogonal property, namely, $\Dd_q\Dd_{q'} f = 0$ for any distribution $f$ and any indexes $q,q'\in \tilde \NN$ that are sufficiently distant, for instance, $|q-q'|\geq 2$.

A further important operator in the Littlewood-Paley theory is the low frequency, homogeneous cut-off operator $S_q$, which takes into account the contribution of any dyadic block before $\Dd_q$:
\begin{equation} \label{eq:S_j}
\Sd_{q} f\,:=\,\sum_{k=-1}^{q-1}\Dd_{k}f,\quad \text{for } q\in \mathbb{N}.
\end{equation}
In other words, for every $q\in \mathbb{N}$, $S_q$ takes into account only the low frequencies of a distribution $f$, localising its Fourier coefficients within a ball of radius proportional to $2^q$. With an analogous procedure as for deriving \cite[Remark 2.11]{BCD}, we infer that the operators  $\Dd_q$ and $\Sd_q$ continuously map any Lebesgue space $L^p(\TT^d)$ to itself, for all $q\in \tilde \NN$ and all $p\in[1,+\infty]$, with a constant of embedding independent of the indexes $q$ and $p$. This fact will be frequently used in this paper.

Let us now introduce the so-called Besov spaces.
\begin{definition}\label{def:Besov-spaces}
   Let $s\in \mathbb{R}$ and $1\leq p,\,r\leq \infty$. The Besov space $B_{p,r}^s=B_{p,r}^s(\mathbb{T}^d)$ consists of all periodic distribution $f$ such that
   \begin{equation*}
       \| f \|_{B_{p,r}^s}:=
       \Big\| \big( 2^{qs}\| \Dd_q f \|_{L^p} \big)_{q\in \tilde \NN} \Big\|_{\ell^r(\tilde \NN)}
       <\infty.
   \end{equation*}
\end{definition}
\noindent 
Moreover, in the case of negative index of regularity, the Besov spaces can additionally be characterised in terms of the operators $S_q$ (cf. \cite[Proposition 2.79]{BCD}):
\begin{lemma}\label{lemma:replacing-Dq-with-Sq}
    Let $s<0$ and $1\leq p,r\leq \infty$. A distribution $f$ belongs to $B_{p,r}^s$ if and only if 
    \begin{equation*}
        \Big\| \big( 2^{qs}\| S_q f \|_{L^p} \big)_{q\in \NN} \Big\|_{\ell^r( \NN)}
       <\infty.
    \end{equation*}
    Furthermore, there exists a constant $C>0$ depending only on the dimension $d$ such that
    \begin{equation*}
        C^{s+1} \| f \|_{B_{p,r}^s}
        \leq \Big\| \big( 2^{qs}\| S_q f \|_{L^p} \big)_{q\in \NN} \Big\|_{\ell^r(\NN)} 
        \leq 
        C\Big( 1 - \frac{1}{s} \Big) \| f \|_{B_{p,r}^s}.
    \end{equation*}
\end{lemma}

\begin{remark}\label{equinorm}
We note that the Besov spaces generalise the elementary function spaces. In particular, for $s\in \mathbb{R}^+\setminus\{\mathbb{N}\}$, $B_{\infty, \infty}^s(\mathbb{T}^d)$ coincides with the H\"older space 
$C^{[s], s-[s]}(\mathbb{T}^d)$, while for any $s\in\mathbb{R}$, $B_{2,2}^s(\mathbb{T}^d)$ coincides with the Sobolev space $H^s(\mathbb{T}^d)$. Throughout this paper, we set (with an abuse of notation) any Sobolev space $H^s(\mathbb{T}^d)$ (for $s \neq 0$) with the equivalent inner product given by $B_{2,2}^s(\mathbb{T}^d)$, i.e.
\begin{equation*}
    \langle f,\,g\rangle_{H^s}= 
    \sum_{q=-1}^\infty 2^{2qs}\int_{\mathbb{T}^d} \Dd_qf(x) \overline{\Dd_qg(x)}\,\dd x.
\end{equation*}
For $s=0$, we however maintain the standard Lebesgue definition of $L^2(\mathbb{T}^d)$.
\end{remark}

\subsection{Some useful lemmas}
An important advantage that arises from the dyadic decomposition and the Besov formalism is reflected by the so-called Bernstein inequalities. These family of inequalities address the derivatives of localised distribution $\Dd_q f$ and provide some uniform bounds in terms of the index $q$.
\begin{lemma}[\textbf{Bernstein inequalities}]\label{prop:Bernstein}
    There exists a positive constant $C>0$ such that  for any $f\in \mathcal{D}'(\TT^d)$,
	\begin{equation}
	\begin{alignedat}{8}
		\mathrm{(i)}\quad &\| \Dd_q \partial^\alpha f\|_{L^r} &&\leq C^{k+1}2^{q\left(k+ d\left(\frac{1}{p}-\frac{1}{r}\right)\right)}
		\| \Dd_q f \|_{L^p},\quad &&&& q=-1,0,1,2,\dots,
		\\
		\mathrm{(ii)}\quad&
		\| S_q \partial^\alpha f  \|_{L^r} && \leq C^{k+1}2^{q\left(k+ d\left(\frac{1}{p}-\frac{1}{r}\right)\right)}
		\| S_q f \|_{L^p},\quad  &&&& q=0,1,2,\dots,\\
		\mathrm{(iii)}\quad&\| \Dd_q f \|_{L^p} &&\leq C2^{-q}
		\| \Dd_q \nabla  f \|_{L^p},\quad &&&& q=0,1,2,\dots,
	\end{alignedat}
	\end{equation}
	for any $p,\,r\in [1,\infty]$ with $r\geq p$, for any $k\in\mathbb{N}$ and any multi-index  $\alpha=(\alpha_1,...,\alpha_d)\in \mathbb{N}^d$ with $|\alpha|=k$, here, we denote $\partial^\alpha=\partial^{\alpha_1}_{x_1}...\partial^{\alpha_d}_{x_d}$.
\end{lemma}
\begin{proof}
  The inequalities (i) and (ii) follow directly from  \cite[Lemma 7]{MR3445609}, the third inequality (iii) can be shown with an analogous procedure as for proving its homologous version in the whole space (cf.  \cite[Lemma 2.1]{BCD}). Indeed, we first remark that, for $q> -1$, $\Dd_q f $ is spectrally supported away from zero. We hence set a cut-off function $\tilde \varphi \in \mathcal{D}(\RR^d)$, whose support is away from zero and is identically $1$ in the support of $\varphi$. Therefore:
  \begin{equation*}
      \Dd_q f (x) 
      = \sum_{n\in\mathbb{Z}^d} \varphi(2^{-q} n) \hat{f}_n e^{\mathrm{i} n\cdot x}
      = \sum_{n\in\mathbb{Z}^d} \frac{\tilde \varphi(2^{-q}n)}{|n|^2}(-\mathrm{i}n)\cdot \mathrm{i} n \,\varphi(2^{-q} n) \hat{f}_n  e^{\mathrm{i} n\cdot x}
      =(\psi_q \ast \Dd_q \nabla f)(x),
  \end{equation*}
  where the symbol $\ast$ denotes the convolution in $\mathbb{T}^d$ and $\psi_q(x) =
  -\sum_{n\in\mathbb{Z}^d}\frac{\tilde \varphi(2^{-q}n)}{|n|^2}\mathrm{i}n e^{\mathrm{i}n\cdot x}$. Then the inequality (iii) follows from Young's inequality and the estimate on the $L^1(\TT^d)$-norm of $\psi_q$. 
  To this end, we invoke the Poisson summation formula, which relates the Fourier series of a periodic function to values of the function's continuous Fourier transform (see \cite[Lemma 6]{MR3445609}):
  \begin{equation*}
      \psi_q (x) = \sum_{k\in \mathbb{Z}^d} 
      \mathcal{F}_{\mathbb{R}^d}^{-1}
      \Big( \frac{\tilde\varphi(2^{-q}\cdot )}{|\cdot |^2}\mathrm{i}\cdot \Big)(x-2k\pi), 
  \end{equation*}
  where $\mathcal{F}_{\mathbb{R}^d}^{-1}$ stands for the inverse Fourier transform in $\mathbb{R}^d$. We thus gather that
  \begin{align*}
      \|\psi_q(x)\|_{L^1}
      &\leq \sum_{k\in \mathbb{Z}^d}\int_{\mathbb{T}^d}\Big|\mathcal{F}_{\mathbb{R}^d}^{-1}\Big( \frac{\tilde\varphi(2^{-q}\cdot )}{|\cdot |^2}\mathrm{i}\cdot \Big)(x-2k\pi)\Big|\,\dd x
      \\
      &\leq  
      \int_{\mathbb{R}^d}\Big|\mathcal{F}_{\mathbb{R}^d}^{-1} \Big( \frac{\tilde\varphi(2^{-q}\cdot )}{|\cdot |^2}\mathrm{i}\cdot \Big)(y)\Big|\,\dd y   \\
      &\leq 
      2^{-q}
       \int_{\mathbb{R}^d}\Big|\mathcal{F}_{\mathbb{R}^d}^{-1} \Big( \frac{\tilde\varphi(2^{-q}\cdot )}{|2^{-q}\cdot |^2}\mathrm{i} 2^{-q}\cdot \Big)(y)\Big|\,\dd y
        \\
      &\leq 
      2^{-q}
      2^{qd } \int_{\mathbb{R}^d}\Big|\mathcal{F}_{\mathbb{R}^d}^{-1} \Big( \frac{\tilde\varphi(\cdot )}{|\cdot |^2}\mathrm{i} \cdot \Big)(2^q y)\Big|\,\dd y
      \leq C2^{-q},
  \end{align*}
  which completes the proof of this lemma.
\end{proof}
A straight consequence of the Bernstein inequalities is the following family of embeddings between Besov spaces.
\begin{prop}
    Let $1\leq p_1\leq p_2 \leq \infty$ and $1\leq r_1\leq r_2\leq \infty$. Then for any $s\in\mathbb{R}$, the space $B_{p_1,r_1}^s$ is continuously embedded in $B_{p_2,r_2}^{s-d(1/p_1-1/p_2)}$, with 
    \begin{equation*}
        \| f \|_{B_{p_2,r_2}^{s-d\left(\frac{1}{p_1}-\frac{1}{p_2}\right)}}
        \leq 
        C\| f \|_{B_{p_1,r_1}^s},
    \end{equation*}
    for any $f\in B_{p_1,r_1}^s$, where the constant $C$ does not depend on the parameters $(p_1,p_2,r_1,r_2,s)$.
\end{prop}
The next result is a further consequence of the Bernstein inequality in $\TT^d$. It is well-known that $H^\frac{d}{2}(\mathbb{T}^d)$ is not embedded in $L^\infty(\mathbb{T}^d)$. Nevertheless, the result holds as long as the function under consideration has Fourier coefficients that are compactly supported. In  particular, the following lemma allows us to recover an uniform bound that depends uniquely on the set where the spectrum is supported: 
\begin{lemma}\label{lemma:SN-infty}
	For any  $N\in \mathbb{N}\setminus\{0\}$, there exists a constant $C>0$ independent of $N$ such that 
	\begin{equation*}
		\| S_N f \|_{L^\infty} \leq C\sqrt{N}\| f \|_{H^\frac{d}{2}},\quad \forall\, f\in H^\frac{d}{2}(\TT^d).
	\end{equation*}
\end{lemma}
\begin{proof}
	The result follows from a direct application of the definition of $S_N$ together with the Bernstein inequalities (ii) and (iii) in Lemma \ref{prop:Bernstein}: 
	\begin{align*}
		\| S_N f \|_{L^\infty} 
		& \leq 
		\sum_{q=-1}^{N-1} \| \Dd_q f \|_{L^\infty}
		\leq 
		\tilde C
		\sum_{q=-1}^{N-1}2^{\frac{d}{2}q} \| \Dd_q f \|_{L^2}
		\leq
		\tilde C
		\left(
			\sum_{q=-1}^{N-1}1
		\right)^\frac{1}{2}
		\left(
			\sum_{q=-1}^{N-1}2^{d q} \| \Dd_q f \|_{L^2}^2
		\right)^\frac{1}{2}\\
		& \leq 
		\tilde C
		\sqrt{1+N}\| f \|_{H^\frac{d}{2}}
		\leq 
		2\tilde C
		\sqrt{N}\| f \|_{H^\frac{d}{2}},
	\end{align*}
	where $\tilde{C}>0$ is a constant independent of $N$. Then we arrive at the conclusion with the choice $C=2\tilde C$.
\end{proof}

Although the Sobolev space  $H^\frac{d}{2}(\mathbb{T}^d)$ is not embedded in $L^\infty(\mathbb{T}^d)$, it is indeed included in $L^p(\mathbb{T}^d)$, for any $p\in [2,\infty)$ (thus of course also for $p\in [1,2)$). A further key tool to prove Proposition \ref{thm:ineq} is the  following lemma, which explicitly expresses the constant of the Sobolev embedding mentioned above. 
\begin{lemma}\label{lemma:eps}
	Let $s\in [0,d/2)$. There exists a constant $C_d>0$ depending only on $d$, such that
	\begin{equation*}
		\| f \|_{L^p}\leq C_d \sqrt{p} \| f \|_{H^s},\quad \text{with }\ p = \frac{2d}{d-2s},\quad \forall\, f\in H^s(\TT^d).
	\end{equation*}
	In particular, when $d=2$, the following interpolation inequality holds for any $\ee \in (0,1]$:
	\begin{equation*}
		\| f \|_{L^\frac{2}{\ee}}
		\leq \frac{C}{\sqrt{\ee}}\| f \|_{L^2}^\ee \| f \|_{H^1}^{1-\ee},\quad \forall\, f\in H^1(\TT^2).
	\end{equation*}
\end{lemma}
\noindent 
This lemma can be proven via an analogous procedure to the one used for the  inequality \cite[(1.30)]{BCD}. For the sake of completeness, we provide its proof in Appendix \ref{Appx:proof-lemma-eps}. 

At last, we address an useful lemma that concerns the commutator estimates. This result will be frequently used when we prove the estimates in Proposition \ref{thm:ineq}, since it allows us to gain one derivative by commutation between the operator $\Dd_q$ and the multiplication by a function with its  gradient in $L^p(\mathbb{T}^d)$. 
\begin{lemma}\label{prop:comm-est}
	There exists a positive constant $C$ such that for any Lipschitz function $f$ with its gradient in 
	$L^p(\TT^d)$ and for any function $g$ in $L^h(\TT^d)$, we have 
	\begin{align*}
		\| [\Dd_q,f]g\|_{L^r} 
		&\leq C2^{-q} 
		\| \nabla f\|_{L^p}\| g\|_{L^h},\\
		\| [S_N,f]g\|_{L^r} 
		&\leq C2^{-N} 
		\| \nabla f\|_{L^p}\| g\|_{L^h},
	\end{align*}	 
	for any $q\in\tilde{\mathbb{N}}$, $N\in\mathbb{N}$, and $r,p,h\geq 1$ satisfying $1/p +1/h = 1/r$.
\end{lemma}
	
\begin{proof}
The two inequalities can be seen as a direct consequence of the commutator estimates in \cite[Lemma 2.97]{BCD} in the framework of the whole space $\mathbb{R}^d$. For the sake of completeness, here we address the periodic case of $\mathbb{T}^d$ and provide their proofs, taking into account just the first inequality, since the second one can be derived with an analogous procedure.

	Recall that the commutator operator $[\Dd_q,f]g$ is defined by means of $[\Dd_q,\,f]g 
	= \Dd_q(fg)-f\Dd_q g= \Phi_q\ast (fg) - f(\Phi_q\ast g)$, with the kernel 
	$$ \Phi_q(x) = \sum_{n\in\ZZ^d} \varphi_q(n)e^{\mathrm{i}n\cdot x}. $$
	Therefore, it holds
	\begin{equation*}
		[\Dd_q,f]g(x) = \int_{\TT^d} \Phi_q(x-y)(f(y)-f(x))g(y)\,\dd y,
	\end{equation*}
	for almost all $x\in\mathbb{T}^d$. 
	Next, introducing the inverse Fourier transform in 
	$\mathbb{R}^d$ of $\varphi(\xi)$ and $\varphi_q(\xi)$ such that
	\begin{equation*}
		\Psi(x) := 
		\big(\mathcal{F}_{\RR^d}^{-1}\varphi \big)(x)
		=
		\frac{1}{(2\pi)^d}
		\int_{\RR^d} \varphi(\xi)e^{-\mathrm{i} x\cdot \xi}\,\dd \xi,\qquad 
		\Psi_q(x) := 
		\big(\mathcal{F}_{\RR^d}^{-1}\varphi_q \big)(x)
		=
		2^{qd} \Psi(2^q x),
	\end{equation*}		
	we invoke again the Poisson summation formula in \cite[Lemma 6]{MR3445609} to obtain
	\begin{align*}
		[\Dd_q,f]g(x) &= 
		\int_{\TT^d} 
		\Big(
		\sum_{k\in \mathbb{Z}^d} 
		\Psi_q(x-y+2k\pi)
		\Big)
		(f(y)-f(x))g(y)\,\dd y\\
		&=
		\int_{\TT^d} 
		\Big(
		\sum_{k\in \mathbb{Z}^d} 
		\Psi_q(x-y+2k\pi)
		(f(y-2k\pi)-f(x))
		\Big)
		g(y)\,\dd y.
	\end{align*}
	Since $f$ is Lipschitz, then for almost all $x,y\in \mathbb{T}^d$, it holds 
	$f(y-2k\pi)-f(x) = \int_0^1 \nabla f(x+s(y-2k\pi -x))\,\dd s \cdot (y-2k\pi-x)$. As a result, we see that 
	\begin{align*}
		\big| [\Dd_q,f]g(x) \big|
		&\leq 
		\int_{\TT^d}\int_0^1 
		\Big(
		\sum_{k\in \mathbb{Z}^d} |\Psi_q(x-y+2k\pi)||x-y+2k\pi||\nabla f(x+s(y-2k\pi-x))|
		\Big)|g(y)|\,\dd s \,
		\dd y\\
		&\leq 
		\int_{\TT^d}\int_0^1 
		\Big(
		\sum_{k\in \mathbb{Z}^d} |\Psi_q(z+2k\pi)||z+2k\pi||\nabla f(x-s(z+2k\pi))|
		\Big)|g(x-z)|\,\dd s \,
		\dd z\\
		&\leq 
		\int_{\TT^d}\int_0^1 
		\Big(
		\sum_{k\in \mathbb{Z}^d} |2^{qd}\Psi(2^q(z+2k\pi))||z+2k\pi||\nabla f(x-s(z+2k\pi))|
		\Big)|g(x-z)|\,\dd s \,
		\dd z.
	\end{align*}
	Now taking the $L^r$-norm of both sides in the above inequality, using the fact that the norm of an integral is less 
	than the integral of the norm, and H\"older's inequality, we get
	\begin{align*}
		\| [\Dd_q,f]g \|_{L^r}
		&\leq 
		\int_{\TT^d}\int_0^1 
		\Big(
		\sum_{k\in \mathbb{Z}^d} |2^{qd}\Psi(2^q(z+2k\pi))||z+2k\pi|
		\|\nabla f(\cdot -s(z+2k\pi))\|_{L^p}
		\Big)\|g(\cdot -z)\|_{L^h}\,\dd s\, 
		\dd z.
	\end{align*}
	The translation invariance of the Lebesgue measure then ensures that
	\begin{align*}
	\| [\Dd_q,f]g \|_{L^r}
		&
		\leq 
		\Big(
		2^{-q}
		\sum_{k\in \mathbb{Z}^d}
		\int_{\TT^d}
		 |2^{qd}\Psi(2^q(z+2k\pi))|2^q|z+2k\pi|\,
		 \dd z
		\Big)
		\|\nabla f\|_{L^p}
		\|g\|_{L^h} \\
		&\leq 
		2^{-q}
		\int_{\RR^d}
		|2^{qd}\Psi(2^qz)||2^qz|
		\,\dd z\,
		\|\nabla f\|_{L^p}
		\|g\|_{L^h}
		\leq 
		C2^{-q}
		\|\nabla f\|_{L^p}
		\|g\|_{L^h},
	\end{align*}
	which completes the proof of this lemma.
\end{proof}

\subsection{The Bony's decomposition}
We conclude this section by introducing the so-called Bony's product decomposition. The construction of paraproduct operators relies on the observation that, formally, any product of two distributions $f$ and $g$, may be decomposed into
\begin{equation*}
	f g = T_f g + T_g f + R(f,g),
\end{equation*}
where we define the paraproducts $T_f g,\,T_g f$ and the reminder $R(f,g)$ by means of
\begin{equation*}
	T_fg 	:=\sum_{q = -1}^\infty \sdm{q}f\Dd_q g,\qquad 
	T_gf 	:=\sum_{q = -1}^\infty \sdm{q}g\Dd_q f,\qquad 
	R(f,g)	:=\sum_{q = -1}^\infty \Dd_q f (\Dd_{q-1}g+\Dd_{q}g+\Dd_{q+1}g).
\end{equation*} 
The paraproduct $T_fg$ (and thus $T_gf$) is always a well-defined distribution.
Heuristically, $T_fg$ behaves at large frequencies like $g$ (and thus retains the same regularity), and $f$ provides only a frequency modulation of $g$. The only difficulty in constructing the product $f g$ for arbitrary distributions lies in handling the reminder term $R(f,g)$.

\smallskip
A first important result that arises from the analysis of the paraproducts and the reminder is the following theorem, concerning the continuity properties of the products and the remainder (whose proof is identical to the one of the homologous version for $\mathbb{R}^d$ in \cite[Theorem 2.82 and Theorem 2.85]{BCD}). 
\begin{theorem}\label{thm:continuity-paraproduct-reminder}
    There exists a positive constant $C$ such that the following inequalities hold:
    \begin{enumerate}[(i)]
        \item For couple  $(s_1,s_2) \in \mathbb{R}^2$ with $s_1$ negative and any $(p,r_1,r_2)$ in $[1,\infty]^3$,
        \begin{align*}
            \| T_f g \|_{B_{p,r}^{s_1+s_2}}\leq \frac{C^{|s_1+s_2|+1}}{|s_1|}\| f \|_{B_{\infty, r_1}^{s_1}} \| g \|_{B_{p, r_2}^{s_2}},
        \end{align*}
        for any $(f,g)\in B_{\infty,r_1}^{s_1}\times B_{p,r_2}^{s_2}$, with $ 1/r = \min\{1, 1/r_1+1/r_2\} $.
        \item For any couple $(s_1,s_2)\in \mathbb{R}^2$ with $s_1+s_2>0$ and for any $(p_1,\,p_2,\,r_1,\,r_2)\in [1,\infty]^4$, with 
        $1/p := 1/p_1+1/p_2\leq 1$ and $1/r = 1/r_1+1/r_2\leq 1$, 
        \begin{equation*}
            \| R(f,g) \|_{B_{p,r}^{s_1+s_2}} \leq \frac{C^{s_1+s_2+1}}{s_1+s_2} \| f \|_{B_{p_1,r_1}^{s_1}}
            \| g \|_{B_{p_2,r_2}^{s_2}},
        \end{equation*}
        for any $(f,g)\in B_{p_1,r_1}^{s_1}\times B_{p_2,r_2}^{s_2}$.
    \end{enumerate}
\end{theorem}
\noindent 
In particular, the following lemma concerning the continuity of the product between Sobolev spaces holds true.
\begin{lemma}\label{lemma:product}
	Let $s,\,t\in\mathbb{R}$ be two constants satisfying $s+t>0$ and $s,t<1$. 
	For any  $f\in H^s(\TT^2)$ and $g\in H^t(\TT^2)$,  the product 
	$fg$ belongs to $H^{s+t-1}(\TT^2)$ and
	\begin{equation*}
		\| fg \|_{H^{s+t-1}}\leq C\| f\|_{H^s}\|g\|_{H^t},
	\end{equation*}
	for certain constant $C>0$ only depending on $s$ and $t$.
\end{lemma}
\begin{proof}
    Making use of the Bony's decomposition, we recast the product $fg$ into $T_fg + T_g f + R(f,g)$. Because of the 
    Besov embeddings $B_{2,1}^{s+t-1}\hookrightarrow B_{2,2}^{s+t-1}$, $B_{2,2}^t\hookrightarrow B_{\infty, \infty}^{t-1}$, we invoke Theorem \ref{thm:continuity-paraproduct-reminder}, part $(i)$ with $s_1 = s-1$ and $s_2 = t$, to remark that
    \begin{equation*}
        \| T_f g \|_{H^{s+t-1}}\leq C\| T_f g \|_{B_{2,1}^{s+t-1}}
        \leq 
        \frac{C^{|s+t-1|+1}}{|s-1|}\| f \|_{B_{\infty, \infty}^{s-1}} \| g \|_{B_{2, 2}^t}
        \leq 
        \frac{C^{|s+t-1|+1}}{|s-1|}\| f \|_{B_{2,2}^{s}} \| g \|_{B_{2, 2}^t}.
    \end{equation*}
    Similarly, we infer that the second paraproduct fulfills 
    \begin{equation*}
        \| T_g f \|_{H^{s+t-1}}\leq \frac{C^{|s+t-1|+1}}{|t-1|}\| g \|_{B_{2,2}^{t}} \| f \|_{B_{2, 2}^s}.
    \end{equation*}
    Finally, we recall the Besov embedding $B_{1,1}^{s+t}\hookrightarrow B_{2,2}^{s+t-1}$ and make use of Theorem \ref{thm:continuity-paraproduct-reminder}, part $(ii)$ with $s_1=s$, $s_2 = t$, $p=1$, $p_1=p_2=2$, $r=1$, $r_1=r_2=2$, to gather
    \begin{equation*}
        \| R(f,g) \|_{B_{2,2}^{s+t-1}} \leq C\| R(f,g) \|_{B_{1,1}^{s+t}} 
        \leq \frac{C^{s+t+1}}{s+t} \| f \|_{B_{2,2}^s} \| g \|_{B_{2,2}^t}.
    \end{equation*}
    This concludes the proof of the lemma.
\end{proof}

The analysis in the subsequent sections requires an additional development of the Bony's decomposition, when we apply a dyadic block $\Dd_q$ to the product. More precisely, applying a commutator between the operators $T_f$ and $\Dd_q$ we gather the following formula:
\begin{align*}
	\Dd_q(T_f g) &= [\Dd_q,\,T_f]g + T_f\Dd_q g\\
	& = 
	\sum_{|j-q|\leq 5}[\Dd_q,\,\sdm{j} f]\Dd_j g + 
	\sum_{|j-q|\leq 5}(\sdm{j}f-\sdm{q}f)\Dd_q\Dd_j g
	+\sum_{|j-q|\leq 5} \sdm{q}f\Dd_q \Dd_j g \\
	& = 
	\sum_{|j-q|\leq 5}[\Dd_q,\,\sdm{j} f]\Dd_j g + 
	\sum_{|j-q|\leq 5}(\sdm{j}f-\sdm{q}f)\Dd_q\Dd_j g
	+\sdm{q}f\Dd_q g,
\end{align*}
which finally leads to 
\begin{equation}\label{bony-decomp}
	\Dd_q(f g) = 
	\bigg(
	\sum_{|j-q|\leq 5} 
	[\Dd_q,\,\sdm{j} f ]\Dd_j g\bigg) +  
	\bigg(\sum_{|j-q|\leq 5} 
	(\sdm{j} -\sdm{q})f\Dd_q \Dd_j g \bigg)
	+ \sdm{q}f\Dd_qg + 
	\bigg(\sum_{j=q-5}^\infty \Dd_q(\Dd_j fS_{j+2}g)\bigg),
\end{equation}
for any $q\in \tilde{\mathbb{N}}$. 

Returning to our uniqueness problem, those most challenging nonlinear terms in the system \eqref{main_system} will be decomposed through the specific form \eqref{bony-decomp} (see, for instance,  \eqref{ineqT2-defIandJ}, \eqref{decomp_deltaAd1od1} and \eqref{def-I...IV-d1A2deltadd1d1} below). It is also worth mentioning that, in each of these decompositions, the third term $\sdm{q}f\Dd_qg$ in \eqref{bony-decomp} will sometimes either cancel out by taking into account the specific coupling structure of system \eqref{main_system}, or yield some additional dissipative terms, in particular, by considering certain physically meaningful combinations in the subsequent estimates (recall Remark \ref{comm_proof}).

\section{\bf Estimate of $\mathcal{T}_1$}\label{sec:estT1}
\setcounter{equation}{0}

Starting from this section, we proceed to prove Proposition \ref{thm:ineq}. In the subsequent analysis, we denote by $C$, $C_j$ generic positive constants that may vary from line to line. Specific dependence will be pointed out if necessary. Besides, we shall  simply write $A\lesssim B$ if $A \leq  CB$.

The proof of Proposition \ref{thm:ineq} will be carried out in Section \ref{sec:estT1} -- Section \ref{sec:double-log-est}. Our aim in this section is to derive estimates for the term $\mathcal{T}_1$ (recall \eqref{Xi1}), namely,
\begin{equation} \label{Xi1a} 
\begin{aligned}
	\mathcal{T}_1 =
	& -
	\langle 	 \delta \uu\cdot \nabla \uu_1,\, \delta \uu 	\rangle_{H^{-\frac{1}{2}}}
	-
	\langle 	 \uu_2\cdot \nabla \delta  \uu,\, \delta \uu 	\rangle_{H^{-\frac{1}{2}}}		
	+
	\langle \nabla \ddd_1\odot \nabla\delta \ddd ,\,\nabla \delta \uu 	\rangle_{H^{-\frac{1}{2}}}	
	+ \langle \nabla \delta \ddd\odot \nabla \ddd_2 ,\,\nabla \delta \uu 	\rangle_{H^{-\frac{1}{2}}} 
	\\
	& 
	- \langle	\delta \ddd\otimes (A_2 \ddd_2),\,
		\nabla \delta \uu \rangle_{H^{-\frac{1}{2}}} 
	-  \langle	(A_2 \ddd_2)\otimes \delta \ddd,\,
		\nabla \delta \uu \rangle_{H^{-\frac{1}{2}}} 
	- \langle
		\delta  \nabla_\ddd W(\ddd))\otimes \ddd_1,\,
		\nabla \delta \uu \rangle_{H^{-\frac{1}{2}}}
		\\
	& +
	\langle
		(\Delta \ddd_2 - \nabla_\ddd W(\ddd_2))\otimes \delta \ddd,\,
		\nabla \delta \uu
	\rangle_{H^{-\frac{1}{2}}}
	- \int_{\TT^2} 
	(\uu_1\cdot \nabla\delta \ddd)\cdot \delta \ddd\,\dd x
	+
	\langle \uu_1\cdot \nabla\delta \ddd ,\,\Delta \delta \ddd 	\rangle_{H^{-\frac{1}{2}}}			
	\\
	&-
	\int_{\TT^2} 
	(\delta \uu\cdot \nabla  \ddd_2)\cdot \delta \ddd\,\dd x
	+
	\langle \delta \uu\cdot \nabla  \ddd_2 ,\,\Delta 
	\delta \ddd 	\rangle_{H^{-\frac{1}{2}}}	
	-
	\int_{\TT^2} 
	\delta \nabla_\ddd W(\ddd)\cdot \delta \ddd \,\dd x		
	+
	\langle \delta \nabla_\ddd W(\ddd),\,\Delta \delta \ddd 	\rangle_{H^{-\frac{1}{2}}}. 
\end{aligned}
\end{equation}
From the technical point of view, the above terms contained in $\mathcal{T}_1$ are relatively ``easier'' to control, in the sense that they can be bounded by using some standard techniques such as the product rule between Sobolev spaces (see Lemma \ref{lemma:product}). 

More precisely, we have 
\begin{prop}
\label{prop:bound-for-xi1}
There exist a positive function $f_1\in L^1(0,T)$ such that for any $\eta\in (0,1)$, the following inequality holds 
	\begin{equation}\label{prop:T1-ineq}
		|\mathcal{T}_1(t) |
		\leq C_1 f_1(t)\Phi(t)  + 10 \eta\, \widehat{\mathfrak{D}}(t),
		\qquad \text{for almost all } t\in (0,T),
	\end{equation}
	where $C_1>0$ is a constant depending on $\eta^{-1}$.
\end{prop}
\begin{proof}[Proof of Proposition \ref{prop:bound-for-xi1}]
Let $\eta\in (0,1)$. We analyze the right-hand side of \eqref{Xi1a} term by term. For the first term, we invoke Lemma \ref{lemma:product} with $s=1/2$ and $t=0$, which states the continuity of the product between $H^{1/2}\times L^2$ and $H^{-1/2}$, to gather that 
\begin{align*}
	\big|
	\langle \delta  \uu \cdot \nabla \uu_1 , \delta \uu \rangle_{H^{-\frac{1}{2}}}
	\big|
	& \leq 
	\| \delta  \uu \cdot \nabla \uu_1 \|_{H^{-\frac{1}{2}}}
	\| \delta \uu \|_{H^{-\frac{1}{2}}}
	\lesssim
	\| \delta  \uu \|_{H^{\frac{1}{2}}} \| \nabla \uu_1 \|_{L^2}
	\| \delta \uu \|_{H^{-\frac{1}{2}}}\\
	&\lesssim
	\| \nabla \delta  \uu \|_{H^{-\frac{1}{2}}} 
	\| \nabla \uu_1 \|_{L^2}
	\| \delta \uu \|_{H^{-\frac{1}{2}}},
\end{align*}
where we use the relation $ \|\delta \uu\|_{H^{1/2}} \leq C\| \nabla \delta  \uu \|_{H^{-1/2}}$, since $\delta \uu(t)$ has null average in $\mathbb{T}^2$, for almost any $t\in (0,T)$. 
By Young's inequality, we deduce that for any $\eta\in (0,1)$, there exists a positive constant $C$ independent of $\eta$ such that 
\begin{align*}
	\big|
		\langle \,\delta  \uu \cdot \nabla \uu_1 ,\, \delta \uu \rangle_{H^{-\frac{1}{2}}}
	\big|	
	& \leq 
	C\eta^{-1} \| \nabla \uu_1 \|_{L^2}^2 
	\| \delta \uu \|_{H^{-\frac{1}{2}}}^2
	\,+\,
	\eta \nu\| \nabla \delta \uu \|_{H^{-\frac{1}{2}}}^2.
\end{align*}
Next, recalling the divergence free condition on $\uu_2$, we estimate the second term as follows
\begin{equation*}
	\big|
	\langle \, \uu_2 \cdot \nabla \delta  \uu ,\, \delta \uu \rangle_{H^{-\frac{1}{2}}}	
	\big|
	=
	\big|\langle \, \uu_2 \otimes \delta  \uu ,\,\nabla  \delta \uu \rangle_{H^{-\frac{1}{2}}}\big|
	\leq 
	\|  \uu_2 \otimes \delta  \uu\|_{H^{-\frac{1}{2}}}
	\| \nabla  \delta \uu \|_{H^{-\frac{1}{2}}}.
\end{equation*}
Then applying Lemma \ref{lemma:product} with $s= 3/4$ and $t= -1/4$ together with a standard interpolation, we obtain 
\begin{align*}
	\big| 
		\langle \, \uu_2 \cdot \nabla \delta  \uu ,\, \delta \uu \rangle_{H^{-\frac{1}{2}}}
	\big|
	&\lesssim 
	\|  \uu_2 \|_{H^\frac{3}{4}} \| \delta  \uu\|_{H^{-\frac{1}{4}}}
	\| \nabla  \delta \uu \|_{H^{-\frac{1}{2}}}\\
	&\lesssim 
	\|  \uu_2 \|_{L^2}^\frac{1}{4}
	\| \uu_2 \|_{H^1}^\frac{3}{4}	
	\| \delta  \uu\|_{H^{-\frac{1}{2}}}^\frac{3}{4}
	\| \nabla  \delta \uu \|_{H^{-\frac{1}{2}}}^\frac{5}{4}	
    \\
	&\leq 
	C{\eta}^{-\frac53}
	\|  \uu_2 \|_{L^2}^\frac{2}{3} \|  \uu_2 \|_{H^1}^2 
	\| \delta  \uu\|_{H^{-\frac{1}{2}}}^2\,+\,\eta \nu \| \nabla  \delta \uu \|_{H^{-\frac{1}{2}}}^2.
\end{align*}
With a similar procedure one can show that the following inequality holds:
\begin{align*}
	& \big| 
		\langle \nabla \ddd_1\odot \nabla\delta \ddd ,\,\nabla \delta \uu 	\rangle_{H^{-\frac{1}{2}}}\big| 							
	+\big| 
	\langle \nabla \delta \ddd\odot \nabla \ddd_2 ,\,\nabla \delta \uu 	\rangle_{H^{-\frac{1}{2}}}
	\big| \\
	&\quad 	\leq 
	C{\eta}^{-
	\frac53}
	 \big(1+\| \nabla\ddd_1 \|_{L^2}^\frac{2}{3}
	 \|\ddd_1\|_{H^2}^2 +\| \nabla\ddd_2 \|_{L^2}^\frac{2}{3} \| \ddd_2\|_{H^2}^2\big)
	 \|\nabla  \delta  \ddd\|_{H^{-\frac{1}{2}}}^2\,+\,\eta \nu \| \nabla  \delta \uu \|_{H^{-\frac{1}{2}}}^2.
\end{align*}
For the terms involving $\delta \ddd\otimes (A_2\ddd_2)$ and $ (A_2\ddd_2)\otimes \delta \ddd$, we apply Lemma \ref{lemma:product} with $s=0$ and $t=1/2$ to conclude that  
\begin{equation}
\begin{aligned}
	& \Big|
		\langle
			\delta \ddd\otimes(A_2\ddd_2)
			,\,
			\nabla \delta \uu
		\rangle_{H^{-\frac{1}{2}}}
	\Big|
	+ \Big|
		\langle
		(A_2\ddd_2) \otimes	\delta \ddd
			,\,
			\nabla \delta \uu
		\rangle_{H^{-\frac{1}{2}}}
	\Big|\\
	&\quad \lesssim
	\| (A_2\ddd_2)\otimes\delta \ddd \|_{H^{-\frac{1}{2}}}
	\| \nabla \delta \uu             \|_{H^{-\frac{1}{2}}}
	\lesssim
	\| A_2\ddd_2 \|_{L^2}\| \delta \ddd\|_{H^\frac{1}{2}}\| \nabla \delta \uu \|_{H^{-\frac{1}{2}}}\\
	&\quad \leq 
	C \eta^{-1}
	\| A_2\ddd_2 \|_{L^2}^2
	\| \delta \ddd\|_{H^\frac{1}{2}}^2
	+\eta \nu 
	\| \nabla \delta \uu \|_{H^{-\frac{1}{2}}}^2.\notag
\end{aligned}
\end{equation}
Recalling the definition of $\nabla_\ddd W$, we see that 
\begin{align*}
	\Big|
		\langle
			(\Delta \ddd_2 - \nabla_\ddd W(\ddd_2))\otimes\delta \ddd			,\,
			\nabla \delta \uu
		\rangle_{H^{-\frac{1}{2}}}
	\Big|
	&\lesssim
	\| (\Delta \ddd_2 - \nabla_\ddd W(\ddd_2))		\|_{L^2}
	\| 	\delta \ddd			\|_{H^\frac{1}{2}}
	\| \nabla \delta \uu	\|_{H^{-\frac{1}{2}}}\\
	&\lesssim
	\big(
	\| \Delta \ddd_2 \|_{L^2} +
    \| \ddd_2 \|_{L^6}^3 +
	\| \ddd_2 \|_{L^2}
	\big)
	\| 	\delta \ddd			\|_{H^\frac{1}{2}}
	\| \nabla \delta \uu	\|_{H^{-\frac{1}{2}}}\\
	&\leq 
	C\eta^{-1}
	\big(
	\| \ddd_2       \|_{H^1}^6+
	\| \ddd_2 		\|_{H^2}^2
	\big)
	\| 	\delta \ddd			\|_{H^\frac{1}{2}}^2
	+
	\eta\nu 
	\| \nabla \delta \uu	\|_{H^{-\frac{1}{2}}}^2.
\end{align*}
On the other hand, from the identity 
	\begin{align*}		
		&\langle
		\delta \nabla_\ddd W(\ddd)\otimes \ddd_1,\,
		\nabla \delta \uu
		\rangle_{H^{-\frac{1}{2}}}	\\
		&\quad =
		\langle 	(|\ddd_1|^2-1)\delta \ddd\otimes \ddd_1,\,\nabla 
		\delta \uu 	\rangle_{H^{-\frac{1}{2}}}		
		+ 	
		\langle  (\delta\ddd\cdot (\ddd_1+\ddd_2)) \ddd_2\otimes \ddd_1 ,\,
		\nabla \delta \uu 	\rangle_{H^{-\frac{1}{2}}}\\
		&\quad =
		-\langle 	\delta \ddd\otimes \ddd_1,\,\nabla 
		\delta \uu 	\rangle_{H^{-\frac{1}{2}}}	+
		\langle 	|\ddd_1|^2 (\delta \ddd \otimes \ddd_1),\,\nabla  
		\delta \uu 	\rangle_{H^{-\frac{1}{2}}}		
		+ 	
		\langle  (\delta\ddd\cdot (\ddd_1+\ddd_2)) \ddd_2 \otimes \ddd_1 ,\,
		\nabla  \delta \uu 	\rangle_{H^{-\frac{1}{2}}},
	\end{align*}
we deduce that 
\begin{align*}
	\big|\langle \delta \ddd \otimes \ddd_1,\,\nabla 
		\delta \uu 	\rangle_{H^{-\frac{1}{2}}}\big|
	& \lesssim
	\| \delta \ddd 			\|_{H^\frac{1}{2}}
	\| \ddd_1 				\|_{L^2}
	\| \nabla \delta \uu 	\|_{H^{-\frac{1}{2}}} \\
	& \leq 
	C\eta^{-1}
	\| \ddd_1 \|_{L^2}^2
	\| \delta \ddd 			\|_{H^\frac{1}{2}}^2
	+
	\eta \nu 
	\| \nabla \delta \uu 	\|_{H^{-\frac{1}{2}}}^2,
\end{align*}
	\begin{align*}
		\big|
		\langle 	
			|\ddd_1|^2(\delta \ddd \otimes \ddd_1),\,
			\nabla \delta \uu 	\rangle_{H^{-\frac{1}{2}}}
		\big|
	 	& \lesssim
	 	\| |\ddd_1|^2\delta \ddd \otimes \ddd_1 \|_{H^{-\frac{1}{2}}}
	 	\| \nabla \delta \uu \|_{H^{-\frac{1}{2}}}
	 	\lesssim
	  	\| 	|\ddd_1|^2\ddd_1	\|_{L^2} 
	  	\|	\delta \ddd			\|_{H^{\frac{1}{2}}}
	  	\| \nabla \delta \uu	\|_{H^{-\frac{1}{2}}}\\
	  	&\lesssim
	  	\| 	\ddd_1 			\|_{L^6}^3 
	  	\|	\delta \ddd			\|_{H^{\frac{1}{2}}}
	  	\| \nabla \delta \uu	\|_{H^{-\frac{1}{2}}}\\
	  	&\leq 
	  	C\eta^{-1} \| 	\ddd_1 				\|_{H^1}^6 
	  	\|	\delta \ddd			\|_{H^{\frac{1}{2}}}^2 +
	  	\eta\nu \| \nabla \delta \uu	\|_{H^{-\frac{1}{2}}}^2,
	\end{align*}
and in a similar manner, 
\begin{equation*}
	\big|
		\langle (\delta\ddd\cdot (\ddd_1+\ddd_2)) \ddd_2\otimes \ddd_1 ,\, \nabla \delta \uu 	\rangle_{H^{-\frac{1}{2}}}
	\big|
	\leq 
	C\eta^{-1}  
	\big( \|\ddd_1\|_{H^1}^6+\|\ddd_2 \|_{H^1}^6\big) 
	\|\delta \ddd\|_{H^\frac{1}{2}}^2 + \eta  \nu \|\nabla \delta \uu\|_{H^{-\frac{1}{2}}}^2.
\end{equation*}
As a consequence, it follows that 
\begin{equation*}
	  \big|
	  \langle
		\delta \nabla_\ddd W(\ddd)\otimes \ddd_1,\,
		\nabla \delta \uu
		\rangle_{H^{-\frac{1}{2}}}
	  \big|
	  \,\leq\,
	  C\eta^{-1} 
	  \big( 1+ \| \ddd_1 \|_{H^1}^6+\| \ddd_2 \|_{H^1}^6\big) 
	  \|\delta \ddd\|_{H^\frac{1}{2}}^2 + 3 \eta \nu \|\nabla \delta \uu \|_{H^{-\frac{1}{2}}}^2. 
\end{equation*}
We now tackle those terms associated to the director equations. 
It easily follows from the fact $\Div \uu_1=0$ and integration by parts that 
$\int_{\TT^2} (\uu_1\cdot \nabla\delta \ddd) \cdot \delta \ddd 
		\,\dd x=0$. 
Next, we infer that  
	\begin{align*}
		\big| \langle \uu_1\cdot \nabla\delta \ddd ,\,\Delta 
		\delta \ddd 	\rangle_{H^{-\frac{1}{2}}}\big| 	
		& = 
		\big| \langle\Div (\uu_1\otimes \delta \ddd) ,\,
					\Delta \delta \ddd 	
		\rangle_{H^{-\frac{1}{2}}}	\big| 
		\lesssim 
		\| \Div(\uu_1\otimes \delta \ddd)\|_{H^{-\frac{1}{2}}} 
		\| \Delta \delta \ddd \|_{H^{-\frac{1}{2}}}\\
		&\lesssim 
		\| \uu_1\otimes \delta \ddd \|_{H^{\frac{1}{2}}} 
		\| \Delta \delta \ddd \|_{H^{-\frac{1}{2}}},
	\end{align*}
from which, applying Lemma \ref{lemma:product} with $s=t=3/4$, we get
	\begin{align*}
		\big|\langle \uu_1\cdot \nabla\delta \ddd ,\,\Delta \delta \ddd 	\rangle_{H^{-\frac{1}{2}}}	\big|
		&\lesssim
		\| \uu_1 				\|_{H^{\frac{3}{4}}} 
		\| \delta \ddd	 		\|_{H^\frac{3}{4}}
		\| \Delta \delta \ddd 	\|_{H^{-\frac{1}{2}}}\\
		&\lesssim
		\| \uu_1 \|_{L^2}^\frac{1}{4}\|\uu_1 \|_{H^1}^\frac{3}{4} 
		\| \delta \ddd\|_{H^{\frac{1}{2}}}^\frac{3}{4}
		\| \delta \ddd\|_{H^{\frac{3}{2}}}^\frac{1}{4}
		\| \Delta \delta \ddd 	\|_{H^{-\frac{1}{2}}}\\
		&\leq 
        C{\eta}^{-\frac53}
		\| \uu_1 \|_{L^2}^\frac{2}{3}
		\| \uu_1 	\|_{H^1}^2
		\| 	\delta \ddd		\|_{H^{\frac{1}{2}}}^2
		+
		C\eta^{-1}\| \uu_1 \|_{L^2}^\frac{1}{2}
		\| \uu_1 	\|_{H^1}^\frac{3}{2}
		\| 	\delta \ddd		\|_{H^{\frac{1}{2}}}^2
		+
		\eta
		\| \Delta \delta \ddd 	\|_{H^{-\frac{1}{2}}}^2\\
		&\leq C{\eta}^{-\frac53}\big(1+
		\| \uu_1 \|_{L^2}^\frac{2}{3}
		\| \uu_1 	\|_{H^1}^2\big)
		\| 	\delta \ddd		\|_{H^{\frac{1}{2}}}^2
		+
		\eta
		\| \Delta \delta \ddd 	\|_{H^{-\frac{1}{2}}}^2.
	\end{align*}
In a similar manner, we have 
	\begin{align*}
		\Big|
			\int_{\TT^2} (\delta \uu \cdot \nabla  \ddd_2) \cdot \delta \ddd 
		\,\dd x\Big|
		&\lesssim \|\delta \uu \cdot \nabla  \ddd_2\|_{H^{-\frac12}}\|\delta \ddd\|_{H^\frac12}
		\lesssim \|\delta \uu\|_{H^\frac12}\|\nabla  \ddd_2\|_{L^2}\|\delta \ddd\|_{H^\frac12}
		\\
		&\leq C\|\delta \uu\|_{H^{-\frac12}}^2+
	C(1+\eta^{-1})
		\| \nabla  \ddd_2 \|_{L^2}^2
		\| \delta \ddd \|_{H^\frac{1}{2}}^2+
		\eta \nu 
		\| \nabla \delta \uu \|_{H^{-\frac{1}{2}}}^2,
	\end{align*}
and
\begin{align*}
	\big|
		\langle \delta \uu\cdot \nabla  \ddd_2 ,\,\Delta 
		\delta \ddd 	\rangle_{H^{-\frac{1}{2}}}	
	\big|
	&\lesssim 
	    \|\delta \uu\cdot  \nabla \ddd_2 \|_{H^{-\frac12}}
	    \|\Delta \delta \ddd \|_{H^{-\frac12}}
		\lesssim 
		\|\delta \uu\|_{H^{-\frac{1}{4}}}
		\| \nabla \ddd_2\|_{H^{\frac 34}}
		\|\Delta \delta \ddd \|_{H^{-\frac12}}
		\\
	&
	\lesssim 
		\|\delta \uu\|_{H^{-\frac{1}{2}}}^\frac{3}{4}
		\|\delta \uu\|_{H^{\frac{1}{2}}}^\frac{1}{4}
		\| \nabla \ddd_2\|_{L^2}^\frac{1}{4}
		\| \nabla \ddd_2\|_{H^1}^\frac{3}{4}
		\|\Delta \delta \ddd \|_{H^{-\frac12}}\\
	&\leq 
			C{\eta}^{-\frac53}\big(1+
	    \| \ddd_2 \|_{H^1}^\frac{2}{3}
		\| \ddd_2 \|_{H^2}^2\big)
		\| \delta \uu \|_{H^{-\frac{1}{2}}}^2
		+
		\eta 
		\| \Delta \delta \ddd \|_{H^{-\frac{1}{2}}}^2
		+
		\eta \nu 
		\| \nabla \delta \uu \|_{H^{-\frac{1}{2}}}^2.
\end{align*}
Concerning the terms involving $\delta \nabla_\ddd W(\ddd)$, we see that 
\begin{align*}
	\int_{\TT^2}\delta \nabla_\ddd W(\ddd)\cdot \delta \ddd\,\dd x
	=
	\int_{\TT^2}(|\ddd_1|^2-1)|\delta \ddd|^2
	\,\dd x +
	\int_{\TT^2} (\delta\ddd\cdot (\ddd_1+\ddd_2))(\ddd_2\cdot \delta \ddd)\,\dd x,
\end{align*}
where 
\begin{align*}
	\Big| 
		\int_{\TT^2}(|\ddd_1|^2-1)|\delta \ddd|^2\,\dd x
	\Big|
	\leq 
	\Big(\int_{\TT^2}(|\ddd_1|^2-1)^2\,\dd x\Big)^\frac{1}{2}
	\| \delta \ddd \|_{L^4}^2
	\lesssim
	\big(1+\|\ddd_1\|_{H^1}^2 \big) \| \delta \ddd \|_{H^\frac{1}{2}}^2,
\end{align*}
and similarly
\begin{align*}
	\Big| 
		\int_{\TT^2} (\delta\ddd\cdot (\ddd_1+\ddd_2))(\ddd_2\cdot \delta \ddd)\,\dd x
	\Big|
	\lesssim 
	\big( 
		\| \ddd_1 \|_{L^4}^2 + \| \ddd_2 \|_{L^4}^2
	\big)
	\| \delta \ddd \|_{L^4}^2
	\lesssim 
	\big( 
		\| \ddd_1 \|_{H^1}^2 + \| \ddd_2 \|_{H^1}^2
	\big)
	\| \delta \ddd \|_{H^\frac{1}{2}}^2.
\end{align*}
It remains to control the last term of  $\mathcal{T}_1$ in \eqref{Xi1a}, that is
	\begin{align*}		
		\langle \delta \nabla_\ddd W(\ddd) ,\,\Delta \delta \ddd 	\rangle_{H^{-\frac{1}{2}}}		
		&=
		\langle 	(|\ddd_1|^2-1)\delta \ddd,\,\Delta 
		\delta \ddd 	\rangle_{H^{-\frac{1}{2}}}		
		+ 	
		\langle (\delta\ddd\cdot (\ddd_1+\ddd_2)) \ddd_2 ,\,
		\Delta \delta \ddd 	\rangle_{H^{-\frac{1}{2}}}\\
		&=
		 \| \nabla \delta \ddd \|_{H^{-\frac{1}{2}}}^2+
		\langle 	|\ddd_1|^2\delta \ddd,\,\Delta 
		\delta \ddd 	\rangle_{H^{-\frac{1}{2}}}		
		+ 	
		\langle  (\delta\ddd\cdot (\ddd_1+\ddd_2)) \ddd_2,\,
		\Delta \delta \ddd 	\rangle_{H^{-\frac{1}{2}}}.	
	\end{align*}
We notice that
	\begin{align*}
		|
		\langle 	
			|\ddd_1|^2\delta \ddd,\,
			\Delta \delta \ddd 	\rangle_{H^{-\frac{1}{2}}}
		|
	 	& \lesssim
	 	\| |\ddd_1|^2\delta \ddd \|_{H^{-\frac{1}{2}}}
	 	\| \Delta \delta \ddd \|_{H^{-\frac{1}{2}}}
	 	 \lesssim
	  	\| 	\ddd_1\|_{L^4}^2 
	  	\|	\delta \ddd			\|_{H^{\frac{1}{2}}}
	  	\| \Delta \delta \ddd 	\|_{H^{-\frac{1}{2}}}\\
	  	&\lesssim
	  	\| 	\ddd_1 \|_{H^1}^2 
	  	\|	\delta \ddd			\|_{H^{\frac{1}{2}}}
	  	\| \Delta \delta \ddd 	\|_{H^{-\frac{1}{2}}}\\
	  	&\leq C\eta^{-1}  
	  \| \ddd_1 \|_{H^1}^4 
	  \|\delta \ddd\|_{H^\frac{1}{2}}^2 + \eta  \|\Delta \delta \ddd\|_{H^{-\frac{1}{2}}}^2,
	\end{align*}
and with a similar procedure, we get
\begin{equation*}
	|\langle  (\delta\ddd\cdot (\ddd_1+\ddd_2)) \ddd_2,\,\Delta \delta \ddd 	\rangle_{H^{-\frac{1}{2}}}
	| \leq 
	C\eta^{-1}    
	\big( \|\ddd_1\|_{H^1}^4+\|\ddd_2\|_{H^1}^4\big) 
	\|\delta \ddd\|_{H^\frac{1}{2}}^2 + \eta  \|\Delta \delta \ddd\|_{H^{-\frac{1}{2}}}^2.
\end{equation*}
Collecting the above estimates together, we can choose the function 
\begin{align*}
f_1(t)&= (1 + \| \uu_1 \|_{L^2}+ \|  \uu_2 \|_{L^2})(\| \uu_1 \|_{H^1}^2 +\| \uu_2 \|_{H^1}^2) + (1+ \|\ddd_1 \|_{H^1} + \|\ddd_2 \|_{H^1})(\|\ddd_1\|_{H^2}^2 
	 +\|\ddd_2\|_{H^2}^2)
		\\
	 & \quad   + \| \ddd_1 \|_{H^1}^6 +	\| \ddd_2 \|_{H^1}^6 + \|A_2\ddd_2\|_{L^2}^2+1,
\end{align*}
which satisfies $f_1(t)\in L^1(0,T)$ due to the definition of $(\uu_i, \ddd_i)$, $i=1,2$. Using Young's inequality, it is straightforward to check that the required inequality \eqref{prop:T1-ineq} holds. The proof of Proposition \ref{prop:bound-for-xi1} is complete. 
\end{proof}

\section{\bf Estimate of $\mathcal{T}_2$}\label{esti-T_2}
\setcounter{equation}{0}

\noindent 
In this section, we analyze the term $\mathcal{T}_2$ given by \eqref{Xi2}, that is  
\begin{equation}\label{Xi2b}
	\begin{aligned}
	\mathcal{T}_2 = \  
	& \frac32  \int_{\TT^2}
			((\nabla \delta \uu)\ddd_1)\cdot \delta\ddd  \,\dd x
			- 
			\frac32 \langle (\nabla \delta \uu)\ddd_1,\, \Delta \delta \ddd\rangle_{H^{-\frac{1}{2}}} 
			+ 
			\frac32 \int_{\TT^2}( (\nabla  \uu_2)\delta \ddd) \cdot \delta \ddd\,\dd x
			-
			\frac32 \langle  (\nabla  \uu_2)\delta \ddd,\, \Delta \delta \ddd\rangle_{H^{-\frac{1}{2}}}  \\
	& +\frac12  \int_{\TT^2}
		((\nabla^{\mathrm{tr}} \delta \uu)\ddd_1)\cdot \delta\ddd  \,\dd x
			- 
			\frac12 \langle (\nabla^{\mathrm{tr}} \delta \uu)\ddd_1,\, \Delta \delta \ddd\rangle_{H^{-\frac{1}{2}}} 
			+ 
			\frac12 \int_{\TT^2}( (\nabla^{\mathrm{tr}}  \uu_2)\delta \ddd) \cdot \delta \ddd\,\dd x
			\\
			& - 
			\frac12 \langle  (\nabla^{\mathrm{tr}}  \uu_2)\delta \ddd,\, \Delta \delta \ddd\rangle_{H^{-\frac{1}{2}}}
			+
			\langle  \Delta \delta \ddd \otimes \ddd_1,\,\nabla\delta \uu 	\rangle_{H^{-\frac{1}{2}}}.
	\end{aligned}
\end{equation}
Comparing with the previous section, the estimate for $\mathcal{T}_2$ turns out to be more involved. More precisely, we have 
\begin{prop}
\label{prop:T2}
There exists a positive function $ f_2 \in L^1(0,T)$ such that 
for any $\eta\in (0,1)$, the following inequality holds 
 \begin{equation}\label{prop:T2-ineq}
	|\mathcal{T}_2(t)| \leq 
	C_2  f_2(t) \Phi(t) +  5\eta \widehat{\mathfrak{D}}(t)+  \sum_{q=-1}^\infty 2^{-q} \int_{\TT^2}|(\Dd_q \delta A)\sdm{q}\ddd_1|^2\,\dd x
		+\frac14\|\Delta \delta \ddd\|_{H^{-\frac12}}^2,
 \end{equation}
 for almost all $t\in (0,T)$, where $C_2>0$ is a constant depending on $\eta^{-1}$.
\end{prop}
\begin{proof}[Proof of Proposition \ref{prop:T2}]
	Let $\eta\in (0,1)$.  We begin with the following estimates 
	\begin{equation}\label{prop:T2-ineq1}
	\begin{aligned}
		&  \left| \frac32 \int_{\TT^2}( (\nabla  \delta\uu)\ddd_1)\cdot \delta \ddd\,\dd x\right|
		+ 
		\left| \frac12
		\int_{\TT^2}
		((\nabla^{\mathrm{tr}} \delta \uu)\ddd_1)\cdot \delta\ddd  \,\dd x
		\right|\\
		&\qquad \lesssim 
		\|\ddd_1\otimes \delta \ddd\|_{H^\frac12} (\|\nabla \delta\uu\|_{H^{-\frac12}} + 
		\|\nabla^{\mathrm{tr}} \delta\uu\|_{H^{-\frac12}})\\
		&\qquad \lesssim
		\| \ddd_1 \|_{H^\frac{3}{4}}
		\| \delta \ddd \|_{H^\frac{3}{4}}
		\| \nabla \delta \uu \|_{H^{-\frac{1}{2}}}\\
		&\qquad \lesssim
		\| \ddd_1 \|_{H^1}
		\| \delta \ddd \|_{H^\frac{1}{2}}^\frac{3}{4}
		\| \delta \ddd \|_{H^\frac{3}{2}}^\frac{1}{4}
		\| \nabla \delta \uu \|_{H^{-\frac{1}{2}}}\\
		&\qquad \lesssim
		\| \ddd_1 \|_{H^1}
		\| \delta \ddd \|_{H^\frac{1}{2}}^\frac{3}{4}
		\big( 
			\| \delta \ddd \|_{H^\frac{1}{2}}+  
			\| \Delta \delta \ddd \|_{H^{-\frac{1}{2}}}
		\big)^\frac{1}{4}
		\| \nabla \delta \uu \|_{H^{-\frac{1}{2}}}\\
		&\qquad \lesssim
		\| \ddd_1 \|_{H^1}
		\| \delta \ddd \|_{H^\frac{1}{2}}
		\| \nabla \delta \uu \|_{H^{-\frac{1}{2}}}+
		\| \ddd_1 \|_{H^1}
		\| \delta \ddd \|_{H^\frac{1}{2}}^\frac{3}{4}
		\| \Delta \delta \ddd \|_{H^{-\frac{1}{2}}}^\frac{1}{4}		
		\| \nabla \delta \uu \|_{H^{-\frac{1}{2}}}\\
		&\qquad \lesssim 
		\big(\eta^{-1}\| \ddd_1 \|_{H^1}^2+ \eta^{-\frac53}\| \ddd_1 \|_{H^1}^\frac{8}{3}\big)
		\| \delta \ddd \|_{H^\frac{1}{2}}^2 + \eta \| \Delta \delta \ddd \|_{H^{-\frac{1}{2}}}^2+ 
		\eta \nu \| \nabla \delta \uu \|_{H^{-\frac{1}{2}}}^2\\
		&\qquad \lesssim \eta^{-\frac53} \big(1+\| \ddd_1 \|_{H^1}^\frac{8}{3}\big)
		\| \delta \ddd \|_{H^\frac{1}{2}}^2 + \eta \| \Delta \delta \ddd \|_{H^{-\frac{1}{2}}}^2+ 
		\eta \nu \| \nabla \delta \uu \|_{H^{-\frac{1}{2}}}^2, 
	\end{aligned}
	\end{equation}
	\begin{equation}\label{prop:T2-ineq1a}
	\begin{aligned}
		&\left|	\frac32 \int_{\TT^2}((\nabla \uu_2) \delta \ddd)\cdot \delta\ddd\,\dd x
		\right|
		+	 \left| \frac12 \int_{\TT^2}( (\nabla^{\mathrm{tr}}  \uu_2)\delta \ddd) \cdot \delta \ddd\,\dd x \right|\\
		&\qquad 
		\lesssim \| \delta \ddd  \|_{L^4}^2 (\| \nabla \uu_2 \|_{L^2} + \| \nabla^{\mathrm{tr}} \uu_2 \|_{L^2})
		\lesssim \| \nabla \uu_2 \|_{L^2} \| \delta \ddd \|_{H^\frac{1}{2}}^2,
	\end{aligned}
	\end{equation}
and
	\begin{equation}\label{prop:T2-ineq2}
	\begin{aligned}
	\left|	\frac32	
		\langle (\nabla \uu_2) \delta \ddd ,\,\Dd \delta \ddd \rangle_{H^{-\frac{1}{2}}}
	\right| 
	+
	 \left| \frac12 \langle  (\nabla^{\mathrm{tr}}  \uu_2)\delta \ddd,\, \Delta \delta \ddd\rangle_{H^{-\frac{1}{2}}}\right|
	&\lesssim 
	(\| \nabla \uu_2	\|_{L^2}+\| \nabla^{\mathrm{tr}} \uu_2 \|_{L^2})
	\| \delta \ddd 	\|_{H^\frac{1}{2}}
	\| \Delta \delta \ddd \|_{H^{-\frac{1}{2}}}\\
	&\leq 
	C{\eta}^{-1}
	\| \nabla \uu_2 \|_{L^2}^2 
	\| \delta \ddd \|_{H^\frac{1}{2}}^2
	+\eta 
	\| \Delta \delta \ddd \|_{H^{-\frac{1}{2}}}^2,
	\end{aligned}
	\end{equation}	
	where we have used Lemma \ref{lemma:product}, the Sobolev embedding theorem $H^{1/2}(\TT^2)\hookrightarrow L^4(\TT^2)$ and Young's inequality. 
	
	In order to complete the proof of Proposition \ref{prop:T2}, 
	we shall show that the remaining terms in \eqref{Xi2b} satisfy the following inequality:
	\begin{equation}\label{prop:T2-ineq3}
	\begin{aligned}
		&\left| -\frac32\langle  (\nabla \delta \uu)\ddd_1,\,\Delta \delta \ddd 	\rangle_{H^{-\frac{1}{2}}}
		- \frac12 \langle (\nabla^{\mathrm{tr}} \delta \uu)\ddd_1,\, \Delta \delta \ddd\rangle_{H^{-\frac{1}{2}}} 
			  +\langle  \Delta \delta \ddd \otimes \ddd_1,\,\nabla\delta \uu \rangle_{H^{-\frac{1}{2}}}
		\right|  \\
		&\qquad 
		\leq 
		C_\eta
		\big(1+\|\ddd_1	\|_{H^1}^2 \big)
		\|	\ddd_1	\|_{H^2}^2 
		\| \delta \uu \|_{H^{-\frac{1}{2}}}^2 
		+ 3
		\eta \nu \| \nabla \delta \uu \|_{H^{-\frac{1}{2}}}^2 + 
		3 \eta\| \Delta \delta \ddd \|_{H^{-\frac{1}{2}}}^2\\
		&\qquad \quad  +  \sum_{q=-1}^\infty 2^{-q} \int_{\TT^2}|(\Dd_q \delta A)\sdm{q}\ddd_1|^2\,\dd x
		+\frac14 \|\Delta \delta \ddd\|_{H^{-\frac12}}^2,
	\end{aligned}
	\end{equation}
	for some constant $C_\eta>0$ depending on $\eta$. If this is the case, then the conclusion of Proposition \ref{prop:T2} is a direct consequence of \eqref{prop:T2-ineq1}--\eqref{prop:T2-ineq3} with the following choice 
	\begin{equation*}
	f_2(t):= 1 + \| \ddd_1(t) \|_{H^1}^3 + \| \nabla \uu_2 (t) 
	\|_{L^2}^2 
	+\big(1+\| \ddd_1(t)\|_{H^1}^2\big)	\| \ddd_1(t)\|_{H^2}^2 
	\in L^1(0,T).
	\end{equation*}
	To obtain \eqref{prop:T2-ineq3}, we make a reformulation of the three terms above as follows 
	\begin{align*}
	 &-\frac32\langle  (\nabla \delta \uu)\ddd_1,\,\Delta \delta \ddd 	\rangle_{H^{-\frac{1}{2}}}
		- \frac12 \langle (\nabla^{\mathrm{tr}} \delta \uu)\ddd_1,\, \Delta \delta \ddd\rangle_{H^{-\frac{1}{2}}} 
			  +\langle  \Delta \delta \ddd \otimes \ddd_1,\,\nabla\delta \uu \rangle_{H^{-\frac{1}{2}}}\\
	 &\quad =  -\langle  (\nabla \delta \uu)\ddd_1,\,\Delta \delta \ddd 	\rangle_{H^{-\frac{1}{2}}}
		 +\langle  \Delta \delta \ddd \otimes \ddd_1,\,\nabla\delta \uu \rangle_{H^{-\frac{1}{2}}} 
		 - \langle ( \delta A)\ddd_1,\, \Delta \delta \ddd\rangle_{H^{-\frac{1}{2}}}. 
	\end{align*}
    Then we perform the estimates by using the specific type of Bony's decomposition given in \eqref{bony-decomp}:
	\begin{equation}\label{ineqT2-defIandJ}
	\begin{aligned}
		- \langle (\nabla \delta \uu)\ddd_1,\,\Delta \delta \ddd 	\rangle_{H^{-\frac{1}{2}}}
		&= \I_1\,+ \I_2 \,+ \I_3 \,+ \I_4, \\
		\langle  \Delta \delta \ddd \otimes \ddd_1,\,\nabla\delta \uu 	\rangle_{H^{-\frac{1}{2}}}=
		\langle  \ddd_1\otimes \Delta \delta \ddd,\,\nabla^{\mathrm{tr}} \delta \uu 	\rangle_{H^{-\frac{1}{2}}}
		&= 
		\J_1 + \J_2 + \J_3 + \J_4,\\
		- \langle ( \delta A)\ddd_1,\, \Delta \delta \ddd\rangle_{H^{-\frac{1}{2}}}
		&=
		\mathcal{K}_1+\mathcal{K}_2+\mathcal{K}_3+\mathcal{K}_4,
	\end{aligned}
	\end{equation}
	where
	\begin{align*}	
		\I_1&:= -\sum_{q=-1}^\infty\sum_{|j-q|\leq 5} 2^{-q} 
		\int_{\TT^2} ([ \Dd_q,\,\sdm{j}\ddd_1]\cdot \Dd_j \nabla \delta \uu)\cdot 
		\Dd_q\Delta \delta \ddd\,\dd x,\\
		\I_2&:=-\sum_{q=-1}^\infty \sum_{|j-q|\leq 5} 2^{-q} 
		\int_{\TT^2}(  (\Dd_q\Dd_j  \nabla \delta \uu)  (\sdm{j}-\sdm{q})\ddd_1 )
		\cdot \Dd_q\Delta \delta \ddd\,\dd x,\\
		\I_3&:=-\sum_{q=-1}^\infty 2^{-q} \int_{\TT^2}((\Dd_q \nabla \delta \uu)\sdm{q}\ddd_1)\cdot 
		\Dd_q\Delta \delta \ddd\,\dd x,\\
		\I_4 &:=
		 -\sum_{q=-1}^\infty \sum_{j = q - 5}^\infty 2^{-q} 
		 \int_{\TT^2}\Dd_q( (\Sd_{j+2} \nabla \delta \uu)\Dd_j \ddd_1)
		 \cdot \Dd_q\Delta \delta \ddd\,\dd x,
	\end{align*}
	\begin{align*}	
		\J_1&:= \sum_{q=-1}^\infty\sum_{|j-q|\leq 5} 2^{-q} 
		\int_{\TT^2} ([ \Dd_q,\,\sdm{j}\ddd_1]\otimes \Dd_j \Delta \delta \ddd): 
		\Dd_q\nabla^{\mathrm{tr}} \delta \uu\,\dd x,\\
		\J_2&:=\sum_{q=-1}^\infty \sum_{|j-q|\leq 5} 2^{-q} 
		\int_{\TT^2}((\sdm{j}-\sdm{q})\ddd_1\otimes \Dd_q\Dd_j  \Delta \delta \ddd )
		: \Dd_q\nabla^{\mathrm{tr}} \delta \uu\,\dd x,\\
		\J_3&:=\sum_{q=-1}^\infty 2^{-q} \int_{\TT^2}(\sdm{q}\ddd_1\otimes \Dd_q \Delta \delta \ddd): 
		\Dd_q\nabla^{\mathrm{tr}} \delta \uu\,\dd x,\\
		\J_4 &:=
		 -\sum_{q=-1}^\infty \sum_{j = q - 5}^\infty 2^{-q} 
		 \int_{\TT^2}\Dd_q( \Dd_j \ddd_1\otimes \Sd_{j+2} \Delta \delta \ddd )
		 \cdot \Dd_q\nabla^{\mathrm{tr}} \delta \uu\,\dd x,
	\end{align*}
	\begin{align*}	
		\mathcal{K}_1&:= -\sum_{q=-1}^\infty\sum_{|j-q|\leq 5} 2^{-q} 
		\int_{\TT^2} ([ \Dd_q,\,\sdm{j}\ddd_1]\cdot \Dd_j \delta A)\cdot 
		\Dd_q\Delta \delta \ddd\,\dd x,\\
		\mathcal{K}_2&:=-\sum_{q=-1}^\infty \sum_{|j-q|\leq 5} 2^{-q} 
		\int_{\TT^2}( (\Dd_q\Dd_j  \delta A)(\sdm{j}-\sdm{q})\ddd_1)
		\cdot \Dd_q\Delta \delta \ddd\,\dd x,\\
		\mathcal{K}_3&:=-\sum_{q=-1}^\infty 2^{-q} \int_{\TT^2}((\Dd_q \delta A)\sdm{q}\ddd_1)\cdot 
		\Dd_q\Delta \delta \ddd\,\dd x,\\
		\mathcal{K}_4 &:=
		 -\sum_{q=-1}^\infty \sum_{j = q - 5}^\infty 2^{-q} 
		 \int_{\TT^2}\Dd_q( (\Sd_{j+2} \delta A)\Dd_j \ddd_1 )
		 \cdot \Dd_q\Delta \delta \ddd\,\dd x.
	\end{align*}

First, we observe that 
$$ \I_3+\J_3 = 0,$$ 
	which is somehow a reminiscent of the intrinsic cancellations we 
	encounter when deriving the energy inequality \eqref{def:energy-inequality-tot} of the system \eqref{main_system}. This allows us to avoid estimating some bad terms that involve the higher order derivatives of the difference of solutions $\nabla \delta \uu$, $\Delta \delta \ddd$ and moreover, strongly depend on the low frequencies of $\ddd_1$, i.e., $\sdm{q}\ddd_1$. 
	For the same reason, there seems no good way to handle $\mathcal{K}_3$ but just by the Cauchy-Schwarz inequality:
	\begin{equation} \label{K3aa}
	\begin{aligned}
	 |\mathcal{K}_3| 
	 &\leq   \sum_{q=-1}^\infty 2^{-q} \int_{\TT^2}|(\Dd_q \delta A)\sdm{q}\ddd_1|^2\,\dd x
		+
		\frac14 \sum_{q=-1}^\infty 2^{-q} \int_{\TT^2}|\Dd_q\Delta \delta \ddd|^2\,\dd x \\
		&=   \sum_{q=-1}^\infty 2^{-q} \int_{\TT^2}|(\Dd_q \delta A)\sdm{q}\ddd_1|^2\,\dd x
		+\frac14 \|\Delta \delta \ddd\|_{H^{-\frac12}}^2.
	\end{aligned}
	\end{equation}
The right-hand side of \eqref{K3aa} cannot be chosen as small as possible (e.g., via a small parameter $\eta\in (0,1)$) and have to be controlled by using certain hidden dissipation of the system (see Proposition \ref{prop:T3} in the next section).

Next, we treat the other nine terms. The basic idea is, thanks to the absence of the low frequencies of $\ddd_1$, we are able to transfer some derivatives to suitable components of the product. To this end, we apply the commutator estimate (Lemma \ref{prop:comm-est}) and Lemma \ref{prop:Bernstein} to $\I_1$, $\J_1$ and $\mathcal{K}_1$, to obtain  
	\begin{align*}
		& |\I_1|+|\J_1|+|\mathcal{K}_1|\\
		&\qquad \lesssim \sum_{q=-1}^\infty\sum_{|j-q|\leq 5} 
		2^{-q}
			\Big(
				\| [\Dd_q,\,\sdm{j}\ddd_1 ]\cdot \Dd_j \nabla \delta \uu 	\|_{L^2}
				\| \Dd_q \Delta \delta \ddd 								\|_{L^2}+
				\| [\Dd_q,\,\sdm{j}\ddd_1 ]\otimes\Dd_j\Delta \delta \ddd	\|_{L^2}
				\| \Dd_q \nabla \delta \uu	 								\|_{L^2}
			\Big)\\
			&\qquad \qquad 
			+ \sum_{q=-1}^\infty\sum_{|j-q|\leq 5} 
		2^{-q}
			\| [\Dd_q,\,\sdm{j}\ddd_1 ]\cdot \Dd_j \delta A	\|_{L^2}
				\| \Dd_q \Delta \delta \ddd 								\|_{L^2}\\
		&\qquad \lesssim \sum_{q=-1}^\infty\sum_{|j-q|\leq 5} 2^{-q}    
			\Big(
				\| \sdm{j}\nabla \ddd_1			\|_{L^4} 
				\| \Dd_j \delta \uu				\|_{L^4}
				\| \Dd_q \Delta \delta \ddd 	\|_{L^2}
				+
				\| \sdm{j}\nabla \ddd_1			\|_{L^4} 
				\| \Dd_j \nabla \delta \ddd		\|_{L^4}
				\| \Dd_q \nabla \delta \uu 		\|_{L^2}
			\Big).
	\end{align*}
	Let us recall that the uniform bound 
	$\| \sdm{j}\nabla \ddd_1 \|_{L^4}\leq C\| \nabla \ddd_1 \|_{L^4}$ holds for any 
	$j\in\mathbb{N}\cup \{-1\}$, for a suitable constant $C>0$. As a consequence, we have
	\begin{equation}\label{ineq:T2-I1J1-middle}
	\begin{aligned}
		|\I_1|+|\J_1|+|\mathcal{K}_1|		
		&\lesssim 
		\|  \nabla \ddd_1 \|_{L^4}
		\sum_{q=-1}^\infty\sum_{|j-q|\leq 5} 2^{-q} 
		\Big( 
		\| \Dd_j \delta \uu					\|_{L^4}
		\| \Dd_q \Delta \delta \ddd 		\|_{L^2} + 
		\| \Dd_j \nabla \delta \ddd			\|_{L^4}
		\| \Dd_q \nabla \delta  \uu		 	\|_{L^2}
		\Big)\\
     	&\lesssim 
		\|  \nabla \ddd_1 \|_{L^4}
		\sum_{q=-1}^\infty\sum_{|j-q|\leq 5} 2^{-q}  
		\Big(\| \Dd_j \delta \uu					\|_{H^\frac{1}{2}}
		\| \Dd_q \Delta \delta \ddd 		\|_{L^2} + 
		\| \Dd_j \nabla \delta \ddd		\|_{H^\frac{1}{2}}
		\| \Dd_q \nabla \delta \uu \|_{L^2}
		\Big),
	\end{aligned}
	\end{equation}
thanks to the Sobolev embedding $H^{1/2}(\TT^2)\hookrightarrow L^4(\TT^2)$. Here we remark that both $\Dd_j\delta \uu$ and $\Dd_j\nabla \delta \ddd$ 
	have null averages for any $j\in \mathbb{N}\cup \{-1\}$, 
	hence $\| \Dd_j \delta \uu\|_{H^\frac{1}{2}}\leq \| \Dd_j \delta \uu \|_{L^2}^\frac{1}{2}
	\| \Dd_j \nabla \delta \uu \|_{L^2}^\frac{1}{2}$ and 
	$\| \Dd_j \nabla \delta \ddd\|_{H^\frac{1}{2}}\leq \| \Dd_j \nabla \delta \ddd \|_{L^2}^\frac{1}{2}
	\| \Dd_j \Delta \delta \ddd \|_{L^2}^\frac{1}{2}$. 
	This observation together with the condition 
	$|j-q|\leq 5$ yields that
	\begin{align*}
		|\I_1|+|\J_1|+|\mathcal{K}_1|
		&\lesssim 
		\|  \nabla \ddd_1 \|_{L^2}^\frac{1}{2}
		\|  \nabla \ddd_1 \|_{H^1}^\frac{1}{2}\\
		&\qquad \times 
		\sum_{q=-1}^\infty\sum_{|j-q|\leq 5} 2^{-q} \Big(  
		\| \Dd_j \delta \uu					\|_{L^2}^\frac{1}{2}
		\| \Dd_j \nabla \delta \uu			\|_{L^2}^\frac{1}{2}
		\| \Dd_q \Delta \delta \ddd 		\|_{L^2} +
		\| \Dd_j \nabla \delta \ddd 	\|_{L^2}^\frac{1}{2}
		\| \Dd_j \Delta \delta \ddd 	\|_{L^2}^\frac{1}{2}
		\| \Dd_q \nabla\delta  \uu 		\|_{L^2}\Big)\\
		&\lesssim 
		\|  \nabla \ddd_1 \|_{L^2}^\frac{1}{2}
		\|  \nabla \ddd_1 \|_{H^1}^\frac{1}{2}
		\bigg[
		\Big(
		\sum_{j=-1}^\infty 
		2^{-j}
		\| \Dd_j \delta \uu					\|_{L^2}^2
		\Big)^\frac{1}{4} 
		\Big(
		\sum_{j=-1}^\infty
		2^{-j}
		\| \Dd_j \nabla \delta \uu			\|_{L^2}^2
		\Big)^\frac{1}{4} 
		\Big(
		\sum_{q=-1}^\infty
		2^{-q}
		\| \Dd_q \Delta \delta \ddd			\|_{L^2}^2
		\Big)^\frac{1}{2}    \\ 
		&\hspace{3cm}+
		\Big(
		\sum_{j=-1}^\infty 
		2^{-j}
		\| \Dd_j \nabla \delta \ddd					\|_{L^2}^2
		\Big)^\frac{1}{4} 
		\Big(
		\sum_{j=-1}^\infty
		2^{-j}
		\| \Dd_j \Delta \delta \ddd			\|_{L^2}^2
		\Big)^\frac{1}{4} 
		\Big(
		\sum_{q=-1}^\infty
		2^{-q}
		\| \Dd_q \nabla \delta \uu			\|_{L^2}^2
		\Big)^\frac{1}{2}
		\bigg]
		\\
		&\lesssim 
		\|  \nabla \ddd_1 			\|_{L^2}^\frac{1}{2}
		\|  \nabla  \ddd_1 			\|_{H^1}^\frac{1}{2}
		\Big(
		\|	\delta \uu	  			\|_{H^{-\frac{1}{2}}}^\frac{1}{2}
		\|	\nabla \delta \uu	  	\|_{H^{-\frac{1}{2}}}^\frac{1}{2}
		\|	\Delta \delta \ddd	  	\|_{H^{-\frac{1}{2}}}
		+
		\|	\nabla \delta \ddd	  			\|_{H^{-\frac{1}{2}}}^\frac{1}{2}
		\|	\Delta \delta \ddd	  	\|_{H^{-\frac{1}{2}}}^\frac{1}{2}
		\|	\nabla \delta \uu	  	\|_{H^{-\frac{1}{2}}}
		\Big),
	\end{align*}
	which finally leads to
	\begin{equation}\label{ineqT1-I1J1}
		|\I_1|+|\J_1|+|\mathcal{K}_1|
		\leq 
		C\eta^{-3}
		\| \nabla \ddd_1	\|_{L^2}^2
		\| \nabla  \ddd_1	\|_{H^1}^2 
		\big(\| \delta \uu 		\|_{H^{-\frac{1}{2}}}^2 +
		\|\delta\ddd\|_{H^\frac12}^2\big)
		+ \eta \nu \| \nabla \delta \uu \|_{H^{-\frac{1}{2}}}^2 + 
		\eta\| \Delta \delta \ddd \|_{H^{-\frac{1}{2}}}^2.
	\end{equation}	
	
	Let us now turn to the estimate of $\I_2$,  $\J_2$ and $\mathcal{K}_2$. It follows that   
	\begin{align*}
		|\I_2|+|\J_2|+|\mathcal{K}_2|
		&\lesssim \sum_{q=-1}^\infty\sum_{|j-q|\leq 5} 2^{-q}  
			\| (\Dd_q\Dd_j \nabla \delta \uu) (\sdm{j}-\sdm{q})\ddd_1 		\|_{L^2}
				\| \Dd_q \Delta \delta \ddd 										\|_{L^2} \\
					&\qquad 
				+\sum_{q=-1}^\infty\sum_{|j-q|\leq 5} 2^{-q} 	\| (\sdm{j}-\sdm{q})\ddd_1 \otimes \Dd_q\Dd_j \Delta\delta \ddd 	\|_{L^2}
				\| \Dd_q \nabla \delta \uu	 										\|_{L^2}
			\\
			&\qquad 
			+ \sum_{q=-1}^\infty\sum_{|j-q|\leq 5} 2^{-q}  
			\| (\Dd_q\Dd_j  \delta A) (\sdm{j}-\sdm{q})\ddd_1 		\|_{L^2}
				\| \Dd_q \Delta \delta \ddd 										\|_{L^2}\\ 
		&\lesssim \sum_{q=-1}^\infty\sum_{|j-q|\leq 5} 2^{-q}  
			\Big(
				\|(\sdm{j}-\sdm{q})\ddd_1		\|_{L^4} 
				\| \Dd_q\Dd_j \nabla\delta \uu	\|_{L^4}
				\| \Dd_q \Delta \delta \ddd 	\|_{L^2}\\
		&\qquad \qquad 	\qquad \qquad
		+ \|(\sdm{j}-\sdm{q})\ddd_1			\|_{L^4} 
				\| \Dd_q\Dd_j \Delta \delta \ddd	\|_{L^4}
				\| \Dd_q \nabla \delta \uu 			\|_{L^2}
			\Big).
	\end{align*}
	Since the indexes $j\in \mathbb{N}\cup \{-1\}$ and $q\in \mathbb{N}\cup \{-1\}$ satisfy $|j-q|\leq 5$, 
	the Fourier transform of the difference  $\sdm{j}\ddd_1 - \sdm{q}\ddd_1$ between low frequencies 
	is localised on an annulus whose radius is proportional to $2^q$. Therefore, it holds 
	$\| \sdm{j}\ddd_1 - \sdm{q}\ddd_1\|_{L^4}\leq C2^{-q}
	\| \sdm{j}\nabla\ddd_1 - \sdm{q}\nabla\ddd_1\|_{L^4}$, which implies that
	\begin{align*}
		|\I_2|+|\J_2|+|\mathcal{K}_2|
		&\lesssim
		\sum_{q=-1}^\infty\sum_{|j-q|\leq 5} 2^{-q}  
			\Big(
				2^{-q}
				\|(\sdm{j}-\sdm{q})\nabla \ddd_1		\|_{L^4} 
				2^q
				\| \Dd_q\Dd_j \delta \uu	\|_{L^4}
				\| \Dd_q \Delta \delta \ddd 	\|_{L^2}
				\\
				&\hspace{3cm}+
				2^{-q}
				\|(\sdm{j}-\sdm{q})\nabla \ddd_1		\|_{L^4} 
				2^q
				\| \Dd_q \Dd_j \nabla \delta \ddd		\|_{L^4}
				\| \Dd_q \nabla \delta \uu 		\|_{L^2}
			\Big)\\
		&\lesssim
		\|\nabla \ddd_1		\|_{L^4} 
		\sum_{q=-1}^\infty\sum_{|j-q|\leq 5} 2^{-q}  
			\Big(
				\| \Dd_j \delta \uu	\|_{L^4}
				\| \Dd_q \Delta \delta \ddd 	\|_{L^2}
				+
				\| \Dd_j \nabla \delta \ddd		\|_{L^4}
				\| \Dd_q \nabla \delta \uu 		\|_{L^2}
			\Big)	,
	\end{align*}
	thanks to the Bernstein inequalities in Lemma \ref{prop:Bernstein}. 
	We note that the right-hand side of the 
	above inequality coincides with the one of inequality \eqref{ineq:T2-I1J1-middle}. Hence, proceeding as for 
	the proof of \eqref{ineqT1-I1J1}, we deduce that 
	\begin{equation}\label{ineq:T1-I2J2}
		|\I_2|+|\J_2|+|\mathcal{K}_2|
		\leq   
		C\eta^{-3}
		\| \nabla \ddd_1	\|_{L^2}^2
		\| \nabla  \ddd_1	\|_{H^1}^2 
		\big(\| \delta \uu 		\|_{H^{-\frac{1}{2}}}^2 +
		\|\delta\ddd\|_{H^\frac12}^2\big) 
		+ \eta \nu \| \nabla \delta \uu \|_{H^{-\frac{1}{2}}}^2 + 
		\eta\| \Delta \delta \ddd \|_{H^{-\frac{1}{2}}}^2.
	\end{equation}
	
	At last, we address the estimate for $\I_4$, $\J_4$ and $\mathcal{K}_4$. Using the following inequality that is a direct consequence of Lemma \ref{prop:Bernstein}: 
	$\| \Dd_q f \|_{L^2}\leq 2^q \| \Dd_q f \|_{L^1}$ for any $q\in\mathbb{N}\cup \{-1\}$, we deduce that 
	\begin{align*}
		&|\I_4|+|\J_4|+|\mathcal{K}_4|\\
		&\qquad \lesssim 
		\sum_{q=-1}^\infty\sum_{j= q - 5}^\infty 2^{-q}  
		\Big(
		\| \Dd_q( \Sd_{j+2}\nabla \delta \uu \Dd_j \ddd_1)\|_{L^2} \| \Dd_q \Delta \delta\ddd \|_{L^2}
		+
		\| \Dd_q(\Dd_j \ddd_1\otimes 
		\Sd_{j+2}\Delta \delta \ddd)\|_{L^2} \| \Dd_q \nabla\delta\uu \|_{L^2}
		\Big)
		\\
		&\qquad \qquad 
		+ \sum_{q=-1}^\infty\sum_{j= q - 5}^\infty 2^{-q}  
		\| \Dd_q( \Sd_{j+2}\delta A \Dd_j \ddd_1 )\|_{L^2} \| \Dd_q \Delta \delta\ddd \|_{L^2}\\
		&\qquad \lesssim
		\sum_{q=-1}^\infty\sum_{j= q - 5}^\infty 
		2^{-q}
		\Big(
		2^q
		\| \Dd_q( \Sd_{j+2}\nabla \delta \uu \Dd_j \ddd_1)\|_{L^1} 
		\| \Dd_q \Delta\delta \ddd \|_{L^2}
		+
		2^q
		\| \Dd_q(\Dd_j \ddd_1\otimes \Sd_{j+2}\Delta \delta \ddd )\|_{L^1} 
		\| \Dd_q \nabla\delta \uu \|_{L^2}
		\Big)\\
		&\qquad \qquad 
		+\sum_{q=-1}^\infty\sum_{j= q - 5}^\infty 
		2^{-q}
		\Big(2^q
		\| \Dd_q( \Sd_{j+2}\delta A \Dd_j \ddd_1 )\|_{L^1} 
		\| \Dd_q \Delta\delta \ddd \|_{L^2}
		\Big)\\
		&\qquad \lesssim
		\sum_{q=-1}^\infty\sum_{j=q - 5}^{\infty}
		\Big(\|\Dd_j \ddd_1\|_{L^2}
		\|\Sd_{j+2}\nabla \delta \uu\|_{L^2}
		2^q \| \Dd_q \nabla\delta \ddd \|_{L^2}
		+ \|\Dd_j \ddd_1\|_{L^2} \|\Sd_{j+2}\Delta \delta \ddd\|_{L^2}
		2^q \| \Dd_q \delta \uu \|_{L^2} 
		\Big)\\
		&\qquad \lesssim
		\sum_{q=-1}^\infty\sum_{j=q - 5}^\infty 2^{\frac{3}{2}(q-j)}2^{-\frac{j}{2}}
		\| \Sd_{j+2}\nabla \delta \uu \|_{L^2} 
		2^{2j}
		\|\Dd_j \ddd_1\|_{L^2}
		2^{-\frac{q}{2}} \| \Dd_q \nabla\delta \ddd \|_{L^2}\\
		&\qquad \qquad 
		+ \sum_{q=-1}^\infty\sum_{j= q - 5}^\infty  2^{\frac{3}{2}(q-j)}2^{-\frac{j}{2}}
		\| \Sd_{j+2}\Delta \delta \ddd \|_{L^2} 
		2^{2j}
		\|\Dd_j \ddd_1\|_{L^2}
		2^{-\frac{q}{2}} \| \Dd_q \delta \uu \|_{L^2} 
		\\
		&\qquad \lesssim
		\Big(\sum_{q=-1}^\infty\sum_{j= q - 5}^\infty 
		2^{3(q-j)}2^{-j}
		\| \Sd_{j+2}\nabla \delta \uu \|_{L^2}^2
		\|\ddd_1\|_{H^2}^2
		\Big)^\frac{1}{2}
		\Big(
		\sum_{q=-1}^\infty
			2^{-q} 
			\| \Dd_q \nabla\delta \ddd \|_{L^2}^2
		\Big)^\frac{1}{2}\\
		&\qquad\qquad  
		+\Big(\sum_{q=-1}^\infty\sum_{j= q - 5}^\infty 
		2^{3(q-j)}2^{-j}
		\| \Sd_{j+2}\Dd \delta \ddd \|_{L^2}^2
		\|\ddd_1\|_{H^2}^2
		\Big)^\frac{1}{2}
		\Big(
		\sum_{q=-1}^\infty
			2^{-q} 
			\| \Dd_q \delta \uu \|_{L^2}^2
		\Big)^\frac{1}{2}
		\\
		&\qquad \lesssim
		\| \ddd_1\|_{H^2}
		\Big(\sum_{q=-\infty}^\infty\sum_{j=-\infty}^\infty 2^{3(q-j)}\mathfrak{1}_{(-\infty,5]}(q-j)
		2^{-j}
		\| \Sd_{j+2}\nabla \delta \uu \|_{L^2}^2
		\mathfrak{1}_{[-6,\infty)}(j)
		\Big)^\frac{1}{2}
		\| \nabla \delta \ddd \|_{H^{-\frac{1}{2}}}\\
		&\qquad \qquad +
		\| \ddd_1\|_{H^2}
		\Big(\sum_{q=-\infty}^\infty\sum_{j=-\infty}^\infty 2^{3(q-j)}\mathfrak{1}_{(-\infty,5]}(q-j)
		2^{-j}
		\| \Sd_{j+2}\Dd \delta \ddd \|_{L^2}^2
		\mathfrak{1}_{[-6,\infty)}(j)
		\Big)^\frac{1}{2}
		\| \delta \uu \|_{H^{-\frac{1}{2}}}.
	\end{align*}
	At this stage we can invoke Young's inequality for series to conlude that 
	\begin{equation}\label{ineqT1-I4J4}
	\begin{aligned} 
		|\I_4|+|\J_4|+|\mathcal{K}_4|
		&\lesssim  
		 \| \ddd_1\|_{H^2}\| \nabla \delta \uu \|_{H^{-\frac{1}{2}}}\| \delta \ddd \|_{H^\frac{1}{2}}
		 +\| \ddd_1\|_{H^2} \| \Dd \delta \ddd \|_{H^{-\frac{1}{2}}} \| \delta \uu \|_{H^{-\frac{1}{2}}}\\
		& \leq
		C\eta^{-1} \| \ddd_1\|_{H^2}^2
		( \| \delta \uu \|_{H^{-\frac{1}{2}}}^2 + \| \delta \ddd \|_{H^\frac{1}{2}}^2 \big) +
		\eta \nu \| \nabla \delta \uu \|_{H^{-\frac{1}{2}}}^2
		+\eta \| \Dd \delta \ddd \|_{H^{-\frac{1}{2}}}^2.
		\end{aligned} 
	\end{equation}
	Collecting the inequalities
	\eqref{K3aa}, \eqref{ineqT1-I1J1}, \eqref{ineq:T1-I2J2} and \eqref{ineqT1-I4J4}, we obtain the required inequality \eqref{prop:T2-ineq3}. This completes the proof of Proposition \ref{prop:T2}.
\end{proof}

\section{\bf Estimate of $\mathcal{T}_3$}\label{sec:single-log-est}
\setcounter{equation}{0}

In this section, we proceed to estimate the term $\mathcal{T}_3$ given by \eqref{Xi3}, that is 
\begin{equation}\label{def_T3}
\begin{aligned}
	\mathcal{T}_3 & :=
	-
	\langle
		\ddd_1 \otimes (\delta A \ddd_1)
		,\,
		\nabla \delta \uu
	\rangle_{H^{-\frac{1}{2}}} 
	-
	\langle 
		\ddd_1 \otimes (A_2 \delta \ddd),\,
		\nabla \delta \uu
	\rangle_{H^{-\frac{1}{2}}}
	-    \langle
		(\delta A \ddd_1)\otimes \ddd_1  
		,\,
		\nabla \delta \uu
	\rangle_{H^{-\frac{1}{2}}} 
	\\
	&\quad\  -	\langle  (A_2 \delta \ddd)\otimes \ddd_1 ,\,
		\nabla \delta \uu
	\rangle_{H^{-\frac{1}{2}}}.
\end{aligned}
\end{equation}
As we shall see below, some further efforts have to be made in order to derive estimates for these inner products. Indeed, we observe that they depend on nonlinear terms like $\ddd_1 \otimes (\delta A \ddd_1)$ and $\ddd_1 \otimes (A_2 \delta \ddd)$ from the momentum equation, which are of higher order comparing with the ones that we have encountered in the previous sections. This feature translates into a major difficulty in the subsequent estimates, which we will overcome by seeking a double-logarithmic type inequality as expressed in the modulus of continuity $\mu(\Phi(t))$ (recall \eqref{muaa}). 

More precisely, we have 
\begin{prop}\label{prop:T3}
There exists a positive function $f_3 \in L^1(0,T)$  such that for any $\eta\in (0,1)$, the following inequality holds
\begin{equation}\label{prop:T3-ineq}
 \begin{aligned}
 	&\Big|
 	\mathcal{T}_3(t)+
 	2
 	\sum_{q=-1}^\infty 2^{-q}\int_{\TT^2}|\Dd_q\delta A \sdm{q}\ddd_1|^2 \,\dd x	
 	\Big| \\
 	&\qquad \leq 
	C_3 f_3(t) \mu(\Phi(t)) 
    +  30\eta\, \widehat{\mathfrak{D}}(t) + 
	10 \eta \sum_{q=-1}^\infty
	2^{-q}
	\int_{\TT^2}
	|\Dd_q \delta A\sdm{q}\ddd_1|^2\,\dd x,
 \quad \text{for almost all } t\in (0,T),
 \end{aligned}
 \end{equation}
 where $\mu$ is defined as in \eqref{muaa} and $C_3>0$ is a constant depending on $\eta^{-1}$.
\end{prop}
\begin{proof}[Proof of Proposition \ref{prop:T3}]
Let $\eta\in (0,1)$. We start the proof by handling the first and third inner products in \eqref{def_T3}, namely,
\begin{align*} 
	&-\langle
		\ddd_1 \otimes (\delta A \ddd_1)
		,\,
		\nabla \delta \uu
	\rangle_{H^{-\frac{1}{2}}}	-    \langle
		(\delta A \ddd_1)\otimes \ddd_1  
		,\,
		\nabla \delta \uu
	\rangle_{H^{-\frac{1}{2}}} \\
	&\quad 
	=
	-\sum_{q=-1}^\infty2^{-q}
	\int_{\TT^2}\Dd_q(	\ddd_1 \otimes (\delta A \ddd_1)):\Dd_q \nabla \delta \uu \, \mathrm{d}x
	-\sum_{q=-1}^\infty2^{-q}
	\int_{\TT^2}\Dd_q(	(\delta A \ddd_1)\otimes \ddd_1 ):\Dd_q \nabla \delta \uu \, \mathrm{d}x.
	\end{align*}
Our aim is to show that
\begin{equation}\label{new-est-T3-n1}
\begin{aligned}
    &	-\langle
		\ddd_1 \otimes (\delta A \ddd_1)
		,\,
		\nabla \delta \uu
	\rangle_{H^{-\frac{1}{2}}} 
	-    \langle
		(\delta A \ddd_1)\otimes \ddd_1  
		,\,
		\nabla \delta \uu
	\rangle_{H^{-\frac{1}{2}}}\\
	&\quad \leq 
	- 2\sum_{q=-1}^\infty 2^{-q}\int_{\TT^2}|\Dd_q\delta A \sdm{q}\ddd_1|^2 \mathrm{d}x
	+ C\eta^{-3} \widetilde{f_3}(t) \mu(\Phi(t))  + 28\eta\, \widehat{\mathfrak{D}}(t) +10 \eta \sum_{q=-1}^\infty 2^{-q}\int_{\TT^2}|\Dd_q\delta A \sdm{q}\ddd_1|^2 \mathrm{d}x,
\end{aligned}
\end{equation}
where 
\begin{equation} \label{f3_1}
\begin{aligned}
 \widetilde{f_3}(t) &=  \big(1+ \| \ddd_1 (t)\|_{H^1}^6 \big) \| \ddd_1(t) \|_{H^2}^2
	+
	\big(\| \uu_1 (t)\|_{L^2}^2 + \| \uu_2 (t)\|_{L^2}^2\big)\big( \| \nabla \uu_1 (t)\|_{L^2}^2 + \| \nabla \uu_2(t) \|_{L^2}^2\big).
\end{aligned}
\end{equation}
satisfying $\widetilde{f_3}(t)\in L^1(0,T)$.

The argument relies on Bony's decomposition \eqref{bony-decomp} as in the proof for Proposition \ref{prop:T2}, however, with some additional technical difficulties to overcome. To this end, we choose $f=\ddd_1$ and $g=\delta A\ddd_1$ in  \eqref{bony-decomp}, which allow us to split the first inner product into four terms such that 
\begin{equation}\label{decomp_deltaAd1od1}
	-\langle
		\ddd_1 \otimes (\delta A \ddd_1)
		,\,
		\nabla \delta \uu
	\rangle_{H^{-\frac{1}{2}}}
	=
	\I+\I\I+\I\I\I+\I\mathcal{V},
\end{equation}
where
\begin{align*}
	\I
	&:=
	-
	\sum_{q=-1}^\infty
	\sum_{|j-q|\leq 5}
	2^{-q}
	\int_{\TT^2}\big([\Dd_q,\,\sdm{j}\ddd_1] \otimes \Dd_j(\delta A \ddd_1)\big):\Dd_q \nabla \delta \uu\,\dd x,\\
	\I\I
	&:=
	-
	\sum_{q=-1}^\infty
	\sum_{|j-q|\leq 5}
	2^{-q}
	\int_{\TT^2}\big((\sdm{j}\ddd_1-\sdm{q}\ddd_1) \otimes \Dd_q\Dd_j(\delta A \ddd_1)\big):\Dd_q \nabla \delta \uu\,\dd x,\\
	\I\I\I
	&:=
	-
	\sum_{q=-1}^\infty
	2^{-q}
	\int_{\TT^2}\big( \sdm{q}\ddd_1 \otimes \Dd_q(\delta A \ddd_1)\big) 
	: \Dd_q(\nabla \delta \uu) \,\dd x,\\
	\I\mathcal{V}
	&:=
	-
	\sum_{q=-1}^\infty
	\sum_{j= q-5}^\infty
	2^{-q}
	\int_{\TT^2} \Dd_q\big(\Dd_{j}\ddd_1\otimes S_{j+2}(\delta A \ddd_1)\big): \Dd_q \nabla \delta \uu\,\dd x.
\end{align*}
The key point is that, later for each of the terms $\I$, $\I\I$, $\I\I\I$ and $\I\mathcal{V}$, we have to make use of the Bony's decomposition one more time. In a similar manner, we denote 
\begin{equation}\label{decomp_deltaAd1od1aa}
	-\langle
		(\delta A \ddd_1)\otimes \ddd_1 
		,\,
		\nabla \delta \uu
	\rangle_{H^{-\frac{1}{2}}}
	=
	\I'+\I\I'+\I\I\I'+\I\mathcal{V}',
\end{equation}
where
\begin{align*}
	\I'
	&:=
	-
	\sum_{q=-1}^\infty
	\sum_{|j-q|\leq 5}
	2^{-q}
	\int_{\TT^2}\big( \Dd_j(\delta A \ddd_1)\otimes [\Dd_q,\,\sdm{j}\ddd_1] \big):\Dd_q \nabla \delta \uu\,\dd x,\\
	\I\I'
	&:=
	-
	\sum_{q=-1}^\infty
	\sum_{|j-q|\leq 5}
	2^{-q}
	\int_{\TT^2}\big( \Dd_q\Dd_j(\delta A \ddd_1) \otimes (\sdm{j}\ddd_1-\sdm{q}\ddd_1) \big):\Dd_q \nabla \delta \uu\,\dd x,\\
	\I\I\I'
	&:=
	-
	\sum_{q=-1}^\infty
	2^{-q}
	\int_{\TT^2}\big(  \Dd_q(\delta A \ddd_1) \otimes \sdm{q}\ddd_1 \big) 
	: \Dd_q(\nabla \delta \uu) \,\dd x,\\
	\I\mathcal{V}'
	&:=
	-
	\sum_{q=-1}^\infty
	\sum_{j= q-5}^\infty
	2^{-q}
	\int_{\TT^2} \Dd_q\big( S_{j+2}(\delta A \ddd_1) \otimes \Dd_{j}\ddd_1\big): \Dd_q \nabla \delta \uu\,\dd x.
\end{align*}

Let us begin with the analysis of $\I\I\I$ and $\I\I\I'$, from which some additional dissipative contributions will appear. This fact will be crucial to control the ``bad'' term $\mathcal{K}_3$ in the previous section (see \eqref{K3aa}). We decompose the term $\Dd_q(\delta A \ddd_1)$ by using \eqref{bony-decomp} and obtain 
\begin{equation*}
	\I\I\I = \I\I\I_1+ \I\I\I_2+ \I\I\I_3+ \I\I\I_4,
\end{equation*}
where
\begin{equation}\label{def-III1234}
\begin{aligned}
	\I\I\I_1
	&:=
	-\sum_{q=-1}^\infty
	\sum_{|j-q|\leq 5}
	2^{-q}
	\int_{\TT^2}\big([\Dd_q,\, \sdm{j}\ddd_1]\cdot \Dd_j \delta A\big) \cdot \big(\Dd_q (\nabla^{\mathrm{tr}} \delta \uu)\sdm{q}\ddd_1\big)\,\dd x,\\
	\I\I\I_2
	&:=
	-\sum_{q=-1}^\infty
	\sum_{|j-q|\leq 5}
	2^{-q}
	\int_{\TT^2}\big((\sdm{j}-\sdm{q})\ddd_1\cdot \Dd_q\Dd_j \delta A\big) \cdot \big(\Dd_q (\nabla^{\mathrm{tr}} \delta \uu)\sdm{q}\ddd_1\big)\,\dd x,\\
	\I\I\I_3
	&:=
	-
	\sum_{q=-1}^\infty
	2^{-q}
	\int_{\mathbb{T}^2} \big(\sdm{q}\ddd_1 \otimes (\Dd_q\delta A \sdm{q} \ddd_1)\big) :\Dd_q(\nabla \delta \uu)\,\dd x,
	\\
	\I\I\I_4
	&:=
	-\sum_{q=-1}^\infty
	\sum_{j=q-5}^\infty 
	2^{-q}
	\int_{\TT^2}\Dd_q(S_{j+2} \delta A\Dd_{j}\ddd_1)\cdot\big(\Dd_q (\nabla^{\mathrm{tr}} \delta \uu)\sdm{q}\ddd_1\big)\,\dd x.
\end{aligned}
\end{equation}
In a similar manner, we write 
\begin{equation*}
	\I\I\I' = \I\I\I_1'+ \I\I\I_2'+ \I\I\I_3'+ \I\I\I_4',
\end{equation*}
where
\begin{equation}\label{def-III1234aa}
\begin{aligned}
	\I\I\I_1'
	&:=
	-\sum_{q=-1}^\infty
	\sum_{|j-q|\leq 5}
	2^{-q}
	\int_{\TT^2}\big([\Dd_q,\, \sdm{j}\ddd_1]\cdot \Dd_j \delta A\big) \cdot \big(\Dd_q (\nabla \delta \uu)\sdm{q}\ddd_1\big)\,\dd x,\\
	\I\I\I_2'
	&:=
	-\sum_{q=-1}^\infty
	\sum_{|j-q|\leq 5}
	2^{-q}
	\int_{\TT^2}\big((\sdm{j}-\sdm{q})\ddd_1\cdot \Dd_q\Dd_j \delta A\big) \cdot \big(\Dd_q (\nabla \delta \uu)\sdm{q}\ddd_1\big)\,\dd x,\\
	\I\I\I_3'
	&:=
	-
	\sum_{q=-1}^\infty
	2^{-q}
	\int_{\mathbb{T}^2} \big((\Dd_q\delta A \sdm{q} \ddd_1)
	\otimes \sdm{q}\ddd_1\big) : \Dd_q(\nabla \delta \uu)\,\dd x,
	\\
	\I\I\I_4'
	&:=
	-\sum_{q=-1}^\infty
	\sum_{j=q-5}^\infty 
	2^{-q}
	\int_{\TT^2}\Dd_q(S_{j+2} \delta A\Dd_{j}\ddd_1)\cdot\big(\Dd_q (\nabla \delta \uu)\sdm{q}\ddd_1\big)\,\dd x.
\end{aligned}
\end{equation}
First, we observe that the term $\I\I\I_3$ together with $\I\I\I_3'$ provides the excepted  additional dissipation. Indeed, it follows that
\begin{align*}
 \I\I\I_3 +  \I\I\I_3'
 &= -
	\sum_{q=-1}^\infty
	2^{-q}
	\int_{\mathbb{T}^2} \big(\sdm{q}\ddd_1 \otimes (\Dd_q\delta A \sdm{q} \ddd_1)
	+ (\Dd_q\delta A \sdm{q} \ddd_1)
	\otimes \sdm{q}\ddd_1\big) : \Dd_q(\delta A+\delta \omega)\,\dd x\\
&=- 2
	\sum_{q=-1}^\infty
	2^{-q}
	\int_{\mathbb{T}^2}|\Dd_q\delta A \sdm{q} \ddd_1|^2 \,\dd x,	
\end{align*}    
since $\Dd_q\delta A\Dd_q\delta \omega = (\Dd_q \nabla\delta \uu)^2/4-(\Dd_q \nabla^{\mathrm{tr}} \delta \uu)^2/4$ is skew adjoint and $\Delta_q\delta A$ is symmetric. 

For $\I\I\I_1$, we introduce two general parameters $N=N(t)\in\mathbb{N}$ and $\ee=\ee(t)\in [0,1/2)$ that will be  
explicitly formulated in the subsequent proof. Then we decompose the director $\ddd_1$ by separating its low and high frequencies such that $\ddd_1 = S_N\ddd_1 + (\Id-S_N)\ddd_1$. Thus, we can split $\I\I\I_1$ into the following form:
\begin{align*}
    \I\I\I_1 = 
    &-\sum_{q=-1}^\infty
	\sum_{|j-q|\leq 5}
	2^{-q}
	\int_{\TT^2}\big([\Dd_q,\, \sdm{j}\ddd_1]\cdot \Dd_j \delta A\big) \cdot \big(\Dd_q (\nabla^{\mathrm{tr}} \delta \uu)\sdm{q}S_N\ddd_1\big)\,\dd x\\
	&-\sum_{q=-1}^\infty
	\sum_{|j-q|\leq 5}
	2^{-q}
	\int_{\TT^2} \big([\Dd_q,\, \sdm{j}\ddd_1]\cdot \Dd_j \delta A\big) \cdot \big(\Dd_q (\nabla^{\mathrm{tr}} \delta \uu)\sdm{q}(\Id-S_N)\ddd_1\big)\,\dd x =:\I\I\I_1^{(1)}+\I\I\I_1^{(2)}.
\end{align*}
We first address the low frequencies part  $\I\I\I_1^{(1)}$ involving $S_N\ddd_1$. Invoking the commutator estimate in Lemma  \ref{prop:comm-est}, we see that 
$\|[\Dd_q,\,\sdm{j}\ddd_1]\cdot \Dd_j\delta A \|_{L^2}
\leq  C 2^{-q}\| \sdm{j} \nabla \ddd_1 \|_{L^{2/\ee}}\| \Dd_j \nabla \delta \uu \|_{L^{2/(1-\ee)}}$. As a consequence, it holds 
\begin{align*}
	|\I\I\I_1^{(1)}|
	&\lesssim 
	\sum_{q=-1}^\infty
	\sum_{|j-q|\leq 5}
	2^{-2q}
	\|\sdm{j}\nabla \ddd_1	        \|_{L^\frac{2}{\ee}}
	\| \Dd_j \nabla \delta \uu	  	\|_{L^\frac{2}{1-\ee}}
	\| \Dd_q \nabla \delta \uu      \|_{L^2}
	\|\sdm{q}S_N\ddd_1              \|_{L^\infty}\\
	&\lesssim
	\|  \nabla \ddd_1	        \|_{L^\frac{2}{\ee}}
	\sum_{q=-1}^\infty
	\sum_{|j-q|\leq 5}
	2^{-2q}
	\| \Dd_j \nabla \delta \uu	  	\|_{L^\frac{2}{1-\ee}}
	\| \Dd_q \nabla \delta \uu      \|_{L^2}
	\|\sdm{q}S_N\ddd_1              \|_{L^\infty},
\end{align*}
where we have used the inequality  $\| \sdm{j}\nabla \ddd_1 \|_{L^{2/\ee}}\lesssim \| \nabla \ddd_1 \|_{L^{2/\ee}}$, for any $j\in \mathbb N\cup \{-1\}$. Next, we make use of the Sobolev embeddings $\| S_N \ddd_1 \|_{L^\infty} \leq C\sqrt{N} \| \ddd_1 \|_{H^1}$ in Lemma \ref{lemma:SN-infty},  
$\| \nabla \ddd_1 \|_{L^{2/\ee}}\leq C\| \nabla \ddd_1 \|_{H^{1-\ee }}/\sqrt{\ee}$ in Lemma \ref{lemma:eps} and the Bernstein inequalities in Lemma \ref{prop:Bernstein}, to gather that	
\begin{align*}
    |\I\I\I_1^{(1)}|
	&\lesssim 
	\sqrt{\frac{N}{\ee }}
	\|		\nabla \ddd_1  \|_{H^{1-\ee }}
	\|  \ddd_1              \|_{H^1}
	\sum_{q=-1}^\infty
	\sum_{|j-q|\leq 5}
	2^{-q}
	\| \Dd_j  \delta \uu	   \|_{L^2}^{1-\ee }
	\| \Dd_j  \delta \uu       \|_{L^\infty}^\ee 
	\| \Dd_q \nabla \delta \uu \|_{L^2}
	\\
	&\lesssim 
	\sqrt{\frac{N}{\ee }}
	\|		\nabla \ddd_1  \|_{H^1}^{1-\ee}
	\|      \ddd_1         \|_{H^1}^{1+\ee}
	\sum_{q=-1}^\infty
	\sum_{|j-q|\leq 5}
	2^{-q}
	\| \Dd_j  \delta \uu	        \|_{L^2}^{1-\ee }
	\| \Dd_j  \nabla \delta \uu     \|_{L^2}^\ee 
	\| \Dd_q \nabla \delta \uu      \|_{L^2}\\
	&\lesssim 
	\sqrt{\frac{N}{\ee }}
	\|		\nabla \ddd_1  \|_{H^1}^{1-\ee}
	\|      \ddd_1         \|_{H^1}^{1+\ee}
    \|      \delta \uu	        \|_{H^{-\frac{1}{2}}}^{1-\ee }
	\|      \nabla \delta \uu   \|_{H^{-\frac{1}{2}}}^{1+\ee }
	\\
	&\leq 
	C\eta^{-\frac{1+\ee}{1-\ee}}
	\|		\nabla \ddd_1		    \|_{H^1}^2
	\|      \ddd_1                  \|_{H^1}^{\frac{2(1+\ee)}{1-\ee}}
	\| \delta \uu                   \|_{H^{-\frac{1}{2}}}^2
	\Big(
	    \frac{N}{\ee}
	\Big)^{\frac{1}{1-\ee}}
	+
	\eta \nu 
	\| \nabla \delta \uu \|_{H^{-\frac{1}{2}}}^2.
\end{align*}
Therefore, it holds 
\begin{equation}\label{ineq-due-to-no-parodi-1}
    |\I\I\I_1^{(1)}|
    \leq  
    C\eta^{-\frac{1+\ee}{1-\ee}}
    \Big(\frac{N}{\ee }\Big)^{\frac{\ee}{1-\ee}}
    \|\nabla \ddd_1			\|_{H^1}^2
	\|		\ddd_1			\|_{H^1}^{\frac{2(1+\ee)}{1-\ee}}
	\| \delta \uu \|_{H^{-\frac{1}{2}}}^2
	\Big(\frac{N}{\ee }\Big)
	+
	\eta\nu 
	\| \nabla \delta \uu \|_{H^{-\frac{1}{2}}}^2.
\end{equation}
 We address now the high frequencies part   $\I\I\I_1^{(2)}$ involving $(\Id-S_N)\ddd_1$ such that
\begin{align*}
    |\I\I\I_1^{(2)}| 
    &\leq 
    	\sum_{q=-1}^\infty
	\sum_{|j-q|\leq 5}
	2^{-2q}
	\|\sdm{j}\nabla \ddd_1	    \|_{L^4}
	\| \Dd_j \nabla \delta \uu	 \|_{L^4}
	\| \Dd_q \nabla \delta \uu \|_{L^2}
	\|\sdm{q}(\Id-S_N)\ddd_1 \|_{L^\infty}\\
	&\lesssim 
	\| \nabla \ddd_1	 \|_{L^4}
    \sum_{q=-1}^\infty
	\sum_{|j-q|\leq 5}
	\| \Dd_j \delta \uu	 \|_{L^4}
	\| \Dd_q \delta \uu  \|_{L^2}
	\|\sdm{q}(\Id-S_N)\ddd_1    \|_{L^\infty}\\
	&\lesssim 
	\| \nabla \ddd_1	    \|_{H^1}^\frac{1}{2}
	\| \nabla \ddd_1	    \|_{L^2}^\frac{1}{2}
    \sum_{q=-1}^\infty
	\sum_{|j-q|\leq 5}
	\| \Dd_j \delta \uu	  	        \|_{L^2}^\frac{1}{2}
	\| \Dd_j \nabla \delta \uu	  	\|_{L^2}^\frac{1}{2}
	\| \Dd_q \delta \uu             \|_{L^2}^\frac{1}{2}
	\| \Dd_q \nabla \delta \uu      \|_{L^2}^\frac{1}{2}
	\|(\Id-S_N)\ddd_1 \|_{L^\infty}\\
	&\lesssim  
	\| \nabla \ddd_1	    \|_{H^1}^\frac{1}{2}
	\| \nabla \ddd_1	    \|_{L^2}^\frac{1}{2}
	\|  \delta \uu	        \|_{L^2}
	\|  \nabla \delta \uu	\|_{L^2}
	\|(\Id-S_N)\ddd_1   \|_{L^\infty}.
\end{align*}
To achieve an uniform bound on $\|(\Id-S_N)\ddd_1\|_{L^\infty}$, we use the fact $(\Id-S_N)\ddd_1 = \sum_{q\geq N}\Dd_q\ddd_1$ and apply the Bernstein inequality $\| \Dd_q \ddd_1 \|_{L^\infty}\leq C2^q\| \Dd_q\ddd_1 \|_{L^2}$ as in Lemma \ref{prop:Bernstein} to conclude  
\begin{align*}
    \|(\Id-S_N)\ddd_1   \|_{L^\infty}
    &\leq 
    \sum_{q= N}^\infty \| \Dd_q \ddd_1 \|_{L^\infty}
    \leq 
    \sum_{q=N}^\infty 2^q\| \Dd_q \ddd_1 \|_{L^2}
    \leq 
    \Big(\sum_{q=N}^\infty 2^{-q}\Big)^{\frac{1}{2}}
    \Big(
    \sum_{q=-1}^\infty
    2^{3q}\| \Dd_q \ddd_1 \|_{L^2}^2
    \Big)^\frac{1}{2}\\
    &\leq 
    C2^{-\frac{N}{2}}\| \ddd_1 \|_{H^\frac{3}{2}}
    \leq 
    C2^{-\frac{N}{2}}
    \| \ddd_1 \|_{H^1}^\frac{1}{2}
    \| \ddd_1 \|_{H^2}^\frac{1}{2}.
\end{align*}
Then we obtain  the estimate
\begin{equation}
     |\I\I\I_1^{(2)}| \leq 
     C
     \| \ddd_1	\|_{H^1}
     \| \ddd_1  \|_{H^2}
	 \| \delta \uu \|_{L^2}
	 \| \nabla \delta \uu \|_{L^2}
	 2^{-\frac{N}{2}},\notag 
\end{equation}
which together with \eqref{ineq-due-to-no-parodi-1} implies that
\begin{align*}
    |\I\I\I_1| & \leq 
   C\eta^{-\frac{1+\ee}{1-\ee}}
    \Big(\frac{N}{\ee }\Big)^{\frac{\ee}{1-\ee}}
    \|\nabla \ddd_1			\|_{H^1}^2
	\|		\ddd_1			\|_{H^1}^{\frac{2(1+\ee)}{1-\ee}}
	\| \delta \uu \|_{H^{-\frac{1}{2}}}^2
	\Big(\frac{N}{\ee }\Big)
	+
    \eta\nu 
	\| \nabla \delta \uu \|_{H^{-\frac{1}{2}}}^2\\
	&\quad +C
     \| \ddd_1	\|_{H^1}
     \| \ddd_1  \|_{H^2}
	 \| \delta \uu \|_{L^2}
	 \| \nabla \delta \uu \|_{L^2}
	 2^{-\frac{N}{2}}.
\end{align*}
We have to explicitly set the parameters $N\in\mathbb{N}$ and $\ee\in (0,1/2)$ in the above estimate. As a first remark, we see that if $\Phi(t) = 0$, then the right-hand side of the previous inequality is identically equal to $0$, hence no further investigation is required. 
If $\Phi(t)>0$, we set $N(t)\in \mathbb{N}$  satisfying 
\begin{align} \label{def-N-T3}
N(t) = 2\Big\lfloor \log_2\Big(1+\frac{1}{\Phi(t)}\Big)\Big\rfloor + 2
\end{align} 
and then 
\begin{align}
\ee(t):=\frac{1}{2+2\ln N(t)}\in \Big(0,\frac12\Big).
\end{align} 
Under this choice, we have 
$2^{-N(t)/2}\leq \Phi(t)$ and 
\begin{equation} \label{esNee1}
    \begin{aligned}
	\left(\frac{N}{\ee}\right)^\frac{\ee}{1-\ee}
	&= \exp\Big( \frac{\ee}{1-\ee}\ln \frac{N}{\ee}\Big)
	=\exp\Big( \frac{1}{1+2\ln N}\ln\Big( 2N\big(1+\ln N\big)\Big)\Big)\\
	&\leq \exp\Big(\frac{\ln N + \ln e + (1+\ln N)}{1+\ln N}\Big)
	\leq e^2.
\end{aligned}
\end{equation}
The above facts and $(1+\ee)/(1-\ee)\in (1,3)$ allow us to conclude that
\begin{align*}
    |\I\I\I_1|& \leq  C\eta^{-3}
    \|\nabla \ddd_1	\|_{H^1}^2
	\big(1+\| \ddd_1\|_{H^1}^{6}\big)
	\Phi(t)
	\Big(1+ \ln\Big(1+\frac{1}{\Phi(t)}\Big)\Big)
	\Big( 1+ \ln\Big(1+\ln\Big(1+\frac{1}{\Phi(t)}\Big)\Big)\Big)
	\\
	&\quad +C
     \| \ddd_1 \|_{H^1}
     \| \ddd_1 \|_{H^2}
	 \big(\| \uu_1\|_{L^2}+\|\uu_2\|_{L^2}\big)        \big(\| \nabla \uu_1\|_{L^2}+\|\nabla \uu_2\|_{L^2}\big)
	\Phi(t)
	 +
	\eta \nu  
	\| \nabla \delta \uu \|_{H^{-\frac{1}{2}}}^2.
\end{align*}
Concerning  $\I\I\I_2$, we observe that the term $(\sdm{q}-\sdm{j})\ddd_1\cdot \Dd_q\Dd_j\delta A$, satisfies a similar property like for the commutator $[\Dd_q,\,\sdm{j}\ddd_1]\cdot \Dd_j\delta A$ such that 
\begin{equation*}
    \|(\sdm{j}\ddd_1-\sdm{q}\ddd_1)\cdot \Dd_q \Dd_j\delta A \|_{L^2}
    \leq 
    C2^{-q}
    \| (\sdm{j}-\sdm{q})\nabla \ddd_1 \|_{L^\frac{2}{\ee}}
    \| \Dd_j \nabla \delta \uu  \|_{L^\frac{2}{1-\ee}}.
\end{equation*}
Then by an analogous procedure for $\I\I\I_1$, we can derive the following inequality:
\begin{align*}
    |\I\I\I_2|
    &\leq  C\eta^{-3}
    \|\nabla \ddd_1 \|_{H^1}^2
	\big(1+\| \ddd_1 \|_{H^1}^{6}\big)
	\Phi(t)
	\Big(1+ \ln\Big(1+\frac{1}{\Phi(t)}\Big)\Big)
	\Big(1+ \ln\Big( 1+\ln\Big(1+\frac{1}{\Phi(t)}\Big)\Big)\Big)
	\\
	&\quad +C
     \| \ddd_1 \|_{H^1}
     \| \ddd_1 \|_{H^2}
	 \big(\| \uu_1\|_{L^2}+\|\uu_2\|_{L^2}\big)        \big(\| \nabla \uu_1\|_{L^2}+\|\nabla \uu_2\|_{L^2}\big)
	\Phi(t)
	 +
	\eta \nu  
	\| \nabla \delta \uu \|_{H^{-\frac{1}{2}}}^2.
\end{align*}

By symmetry, the terms $\I\I\I_i'$, $i=1,2$, can be handled in the same way as for $\I\I\I_i$. Hence, it remains to control $\I\I\I_4$ and $\I\I\I_4'$. We note that these two terms have to  and can be estimated together thanks to the compatibility conditions $\lambda_1=\mu_2-\mu_3$, $\lambda_2=\mu_5-\mu_6$ (cf. Remark \ref{comm_proof}). Recalling the Bernstein inequality $\| \Dd_q f \|_{L^2}\leq C2^q\| \Dd_q f \|_{L^1}$  as depicted in Lemma \ref{prop:Bernstein}, we infer that
\begin{align*}
	|\I\I\I_4+ \I\I\I_4' |
	&= \left| - 2\sum_{q=-1}^\infty
	\sum_{j=q-5}^\infty 
	2^{-q}
	\int_{\TT^2}\Dd_q(S_{j+2} \delta A\Dd_{j}\ddd_1)\cdot\big(\Dd_q \delta A \sdm{q}\ddd_1\big)\,\dd x \right| \\
	&\lesssim
	\sum_{q=-1}^\infty
	\sum_{j=q-5}^\infty 
	2^{-q}
	\| \Dd_q(S_{j+2} \delta A\Dd_{j}\ddd_1)	\|_{L^2}
	\| \Dd_q \delta A \sdm{q}\ddd_1\|_{L^2}
	\\
	&\lesssim
	\sum_{q=-1}^\infty
	\sum_{j=q-5}^\infty 
	2^{\frac q2 }
	\| \Dd_q(S_{j+2} \delta A \Dd_{j}\ddd_1)	\|_{L^1}
	2^{-\frac{q}{2}}
	 \| \Dd_q \delta A\sdm{q}\ddd_1	\|_{L^2}
	\\
	&\lesssim
	\sum_{q=-1}^\infty
	\sum_{j=q-5}^\infty
	2^{\frac q2 }
	\| S_{j+2} \nabla \delta \uu \|_{L^2} 
	\| \Dd_j\ddd_1	\|_{L^2}
	2^{-\frac{q}{2}}
	\| \Dd_q \delta A \sdm{q}\ddd_1	\|_{L^2}\\
	&\lesssim
	\sum_{q=-1}^\infty
	\sum_{j=q-5}^\infty 
	2^{\frac{q-j}{2}}
	2^{-\frac{j}{2}}
	\| S_{j+2}  \delta \uu \|_{L^2}
	2^{2j} 
	\| \Dd_j\ddd_1	\|_{L^2}
	2^{-\frac{q}{2}}
	\| \Dd_q \delta A\sdm{q}\ddd_1	\|_{L^2}\\
	&\lesssim
	\| \ddd_1 		\|_{H^2}
	\| \delta \uu 	\|_{H^{-\frac{1}{2}}}
	\Big( \sum_{q=-1}^\infty 2^{-q}	\| \Dd_q \delta A\sdm{q}\ddd_1\|_{L^2}^2 \Big)^\frac{1}{2}\\
	&\leq C\eta^{-1}
	\| \ddd_1 \|_{H^2}^2
	\| \delta \uu \|_{H^{-\frac{1}{2}}}^2
	+
	\eta \sum_{q=-1}^\infty
	2^{-q}
	\int_{\TT^2}
	|\Dd_q \delta A\sdm{q}\ddd_1|^2\,\dd x.
\end{align*}
Collecting  the above estimates, we infer the following bound: 
\begin{equation}\label{decomp_deltaAd1od1_1}
\begin{aligned}
 &	\Big| \I\I\I  +   \I\I\I'
		+ 2 \sum_{q=-1}^\infty
		2^{-q} \int_{\TT^2}
		|\Dd_q \delta A\sdm{q}\ddd_1|^2\,\dd x
	\Big|\\
	&\quad \leq 
	 C\eta^{-3}
    \|\nabla \ddd_1	\|_{H^1}^2
	\big(1+\| \ddd_1 \|_{H^1}^{6}\big)
	\Phi(t)
	\Big(1+ \ln\Big(1+\frac{1}{\Phi(t)}\Big)\Big)
	\Big(1+ \ln\Big(1+\ln\Big(1+\frac{1}{\Phi(t)}\Big)\Big)\Big)
	\\
	&\qquad +C
     \| \ddd_1 \|_{H^1}
     \| \ddd_1 \|_{H^2}
	 \big(\| \uu_1\|_{L^2}+\|\uu_2\|_{L^2}\big)        \big(\| \nabla \uu_1\|_{L^2}+\|\nabla \uu_2\|_{L^2}\big)
	\Phi(t) + 
	C\eta^{-1}	\| \ddd_1 \|_{H^2}^2 \Phi(t)
	 \\ 
	& \qquad 
	+
	4\eta \nu  
	\| \nabla \delta \uu \|_{H^{-\frac{1}{2}}}^2
	+
	\eta
	\sum_{q=-1}^\infty
	2^{-q}
	\int_{\TT^2}
	|\Dd_q \delta A\sdm{q}\ddd_1|^2\,\dd x.
\end{aligned}
\end{equation}

Next, we estimate the term $\I$ in \eqref{decomp_deltaAd1od1}. Applying Bony's decomposition \eqref{bony-decomp} to $\Dd_j(\delta A\ddd_1)$, we get 
\begin{equation*}
	\I=\I_1+\I_2+\I_3+\I_4,
\end{equation*}
where
\begin{align*}
	\I_1
	&:=
	-
	\sum_{q=-1}^\infty
	\sum_{|j-q|\leq 5}
	\sum_{|l-j|\leq 5}
	2^{-q}
	\int_{\TT^2}\big([\Dd_q,\,\sdm{j}\ddd_1] \otimes 
	([\Dd_j,\,\sdm{l}\ddd_1]\cdot \Dd_l\delta A)\big): \Dd_q \nabla \delta \uu\,\dd x,\\
	\I_2
	&:=-
	\sum_{q=-1}^\infty
	\sum_{|j-q|\leq 5}
	\sum_{|l-j|\leq 5}
	2^{-q}
	\int_{\TT^2}\big([\Dd_q,\,\sdm{j}\ddd_1] \otimes 
	((\sdm{l}-\sdm{j})\ddd_1\cdot \Dd_j\Dd_l\delta A)\big): \Dd_q \nabla \delta \uu\,\dd x,\\
	\I_3
	&:=-
	\sum_{q=-1}^\infty
	\sum_{|j-q|\leq 5}
	2^{-q}
	\int_{\TT^2}\big([\Dd_q,\,\sdm{j}\ddd_1] \otimes (\Dd_j\delta A\,\sdm{j}\ddd_1)\big): \Dd_q \nabla \delta \uu\,\dd x,\\
	\I_4
	&:=-
	\sum_{q=-1}^\infty
	\sum_{|j-q|\leq 5}
	\sum_{l=j-5}^\infty 
	2^{-q}
	\int_{\TT^2}\big([\Dd_q,\,\sdm{j}\ddd_1] \otimes \Dd_j(S_{l+2}\delta A\, \Dd_l\ddd_1)\big): \Dd_q \nabla \delta \uu\,\dd x.
\end{align*}
Applying the commutator estimates in Lemma  \ref{prop:comm-est}, we see that 
\begin{align*}
	\big| \I_1 \big| 
	&\leq 
	\sum_{q=-1}^\infty
	\sum_{|j-q|\leq 5}
	\sum_{|l-j|\leq 5}
	2^{-2q}
	\| \sdm{j} \nabla \ddd_1 	\|_{L^\infty}
	2^{-j}
	\| \sdm{l} \nabla \ddd_1	\|_{L^\infty}
	\| \Dd_l \delta A 	  	\|_{L^2}
	\| \Dd_q \nabla \delta \uu 			\|_{L^2}\\
	&\leq 
	\sum_{q=-1}^\infty
	\sum_{|j-q|\leq 5}
	\sum_{|l-j|\leq 5}
	2^{-2q}
	\| \sdm{j} \nabla \ddd_1 			\|_{L^\infty}
	2^{-j}
	\| \sdm{l} \nabla \ddd_1			\|_{L^\infty}
	\| \Dd_l\nabla \delta \uu 	\|_{L^2}
	\| \Dd_q \nabla \delta \uu			\|_{L^2},
\end{align*}
which together with the Bernstein inequality $\| \sdm{j}\nabla \ddd_1\|_{L^\infty}\leq 2^\frac{j}{2}\| 
 \sdm{j}\nabla \ddd_1\|_{L^4}$ as depicted in Lemma \ref{prop:Bernstein} yields 
\begin{align*}
	|\I_1|&\leq 
	\sum_{q=-1}^\infty
	\sum_{|j-q|\leq 5}
	\sum_{|l-j|\leq 5}
	2^{-2q}2^{\frac{j}{2}}
	\| \sdm{j} \nabla \ddd_1 	\|_{L^4}
	2^{-j}2^{\frac{l}{2}}
	\| \sdm{l} \nabla \ddd_1	\|_{L^4}
	\| \Dd_l \nabla \delta \uu  	  	\|_{L^2}
	2^q
	\| \Dd_q  \delta \uu 			\|_{L^2}\\
	&\lesssim 
	\| \nabla \ddd_1 	\|_{L^4}^2
	\sum_{q=-1}^\infty
	\sum_{|j-q|\leq 5}
	\sum_{|l-j|\leq 5}
	2^{-\frac{l}{2}}	
	2^{l-j}
	\| \Dd_l \nabla \delta \uu  	  	\|_{L^2}
	2^{-\frac{q}{2}} 2^{\frac{j-q}{2}}
	\| \Dd_q  \delta \uu \|_{L^2}
	\\
	&\lesssim 
	\| \nabla \ddd_1 \|_{L^2}
	\| \nabla  \ddd_1 \|_{H^1}
	\|	\delta \uu 				\|_{H^{-\frac{1}{2}}}
	\| \nabla	\delta \uu 		\|_{H^{-\frac{1}{2}}}\\
	&\lesssim 
	C\eta^{-1}
	\| \nabla \ddd_1 \|_{L^2}^2
	\| \nabla  \ddd_1 \|_{H^1}^2
	\|	\delta \uu 				\|_{H^{-\frac{1}{2}}}^2
	+
	\eta \nu 
	\| \nabla	\delta \uu 		\|_{H^{-\frac{1}{2}}}^2.
\end{align*}
The argument for the estimate of $\I_2$ is analogous to the one for $\I_1$. Hence, we have 
\begin{equation*}
	\big| \I_2 \big| 
	\leq 
	C\eta^{-1}
	\| \nabla \ddd_1 \|_{L^2}^2
	\| \nabla \ddd_1 \|_{H^1}^2
	\|	\delta \uu 				\|_{H^{-\frac{1}{2}}}^2
	+
	\eta\nu  
	\| \nabla	\delta \uu 		\|_{H^{-\frac{1}{2}}}^2.
\end{equation*}
Concerning $\I_3$, it follows from the commutator estimates in Lemma \ref{prop:comm-est} that
\begin{align*}
	|\I_3|
	&\lesssim 
	\sum_{q=-1}^\infty
	\sum_{|j-q|\leq 5}
	2^{-2q}
	\| \sdm{j} \nabla \ddd_1		 \|_{L^4}
	\| \Dd_j\delta A\,\sdm{j}\ddd_1 \|_{L^2} 
	\|\Dd_q \nabla \delta \uu 		 \|_{L^4}\\
	&\lesssim
	\sum_{q=-1}^\infty
	\sum_{|j-q|\leq 5}
	2^{-2q}
	\| \sdm{j} \nabla \ddd_1		 \|_{L^4}
	\| \Dd_j\delta A\,\sdm{j}\ddd_1 \|_{L^2}
	2^q 
	\|\Dd_q \delta \uu 		 \|_{L^4}\\
	&\lesssim
	\| \nabla \ddd_1		 \|_{L^4}
	\sum_{q=-1}^\infty
	\sum_{|j-q|\leq 5}
	2^{-q}
	\| \Dd_j\delta A\,\sdm{j}\ddd_1 \|_{L^2} 
	\|\Dd_q  \delta \uu 		 	 \|_{L^2}^\frac{1}{2}
	\|\Dd_q \nabla \delta \uu 		 \|_{L^2}^\frac{1}{2}\\
	&\lesssim
	\| \nabla \ddd_1		 \|_{L^2}^\frac{1}{2}
	\| \nabla \ddd_1		 \|_{H^1}^\frac{1}{2}
	\sum_{q=-1}^\infty
	\sum_{|j-q|\leq 5}
	2^{\frac{j-q}{2}} 2^{-\frac{j}{2}}
	\| \Dd_j\delta A\,\sdm{j}\ddd_1 \|_{L^2} 
	\big(2^{-\frac{q}{2}}\|\Dd_q  \delta \uu 		 	 \|_{L^2}\big)^\frac{1}{2}
	\big(2^{-\frac{q}{2}}\|\Dd_q \nabla \delta \uu 	 \|_{L^2}\big)^\frac{1}{2}\\
	&\lesssim
	\| \nabla \ddd_1		 \|_{L^2}^\frac{1}{2}
	\| \nabla \ddd_1		 \|_{H^1}^\frac{1}{2}
	\Big(
	\sum_{j=-1}^\infty
	2^{-j}
	\| \Dd_j\delta A\,\sdm{j}\ddd_1 \|_{L^2}^2
	\Big)^\frac{1}{2}
	\| \delta \uu 		 	 \|_{H^{-\frac{1}{2}}}^\frac{1}{2}
	\| \nabla \delta \uu	 \|_{H^{-\frac{1}{2}}}^\frac{1}{2}\\
	&\leq 
	C\eta^{-3}
	\| \nabla \ddd_1 \|_{L^2}^2
	\| \nabla \ddd_1 \|_{H^1}^2
	\| \delta \uu 		 	 \|_{H^{-\frac{1}{2}}}^2
	+
	\eta \nu 
	\| \nabla \delta \uu 	\|_{H^{-\frac{1}{2}}}^2 
	+ \eta   \sum_{j=-1}^\infty 2^{-j}
	\int_{\TT^2}| \Dd_j\delta A\,\sdm{j}\ddd_1|^2\,\dd x.
\end{align*}
Then we estimate the last term $\I_4$: 
\begin{align*}
	|\I_4|
	&\lesssim
	\sum_{q=-1}^\infty
	\sum_{|j-q|\leq 5}
	\sum_{l=j-5}^\infty 
	2^{-2q}
	\| \sdm{j} \nabla \ddd_1 					\|_{L^\infty} 
	\| \Dd_j(S_{l+2}\delta A\, \Dd_l\ddd_1)	\|_{L^2} 
	\| \Dd_q \nabla \delta \uu \|_{L^2}\\
	&\lesssim
	\sum_{q=-1}^\infty
	\sum_{|j-q|\leq 5}
	\sum_{l=j-5}^\infty 
	2^{-q}
	\| \sdm{j} \nabla \ddd_1 					\|_{L^2} 
	\| \Dd_j(S_{l+2}\delta A\, \Dd_l\ddd_1)	\|_{L^2} 
	\| \Dd_q\nabla  \delta \uu 				 	\|_{L^2}
	\\
	&\lesssim
	\sum_{q=-1}^\infty
	\sum_{|j-q|\leq 5}
	\sum_{l=j-5}^\infty 
	2^{-\frac q2 }
	\| \sdm{j} \nabla \ddd_1 					\|_{L^2} 
	2^j
	\| \Dd_j(S_{l+2}\delta A\, \Dd_l\ddd_1)	\|_{L^1} 
	2^{-\frac q2 }
	\| \Dd_q 	 \delta \uu 				 	\|_{L^2}
	\\
	&\lesssim
	\| 			\nabla \ddd_1 					\|_{L^2} 
	\sum_{q=-1}^\infty
	\sum_{|j-q|\leq 5}
	\sum_{l=j-5}^\infty 
	2^{\frac{j-q}{2}}
	2^{\frac{j-l}{2}}
	2^{-\frac{l}{2}}
	\| S_{l+2}\delta A 							\|_{L^2}
	2^{l}
	\|\Dd_l\ddd_1 								\|_{L^2} 
	2^{-\frac q2 }
	\| \Dd_q 	 \delta \uu 				 	\|_{L^2}
	\\
	&\lesssim
	\| \nabla \ddd_1 		\|_{L^2}
	\| \ddd_1 				\|_{H^1}
	\| \delta \uu			\|_{H^{-\frac{1}{2}}}
	\| \nabla \delta \uu 	\|_{H^{-\frac{1}{2}}}\\
	& \leq 
	C\eta^{-1}\| \ddd_1 				\|_{H^1}^4
	\| \delta \uu			\|_{H^{-\frac{1}{2}}}^2 
	+ \eta \nu 
	\| \nabla \delta \uu			\|_{H^{-\frac{1}{2}}}^2.
\end{align*}
As a consequence, we infer from the above estimates that 
\begin{align}
	|\I| & \leq C\eta^{-3} \| \ddd_1 \|_{H^1}^2\| \ddd_1 \|_{H^2}^2
	\| \delta \uu			\|_{H^{-\frac{1}{2}}}^2 +
	4 \eta\nu 
	\| \nabla \delta \uu			\|_{H^{-\frac{1}{2}}}^2 
	+ \eta \sum_{j=-1}^\infty 2^{-j}
	\int_{\TT^2}| \Dd_j\delta A\,\sdm{j}\ddd_1|^2\,\dd x. 
	\label{decomp_deltaAd1od1_2}
\end{align}
With a similar procedure, we can show that $\I\I$ satisfies
\begin{equation}\label{decomp_deltaAd1od1_3}
	|\I\I|\leq C\eta^{-1}\| \ddd_1 \|_{H^1}^2\| \ddd_1 \|_{H^2}^2 
	\| \delta \uu			\|_{H^{-\frac{1}{2}}}^2 +
	4 \eta \nu 
	\| \nabla \delta \uu			\|_{H^{-\frac{1}{2}}}^2 
	+ \eta  \sum_{j=-1}^\infty 2^{-j}
	\int_{\TT^2}| \Dd_j\delta A\,\sdm{j}\ddd_1|^2\,\dd x.
\end{equation}
Again by symmetry, $\I'$, $\I\I'$ can be handled in the same way as for $\I$ and $\I\I$, respectively. 

It remains to analyse $\I\mathcal{V}$ as well as $\I\mathcal{V}'$. Decomposing the term $\I\mathcal{V}$ by means of the Bony's decomposition \eqref{bony-decomp}, we obtain 
\begin{align*}
	\I\mathcal{V} 
	= 
	&
	-
	\sum_{q=-1}^\infty
	\sum_{j=q-5}^\infty 
	2^{-q}
	\int_{\TT^2}\Dd_q\big(\Dd_{j}\ddd_1\otimes S_{j+2}(\delta A \ddd_1)\big):\Dd_q \nabla \delta \uu \,\dd x
	\\
	=
	&-
	\sum_{q=-1}^\infty
	\sum_{j=q-5}^\infty 
	\sum_{l=-1}^{j+6}
	2^{-q}
	\int_{\TT^2}\Dd_q \big( \Dd_{j}\ddd_1 \otimes S_{j+2} \big(\Dd_l(\delta A \ddd_1)\big)\big) :\Dd_q \nabla \delta \uu \,\dd x
	\\
	=
	&
	-\sum_{q=-1}^\infty
	\sum_{j=q-5}^\infty 
	\sum_{l=-1}^{j+6}
	\sum_{|t-l|\leq 5}
	2^{-q}
	\int_{\TT^2}\Dd_q \big(\Dd_{j}\ddd_1 \otimes S_{j+2}\big([\Dd_l,\, \sdm{t}\ddd_1]\cdot\Dd_t \delta A\big)\big):
	\Dd_q \nabla \delta \uu \,\dd x \\
	&-
	\sum_{q=-1}^\infty
	\sum_{j=q-5}^\infty 
	\sum_{l=-1}^{j+6}
	\sum_{|t-l|\leq 5}
	2^{-q}
	\int_{\TT^2}\Dd_q \big(\Dd_{j}\ddd_1 \otimes S_{j+2}\big((\sdm{t}-\sdm{l})\ddd_1\cdot\Dd_l\Dd_t \delta A\big)\big) :
	\Dd_q \nabla \delta \uu \,\dd x\\
	&-\sum_{q=-1}^\infty
	\sum_{j=q-5}^\infty 
	\sum_{l=-1}^{j+6}
	2^{-q}
	\int_{\TT^2}\Dd_q \big(\Dd_{j}\ddd_1 \otimes S_{j+2}\big(\Dd_l \delta A \sdm{l}\ddd_1\big)\big):
	\Dd_q \nabla \delta \uu \,\dd x\\
	&-
	\sum_{q=-1}^\infty
	\sum_{j=q-5}^\infty 
	\sum_{l=-1}^{j+6}
	\sum_{t=l-5}^\infty 
	2^{-q}
	\int_{\TT^2}\Dd_q \big(\Dd_{j}\ddd_1 \otimes S_{j+2}\big(\Dd_l(\Dd_{t}\ddd_1\cdot S_{t+2} \delta A)\big)\big):
	\Dd_q \nabla \delta \uu \,\dd x \\
	=: &\  
	\I\mathcal{V}_1+\I\mathcal{V}_2+\I\mathcal{V}_3+\I\mathcal{V}_4.
\end{align*}
We begin with the estimate for $\I\mathcal{V}_1$:
\begin{align*}
	|\I\mathcal{V}_1|
	&\lesssim 
	\sum_{q=-1}^\infty
	\sum_{j=q-5}^\infty
	\sum_{l=-1}^{j+6}
	\sum_{|t-l|\leq 5}
	2^{-q}
	\|\Dd_q(\Dd_{j}\ddd_1\otimes S_{j+2}([\Dd_l,\, \sdm{t}\ddd_1]\cdot\Dd_t \delta A))\|_{L^2}
	\|\Dd_q \nabla \delta \uu \|_{L^2}
	\\
	&\lesssim
	\sum_{q=-1}^\infty
	\sum_{j=q-5}^\infty 
	\sum_{l=-1}^{j+6}
	\sum_{|t-l|\leq 5}
	\| \Dd_q(\Dd_{j}\ddd_1\otimes S_{j+2}([\Dd_l,\, \sdm{t}\ddd_1]\cdot\Dd_t \delta A))\|_{L^1}
	\| \Dd_q \nabla \delta \uu \|_{L^2}
	\\
	&\lesssim
	\sum_{q=-1}^\infty
	\sum_{j=q-5}^\infty 
	\sum_{l=-1}^{j+6}
	\sum_{|t-l|\leq 5}
	\| \Dd_{j}\ddd_1 \|_{L^2}
	\| S_{j+2}([\Dd_l,\, \sdm{t}\ddd_1]\cdot\Dd_t \delta A))\|_{L^2}
	2^q
	\| \Dd_q  \delta \uu \|_{L^2}
	\\
	&\lesssim
	\sum_{q=-1}^\infty
	\sum_{j=q-5}^\infty 
	\sum_{l=-1}^{j+6}
	\sum_{|t-l|\leq 5}
	\| \Dd_{j}\ddd_1 \|_{L^2}
	2^{-l}
	\| \sdm{t}\nabla \ddd_1\|_{L^\infty} \| \Dd_t \delta A \|_{L^2}
	2^q
	\| \Dd_q  \delta \uu \|_{L^2}
	\\
	&\lesssim
	\sum_{q=-1}^\infty
	\sum_{j=q-5}^\infty 
	\sum_{t=-1}^{j+11}
	2^{\frac{3}{2}(q-j)}
	2^{\frac{3}{2}j}
	\| \Dd_{j} \ddd_1 \|_{L^2}
	\| \sdm{t}\nabla \ddd_1\|_{L^2}
	2^{\frac{t}{2}} 
	2^{-\frac{t}{2}}
	\| \Dd_t \nabla \delta \uu \|_{L^2}
	2^{-\frac{q}{2}}
	\| \Dd_q  \delta \uu \|_{L^2}
	\\
	&\lesssim
	\| \nabla \ddd_1\|_{L^2}
	\sum_{q=-1}^\infty
	\sum_{j=q-5}^\infty 
	2^{\frac{3}{2}(q-j)}
	2^{\frac{3}{2}j}
	\| \Dd_{j} \ddd_1 \|_{L^2}
	\Big(
	\sum_{t=-1}^{j+11}
	2^{t}
	\Big)^\frac{1}{2}
	\Big( \sum_{t=-1}^{j+11}
		2^{-t} \| \Dd_t \nabla \delta \uu \|_{L^2}^2 \Big)^\frac{1}{2}
	2^{-\frac{q}{2}}
	\| \Dd_q  \delta \uu \|_{L^2}
	\\
	&\lesssim
	\| \nabla \ddd_1\|_{L^2}
	\| \nabla \delta \uu \|_{H^{-\frac{1}{2}}}
	\sum_{q=-1}^\infty
	\sum_{j=q-5}^\infty 
	2^{2j}
	\| \Dd_{j} \ddd_1 \|_{L^2}
	2^{-\frac{q}{2}}
	\| \Dd_q  \delta \uu \|_{L^2}
	\\
	&\lesssim
	\| \nabla \ddd_1\|_{L^2}
	\| \nabla \delta \uu \|_{H^{-\frac{1}{2}}}
	\|   \ddd_1\|_{H^2}
	\|  \delta \uu \|_{H^{-\frac{1}{2}}}\\
	&\leq C\eta^{-1} \| \nabla \ddd_1\|_{L^2}^2
	\|   \ddd_1\|_{H^2}^2
	\|  \delta \uu \|_{H^{-\frac{1}{2}}}^2
	+\eta\nu \| \nabla \delta \uu \|_{H^{-\frac{1}{2}}}^2.
\end{align*}
With a similar procedure, we can deduce the estimate for $\I\mathcal{V}_2$: 
\begin{equation*}
	|\I\mathcal{V}_2|\leq C\eta^{-1} \| \nabla \ddd_1\|_{L^2}^2
	\|   \ddd_1\|_{H^2}^2
	\|  \delta \uu \|_{H^{-\frac{1}{2}}}^2
	+ \eta\nu   \| \nabla \delta \uu \|_{H^{-\frac{1}{2}}}^2.
\end{equation*}
Next, we handle $\I\mathcal{V}_3$. It follows that 
\begin{align*}
	|\I\mathcal{V}_3|
	&\lesssim
	\sum_{q=-1}^\infty
	\sum_{j=q-5}^\infty 
	\sum_{l=-1}^{j+6}
	2^{-q}
	\|\Dd_q(\Dd_{j}\ddd_1\otimes S_{j+2}(\Dd_l \delta A \sdm{l}\ddd_1))\|_{L^2}
	\| \Dd_q \nabla \delta \uu \|_{L^2}\\
	&\lesssim
	\sum_{q=-1}^\infty
	\sum_{j=q-5}^\infty 
	\sum_{l=-1}^{j+6}
	\|\Dd_q(\Dd_{j}\ddd_1\otimes S_{j+2}(\Dd_l \delta A \sdm{l}\ddd_1))\|_{L^1}
	\| \Dd_q \nabla \delta \uu \|_{L^2}\\
	&\lesssim
	\sum_{q=-1}^\infty
	\sum_{j=q-5}^\infty 
	\sum_{l=-1}^{j+6}
	\|\Dd_{j}\ddd_1 							\|_{L^2}
	\| S_{j+2}(\Dd_l \delta A \sdm{l}\ddd_1)	\|_{L^2}
	\| \Dd_q \nabla \delta \uu 					\|_{L^2}\\
	&\lesssim
	\sum_{q=-1}^\infty
	\sum_{j=q-5}^\infty 
	\sum_{l=-1}^{j+6}
	2^{\frac{3}{2}(q-j)}
	2^{\frac 32 j}
	\|\Dd_{j}\ddd_1 						\|_{L^2}
	2^{\frac{l}{2}}\big(
	2^{-\frac{l}{2}}
	\| \Dd_l \delta A 	\sdm{l}\ddd_1		\|_{L^2}\big)
	2^{-\frac{q}{2}}
	\| \Dd_q \delta \uu 					\|_{L^2}\\
	&\lesssim
	\sum_{q=-1}^\infty
	\sum_{j=q-5}^\infty 
	2^{\frac{3}{2}(q-j)}
	2^{\frac 32 j}
	\|\Dd_{j}\ddd_1 						\|_{L^2}
	\Big( \sum_{l=-1}^{j+6}
		2^{l} \Big)^\frac{1}{2}
	\Big( \sum_{l=-1}^{j+6}
		2^{-l} \| \Dd_l \delta A 	\sdm{l}\ddd_1 \|_{L^2}^2 \Big)^\frac{1}{2}	
	2^{-\frac{q}{2}}
	\| \Dd_q \delta \uu 					\|_{L^2}\\
	&\lesssim
	\Big(
		\sum_{l=-1}^{+\infty}
		2^{-l} \| \Dd_l \delta A 	\sdm{l}\ddd_1 \|_{L^2}^2
	\Big)^\frac{1}{2}
	\sum_{q=-1}^\infty
	\sum_{j=q-5}^\infty 
	2^{2j}
	\|\Dd_{j}\ddd_1 						\|_{L^2}
	2^{-\frac{q}{2}}
	\| \Dd_q \delta \uu 					\|_{L^2}\\
	&\lesssim
	\Big(
		\sum_{l=-1}^{+\infty}
		2^{-l}
		\| \Dd_l \delta A \sdm{l}\ddd_1		\|_{L^2}^2
	\Big)^\frac{1}{2}
	\| \ddd_1 \|_{H^2}
	\| \delta \uu \|_{H^{-\frac{1}{2}}}\\
	&\leq 
	C\eta^{-1}
	\| \ddd_1 \|_{H^2}^2
	\| \delta \uu \|_{H^{-\frac{1}{2}}}^2
	+
	\eta \sum_{q=-1}^{+\infty}
		2^{-q} \int_{\TT^2}| \Dd_q \delta A 	\sdm{q}\ddd_1|^2.
\end{align*}
Concerning the last term $\I\mathcal{V}_4$, we have 
\begin{align*}
	|\I\mathcal{V}_4|
	&\lesssim
	\sum_{q=-1}^\infty
	\sum_{j=q-5}^\infty 
	\sum_{l=-1}^{j+6}
	\sum_{t=l-5}^\infty 
	2^{-q}
	\| 
		\Dd_q(\Dd_{j}\ddd_1\otimes S_{j+2}(\Dd_l(\Dd_{t}\ddd_1\cdot S_{t+2} \delta A)))\|_{L^2}
	\| \Dd_q \nabla \delta \uu \|_{L^2}\\
	&\lesssim
	\sum_{q=-1}^\infty
	\sum_{j=q-5}^\infty 
	\sum_{l=-1}^{j+6}
	\sum_{t=l-5}^\infty 
	\| \Dd_q(\Dd_{j}\ddd_1\otimes S_{j+2}(\Dd_l(\Dd_{t}\ddd_1\cdot S_{t+2} \delta A)))\|_{L^1}
	\| \Dd_q \nabla \delta \uu \|_{L^2}\\
	&\lesssim
	\sum_{q=-1}^\infty
	\sum_{j=q-5}^\infty 
	\sum_{l=-1}^{j+6}
	\sum_{t=l-5}^\infty 
	2^{q}
	\| \Dd_{j}\ddd_1 \|_{L^2}
	\|	\Dd_l(\Dd_{t}\ddd_1\cdot S_{t+2} \delta A)\|_{L^2}
	\| \Dd_q \delta \uu \|_{L^2}\\
	&\lesssim
	\sum_{q=-1}^\infty
	\sum_{j=q-5}^\infty 
	\sum_{l=-1}^{j+6}
	\sum_{t=l-5}^\infty 
	2^{q}
	\| \Dd_{j}\ddd_1 \|_{L^2}
	2^l
	\|	\Dd_l(\Dd_{t}\ddd_1\cdot S_{t+2} \delta A)\|_{L^1}
	\| \Dd_q \delta \uu \|_{L^2}\\
	&\lesssim
	\sum_{q=-1}^\infty
	\sum_{j=q-5}^\infty 
	\sum_{l=-1}^{j+6}
	\sum_{t=l-5}^\infty 
	2^{q}
	\| 	\Dd_{j}\ddd_1 \|_{L^2}
	2^l
	\|	\Dd_{t}\ddd_1 \|_{L^2}
	\| 	S_{t+2} \delta A \|_{L^2}
	\| \Dd_q \delta \uu \|_{L^2}
	\\
	&\lesssim
	\sum_{q=-1}^\infty
	\sum_{j=q-5}^\infty 
	2^{q}
	\| 	\Dd_{j}\ddd_1 \|_{L^2}
	\| \Dd_q \delta \uu \|_{L^2}
	\sum_{l=-1}^{j+6}
	2^\frac{l}{2}
	\sum_{t=l-5}^\infty 
	2^{\frac{1}{2}(l-t)}\mathbf{1}_{(-\infty,5]}(l-t)
	2^{t}
	\|	\Dd_{t}\ddd_1 \|_{L^2}
	2^{-\frac{t}{2}}
	\| 	S_{t+2} \delta A \|_{L^2}
	\\
	&\lesssim 
	\sum_{q=-1}^\infty
	\sum_{j=q-5}^\infty 
	2^{q}
	\| 	\Dd_{j}\ddd_1 \|_{L^2}
	\| \Dd_q \delta \uu \|_{L^2}
	\Big( \sum_{l=-1}^{j+6}
		2^{l} \Big)^\frac{1}{2}
	\Big( \sum_{l=-1}^{\infty}
		\Big| \sum_{t=1}^\infty
			2^{\frac{1}{2}(l-t)}\mathbf{1}_{(-\infty,5]}(l-t)
			2^{ t}
			\| \Dd_{t}\ddd_1 \|_{L^2}
			2^{-\frac{t}{2}}
			\| S_{t+2} \delta A \|_{L^2}
		\Big|^2 
	\Big)^\frac{1}{2}\\
	&\lesssim
	\|	\ddd_1	\|_{H^1}
	\| \nabla \delta \uu	\|_{H^{-\frac{1}{2}}}
	\sum_{q=-1}^\infty
	\sum_{j=q-5}^\infty 
	2^{q} 2^\frac{j}{2}
	\| 	\Dd_{j}\ddd_1 \|_{L^2}
	\| \Dd_q \delta \uu \|_{L^2}\\
	&\lesssim
	\|	\ddd_1 \|_{H^1}
	\| \nabla \delta \uu	\|_{H^{-\frac{1}{2}}}
	\sum_{q=-1}^\infty
	\sum_{j=q-5}^{\infty}
	2^{\frac 32 (q-j)}
	2^{2j}
	\| 	\Dd_{j}\ddd_1 \|_{L^2}
	2^{-\frac{q}{2}}
	\| \Dd_q \delta \uu \|_{L^2}\\
	&\lesssim
	\|	\ddd_1	\|_{H^1}
	\| 	\ddd_1	\|_{H^2}
	\| \nabla \delta \uu	\|_{H^{-\frac{1}{2}}}
	\| \delta \uu 			\|_{H^{-\frac{1}{2}}}\\
	&\leq C\eta^{-1}
	\|	\ddd_1	\|_{H^1}^2
	\| 	\ddd_1	\|_{H^2}^2
	\| \delta \uu 			\|_{H^{-\frac{1}{2}}}^2
	+ \eta\nu \| \nabla \delta \uu	\|_{H^{-\frac{1}{2}}}^2.
\end{align*}
Collecting the above estimates, we conclude that 
\begin{equation}\label{decomp_deltaAd1od1_4}
	|\I\mathcal{V}|\leq C\eta^{-1}
	(1+\|	\ddd_1				\|_{H^1}^2)
	\| \ddd_1 		\|_{H^2}^2
	\| \delta \uu 	\|_{H^{-\frac{1}{2}}}^2
	+
	\eta 
	\sum_{q=-1}^{+\infty}
		2^{-q}
	\int_{\TT^2}| \Dd_q \delta A 	\sdm{q}\ddd_1|^2\,\dd x
	+3\eta \nu\| \nabla \delta \uu	\|_{H^{-\frac{1}{2}}}^2.
\end{equation}
Again, the term $\I\mathcal{V}'$ can be estimated in the same way as for $\I\mathcal{V}.$ 

Hence, combining the above estimates  \eqref{decomp_deltaAd1od1_1}, \eqref{decomp_deltaAd1od1_2}, \eqref{decomp_deltaAd1od1_3} and \eqref{decomp_deltaAd1od1_4}, we finally deduce the inequality \eqref{new-est-T3-n1}. 

Now we turn to estimate the second term of $\mathcal{T}_3$, namely, $-\langle  \ddd_1\otimes (A_2 \delta \ddd),\,\nabla \delta \uu \rangle_{H^{-1/2}}$. We anticipate that the estimate of this term depends on the logarithm of the function $\Phi(t)$, as in \eqref{def-Phi-D}. Like before, if $\Phi(t)=0$ then it is identically null and no further discussion is needed. Thus, we address the scenario with  $\Phi(t)>0$. We begin with splitting the low and high frequencies of $\ddd_1 = S_N\ddd_1 + (\Id-S_N)\ddd_1$, where the parameter $N\in \mathbb{N}$ will be chosen later. Then we have  
\begin{equation}\label{A2deltadotimesd1-first}
\begin{aligned}
	\big|\langle
		\ddd_1\otimes (A_2\delta \ddd)
		,\,
		\nabla \delta \uu
	\rangle_{H^{-\frac{1}{2}}}
	\big|
	&\leq 
	\big|
	\langle
		S_N\ddd_1\otimes (A_2\delta \ddd)
		,\,
		\nabla \delta \uu
	    \rangle_{H^{-\frac{1}{2}}}
	\big|+
	\big|
	\langle
		(\Id- S_N)\ddd_1\otimes (A_2\delta \ddd)
		,\,
		\nabla \delta \uu
	\rangle_{H^{-\frac{1}{2}}}
	\big|.
\end{aligned}
\end{equation}
We first handle the low frequencies part involving  $S_N\ddd_1$. Invoking Lemma \ref{lemma:product} about the continuity of the product in $L^2\times H^{1/2}\to H^{-1/2}$ and Lemma \ref{lemma:SN-infty} about the Sobolev embedding $\| S_N \ddd_1 \|_{L^\infty} \leq C \sqrt{N} \| \ddd_1 \|_{H^1}$, we see that 
\begin{equation}\label{A2deltadotimesd1-second}
\begin{aligned}
	\big|
	\langle
		S_N\ddd_1\otimes (A_2\delta \ddd)
		,\,
		\nabla \delta \uu
	\rangle_{H^{-\frac{1}{2}}}
	\big|
	&\lesssim
	\| (A_2\delta \ddd)\otimes S_N\ddd_1  \|_{H^{-\frac{1}{2}}}
	\| \nabla \delta \uu	\|_{H^{-\frac{1}{2}}}\\
	&\lesssim
	\| A_2 	\|_{L^2}
	\| S_N\ddd_1 \|_{L^\infty} 
	\| \delta \ddd 			\|_{H^{\frac{1}{2}}} 
	\| \nabla \delta \uu	\|_{H^{-\frac{1}{2}}}\\
	&\lesssim
	\| \nabla\uu_2 	\|_{L^2}
	\sqrt{N}
	\| \ddd_1 \|_{H^1}
	\| \delta \ddd 			\|_{H^{\frac{1}{2}}} 
	\| \nabla \delta \uu	\|_{H^{-\frac{1}{2}}}
	\\
	&\leq 
	C\eta^{-1}
	\| \nabla\uu_2 	\|_{L^2}^2
	\| \ddd_1 		\|_{H^1}^2	
	\| \delta \ddd 			\|_{H^{\frac{1}{2}}}^2
	N
	+\eta 
	\| \nabla \delta \uu	\|_{H^{-\frac{1}{2}}}^2.
\end{aligned}
\end{equation}
Next, we control the high frequency part. Keeping in mind Lemma \ref{lemma:product} about the continuity of the product between $L^2\times H^{1/4}$ and $H^{-3/4}$, we deduce that
\begin{align*}
	&\big|\langle
		(\Id- S_N)\ddd_1\otimes (A_2\delta \ddd)
		,\,
		\delta A
	\rangle_{H^{-\frac{1}{2}}}\big|\\
	&\quad \lesssim
	\|(A_2\delta \ddd)\otimes(\Id- S_N)\ddd_1 \|_{H^{-\frac{3}{4}}} 
	\| \nabla \delta \uu \|_{H^{-\frac{1}{4}}}\\
	&\quad \lesssim
	\|A_2 \delta \ddd   \|_{L^2}
	\|(\Id- S_N)\ddd_1	\|_{H^\frac{1}{4}}
	\| \delta \uu 		\|_{H^\frac{3}{4}}\\
	&\quad \lesssim
	\|	\nabla\uu_2		\|_{L^2}
	\| \delta \ddd 		\|_{L^\infty}
	\|(\Id- S_N)\ddd_1	\|_{H^\frac{1}{4}}
	\| \delta \uu 		\|_{L^2}^\frac{1}{4}
	\| \nabla \delta \uu\|_{L^2}^\frac{3}{4}
	\\
	&\quad \lesssim
	\big(\| \uu_1\|_{L^2}^\frac14 + \|\uu_2  \|_{L^2}^\frac{1}{4}\big)
	\|\nabla \uu_2\|_{L^2}
	\big(\|	\nabla\uu_1\|_{L^2}^\frac34 + \|\nabla\uu_2  \|_{L^2}^\frac{3}{4}\big)
	\big(\| \ddd_1\|_{H^\frac{5}{4}} +\|\ddd_2 \|_{H^\frac{5}{4}}\big)
	\|(\Id- S_N)\ddd_1\|_{H^\frac{1}{4}}. 
\end{align*}
Using the interpolation 
$\|\ddd_i\|_{H^{5/4}}\lesssim 
\|\ddd_i\|_{H^1}^{3/4}
\|\ddd_i\|_{H^2}^{1/4}$, $i=1,2$,  
and the following estimate
\begin{equation*}
    \|(\Id- S_N)\ddd_1	\|_{H^\frac{1}{4}}
    \lesssim \Big( \sum_{q = N-5}^\infty 2^{\frac{q}{2}}\| \Dd_q \ddd_1 \|_{L^2}^2 \Big)^\frac{1}{2}
    \lesssim 
    2^{-\frac{3}{4}N} \|\ddd_1 \|_{H^1},
\end{equation*}
we infer that 
\begin{equation}\label{A2deltadotimesd1-third}\begin{aligned} 
    & \big|\langle
		(\Id- S_N)\ddd_1\otimes  (A_2\delta \ddd),\, \delta A
	\rangle_{H^{-\frac{1}{2}}}\big|\\
	&\quad 	\lesssim
	(\|\uu_1\|_{L^2}^\frac14
	  +\|\uu_2\|_{L^2}^\frac{1}{4})
	(\|\ddd_1\|_{H^1}^\frac74 +\|\ddd_2 		      \|_{H^1}^\frac{7}{4})
	\big(\|	\nabla\uu_1\|_{L^2}^\frac74 + \|\nabla\uu_2  \|_{L^2}^\frac{7}{4}\big)
	(\|\ddd_1\|_{H^2}^\frac14+\|\ddd_2 		      \|_{H^2}^\frac{1}{4})
	2^{-\frac{3}{4}N}.
\end{aligned} 
\end{equation}
For any time $t\in (0,T)$, we take the integer $N= N(t)\in\mathbb{N}$ as in \eqref{def-N-T3},  namely, 
$N(t) = 2\big\lfloor \log_2\big(1+\Phi(t)^{-1}\big)\big\rfloor + 2$,
which implies 
\begin{equation}\label{to-cite-N-1}
2^{-\frac{3}{4}N(t)} \lesssim  \frac{\Phi(t)^\frac32}{(1+\Phi(t))^\frac32} \lesssim  \Phi(t).
\end{equation}
Thus, combining the estimates  \eqref{A2deltadotimesd1-first}, \eqref{A2deltadotimesd1-second} and  \eqref{A2deltadotimesd1-third}, 
we finally gather that
\begin{equation} \label{decomp_deltaA2dod1_1}
\begin{aligned}
	&\Big|
		\langle
		\ddd_1\otimes (A_2\delta \ddd)
		,\,
		\nabla \delta \uu
	\rangle_{H^{-\frac{1}{2}}}
	\Big|\\
	&\quad  \lesssim 
	C\eta^{-1}
	\| \nabla\uu_2 	\|_{L^2}^2
	\| \ddd_1 		\|_{H^1}^2	
	\Phi(t)
	\Big(
		1+
		\ln\left(1+\frac{1}{\Phi(t)}\right)
	\Big)
	+\eta \nu 
	\| \nabla \delta \uu	\|_{H^{-\frac{1}{2}}}^2\\
	&\qquad +
	(\|\uu_1\|_{L^2}^\frac14
	  +\|\uu_2\|_{L^2}^\frac{1}{4})
	(\|\ddd_1\|_{H^1}^\frac74 +\|\ddd_2 		      \|_{H^1}^\frac{7}{4})
	\big(\|	\nabla\uu_1\|_{L^2}^\frac74 + \|\nabla\uu_2  \|_{L^2}^\frac{7}{4}\big)
	(\|\ddd_1\|_{H^2}^\frac14+\|\ddd_2 		      \|_{H^2}^\frac{1}{4})
	\Phi(t).
\end{aligned}
\end{equation}
By symmetry, the fourth term in $\mathcal{T}_3$ can be treated in a similar manner for \eqref{decomp_deltaA2dod1_1}.

Hence, in view of \eqref{new-est-T3-n1} and \eqref{decomp_deltaA2dod1_1}, using Young's inequality, we can choose function 
\begin{align*}
f_3(t) & = \big(1+\|\uu_1\|_{L^2}^2
	  +\|\uu_2\|_{L^2}^2+\| \ddd_1 \|_{H^1}^{6}\big) \big(\|\ddd_1 \|_{H^2}^2 +\|\ddd_2 \|_{H^2}^2\big)
      \\
      &\quad + 
	 \big(\| \uu_1\|_{L^2} ^2 +\|\uu_2\|_{L^2}^2 +\| \ddd_1 \|_{H^1}^2 +\| \ddd_2 \|_{H^1}^2 \big) 
	 \big(\| \nabla \uu_1\|_{L^2}^2 +\|\nabla \uu_2\|_{L^2}^2 \big)+1,
\end{align*}
such that $f_3(t)\in L^1(0,T)$. This eventually leads to the conclusion \eqref{prop:T3-ineq}. The proof of Proposition \ref{prop:T3} is complete.
\end{proof}

\section{\bf Estimate of $\mathcal{T}_4$}
\label{sec:double-log-est}
\setcounter{equation}{0}

In this section, we deal with the final term $\mathcal{T}_4=-
	\langle 	\delta \bm{\sigma}_1,\,\nabla\delta \uu 	\rangle_{H^{-\frac{1}{2}}}$ of the highest order, which involves the Leslie stress tensor
\begin{equation}\label{delta-sigma}
\begin{aligned}
	\delta \bm{\sigma}_1 =&\, 
	(\ddd_1\cdot (\delta A \ddd_1 	))\ddd_1\otimes\ddd_1 + 
	(\ddd_1\cdot (A_2 \delta\ddd	))\ddd_1\otimes\ddd_1 + 
	(\delta\ddd	\cdot (A_2 \ddd_2	))\ddd_1\otimes\ddd_1  \\
	& 
	+(\ddd_2	\cdot ( A_2 \ddd_2	))\delta \ddd\otimes\ddd_1
	+(\ddd_2	\cdot (A_2 \ddd_2	))\ddd_2\otimes \delta \ddd.
\end{aligned}
\end{equation}
We recall that in the previous section, the lower order nonlinear terms like $\ddd_1 \otimes (\delta A \ddd_1)$ has already brought us many difficulties. 
Now the mathematical difficulties related to $\delta \bm{\sigma}_1$ are even harder, for instance, in the term $(\ddd_1\cdot ( \delta A \ddd_1 )) \ddd_1\otimes\ddd_1$, the velocity gradient $\nabla \delta \uu$ is multiplied by $\ddd_1$ to the power of four. This is very difficult to handle due to the loss of control on the $L^\infty$ norm of the director. The key idea is to carefully explore the nonlinear structure and establish an inequality still being possible with a double-logarithmic structure. To this end, we divide the estimate of $\delta \bm{\sigma}_1$ into two parts and prove the following result:
\begin{prop}\label{prop:last-ineq-main-thm}
There exist two functions $g_1,\,g_2 \in L^1(0,T)$ such that for any $\eta\in (0,1)$, the following inequalities hold 
\begin{equation*}
\begin{alignedat}{4}
	&\mathrm{(a)}\quad
	\Big|\big\langle \delta \bm{\sigma}_1(t) - [(\ddd_1\cdot (\delta A \ddd_1  ))\ddd_1\otimes\ddd_1](t),\,\nabla\delta\uu(t)\big\rangle_{H^{-\frac{1}{2}}} \Big| \\ 
	&\qquad \qquad  \leq  C_4 g_1(t)\mu( \Phi(t) ) + 6\eta\nu \|\nabla \delta \uu\|_{H^{-\frac12}}^2 + 2 \eta 	\sum_{q=-1}^\infty \!\!2^{-q}\!\!\int_{\TT^2}|\sdm{q}(\ddd_1\otimes\ddd_1):\Dd_q\delta A|^2(t,x)\, \dd x,
	\\
	&\mathrm{(b)}\quad
	\Big|\big\langle [(\ddd_1\cdot (\delta A\ddd_1)) \ddd_1\otimes\ddd_1](t),\,
	\nabla\delta\uu(t)\big\rangle_{H^{-\frac{1}{2}}}
	- \sum_{q=-1}^\infty \!\!2^{-q}\!\!\int_{\TT^2}|\sdm{q}(\ddd_1\otimes\ddd_1):\Dd_q\delta A|^2(t,x)\, \dd x	\Big|\\
	&\qquad \qquad 	\leq C_4'g_2(t)\mu( \Phi(t) )+ 8\eta\nu \|\nabla \delta \uu\|_{H^{-\frac12}}^2 + 10 \eta 	\sum_{q=-1}^\infty \!\!2^{-q}\!\!\int_{\TT^2}|\sdm{q}(\ddd_1\otimes\ddd_1):\Dd_q\delta A|^2(t,x)\, \dd x,
\end{alignedat}
\end{equation*}
for almost all time $t\in (0,T)$, where where $\mu$ is defined as in \eqref{muaa} and $C_4, C_4'>0$ are constants depending on $\eta^{-1}$.
\end{prop}
\begin{remark}
Once Proposition \ref{prop:last-ineq-main-thm} is proven, we can easily verify that the last inequality of Proposition \ref{thm:ineq} holds, by setting $f_4(t):=g_1(t)+g_2(t)$ for $t\in (0,T)$.
\end{remark}
\begin{proof}[Proof of Proposition \ref{prop:last-ineq-main-thm}, Part $\mathrm{(a)}$] Let $\eta\in (0,1)$. 
We begin with the estimates involving the last two terms of \eqref{delta-sigma}, namely, $\ddd_2\cdot(A_2 \ddd_2)\delta \ddd\otimes \ddd_2$ and  $\ddd_2\cdot(A_2 \ddd_2)\ddd_1\otimes \delta \ddd$.
Let $N\geq 1$ be a positive integer that will be chosen later, depending on the value of function $\Phi(t)$ (see  \eqref{def-N-prop-deltasigma1-first} below). We treat the low frequencies $S_N\ddd_i$ and the high frequencies $(\Id-S_N)\ddd_i$ of both  director fields $\ddd_i$, $i=1,2$ separately such that
\begin{align*}
	&\langle 
		(\ddd_2	\cdot (A_2 \ddd_2	))(\delta \ddd\otimes\ddd_1
		+\ddd_2\otimes \delta \ddd)
		,\,
		\nabla \delta \uu
	\rangle_{H^{-\frac{1}{2}}}\\ 
	&\quad =\, 
	\langle 
		(\ddd_2	\cdot (A_2 \ddd_2	))(\delta \ddd\otimes S_N\ddd_1+
		S_N\ddd_2\otimes \delta \ddd
		,\,
		\nabla \delta \uu
	\rangle_{H^{-\frac{1}{2}}}  \\
	&\qquad\ +
	\langle 
		(\ddd_2	\cdot (A_2 \ddd_2	))(\delta \ddd\otimes (\Id- S_N)\ddd_1+
		(\Id- S_N)\ddd_2\otimes \delta \ddd
		,\,
		\nabla \delta \uu
	\rangle_{H^{-\frac{1}{2}}}.
\end{align*}
At this stage, we invoke Lemma \ref{lemma:SN-infty} about the Sobolev embedding $\| S_N\ddd_1 \|_{L^\infty}\leq C\sqrt{N}\| \ddd_1 \|_{H^1}$ and Lemma \ref{lemma:product} about the continuity of the product between $H^{1/2} \times L^2$ and $H^{-1/2}$. Then the inner product involving the low frequency part $S_N\ddd_1$ can be controlled as follows:
\begin{equation}\label{ineq_SNdAddd1}
\begin{aligned}
	\big| 
		\langle 
		(\ddd_2	\cdot (A_2 \ddd_2	))\delta \ddd\otimes S_N\ddd_1
		,\,
		\nabla \delta \uu
		\rangle_{H^{-\frac{1}{2}}}
	\big|
	&\leq 
	\| (\ddd_2	\cdot (A_2 \ddd_2	))\delta \ddd\otimes S_N\ddd_1
	\|_{H^{-\frac{1}{2}}}
	\| \nabla \delta\uu \|_{H^{-\frac{1}{2}}}\\
	&\lesssim 
	\| (\ddd_2	\cdot (A_2 \ddd_2)) S_N\ddd_1 \|_{L^2}
	\| \delta \ddd \|_{H^\frac{1}{2}}
	\| \nabla \delta \uu \|_{H^{-\frac{1}{2}}}\\
	&\lesssim 
	\| \ddd_2	\cdot (A_2 \ddd_2) \|_{L^2}
	\| S_N \ddd_1 \|_{L^\infty}
	\| \delta \ddd \|_{H^\frac{1}{2}}
	\| \nabla \delta \uu \|_{H^{-\frac{1}{2}}}\\
	&\lesssim 
	\| \ddd_2	\cdot (A_2 \ddd_2) \|_{L^2}
	\| \ddd_1 \|_{H^1}\sqrt{N}
	\| \delta \ddd \|_{H^\frac{1}{2}}
	\| \nabla \delta \uu \|_{H^{-\frac{1}{2}}}
	\\
	&\leq
	C\eta^{-1} 
	\| \ddd_2	\cdot (A_2 \ddd_2) \|_{L^2}^2
	\| \ddd_1 \|_{H^1}^2
	\| \delta \ddd \|_{H^\frac{1}{2}}^2
	N
	+
	\eta \nu \| \nabla \delta \uu \|_{H^{-\frac{1}{2}}}^2.
\end{aligned}
\end{equation}
With an analogous procedure, we have 
\begin{equation}\label{ineq_SNdAddd2}
	\big| 
		\langle 
		(\ddd_2	\cdot (A_2 \ddd_2	))S_N\ddd_2\otimes \delta \ddd
		,\,
		\nabla \delta \uu
		\rangle_{H^{-\frac{1}{2}}}
	\big|
	\leq 
	C\eta^{-1} 
	\| \ddd_2	\cdot (A_2 \ddd_2) \|_{L^2}^2
	\| \ddd_2 \|_{H^1}^2
	\| \delta \ddd \|_{H^\frac{1}{2}}^2
	N
	+
	\eta \nu  \| \nabla \delta \uu \|_{H^{-\frac{1}{2}}}^2.
\end{equation}
Now let us turn to the estimate of the high frequencies part involving  $(\Id-S_N)\ddd_1$ and $(\Id-S_N)\ddd_2$, which are the dominant ones. 
First, from Lemma \ref{lemma:product} we infer the continuity of the product between $H^{-1/4} \times H^{3/4}$ and $H^{-1/2}$ and get  
\begin{align*}
	\big| 
		\langle 
		(\ddd_2	\cdot (A_2 \ddd_2	))\delta \ddd\otimes (\Id -S_N)\ddd_1
		,\,
		\nabla \delta \uu
		\rangle_{H^{-\frac{1}{2}}}
	\big|
	&\leq
	\| (\ddd_2	\cdot (A_2 \ddd_2	))\delta \ddd\otimes (\Id -S_N)\ddd_1
	\|_{H^{-\frac{1}{2}}}
	\| \nabla \delta \uu \|_{H^{-\frac{1}{2}}}\\
	&\lesssim
	\| (\ddd_2	\cdot (A_2 \ddd_2))  \delta \ddd  \|_{H^{-\frac 14}}
	\| 	 (\Id -S_N)\ddd_1 \|_{H^\frac{3}{4}}
	\|  \delta \uu \|_{L^2}^\frac{1}{2}\|  \nabla \delta \uu \|_{L^2}^\frac{1}{2}.
\end{align*}
From the definition of $(\Id-S_N)\ddd_1 = \sum_{q=N}^\infty \Dd_q\ddd_1$, it follows  that 
$\|  (\Id -S_N)\ddd_1 \|_{H^\frac{3}{4}} \leq 2^{-N/4} \| \ddd_1 \|_{H^1}$.  Therefore, we obtain
\begin{equation}\label{ineq_SNdAddd3}
\begin{aligned}
	&\big| 
		\langle 
		(\ddd_2	\cdot (A_2 \ddd_2	))\delta \ddd\otimes (\Id -S_N)\ddd_1
		,\,
		\nabla \delta \uu
		\rangle_{H^{-\frac{1}{2}}}
	\big|\\
	&\quad \lesssim
	\| \ddd_2	\cdot (A_2 \ddd_2	)\|_{L^2}
	\| \delta \ddd   \|_{H^{\frac 34}}
	\| 	 \ddd_1  \|_{H^1}2^{-\frac{N}{4}}
	\|  \delta \uu \|_{L^2}^\frac{1}{2}\|  \nabla \delta \uu \|_{L^2}^\frac{1}{2}
	\\
	&\quad \leq
	C
	\| \ddd_2 \cdot (A_2 \ddd_2) \|_{L^2}
	\big(\|\uu_1\|_{L^2}+ \|\uu_2\|_{L^2}\big)^\frac{1}{2} 
	\big(\| \nabla \uu_1\|_{L^2}+ \|\nabla \uu_2 \|_{L^2}\big)^\frac{1}{2}
	\| \ddd_1 \|_{H^1} \big(\| \ddd_1 \|_{H^1} + \| \ddd_2 \|_{H^1} \big)
	2^{-\frac{N}{4}}.
\end{aligned}
\end{equation}
With an analogous procedure, one can estimate the term containing high frequencies of $(\Id-S_N)\ddd_2$ through
\begin{equation}\label{ineq_SNdAddd4}
\begin{aligned} 
	&\big|
	\langle 
		(\ddd_2	\cdot (A_2 \ddd_2	))
		((\Id- S_N)\ddd_2\otimes \delta \ddd)
		,\,
		\nabla \delta \uu
	\rangle_{H^{-\frac{1}{2}}}
	\big|\\
	&\quad \leq
	C
	\| \ddd_2	\cdot (A_2 \ddd_2	)	\|_{L^2}
	\big(\|\uu_1\|_{L^2}+ \|\uu_2\|_{L^2}\big)^\frac{1}{2} 
	\big(\| \nabla \uu_1\|_{L^2}+ \|\nabla \uu_2 \|_{L^2}\big)^\frac{1}{2}
	\| \ddd_2 \|_{H^1} \big(\| \ddd_1 \|_{H^1} + \| \ddd_2 \|_{H^1} \big)
	2^{-\frac{N}{4}}.
\end{aligned} 
\end{equation}
Combining \eqref{ineq_SNdAddd1}--\eqref{ineq_SNdAddd4}, we finally deduce that
\begin{equation}\label{d2d2nablau2deltadd1d2-final}
\begin{aligned}
	&\big|
	\langle 
		(\ddd_2	\cdot (A_2 \ddd_2	))(\delta \ddd\otimes\ddd_1
		\!+\!\ddd_2\otimes \delta \ddd)
		,
		\nabla \delta \uu
	\rangle_{H^{-\frac{1}{2}}}\!
	\big|\\
	&\quad \leq\  
	C\eta^{-1} 
	\| \ddd_2	\cdot (A_2 \ddd_2) \|_{L^2}^2
	\big( \| \ddd_1 \|_{H^1}^2 +\| \ddd_2 \|_{H^1}^2\big)
	\| \delta \ddd \|_{H^\frac{1}{2}}^2  N
	+ 
	2 \eta \nu \| \nabla \delta \uu \|_{H^{-\frac{1}{2}}}^2  
	\\
	&\qquad\ +C
	\| \ddd_2	\cdot (A_2 \ddd_2	)	\|_{L^2}
	\big(\|\uu_1\|_{L^2}+ \|\uu_2\|_{L^2}\big)^\frac{1}{2} 
	\big(\| \nabla \uu_1\|_{L^2}+ \|\nabla \uu_2 \|_{L^2}\big)^\frac{1}{2}
	\big(\| \ddd_1 \|_{H^1}^2 + \| \ddd_2 \|_{H^1}^2 \big)
	2^{-\frac{N}{4}},
\end{aligned}
\end{equation}
where the positive constant $C$ does not depend on $N$ and $\eta$. 
Like before, if $\Phi(t)=0$ then the right-hand side is identically null. We hence focus to the case of $\Phi(t)>0$ and  set 
\begin{equation}\label{def-N-prop-deltasigma1-first}
	N(t) = \Big\lfloor 4\log_2\Big(1+\frac{1}{\Phi(t)}\Big)\Big\rfloor +4.
\end{equation}
Thus, we see that $N(t)/4>\log_2(1+1/\Phi(t))$, which implies 
\begin{equation}\label{N-Phi}
	2^{-\frac{N(t)}{4}} \leq 2^{-\log_2\big(1+\frac{1}{\Phi(t)}\big)}
	= \frac{\Phi(t)}{1+\Phi(t)}\leq \Phi(t).
\end{equation}
Summarizing, we infer from the above estimates that 
\begin{equation}\label{d2d2nablau2deltadd1-final-estimate}
\begin{aligned}
	& \big|
	\langle 
		(\ddd_2	\cdot (A_2 \ddd_2	))(\delta \ddd\otimes\ddd_1
		+\ddd_2\otimes \delta \ddd)
		,\, \nabla \delta \uu \rangle_{H^{-\frac{1}{2}}}
	\big|\\
	&\quad \leq 
	C\eta^{-1} 
	\| \ddd_2	\cdot (A_2 \ddd_2) \|_{L^2}^2
	\big( \| \ddd_1 \|_{H^1}^2 +\| \ddd_2 \|_{H^1}^2\big)
	\Phi(t)
	\Big(1+\ln\Big(1+\frac{1}{\Phi(t)}\Big)\Big)\\
	&\qquad  
	+C
	\| \ddd_2	\cdot (A_2 \ddd_2	)	\|_{L^2}
	\big(\|\uu_1\|_{L^2}+ \|\uu_2\|_{L^2}\big)^\frac{1}{2} 
	\big(\| \nabla \uu_1\|_{L^2}+ \|\nabla \uu_2 \|_{L^2}\big)^\frac{1}{2}
	\big(\| \ddd_1 \|_{H^1}^2 + \| \ddd_2 \|_{H^1}^2 \big) 
	\Phi(t)\\
	&\qquad + 2\eta \nu \| \nabla \delta \uu \|_{H^{-\frac{1}{2}}}^2.
\end{aligned}
\end{equation}

Now we proceed to derive estimate for the inner product involving the second term in $\delta \bm{\sigma}_1$, i.e., $(\ddd_1\cdot (A_2 \delta\ddd	))\ddd_1\otimes\ddd_1$ in \eqref{delta-sigma}. To achieve this goal, we make use of the Bony's decomposition  \eqref{bony-decomp} to split the inner product as follows
\begin{equation*}
    \langle 
		(\ddd_1\cdot  (A_2\delta \ddd))\ddd_1\otimes \ddd_1,\,\nabla \delta \uu
	\rangle_{H^{-\frac{1}{2}}} = 
	\sum_{q=-1}^\infty
	2^{-q}
	\int_{\mathbb{T}^2}
	\Dd_q(\ddd_1 \cdot (A_2 \delta \ddd) \ddd_1 \otimes \ddd_1):\Dd_q \nabla \delta \uu\,\dd x
	=:
	\I + \I\I + \I\I\I + \I\mathcal{V},
\end{equation*}
where the four terms $\I$, $\I\I$, $\I\I\I$ and $\I\mathcal{V}$ on the right-hand side are defined by
\begin{equation}\label{def-I...IV-d1A2deltadd1d1}
\begin{aligned}
    \I
	&=
	\sum_{q=-1}^\infty
	\sum_{|j-q|\leq 5}
	2^{-q}
	\int_{\mathbb{T}^2}
	[\Dd_q,\,\sdm{j}(\ddd_1 \otimes \ddd_1)]\Dd_j(\ddd_1 \cdot (A_2 \delta \ddd)):\Dd_q \nabla \delta \uu\,\dd x,\\
	\I\I &=
	\sum_{q=-1}^\infty
	\sum_{|j-q|\leq 5}
	2^{-q}
	\int_{\mathbb{T}^2}
	(\sdm{j}-\sdm{q})(\ddd_1\otimes \ddd_1)
	\Dd_q\Dd_j(\ddd_1 \cdot (A_2 \delta \ddd)):\Dd_q \nabla \delta \uu\,\dd x,\\
	\I\I\I &=
	\sum_{q=-1}^\infty
	2^{-q}
	\int_{\mathbb{T}^2}
	\Dd_q(\ddd_1 \cdot (A_2 \delta \ddd))\sdm{q}(\ddd_1\otimes \ddd_1):\Dd_q \nabla \delta \uu\,\dd x,
	\\
	\I\mathcal{V} &=
	\sum_{q=-1}^\infty
	\sum_{j=q-5}^\infty
	2^{-q}
	\int_{\mathbb{T}^2}
	\Dd_q(S_{j+2}(\ddd_1 \cdot (A_2 \delta \ddd))\Dd_j(\ddd_1\otimes \ddd_1)):\Dd_q \nabla \delta \uu\,\dd x.
	\end{aligned}
\end{equation}

First, using the symmetry and skew symmetry of the tensors $\ddd_1\otimes\ddd_1$ and $\delta \omega$, respectively, we can handle the third term $\I\I\I$ in the following way
\begin{align*}
    |\I\I\I| 
    & 
    =\left|	\sum_{q=-1}^\infty
	2^{-q}
	\int_{\mathbb{T}^2}
	\Dd_q(\ddd_1 \cdot (A_2 \delta \ddd))\sdm{q}(\ddd_1\otimes \ddd_1):\Dd_q (\delta A+\delta\omega)\,\dd x\right|
	\\
	& 
    =\left|	\sum_{q=-1}^\infty
	2^{-q}
	\int_{\mathbb{T}^2}
	\Dd_q(\ddd_1 \cdot (A_2 \delta \ddd))\sdm{q}(\ddd_1\otimes \ddd_1):\Dd_q \delta A\,\dd x\right|
	\\
    &\leq 
    \sum_{q=-1}^\infty
	2^{-q}
	\|  \Dd_q(\ddd_1 \cdot (A_2 \delta \ddd)) \|_{L^2}
	\| \sdm{q}(\ddd_1\otimes \ddd_1):\Dd_q  \delta A \|_{L^2}\\
	&\leq 
	\Big(
	    \sum_{q=-1}^\infty
	    2^{-q}
	    \|  \Dd_q(\ddd_1 \cdot (A_2 \delta \ddd)) \|_{L^2}^2
	\Big)^\frac{1}{2}
	\Big( 
	     \sum_{q=-1}^\infty
	    2^{-q}
	    \| \sdm{q}(\ddd_1\otimes \ddd_1):\Dd_q  \delta A \|_{L^2}^2
	\Big)^\frac{1}{2}\\
	&\leq  C\eta^{-1} \| \ddd_1 \cdot (A_2 \delta \ddd) \|_{H^{-\frac{1}{2}}}^2+ \eta 	\sum_{q=-1}^\infty \!\!2^{-q}\!\!\int_{\TT^2}|\sdm{q}(\ddd_1\otimes\ddd_1):\Dd_q\delta A|^2(t,x)\, \dd x.
\end{align*}
 With an analogous procedure as the one used in the previous estimate, we separately analyse the low and high frequencies of the vector $\ddd_1$. More precisely, for a certain  positive integer $N\in \mathbb{N}$ that will be determined later on, we make the  decomposition $\ddd_1 = S_N \ddd_1 + (\Id- S_N )\ddd_1$ and obtain 
\begin{equation}\label{d1A2ddH12}
    \begin{aligned}
	\| \ddd_1\cdot( A_2 \delta \ddd) \|_{H^{-\frac{1}{2}}}^2
	&\lesssim 
	\| S_N \ddd_1 \cdot (A_2 \delta \ddd) \|_{H^{-\frac{1}{2}}}^2 +
	\| (\Id -S_N)\ddd_1 \cdot (A_2 \delta \ddd) \|_{H^{-\frac{1}{2}}}^2\\
	&\lesssim
	\| A_2(S_N \ddd_1) \|_{L^2}^2\| \delta \ddd \|_{H^\frac{1}{2}}^2+
	\| A_2 (\Id -S_N)\ddd_1 \|_{H^{-\frac{1}{4}}}^2 
	\| \delta \ddd \|_{H^{\frac{3}{4}}}^2
	\\
	&\lesssim
	\| S_N \ddd_1 		\|_{L^\infty}^2
	\| \nabla \uu_2 	\|_{L^2}^2
	\| \delta \ddd \|_{H^\frac{1}{2}}^2+
	\|   (\Id -S_N)\ddd_1 \|_{H^{\frac{3}{4}}}^2
	\| A_2 				\|_{L^2}^2
	\big(\|\ddd_1\|_{H^1}^2+\|\ddd_2\|_{H^{1}}^2\big)\\
	&\lesssim
	\|	\ddd_1 		\|_{H^1}^2
	\| \nabla \uu_2 	\|_{L^2}^2
	\| \delta \ddd \|_{H^\frac{1}{2}}^2N+
	\| \ddd_1 						\|_{H^{1}}^2 
	\big(\|\ddd_1\|_{H^1}^2+\|\ddd_2\|_{H^{1}}^2\big) 
	\| \nabla \uu_2 				\|_{L^2}^2
	2^{-\frac{N}{2}},
\end{aligned}
\end{equation}
which eventually leads to
\begin{equation}\label{III-d1A2deltad}
\begin{aligned}
    |\I\I\I| 
    &\leq 
    C\eta^{-1}
    \|	\ddd_1 		\|_{H^1}^2
	\| \nabla \uu_2 	\|_{L^2}^2
	\| \delta \ddd \|_{H^\frac{1}{2}}^2N
	+
	C\eta^{-1}
	\big(\|\ddd_1\|_{H^1}^4+\|\ddd_2\|_{H^{1}}^4\big) 
	\| \nabla \uu_2 				\|_{L^2}^2
	2^{-\frac{N}{2}}\\
	&\quad + \eta 	\sum_{q=-1}^\infty \!\!2^{-q}\!\!\int_{\TT^2}|\sdm{q}(\ddd_1\otimes\ddd_1):\Dd_q\delta A|^2(t,x)\, \dd x.
	\end{aligned}
\end{equation}

Let us turn to the estimate of the first term $\I$ in \eqref{def-I...IV-d1A2deltadd1d1}. We separately analyse the low and high frequencies by splitting the tensor $\ddd_1\otimes \ddd_1 = 
S_N (\ddd_1\otimes \ddd_1) + (\Id-S_N)(\ddd_1\otimes \ddd_1)$, so that $\I$ can be decomposed as follows
\begin{equation}\label{I-d1d1d1A2deltad}
\begin{aligned}
    \I &= 
    \sum_{q=-1}^\infty
	\sum_{|j-q|\leq 5}
	2^{-q}
	\int_{\mathbb{T}^2}
	[\Dd_q,\,\sdm{j}S_N(\ddd_1 \otimes \ddd_1)]\Dd_j(\ddd_1 \cdot (A_2 \delta \ddd)):\Dd_q \nabla \delta \uu \,\dd x
	\\
	&\quad +
	\sum_{q=-1}^\infty
	\sum_{|j-q|\leq 5}
	2^{-q}
	\int_{\mathbb{T}^2}
	[\Dd_q,\,\sdm{j}(\Id-S_N)(\ddd_1 \otimes \ddd_1)]\Dd_j(\ddd_1 \cdot (A_2 \delta \ddd)):\Dd_q \nabla \delta \uu\,\dd x \\
	& =: \I_1 + \I_2.
\end{aligned}
\end{equation}
The term $\I_2$ that involves the high frequencies can be estimated by using the commutator estimate of Lemma \ref{prop:comm-est} and the Bernstein inequality, namely, 
\begin{align*}
    |\I_2 |
    &\leq 
    \sum_{q=-1}^\infty
	\sum_{|j-q|\leq 5}
	2^{-q}
	2^{-q}
	\| \sdm{j}(\Id-S_N)\nabla (\ddd_1 \otimes \ddd_1)\|_{L^\infty}
	\| \Dd_j(\ddd_1 \cdot (A_2 \delta \ddd)) \|_{L^2}
	\| \Dd_q \nabla \delta \uu \|_{L^2}\\
	&\lesssim 
    \sum_{q=-1}^\infty
	\sum_{|j-q|\leq 5}
	2^{-2q}
	2^{j}
	\| \sdm{j}(\Id-S_N)(\nabla \ddd_1 \otimes \ddd_1)\|_{L^2}
	\| \Dd_j(\ddd_1 \cdot (A_2 \delta \ddd)) \|_{L^2}
	\| \Dd_q \nabla \delta \uu \|_{L^2}
	\\
	&\lesssim 
    \sum_{q=-1}^\infty
	\sum_{|j-q|\leq 5}
	2^{2(j-q)} 
	2^{- \frac{j}{2}}
	\| \sdm{j}(\Id-S_N)(\nabla \ddd_1 \otimes \ddd_1)\|_{L^2}
	2^{-\frac{j}{2}}
	\| \Dd_j(\ddd_1 \cdot (A_2 \delta \ddd)) \|_{L^2}
	\| \Dd_q \nabla \delta \uu \|_{L^2}\\
	&\lesssim 
	\| (\Id-S_N)(\nabla \ddd_1 \otimes \ddd_1 ) \|_{H^{-\frac{1}{2}}}
	\| \ddd_1 \cdot (A_2 \delta \ddd   )         \|_{H^{-\frac{1}{2}}}
	\| \nabla \delta \uu                        \|_{L^2}
	\\
	&\lesssim 
	2^{-\frac{N}{4}}
	\| (\Id-S_N)(\nabla \ddd_1 \otimes \ddd_1 ) \|_{H^{-\frac{1}{4}}}
	\| \ddd_1 \otimes \delta \ddd               \|_{H^\frac{1}{2}} 
	\| A_2                                      \|_{L^2}
	\| \nabla \delta \uu                        \|_{L^2}
	\\
	&\lesssim 
	2^{-\frac{N}{4}}
	(\| \nabla \ddd_1    \|_{L^2} 
    \|  \ddd_1          \|_{H^{\frac{3}{4}}})
	(\| \ddd_1           \|_{H^{\frac{3}{4}}}
	\| \delta \ddd      \|_{H^\frac{3}{4}}) 
	\| \nabla \uu_2            \|_{L^2}
	\big(\| \nabla \uu_1\|_{L^2}+\|\nabla \uu_2  \|_{L^2}\big)\\
	&\leq C
	\|\ddd_1\|_{H^1}^3\big(\| \ddd_1\|_{H^1}+\|\ddd_2\|_{H^1}\big)
	\| \nabla \uu_2            \|_{L^2}
	\big(\| \nabla \uu_1\|_{L^2}+\|\nabla \uu_2  \|_{L^2}\big)
	2^{-\frac{N}{4}},
\end{align*}
where the constant $C>0$ is indpendent of $N$. Next, concerning the low frequencies part, namely, the term $\I_1$ in \eqref{I-d1d1d1A2deltad}, applying the commutator estimate of Lemma \ref{prop:comm-est}, we deduce that 
\begin{align*}
    |\I_1 |
    &\leq
    \sum_{q =-1}^\infty 
    \sum_{|j-q|\leq 5}
    2^{-q}2^{-q}
    \| \sdm{j} S_N \nabla (\ddd_1\otimes \ddd_1) \|_{L^\frac{2}{\ee}}
    \| \Dd_j(\ddd_1\cdot (A_2 \delta\ddd)) \|_{L^2}
    \| \Dd_q \nabla \delta \uu \|_{L^\frac{2}{1-\ee}}\\
    &\lesssim
    \sum_{q =-1}^\infty 
    \sum_{|j-q|\leq 5}
    2^{-q}
    \| \sdm{j} S_N(\nabla \ddd_1\otimes \ddd_1) \|_{L^\frac{2}{\ee}}
    \| \Dd_j(\ddd_1\cdot (A_2 \delta\ddd)) \|_{L^2}
    \| \Dd_q \delta \uu \|_{L^\frac{2}{1-\ee}}\\
    &\lesssim 
    \sum_{q =-1}^\infty 
    \sum_{|j-q|\leq 5}
    2^{-q}
    \| \sdm{j} ([S_N,\,\nabla \ddd_1]\otimes \ddd_1) \|_{L^\frac{2}{\ee}}
    \| \Dd_j(\ddd_1\cdot (A_2 \delta\ddd)) \|_{L^2}
    \| \Dd_q \delta \uu \|_{L^\frac{2}{1-\ee}}
    \\
    &\quad +
    \sum_{q =-1}^\infty 
    \sum_{|j-q|\leq 5}
    2^{-q}
    \| \sdm{j} (\nabla \ddd_1\otimes S_N\ddd_1) \|_{L^\frac{2}{\ee}}
    \| \Dd_j(\ddd_1\cdot (A_2 \delta\ddd)) \|_{L^2}
    \| \Dd_q \delta \uu \|_{L^\frac{2}{1-\ee}}\\
    &=:
    \I_1^{(1)} + \I_1^{(2)}.
\end{align*}
The term $\I_1^{(1)}$ that involves the commutator for $S_N$ can be handled by
\begin{align*}
    \I_1^{(1)} 
    &\lesssim 
    \sum_{q =-1}^\infty 
    \sum_{|j-q|\leq 5}
    2^{-q}
    2^{j\big(\frac{3}{2}-\ee\big)}
    \| \sdm{j} ([S_N,\,\nabla \ddd_1]\otimes \ddd_1) \|_{L^\frac{4}{3}}
    \| \Dd_j(\ddd_1\cdot (A_2 \delta\ddd)) \|_{L^2}
    2^{q\ee}
    \| \Dd_q \delta \uu \|_{L^2}\\
    &\lesssim 
    \| [S_N,\,\nabla \ddd_1]\otimes \ddd_1 \|_{L^\frac{4}{3}}
    \sum_{q =-1}^\infty 
    \sum_{|j-q|\leq 5} 2^{\frac74{(j-q)}}2^{(q-j)\ee} 
    2^{-\frac{j}{4}}
    \| \Dd_j(\ddd_1\cdot (A_2 \delta\ddd)) \|_{L^2}
    2^{\frac{3}{4}q}
    \| \Dd_q \delta \uu \|_{L^2}\\
    &\lesssim 
    \| [S_N,\,\nabla \ddd_1]\otimes \ddd_1 \|_{L^\frac{4}{3}}
    \sum_{q =-1}^\infty 
    \sum_{|j-q|\leq 5} 
    2^{-\frac{j}{4}}
    \| \Dd_j(\ddd_1\cdot (A_2 \delta\ddd)) \|_{L^2}
    2^{\frac{3}{4}q}
    \| \Dd_q \delta \uu \|_{L^2},
\end{align*}
where we implicitly use the fact $\ee\in(0,1/2]$ (see below).  
On one hand, by definition of the commutator, it holds 
\begin{align*}
    \| [S_N,\,\nabla \ddd_1]\otimes \ddd_1 \|_{L^\frac{4}{3}}
    &\leq  
    \| S_N(\nabla \ddd_1 \otimes \ddd_1) \|_{L^\frac{4}{3}}+
    \| \nabla \ddd_1 \otimes S_N \ddd_1  \|_{L^\frac{4}{3}}\\
    &\lesssim 
    \| \nabla \ddd_1\otimes \ddd_1 \|_{L^\frac43}+
    \| \nabla \ddd_1 \|_{L^2} \|S_N \ddd_1 \|_{L^4}\\
    &\lesssim   
    \| \nabla \ddd_1 \|_{L^2}
    \| \ddd_1 \|_{L^4},
\end{align*}
and on the other hand, it follows from the commutator estimate in Lemma \ref{prop:comm-est} that 
\begin{align*}
    \| [S_N,\,\nabla \ddd_1]\otimes \ddd_1 \|_{L^\frac{4}{3}}
    \leq  
    C2^{-N}
    \| \nabla  \ddd_1 \|_{H^1}
    \| \ddd_1 \|_{L^4}.
\end{align*}
Then a standard interpolation leads to
\begin{equation}
    \| [S_N,\,\nabla \ddd_1]\otimes \ddd_1 \|_{L^\frac{4}{3}}
    \leq  
    C2^{-\frac N4}
    \| \ddd_1 \|_{L^4}
    \| \nabla \ddd_1 \|_{L^2}^\frac{3}{4}
    \| \nabla \ddd_1 \|_{H^1}^\frac{1}{4}. 
    \label{interSN43}
\end{equation}
Therefore, we can deduce that  
\begin{align*}
    \I_1^{(1)}
    &\lesssim 
    2^{-\frac N4}
    \| \ddd_1 \|_{L^4}
    \| \nabla \ddd_1 \|_{L^2}^\frac{3}{4}
    \| \nabla  \ddd_1 \|_{H^1}^\frac{1}{4}
    \sum_{q =-1}^\infty 
    \sum_{|j-q|\leq 5}
    2^{-\frac{j}{4}}
    \| \Dd_j(\ddd_1\cdot (A_2 \delta\ddd)) \|_{L^2}
    2^{\frac{3}{4}q}
    \| \Dd_q \delta \uu \|_{L^2}\\
    &\lesssim 
    2^{-\frac N4}
    \| \ddd_1 \|_{H^1}^\frac{7}{4}
    \| \nabla \ddd_1 \|_{H^1}^\frac{1}{4}
    \| \ddd_1\cdot (A_2 \delta\ddd) \|_{H^{-\frac{1}{4}}}
    \|  \delta \uu                  \|_{H^\frac{3}{4}}\\
    &\lesssim 
    2^{-\frac N4}
    \| \ddd_1 \|_{H^1}^\frac{7}{4}
    \| \nabla \ddd_1 \|_{H^1}^\frac{1}{4}
    \| \ddd_1 \otimes \delta \ddd    \|_{H^\frac{3}{4}}
    \| A_2  \|_{L^2}
    \|  \delta \uu                   \|_{L^2}^\frac{1}{4}
    \|  \nabla \delta \uu                   \|_{L^2}^\frac{3}{4}\\
    &\lesssim 
    \| \ddd_1 \|_{H^1}^\frac{7}{4}
    \|\ddd_1\|_{H^\frac78}
    \|\delta \ddd\|_{H^\frac78}
    \| \nabla  \ddd_1 \|_{H^1}^\frac{1}{4}
    \big(\|\uu_1\|_{L^2}+\|\uu_2                  \|_{L^2}\big)^\frac{1}{4}
    \big(\| \nabla \uu_1\|_{L^2}+\|\nabla \uu_2 \|_{L^2}\big)^\frac{3}{4}
    \| \nabla \uu_2  \|_{L^2}
    2^{-\frac N4}
    \\
    &\lesssim 
    \|\ddd_1\|_{H^1}^\frac{11}{4}
    \big(\|\ddd_1\|_{H^1}+\|\ddd_2\|_{H^1}\big)
    \big(\|\uu_1\|_{L^2}+\|\uu_2                  \|_{L^2}\big)^\frac{1}{4}\| \nabla  \ddd_1 \|_{H^1}^\frac{1}{4}
    \big(\| \nabla \uu_1\|_{L^2}+\|\nabla \uu_2 \|_{L^2}\big)^\frac{3}{4}
    \| \nabla \uu_2  \|_{L^2}
    2^{-\frac N4}.
\end{align*}
Next, for $\I_1^{(2)}$, we infer from Lemma \ref{lemma:SN-infty} and Lemma \ref{lemma:eps} that 
\begin{align*}
\I_1^{(2)}
&\leq 
\sum_{q =-1}^\infty \sum_{|j-q|\leq 5}
2^{-q}
\| \nabla \ddd_1 \|_{L^\frac{2}{\ee}}
\| S_N \ddd_1 \|_{L^\infty}
\| \Dd_j (\ddd_1\cdot (A_2\delta \ddd))\|_{L^2}
\| \Dd_q \delta \uu \|_{L^\frac{2}{1-\ee}}\\
&\lesssim 
\sum_{q =-1}^\infty \sum_{|j-q|\leq 5}
2^{-q}
\frac{1}{\sqrt{\ee}}
\| \nabla \ddd_1 \|_{L^2}^{\ee}
\| \nabla \ddd_1 \|_{H^1}^{1-\ee}
\sqrt{N}
\| \ddd_1 \|_{H^1}
\| \Dd_j (\ddd_1\cdot (A_2\delta \ddd))\|_{L^2}
\| \Dd_q \delta \uu \|_{L^2}^{1-\ee}
\| \Dd_q \delta \uu \|_{L^\infty}^{\ee}\\
&\lesssim 
\sqrt{\frac{N}{\ee}}
\| \nabla \ddd_1 \|_{L^2}^{\ee}
\| \nabla \ddd_1 \|_{H^1}^{1-\ee}
\| \ddd_1 \|_{H^1}
\sum_{q =-1}^\infty \sum_{|j-q|\leq 5}
2^{-q}
\| \Dd_j (\ddd_1\cdot (A_2\delta \ddd))\|_{L^2}
\| \Dd_q \delta \uu \|_{L^2}^{1-\ee}
(2^q\| \Dd_q \delta \uu \|_{L^2})^{\ee}\\
&\lesssim 
\sqrt{\frac{N}{\ee}}
\| \ddd_1 \|_{H^1}^{1+\ee}
\| \nabla \ddd_1 \|_{H^1}^{1-\ee}
\| \ddd_1\cdot (A_2\delta \ddd))\|_{H^{-\frac{1}{2}}}
\| \delta \uu \|_{H^{-\frac{1}{2}}}^{1-\ee}
\| \nabla \delta \uu \|_{H^{-\frac{1}{2}}}^{\ee}\\
    &\leq C\eta^{-\frac{\ee}{1-\ee}}\left(\frac{N}{\ee}\right)^\frac{1}{1-\ee}   
    \|\ddd_1\|_{H^1}^\frac{2(1+\ee)}{1-\ee}
    \|\ddd_1\|_{H^2}^{2}\|\delta \uu \|_{H^{-\frac12}}^2
    + \|\ddd_1\cdot (A_2 \delta\ddd)\|_{H^{-\frac12}}^2
    + \eta \nu \|\nabla \delta \uu \|_{H^{-\frac12}}^{2},
\end{align*}
where the constant $C>0$ is independent of $\ee$ and $N$. Keeping in mind that we will take $\eta\in (0,1)$ and $\ee\in (0,1/2]$ later, then it follows that 
\begin{align*}
\I_1^{(2)} \leq
C\eta^{-1}\left(\frac{N}{\ee}\right)^\frac{\ee}{1-\ee}   
    \|\ddd_1\|_{H^1}^\frac{2(1+\ee)}{1-\ee}
    \|\ddd_1\|_{H^2}^{2}\|\delta \uu \|_{H^{-\frac12}}^2\left(\frac{N}{\ee}\right)  
    + \|\ddd_1\cdot (A_2 \delta\ddd)\|_{H^{-\frac12}}^2
    + \eta \nu \|\nabla \delta \uu \|_{H^{-\frac12}}^{2}.
\end{align*}
Collecting the above estimates and using \eqref{d1A2ddH12}, we arrive at 
\begin{equation}\label{I-d1A2deltad}
\begin{aligned}
\I&\leq C
\big(\| \ddd_1\|_{H^1}^4+\|\ddd_2\|_{H^1}^4+
\|\uu_2\|_{L^2}^4+\|\uu_2\|_{L^2}^4\big)
	\big(\|\nabla \ddd_1\|_{H^1}^2+ \| \nabla \uu_1\|_{L^2}^2+\|\nabla \uu_2  \|_{L^2}^2\big)
	2^{-\frac{N}{4}}\\
	 &\quad + C\|	\ddd_1 		\|_{H^1}^2
	\| \nabla \uu_2 	\|_{L^2}^2
	\| \delta \ddd \|_{H^\frac{1}{2}}^2N
	 + C\eta^{-1}\left(\frac{N}{\ee}\right)^\frac{\ee}{1-\ee}   
    \|\ddd_1\|_{H^1}^\frac{2(1+\ee)}{1-\ee}
    \|\ddd_1\|_{H^2}^{2}\|\delta \uu \|_{H^{-\frac12}}^2\left(\frac{N}{\ee}\right)  \\
    &\quad  + \eta \nu \|\nabla \delta \uu \|_{H^{-\frac12}}^{2}.
\end{aligned}
\end{equation}

Now concerning the term $\I\I$, we observe that 
\begin{align*}
    \I\I &=
	\sum_{q=-1}^\infty
	\sum_{|j-q|\leq 5}
	2^{-q}
	\int_{\mathbb{T}^2}
	(\sdm{j}-\sdm{q})S_N(\ddd_1\otimes \ddd_1)
	\Dd_q\Dd_j(\ddd_1 \cdot (A_2 \delta \ddd)):\Dd_q \nabla \delta \uu\,\mathrm{d}x\\
	&\quad + \sum_{q=-1}^\infty
	\sum_{|j-q|\leq 5}
	2^{-q}
	\int_{\mathbb{T}^2}
	(\sdm{j}-\sdm{q})(\mathrm{Id}-S_N)(\ddd_1\otimes \ddd_1)
	\Dd_q\Dd_j(\ddd_1 \cdot (A_2 \delta \ddd)):\Dd_q \nabla \delta \uu \,\mathrm{d}x.
\end{align*}
For $|j-q|\leq 5$, it holds 
\begin{align*}
&\|(\sdm{j}-\sdm{q})(\mathrm{Id}-S_N)(\ddd_1\otimes \ddd_1))
	\Dd_q\Dd_j(\ddd_1 \cdot (A_2 \delta \ddd))\|_{L^2}\\
&\quad \lesssim  \|(\sdm{j}-\sdm{q})(\mathrm{Id}-S_N)(\ddd_1\otimes \ddd_1)\|_{L^\infty}
	\|\Dd_q\Dd_j(\ddd_1 \cdot (A_2 \delta \ddd))\|_{L^2}\\
	&\quad \lesssim 2^{-q}  \|(\sdm{j}-\sdm{q})(\mathrm{Id}-S_N)\nabla (\ddd_1\otimes \ddd_1)\|_{L^\infty}
	\|\Dd_q\Dd_j(\ddd_1 \cdot (A_2 \delta \ddd))\|_{L^2}
	\\
	&\quad  \lesssim  \|(\sdm{j}-\sdm{q})(\mathrm{Id}-S_N)\nabla (\ddd_1\otimes \ddd_1)\|_{L^2}
	\|\Dd_q\Dd_j(\ddd_1 \cdot (A_2 \delta \ddd))\|_{L^2},
\end{align*}
and
\begin{align*}
\| (\sdm{j}-\sdm{q})S_N(\ddd_1\otimes \ddd_1)\|_{L^\frac{2}{\ee}}
& \lesssim 
2^{-q}\|(\sdm{j}-\sdm{q})S_N\nabla (\ddd_1\otimes \ddd_1)\|_{L^\frac{2}{\ee}}.
\end{align*}
Hence, by applying a similar argument like for $\I$, we can deduce that 
\begin{equation}\label{II-d1A2deltad}
\begin{aligned}
\I\I &\leq C
\big(\| \ddd_1\|_{H^1}^4+\|\ddd_2\|_{H^1}^4+
\|\uu_2\|_{L^2}^4+\|\uu_2\|_{L^2}^4\big)
	\big(\|\nabla \ddd_1\|_{H^1}^2+ \| \nabla \uu_1\|_{L^2}^2+\|\nabla \uu_2  \|_{L^2}^2\big)
	2^{-\frac{N}{4}}\\
	 &\quad + C\|	\ddd_1 		\|_{H^1}^2
	\| \nabla \uu_2 	\|_{L^2}^2
	\| \delta \ddd \|_{H^\frac{1}{2}}^2N
	 + C\eta^{-1}\left(\frac{N}{\ee}\right)^\frac{\ee}{1-\ee}   
    \|\ddd_1\|_{H^1}^\frac{2(1+\ee)}{1-\ee}
    \|\ddd_1\|_{H^2}^{2}\|\delta \uu \|_{H^{-\frac12}}^2\left(\frac{N}{\ee}\right)  \\
    &\quad  + \eta \nu \|\nabla \delta \uu \|_{H^{-\frac12}}^{2}.
\end{aligned}
\end{equation}

It remains to estimate the last term $\I\mathcal{V}$ in \eqref{def-I...IV-d1A2deltadd1d1}. Once again, we proceed with localising the high and low frequencies related to the tensor $\ddd_1\otimes \ddd_1$ with respect to the operator $S_N$ such that
\begin{equation}\label{IV-d1A2deltadd1d1}
\begin{aligned}
\I\mathcal{V} &=
	\sum_{q=-1}^\infty
	\sum_{j=q-5}^\infty
	2^{-q}
	\int_{\mathbb{T}^2}
	\Dd_q(S_{j+2}(\ddd_1 \cdot (A_2 \delta \ddd))\Dd_j(S_N(\ddd_1\otimes \ddd_1))):\Dd_q \nabla \delta \uu\,\dd x\\
	&\quad +
	\sum_{q=-1}^\infty
	\sum_{j=q-5}^\infty
	2^{-q}
	\int_{\mathbb{T}^2}
	\Dd_q(S_{j+2}(\ddd_1 \cdot (A_2 \delta \ddd))\Dd_j((\mathrm{Id}-S_N)(\ddd_1\otimes \ddd_1))) :\Dd_q \nabla \delta \uu\,\dd x =:\I\mathcal{V}_1+\I\mathcal{V}_2.  
\end{aligned}
\end{equation}
We begin our analysis by dealing with $\I\mathcal{V}_2$, which treats the high frequencies of $(\Id-S_N)(\ddd_1\otimes \ddd_1)$. Invoking the Bernstein inequality in Lemma \ref{prop:Bernstein}, we see that $\| \Dd_q f \|_{L^2}\leq 2^q \| \Dd_q f \|_{L^1}$ for any function $f$ in $L^1(\TT^2)$. Besides, the operator $\Dd_q$ maps $L^1(\TT^2)$ into itself (with a constant of embedding, not depending on the index $q$). Then we deduce that 
\begin{align*}
    |\I\mathcal{V}_2| 
    &\leq 
    \sum_{q=-1}^\infty
	\sum_{j=q-5}^\infty
	2^{-q}
	\| 	\Dd_q(S_{j+2}(\ddd_1 \cdot (A_2 \delta \ddd))\Dd_j((\mathrm{Id}-S_N)(\ddd_1\otimes \ddd_1))) \|_{L^2}
	\| \Dd_q \nabla \delta \uu \|_{L^2}
	\\
	&\lesssim
    \sum_{q=-1}^\infty
	\sum_{j=q-5}^\infty
	2^{-q}
	2^{q}
	\| 	\Dd_q(S_{j+2}(\ddd_1 \cdot (A_2 \delta \ddd))\Dd_j((\mathrm{Id}-S_N)(\ddd_1\otimes \ddd_1))) \|_{L^1}
	\| \Dd_q \nabla  \delta \uu \|_{L^2}
	\\
	&\lesssim
    \sum_{q=-1}^\infty
	\sum_{j=q-5}^\infty
	\| 	S_{j+2}(\ddd_1 \cdot (A_2 \delta \ddd))\Dd_j((\mathrm{Id}-S_N)(\ddd_1\otimes \ddd_1)) \|_{L^1}
	\| \Dd_q \nabla \delta \uu \|_{L^2}
	\\
	&\lesssim
    \sum_{q=-1}^\infty
	\sum_{j=q-5}^\infty
	\| 	S_{j+2}(\ddd_1 \cdot (A_2 \delta \ddd)) \|_{L^2}
	\| \Dd_j((\mathrm{Id}-S_N)(\ddd_1\otimes \ddd_1)) \|_{L^2}
	\| \Dd_q \nabla \delta \uu \|_{L^2}.
\end{align*}
Since $(\mathrm{Id}-S_N)(\ddd_1\otimes \ddd_1)$ has null average (indeed $(\mathrm{Id}-S_N)$ erases the low frequencies of any function), then we see that  
$\| \Dd_j((\mathrm{Id}-S_N)(\ddd_1\otimes \ddd_1)) \|_{L^2} = 
2^{-j}2^j\| \Dd_j((\mathrm{Id}-S_N)(\ddd_1\otimes \ddd_1)) \|_{L^2}\leq 
C2^{-j}\| \Dd_j\nabla ((\mathrm{Id}-S_N)(\ddd_1\otimes \ddd_1)) \|_{L^2}$, thanks to the Bernstein inequality in Lemma \ref{prop:Bernstein}. Recalling that $\nabla$ and $(\Id-S_N)$ commute as operators, we gather that
\begin{align*}
    |\I\mathcal{V}_2| 
	&\lesssim
    \sum_{q=-1}^\infty
	\sum_{j=q-5}^\infty
	\| 	S_{j+2}(\ddd_1 \cdot (A_2 \delta \ddd)) \|_{L^2}2^{-j}
	\| \Dd_j((\mathrm{Id}-S_N)\nabla (\ddd_1\otimes \ddd_1)) \|_{L^2}
	\| \Dd_q \nabla \delta \uu \|_{L^2}
	\\
	&\lesssim
    \sum_{q=-\infty}^\infty
	\sum_{j=q-5}^\infty
	2^{\frac{q-j}{2}}
	2^{-\frac{j}{4}}
	\| 	S_{j+2}(\ddd_1 \cdot (A_2 \delta \ddd)) \|_{L^2}
	2^{-\frac{j}{4}}
	\| \Dd_j((\mathrm{Id}-S_N)(\nabla \ddd_1\otimes \ddd_1)) \|_{L^2}
	2^\frac{q}{2}
	\| \Dd_q \delta \uu \|_{L^2}
	\mathbf{1}_{[-1,\infty)}(q)
	\\
	&\lesssim
	\Big(
    \sum_{q=-\infty}^\infty
    \Big|
	\sum_{j=-\infty}^\infty
	2^{\frac{q-j}{2}}
	\mathbf{1}_{(-\infty,5]}(q-j)
	2^{-\frac{j}{4}}
	\| 	S_{j+2}(\ddd_1 \cdot (A_2 \delta \ddd)) \|_{L^2}
	2^{-\frac{j}{4}}
	\| \Dd_j((\mathrm{Id}-S_N)(\nabla \ddd_1\otimes \ddd_1)) \|_{L^2}
	\mathbf{1}_{[-1,\infty)}(j)
	\Big|^2
	\Big)^\frac{1}{2}
	\\ 
	&\qquad {\times }
	\Big(
	\sum_{q=-\infty}^\infty
	2^q
	\| \Dd_q \delta \uu \|_{L^2}^2\mathbf{1}_{[-1,\infty)}
	\Big)^\frac{1}{2}.
\end{align*}
Applying Young's inequality on convolution of sequences, we obtain 
\begin{align*}
    |\I\mathcal{V}_2| 
    &\lesssim 
    \Big(
        \sum_{q=-\infty}^\infty 
        2^{\frac{q}{2}}\mathbf{1}_{(-\infty,5]}(q)
    \Big)
    \Big(
        \sum_{j = -\infty}^\infty 
        \Big|
        2^{-\frac{j}{4}}
	\| 	S_{j+2}(\ddd_1 \cdot (A_2 \delta \ddd)) \|_{L^2}
	2^{-\frac{j}{4}}
	\| \Dd_j((\mathrm{Id}-S_N)(\nabla \ddd_1\otimes \ddd_1)) \|_{L^2}
	\mathbf{1}_{[-1,\infty)}(j)
	\Big|^2
	\Big)^\frac{1}{2}
	\\ &\qquad {\times }
	\Big(
	\sum_{q=-1}^\infty
	2^q
	\| \Dd_q \delta \uu \|_{L^2}^2
	\Big)^\frac{1}{2}.
\end{align*}
While the first term on the right-hand side of the last inequality is uniformly bounded, the remaining ones lead to the following estimate:
\begin{align*}
    |\I\mathcal{V}_2| 
    &\lesssim
    \| \ddd_1 \cdot (A_2\delta \ddd) \|_{H^{-\frac{1}{4}}}
    \| (\Id -S_N) (\nabla \ddd_1\otimes \ddd_1) \|_{B^{-\frac{1}{4}}_{2,\infty}}
    \| \delta \uu \|_{H^\frac{1}{2}}\\
    &\lesssim
    \| \ddd_1 \cdot (A_2\delta \ddd) \|_{H^{-\frac{1}{4}}}
    \| (\Id -S_N) (\nabla \ddd_1\otimes \ddd_1) \|_{H^{-\frac{1}{4}}}
    \| \delta \uu \|_{H^\frac{1}{2}}, 
\end{align*}
where we have invoked the following embedding between Besov spaces $H^{-1/4}= B_{2,2}^{-1/4}  \hookrightarrow B_{2,\infty}^{-1/4}$. Recalling Lemma \ref{lemma:product} that states the continuity of the product from $H^{3/4}\times L^2$ to $H^{-1/4}$ and from $H^{7/8}\times H^{7/8}$ to $H^{3/4}$, then by Ladyzhenskaya's inequality $\| f \|_{L^4}\lesssim \| f \|_{L^2}^{1/2}\|  f \|_{H^1}^{1/2}$ in two dimensions, we get
\begin{align*}
    |\I\mathcal{V}_2| 
    &\lesssim 
    \| \ddd_1 \otimes \delta \ddd \|_{H^\frac{3}{4}}
    \| A_2 \|_{L^2}
    2^{-\frac{N}{4}}
    \| \nabla \ddd_1\otimes \ddd_1  \|_{L^2}
    \| \delta \uu                   \|_{L^2}^\frac{1}{2}
    \| \nabla \delta \uu            \|_{L^2}^\frac{1}{2}\\
    &\lesssim 
    \| \ddd_1  \|_{H^\frac{7}{8}}
    \| \delta \ddd \|_{H^\frac{7}{8}}
    \| \nabla \uu_2 \|_{L^2}
    \| \nabla \ddd_1 \|_{L^4}
    \| \ddd_1  \|_{L^4}
    \| \delta \uu                   \|_{L^2}^\frac{1}{2}
    \| \nabla \delta \uu            \|_{L^2}^\frac{1}{2}
    2^{-\frac{N}{4}}\\
    &\lesssim 
    \|\ddd_1\|_{H^1}^2(\|\ddd_1\|_{H^1} +\|\ddd_2 \|_{H^1})
    \| \nabla \uu_2 \|_{L^2}
    \| \nabla \ddd_1 \|_{L^2}^\frac{1}{2}
    \| \nabla \ddd_1 \|_{H^1}^\frac{1}{2}
    (\| \uu_1\|_{L^2} +\|\uu_2                \|_{L^2})^\frac{1}{2}
    \\
    &\quad \times (\| \nabla \uu_1\|_{L^2}+ \|\nabla \uu_2\|_{L^2})^\frac{1}{2}
    2^{-\frac{N}{4}}\\
    &\leq 
    C
    (\| \ddd_1\|_{H^1}^\frac72 +\|\ddd_2  \|_{H^1}^\frac{7}{2})
    (\| \uu_1\|_{L^2}^\frac{1}{2} +\|\uu_2                \|_{L^2}^\frac{1}{2})
    (\| \nabla \uu_1\|_{L^2}^\frac{3}{2} + \|\nabla \uu_2\|_{L^2}^\frac{3}{2})
    \| \nabla \ddd_1 \|_{H^1}^\frac{1}{2}
    2^{-\frac{N}{4}},
\end{align*}
where the constant $C>0$ does not depend on the parameter $N\in \mathbb{N}$.  We now turn our attention to the low frequencies part  $\I\mathcal{V}_1$. Using again the Bernstein inequality such that  $\|\Dd_q f \|_{L^2}\lesssim 2^q \| \Dd_q f \|_{L^1}$, and invoking the continuity of $\Dd_q$ from $L^1$ to itself, we infer that
\begin{align*}
    |\I \mathcal{V}_1|
    &\leq 
    \sum_{q=-1}^\infty
	\sum_{j=q-5}^\infty
	2^{-q}
	\| \Dd_q(S_{j+2}(\ddd_1 \cdot (A_2 \delta \ddd))\Dd_j(S_N(\ddd_1\otimes \ddd_1))) \|_{L^2}
	\| \Dd_q \nabla \delta \uu \|_{L^2}
	\\
	&\lesssim
	\sum_{q=-1}^\infty
	\sum_{j=q-5}^\infty
	\| \Dd_q(S_{j+2}(\ddd_1 \cdot (A_2 \delta \ddd))\Dd_j(S_N(\ddd_1\otimes \ddd_1))) \|_{L^1}
	\| \Dd_q \nabla \delta \uu \|_{L^2}
	\\
	&\lesssim
	\sum_{q=-1}^\infty
	\sum_{j=q-5}^\infty
	\| S_{j+2}(\ddd_1 \cdot (A_2 \delta \ddd)) \|_{L^2}
	\| \Dd_jS_N (\ddd_1\otimes \ddd_1) \|_{L^2}
	2^q
	\| \Dd_q \delta \uu \|_{L^2}
	\\
	&\lesssim
	\sum_{q=-1}^\infty
	\sum_{j=q-5}^\infty
	2^{\frac{q-j}{2}}
	2^{-\frac{j}{2}}
	\| S_{j+2}(\ddd_1 \cdot (A_2 \delta \ddd)) \|_{L^2}
	2^{j}
	\| \Dd_j S_N (\ddd_1\otimes \ddd_1) \|_{L^2}
	2^{-\frac{q}{2}}
	\| \Dd_q \nabla \delta \uu \|_{L^2}.
\end{align*}
    On the one hand, if $j \leq -1$, then $2^{j}
	\| \Dd_j S_N (\ddd_1\otimes \ddd_1) \|_{L^2}
	\leq 2^{-1}\| \Dd_j S_N (\ddd_1\otimes \ddd_1) \|_{L^2}$, 
	on the other hand, if $j\geq 0$, then the Bernstein inequality in Lemma \ref{prop:Bernstein}
	yields 
	$2^{j}
	\| \Dd_j S_N (\ddd_1\otimes \ddd_1) \|_{L^2}
	\leq C\| \Dd_j S_N \nabla (\ddd_1\otimes \ddd_1) \|_{L^2}$. As a consequence, we deduce that
\begin{align*}
	 |\I \mathcal{V}_1|
	 &\lesssim
	\sum_{q=-1}^\infty
	\sum_{j=q-5}^\infty
	2^{\frac{q-j}{2}}
	2^{-\frac{j}{2}}
	\| S_{j+2}(\ddd_1 \cdot (A_2 \delta \ddd)) \|_{L^2}
	\big(
	\| \Dd_j S_N (\nabla \ddd_1\otimes \ddd_1) \|_{L^2}+
	\| \Dd_j S_N (\ddd_1\otimes \ddd_1) \|_{L^2}
	\big)
	2^{-\frac{q}{2}}
	\| \Dd_q \nabla \delta \uu \|_{L^2}\\
	&\lesssim
	\I\mathcal{V}_1^{(1)}+ \I\mathcal{V}_1^{(2)}+
	\I\mathcal{V}_1^{(3)},
\end{align*}
where
\begin{align*}
    \I\mathcal{V}_1^{(1)}
    &:=
	\sum_{q=-1}^\infty
	\sum_{j=q-5}^\infty
	2^{\frac{q-j}{2}}
	2^{-\frac{j}{2}}
	\| S_{j+2}(\ddd_1 \cdot (A_2 \delta \ddd)) \|_{L^2}
	\| \Dd_j ([S_N,\, \nabla \ddd_1]\otimes \ddd_1) \|_{L^2}
	2^{-\frac{q}{2}}
	\| \Dd_q \nabla \delta \uu \|_{L^2},
	\\
	\I\mathcal{V}_1^{(2)}
    &:=
	\sum_{q=-1}^\infty
	\sum_{j=q-5}^\infty
	2^{\frac{q-j}{2}}
	2^{-\frac{j}{2}}
	\| S_{j+2}(\ddd_1 \cdot (A_2 \delta \ddd)) \|_{L^2}
	\| \Dd_j ([S_N,\, \ddd_1]\otimes \ddd_1) \|_{L^2}
	2^{-\frac{q}{2}}
	\| \Dd_q \nabla \delta \uu \|_{L^2},\\
	\I\mathcal{V}_1^{(3)}
	&:=
	\sum_{q=-1}^\infty
	\sum_{j=q-5}^\infty
	2^{\frac{q-j}{2}}
	2^{-\frac{j}{2}}
	\| S_{j+2}(\ddd_1 \cdot (A_2 \delta \ddd)) \|_{L^2}
	\big(
	\| \Dd_j(\nabla \ddd_1 \otimes S_N \ddd_1) \|_{L^2} +
	\| \Dd_j (\ddd_1 \otimes S_N \ddd_1) \|_{L^2}
	\big)
	2^{-\frac{q}{2}}
	\| \Dd_q \nabla \delta \uu \|_{L^2}.
\end{align*}
The term $\I\mathcal{V}_1^{(1)}$ that involves the commutator for $S_N$ can be controlled by
\begin{align*}
    \I\mathcal{V}_1^{(1)}
	&\lesssim 
	\sum_{q=-1}^\infty
	\sum_{j=q-5}^\infty
	2^{\frac{q-j}{2}}
	2^{-\frac{j}{2}}
	\| S_{j+2}(\ddd_1 \cdot (A_2 \delta \ddd)) \|_{L^2}
	2^{\frac{j}{2}}
	\| \Dd_j ([S_N,\, \nabla \ddd_1]\otimes \ddd_1) \|_{L^\frac{4}{3}}
	2^{-\frac{q}{2}}
	\| \Dd_q \nabla \delta \uu \|_{L^2}\\
	&\lesssim 
	\| \ddd_1 \|_{L^4}
	\| \nabla \ddd_1 \|_{L^2}^\frac{3}{4}
	\| \nabla \ddd_1 \|_{H^1}^\frac{1}{4}
	2^{-\frac{N}{4}}
	\sum_{q=-1}^\infty
	\sum_{j=q-5}^\infty
	2^{\frac{q-j}{2}}
	2^{-\frac{j}{2}}
	\| S_{j+2}(\ddd_1 \cdot (A_2 \delta \ddd)) \|_{L^2}
	2^{\frac{j}{2}}
	2^{-\frac{q}{2}}
	\| \Dd_q \nabla \delta \uu \|_{L^2}\\
	&\lesssim 
	\| \ddd_1 \|_{L^4}
	\| \nabla \ddd_1 \|_{L^2}^\frac{3}{4}
	\| \nabla \ddd_1 \|_{H^1}^\frac{1}{4}
	2^{-\frac{N}{4}}
	\sum_{q=-1}^\infty
	\sum_{j=-1}^\infty
	2^{\frac{q-j}{4}}\mathbf{1}_{(-\infty, 5]}(q-j)
	2^{-\frac{j}{4}}
	\| S_{j+2}(\ddd_1 \cdot (A_2 \delta \ddd)) \|_{L^2}
	2^{-\frac{q}{4}}
	\| \Dd_q \nabla \delta \uu \|_{L^2},
\end{align*}
where we have  used \eqref{interSN43}.
At this stage, invoking again Young's inequality for the convolution of sequences, we gather that
\begin{align*}
    \I\mathcal{V}_1^{(1)}
    &\lesssim 
    \| \ddd_1 \|_{L^4}
	\| \nabla \ddd_1 \|_{L^2}^\frac{3}{4}
	\| \nabla \ddd_1 \|_{H^1}^\frac{1}{4}
	\| \ddd_1 \cdot (A_2 \delta \ddd) \|_{H^{-\frac{1}{4}}}
	\| \nabla \delta \uu \|_{H^{-\frac{1}{4}}}
	2^{-\frac{N}{4}}
	\\
	&\lesssim
    \| \ddd_1 \|_{L^4}
	\| \nabla \ddd_1 \|_{L^2}^\frac{3}{4}
	\| \nabla \ddd_1 \|_{H^1}^\frac{1}{4}
	\| \ddd_1 \otimes \delta \ddd \|_{H^\frac{3}{4}}
	\| A_2 \|_{L^2}
	\| \delta \uu \|_{H^{\frac{3}{4}}}
	2^{-\frac{N}{4}}
	\\
	&\lesssim
    \| \ddd_1           \|_{L^4}
	\| \nabla \ddd_1    \|_{L^2}^\frac{3}{4}
	\| \nabla \ddd_1    \|_{H^1}^\frac{1}{4}
	\| \ddd_1           \|_{H^\frac{7}{8}}
	\| \delta \ddd      \|_{H^\frac{7}{8}}
	\| \nabla \uu_2     \|_{L^2}
	(\|\uu_1\|_{L^2} +\|\uu_2\|_{L^2})^\frac{1}{4}
	\\
	&\quad \times (\| \nabla \uu_1\|_{L^2}+\|\nabla \uu_2 \|_{L^2})^\frac{3}{4}
	2^{-\frac{N}{4}}
	\\
	&\lesssim
    \|\ddd_1\|_{H^1}^\frac{11}{4}(\| \ddd_1\|_{H^1}+ \|\ddd_2\|_{H^1})
    (\|\uu_1\|_{L^2} +\|\uu_2\|_{L^2})^\frac{1}{4}
	\| \nabla \ddd_1    \|_{H^1}^\frac{1}{4}
	\|\nabla \uu_2\|_{L^2} (\| \nabla \uu_1\|_{L^2}+\|\nabla \uu_2 \|_{L^2})^\frac{3}{4}
	2^{-\frac{N}{4}}
	\\
	&\lesssim
	\big(
    \| \ddd_1 \|_{H^1}^4+
    \| \ddd_2 \|_{H^1}^4+
    \| \uu_1 \|_{L^2}^4+
    \| \uu_2 \|_{L^2}^4
    \big)
    \big(
	\| \nabla \ddd_1    \|_{H^1}^2+
	\| \nabla \uu_1 \|_{L^2}^2+
	\| \nabla \uu_2 \|_{L^2}^2
	\big)
	2^{-\frac{N}{4}}.
\end{align*}
Next, we address the estimate for $\I\mathcal{V}_1^{(2)}$, which is indeed easier. Keeping in mind the argument for $\I\mathcal{V}_1^{(1)}$, we infer from the fact $\|[S_N,\ddd_1]\otimes \ddd_1\|_{L^\frac43}\leq C2^{-N}\|\nabla \ddd_1\|_{L^2}\|\ddd_1\|_{L^4}$ that 
\begin{align*}
    \I\mathcal{V}_1^{(2)}
	&\lesssim 
	\sum_{q=-1}^\infty
	\sum_{j=q-5}^\infty
	2^{\frac{q-j}{2}}
	2^{-\frac{j}{2}}
	\| S_{j+2}(\ddd_1 \cdot (A_2 \delta \ddd)) \|_{L^2}
	2^{\frac{j}{2}}
	\| \Dd_j ([S_N,\, \ddd_1]\otimes \ddd_1) \|_{L^\frac{4}{3}}
	2^{-\frac{q}{2}}
	\| \Dd_q \nabla \delta \uu \|_{L^2}\\
	&\lesssim 
	2^{-N}\|\nabla \ddd_1\|_{L^2}\|\ddd_1\|_{L^4}
	\sum_{q=-1}^\infty
	\sum_{j=q-5}^\infty
	2^{\frac{q-j}{2}}
	2^{-\frac{j}{2}}
	\| S_{j+2}(\ddd_1 \cdot (A_2 \delta \ddd)) \|_{L^2}
	2^{\frac{j}{2}}
	2^{-\frac{q}{2}}
	\| \Dd_q \nabla \delta \uu \|_{L^2}\\
	&\lesssim 
	2^{-N}\|\nabla \ddd_1\|_{L^2}\|\ddd_1\|_{L^4}
	\| \ddd_1 \cdot (A_2 \delta \ddd) \|_{H^{-\frac{1}{4}}}
	\| \nabla \delta \uu \|_{H^{-\frac{1}{4}}}\\
	&\lesssim \|\ddd_1\|_{H^1}^3(\| \ddd_1\|_{H^1}+ \|\ddd_2\|_{H^1})
    (\|\uu_1\|_{L^2} +\|\uu_2\|_{L^2})^\frac{1}{4}
	\|\nabla \uu_2\|_{L^2} (\| \nabla \uu_1\|_{L^2}+\|\nabla \uu_2 \|_{L^2})^\frac{3}{4}
	2^{-N}\\
	&\lesssim (\|\ddd_1\|_{H^1}^4+ \|\ddd_2\|_{H^1}^4)
    (\|\uu_1\|_{L^2}^\frac{1}{4} +\|\uu_2\|_{L^2}^\frac{1}{4}) 
	(\| \nabla \uu_1\|_{L^2}^\frac{7}{4} +\|\nabla \uu_2 \|_{L^2}^\frac{7}{4})
	2^{-\frac{N}{4}},
\end{align*}
where we have used the fact $N\geq 0$ (see the choice of $N$ below). 
Finally, for $\I\mathcal{V}_1^{(3)}$, we deduce from Lemma \ref{lemma:SN-infty} and \eqref{d1A2ddH12} that 
\begin{align*}
    \I\mathcal{V}_1^{(3)}
	&\lesssim
    \sum_{q=-1}^\infty
	\sum_{j=q-5}^\infty
	2^{\frac{q-j}{2}}
	2^{-\frac{j}{2}}
	\| S_{j+2}(\ddd_1 \cdot (A_2 \delta \ddd)) \|_{L^2}
	2^{-j}
    \big(
    2^j
	\| \Dd_j(\nabla \ddd_1 \otimes S_N \ddd_1) \|_{L^2}
	+
	2^j
	\| \Dd_j(\ddd_1 \otimes S_N \ddd_1) \|_{L^2}
	\big)\\
	&\qquad\qquad \times 
	2^{-\frac{q}{2}}
	(2^q
	\| \Dd_q \delta \uu \|_{L^2})
	\\
	&\lesssim
    \sum_{q=-1}^\infty
	\sum_{j=q-5}^\infty
	2^{\frac{3}{2}(q-j)}
	2^{-\frac{j}{2}}
	\| S_{j+2}(\ddd_1 \cdot (A_2 \delta \ddd)) \|_{L^2}
	\big(
	\| \Dd_j\nabla(\nabla  \ddd_1 \otimes S_N \ddd_1) \|_{L^2}
	+
	\| \Dd_j\nabla (\ddd_1 \otimes S_N \ddd_1) \|_{L^2}
	\big)
	2^{-\frac{q}{2}}
	\| \Dd_q \delta \uu \|_{L^2}
	\\
	&\lesssim
	\| \ddd_1 \cdot (A_2 \delta \ddd) \|_{H^{-\frac{1}{2}}}
	\big(
	\| \nabla  \ddd_1\|_{H^1}\|S_N \ddd_1 \|_{L^\infty} 
	+
	\| \ddd_1\|_{H^1}\|S_N\nabla \ddd_1 \|_{L^\infty}
	\big)
	\| \delta \uu \|_{H^{-\frac{1}{2}}}
	\\
	&\lesssim
	\| \ddd_1 \cdot (A_2 \delta \ddd) \|_{H^{-\frac{1}{2}}}^2
	+
	\big(
	\| \nabla \ddd_1 \|_{H^1}^2\|  S_N \ddd_1 \|_{L^\infty}^2
	+
	\|\ddd_1 \|_{H^1}^2 \| S_N \nabla \ddd_1 \|_{L^\infty}^2
	\big)
	\| \delta \uu \|_{H^{-\frac{1}{2}}}^2
	\\
	&\leq
	C
	\|	\ddd_1 		\|_{H^1}^2
	\| \nabla \uu_2 	\|_{L^2}^2
	\| \delta \ddd \|_{H^\frac{1}{2}}^2N+
	\big(\|\ddd_1\|_{H^1}^4+\|\ddd_2\|_{H^{1}}^4\big) 
	\| \nabla \uu_2 				\|_{L^2}^2
	2^{-\frac{N}{2}}+
	\| \ddd_1 \|_{H^2}^2\| \ddd_1 \|_{H^1}^2
	\| \delta \uu \|_{H^{-\frac{1}{2}}}^2N.
\end{align*}
Collecting the above estimates and applying Young's inequality, we obtain the following estimate for $\I\mathcal{V}$ such that
\begin{equation}\label{IV-d1A2deltad}
\begin{aligned}
	\I\mathcal{V}
	&\leq 
	C
	\big(
    \| \ddd_1 \|_{H^1}^4+
    \| \ddd_2 \|_{H^1}^4+
    \| \uu_1 \|_{L^2}^4+
    \| \uu_2 \|_{L^2}^4
    \big)
    \big(
	\| \nabla \ddd_1    \|_{H^1}^2+
	\|  \uu_1 \|_{H^1}^2+
	\|  \uu_2 \|_{H^1}^2
	\big)
	2^{-\frac{N}{4}}  \\
	&\quad +
	C
	\|	\ddd_1 		\|_{H^1}^2
	\| \nabla \uu_2 	\|_{L^2}^2
	\| \delta \ddd \|_{H^\frac{1}{2}}^2 N
	+
	\| \ddd_1 \|_{H^2}^2\| \ddd_1 \|_{H^1}^2
	\| \delta \uu \|_{H^{-\frac{1}{2}}}^2N,
\end{aligned}
\end{equation}
where $C>0$ is a constant independent of $N$.

Hence, from inequalities \eqref{III-d1A2deltad}, \eqref{I-d1A2deltad}, \eqref{II-d1A2deltad} and 
\eqref{IV-d1A2deltad}, we finally arrive at 
\begin{equation}
\label{est-d1d1d1A2deltad-with-eps}
\begin{aligned}
 &\big|\langle (\ddd_1\cdot (A_2\delta \ddd))\ddd_1\otimes\ddd_1,\, 
 \nabla \delta \uu \rangle_{H^{-\frac{1}{2}}}\big|\\
 &\quad \leq 
 	C(1+\eta^{-1})
    \|	\ddd_1 		\|_{H^1}^2
    \Big(
		\| \nabla \uu_2 	\|_{L^2}^2+
		\| \ddd_1			\|_{H^2}^2
	\Big)
	\Big(
		\| \delta \ddd \|_{H^\frac{1}{2}}^2+
		\| \delta \uu \|_{H^{-\frac{1}{2}}}^2
	\Big)
	N
    \\
    &\qquad + C(1+\eta^{-1})
\big(\| \ddd_1\|_{H^1}^4+\|\ddd_2\|_{H^1}^4+
\|\uu_2\|_{L^2}^4+\|\uu_2\|_{L^2}^4\big)
	\big(\| \ddd_1\|_{H^2}^2+ \|  \uu_1\|_{H^1}^2+\|\uu_2  \|_{H^1}^2\big)
	2^{-\frac{N}{4}}\\
	&\qquad 
	 + C\eta^{-1}\left(\frac{N}{\ee}\right)^\frac{\ee}{1-\ee}   
    \|\ddd_1\|_{H^1}^\frac{2(1+\ee)}{1-\ee}
    \|\ddd_1\|_{H^2}^{2}\|\delta \uu \|_{H^{-\frac12}}^2\left(\frac{N}{\ee}\right) + 2\eta\nu \|\nabla \delta \uu\|_{H^{-\frac12}}^2\\
    &\qquad + \eta 	\sum_{q=-1}^\infty \!\!2^{-q}\!\!\int_{\TT^2}|\sdm{q}(\ddd_1\otimes\ddd_1):\Dd_q\delta A|^2(t,x)\, \dd x.
\end{aligned}
\end{equation}
To complete the estimate, we need to explicitly set the values of both $N$ and $\ee$ in term of the function $\Phi(t)$. For $N$, we make the same choice as in \eqref{def-N-prop-deltasigma1-first} 
$$ N = N(t)= \Big\lfloor 4\log_2\Big(1+\frac{1}{\Phi(t)}\Big)\Big\rfloor +4,\quad \text{such that }\ 2^{-\frac{N(t)}{4}}\leq \Phi(t).$$
Then we take the parameter $\ee$ as
\begin{align*} 
\ee=\ee(t)= \frac{1}{2+2\ln(1+N(t))}\in \Big(0,\,\frac12\Big),\quad \text{so that }
	\left(\frac{N}{\ee}\right)^\frac{\ee}{1-\ee} \leq e^2,
\end{align*}
where the last inequality follows from \eqref{esNee1}. 
Since
 $(1+\ee)/(1-\ee)<3$, we can finally recast the estimate \eqref{est-d1d1d1A2deltad-with-eps} into the following form 
\begin{equation}\label{double-loc-est-deltad}
\begin{aligned}
	& \Big|
		\langle 
			(\ddd_1\cdot (A_2\delta \ddd))\ddd_1\otimes \ddd_1,\nabla \delta \uu
		\rangle_{H^{\rule[1pt]{3pt}{0.4pt}\frac{1}{2}}}(t)
	\Big|\\
	&\quad \leq 
	C
	(1+\eta^{-1})
	\Big[
	\big(\| \ddd_1 \|_{H^1}^2 + \| \ddd_1\|_{H^1}^6+\|\ddd_2\|_{H^1}^4+
\|\uu_2\|_{L^2}^4+\|\uu_2\|_{L^2}^4\big)
	\big(\| \ddd_1\|_{H^2}^2+ \| \uu_1\|_{H^1}^2+\|\uu_2  \|_{H^1}^2\big)
	\Big]\mu(\Phi(t))\\
	&\qquad + 2\eta\nu \|\nabla \delta \uu\|_{H^{-\frac12}}^2 + \eta 	\sum_{q=-1}^\infty \!\!2^{-q}\!\!\int_{\TT^2}|\sdm{q}(\ddd_1\otimes\ddd_1):\Dd_q\delta A|^2(t,x)\, \dd x,
\end{aligned}
\end{equation}
for any time $t\in (0,T)$, where $\mu$ is given by \eqref{muaa}. 

To conclude, we note that an analogous procedure as the one used to prove \eqref{double-loc-est-deltad} leads to the following estimate for the third term of $\delta \bm{\sigma}_1$ in \eqref{delta-sigma}
\begin{equation}\label{double-loc-est-deltad2}
\begin{aligned}
	& \Big|
		\langle 
			( \delta \ddd\cdot  (A_2\ddd_2))(\ddd_1\otimes \ddd_1),\nabla \delta \uu
		\rangle_{H^{\rule[1pt]{3pt}{0.4pt}\frac{1}{2}}}(t)
	\Big|\\
	&\quad \leq 
		C
	(1+\eta^{-1})
	\Big[
	\big(\|\ddd_1\|_{H^1}^2+ \| \ddd_2 \|_{H^1}^2 +  \| \ddd_1\|_{H^1}^6 + \| \ddd_2\|_{H^1}^4+
\|\uu_2\|_{L^2}^4+\|\uu_2\|_{L^2}^4\big)
	\big(\| \ddd_1\|_{H^2}^2+ \| \uu_1\|_{H^1}^2+\| \uu_2  \|_{H^1}^2\big)
	\Big]\\
	&\qquad\quad  \times 
	\mu(\Phi(t))  + 2\eta\nu \|\nabla \delta \uu\|_{H^{-\frac12}}^2 + \eta 	\sum_{q=-1}^\infty \!\!2^{-q}\!\!\int_{\TT^2}|\sdm{q}(\ddd_1\otimes\ddd_1):\Dd_q\delta A|^2(t,x)\, \dd x.
\end{aligned}
\end{equation}
In summary, combining the estimates  \eqref{d2d2nablau2deltadd1-final-estimate}, \eqref{double-loc-est-deltad}, \eqref{double-loc-est-deltad2} and using Young's inequality, we can choose the function $g_1(t)\in L^1(0,T)$ by means of
\begin{align*}
	g_1(t)&:= \| \ddd_2	\cdot (A_2 \ddd_2) \|_{L^2}^2
	\big( \| \ddd_1 \|_{H^1}^2 +\| \ddd_2 \|_{H^1}^2\big)+
	\big(\|\uu_1\|_{L^2}+ \|\uu_2\|_{L^2}\big)
	\big(\| \ddd_1 \|_{H^1}^2 + \| \ddd_2 \|_{H^1}^2 \big)
	\big(\| \nabla \uu_1\|_{L^2}+ \|\nabla \uu_2 \|_{L^2}\big)
	 \\
	&\quad\ \  +	\big(\|\ddd_1\|_{H^1}^2+ \| \ddd_2 \|_{H^1}^2 +  \| \ddd_1\|_{H^1}^6 + \| \ddd_2\|_{H^1}^6+
\|\uu_2\|_{L^2}^4+\|\uu_2\|_{L^2}^4\big)
	\big(\| \ddd_1\|_{H^2}^2+ \| \nabla \uu_1\|_{L^2}^2+\|\nabla \uu_2  \|_{L^2}^2\big),
\end{align*}
and conclude that
\begin{align*} 
& \Big|\langle 	\delta \bm{\sigma}_1(t)-[(\ddd_1\cdot (\delta A\ddd_1))\ddd_1\otimes\ddd_1](t),\,\nabla\delta\uu(t)\rangle_{H^{-\frac{1}{2}}}
	\Big|\\
	&\qquad 	\leq  C_4 g_1(t)\mu( \Phi(t) )+ 6\eta\nu \|\nabla \delta \uu\|_{H^{-\frac12}}^2 + 2\eta 	\sum_{q=-1}^\infty \!\!2^{-q}\!\!\int_{\TT^2}|\sdm{q}(\ddd_1\otimes\ddd_1):\Dd_q\delta A|^2(t,x)\, \dd x,
	\end{align*}
where the constant $C_4>0$ depends on $\eta^{-1}$. This concludes part $\mathrm{(a)}$ of Proposition \ref{prop:last-ineq-main-thm}.
\end{proof}

In what follows, we turn to the estimate for the inner product $(\ddd_1\cdot(\delta A \ddd_1))\ddd_1\otimes \ddd_1$ in  \eqref{delta-sigma}. It is worth mentioning that one main dissipative term in $\mathfrak{D}(t)$ (recall  \eqref{def-Phi-D}) that has played an important role in the previous estimates, appears as a consequence of a deep decomposition of this specific nonlinear tensor with symmetric structure. 
\begin{proof}[Proof of Proposition \ref{prop:last-ineq-main-thm}, Part $\mathrm{(b)}$]
Using the symmetry of $\ddd \otimes \ddd$ and applying Bony's decomposition \eqref{bony-decomp}, we obtain  
\begin{equation}\label{last-term-of-second-term}
	\begin{aligned}
		&\langle 	
			( 
			(\ddd_1\cdot(\delta A \ddd_1))\ddd_1 \otimes \ddd_1
			,\, 
			\nabla \delta \uu 
		\rangle_{H^{-\frac{1}{2}}}\\	
		&\quad =\,		
		\sum_{q=-1}^\infty 
		2^{-q}
		\int_{\TT^2}
		\Dd_q(\ddd_1\otimes\ddd_1(\ddd_1\otimes\ddd_1:\delta A)):\Dd_q \delta A \,\dd x 
		\\	
		&\quad =\,
		\sum_{q=-1}^\infty
		\sum_{|j-q|\leq 5}
		2^{-q}
		\int_{\TT^2}
		[\Dd_q,\,\sdm{j}(\ddd_1\otimes\ddd_1)]
		\Dd_j (\ddd_1 \otimes \ddd_1:\delta A):\Dd_q \delta A\,\dd x\\
		&\qquad +\sum_{q=-1}^\infty
		\sum_{|j-q|\leq 5}
		2^{-q}
		\int_{\TT^2}
		( \sdm{j}\,- \,\sdm{q})(\ddd_1\otimes\ddd_1)
		\Dd_q\Dd_j (\ddd_1 \otimes \ddd_1:\delta A):\Dd_q \delta A\,\dd x\\
		&\qquad +\sum_{q=-1}^\infty
		2^{-q}
		\int_{\TT^2}
		\Dd_q (\ddd_1 \otimes \ddd_1:\delta A)(\sdm{q}(\ddd_1\otimes\ddd_1):\Dd_q \delta A)\,\dd x\\
		&\qquad +\sum_{q=-1}^\infty
		\sum_{j= q - 5}^\infty 
		2^{-q}
		\int_{\TT^2}
		\Dd_q\big(\Dd_j(\ddd_1\otimes\ddd_1)S_{j+2} (\ddd_1 \otimes \ddd_1:\delta A)\big):\Dd_q \delta A\,\dd x\\
		&\quad =:\, \mathcal{K}_1\,+\,\mathcal{K}_2\,+\,\mathcal{K}_3\,+\,\mathcal{K}_4.
	\end{aligned}
\end{equation}
We continue our analysis by estimating each term $\mathcal{K}_1$, $\mathcal{K}_2$, $\mathcal{K}_3$ and $\mathcal{K}_4$. These estimates will again lead to an inequality with double logarithmic structure similar to  Proposition \ref{prop:last-ineq-main-thm}, part $\mathrm{(a)}$. 

First, we present a lemma that is useful in the subsequent proofs. It gives the control on $H^{-1/2}$ norm of the quantity $\ddd_1 \otimes \ddd_1:\delta A$, depending on two parameters $N$ and $\ee$. We postpone its proof to the Appendix (see Lemma \ref{lemma:juve16b}).
\begin{lemma}\label{lemma:juve16}
	The following inequality holds for any positive integer $N\in \NN$ and  any $\ee\in (0, 1/2)$:
	\begin{equation}\label{juve16}
	\begin{aligned}
		\|\ddd_1 \otimes \ddd_1:\delta A \|_{H^{-\frac{1}{2}}}
		&\lesssim 
			\Big(
			\sum_{q=-1}^\infty
			2^{-q}
			\int_{\TT^2}|\sdm{q} (\ddd_1 \otimes \ddd_1): \Dd_q \delta A |^2\,\dd x
			\Big)^\frac{1}{2}
			\\ 
			&\quad + \sqrt{\frac{N}{\ee}}
			\|\ddd_1				\|_{H^1}^{1+\ee}
			\|\nabla \ddd_1			\|_{H^1}^{1-\ee}
			\|\delta \uu 			\|_{H^{-\frac{1}{2}}}^{1-\ee}
			\|\nabla \delta \uu 	\|_{H^{-\frac{1}{2}}}^\ee
			+ 
			\|\ddd_1\|_{H^1}
			\|\ddd_1\|_{H^2}
			\|\delta \uu\|_{H^{-\frac{1}{2}}}
			\sqrt{N}
			\\
			&\quad 
			+ (\|\ddd_1\|_{H^1}^2+ \|\uu_1\|_{L^2}^2 + \|\uu_2\|_{L^2}^{2}) 
			(\|\ddd_1\|_{H^2}+ \|\nabla \uu_1\|_{L^2}+ \|\nabla \uu_2\|_{L^2})
			2^{-\frac{N}{2}}.
\end{aligned}
\end{equation}
\end{lemma}
\noindent Concerning $\mathcal K_1$, we invoke the commutator estimate in Lemma \ref{prop:comm-est}, which states that for any $q\in \tilde \NN$, $|j-q|\leq 5$, it holds $$\|  [\Dd_q,\,\sdm{j} (\ddd_1\otimes\ddd_1)]\Dd_j (\ddd_1 \otimes \ddd_1:\delta A) \|_{L^{\frac{2}{1+\ee}}}\leq 2^{-q} \| \sdm{j} \nabla (\ddd_1\otimes\ddd_1) \|_{L^{\frac{2}{\ee}}} \| \Dd_j (\ddd_1 \otimes \ddd_1:\delta A)\|_{L^2}.$$ 
 Therefore, it follows that 
\begin{align*}
		|\mathcal K_1|
		\,&\leq \,\sum_{q=-1}^\infty
		\sum_{|j-q|\leq 5}
		2^{-q}\Big| 
		\int_{\TT^2}
		[\Dd_q,\,\sdm{j} (\ddd_1\otimes\ddd_1)]\Dd_j (\ddd_1 \otimes \ddd_1:\delta A):\Dd_q \delta A\,\dd x\Big|
		\\
		\,&\lesssim\,
		\sum_{q=-1}^\infty
		\sum_{|j-q|\leq 5}
		2^{-2q}
		\|\,\sdm{j} \nabla (\ddd_1\otimes\ddd_1)\,\|_{L^\frac{2}{\ee}}
		\|\,\Dd_j (\ddd_1 \otimes \ddd_1:\delta A)\,\|_{L^2}
		\|\,\Dd_q \delta A\,\|_{L^\frac{2}{1-\ee}}
		\\
		\,&\lesssim\, 
		\sum_{q=-1}^\infty
		\sum_{|j-q|\leq 5}
		2^{\frac{j-q}{2}}
		\|\,\sdm{j} (\nabla \ddd_1 \otimes\ddd_1)\,\|_{L^\frac{2}{\ee}}
		2^{-\frac{j}{2}}\|\,\Dd_j (\ddd_1 \otimes \ddd_1:\delta A)\,\|_{L^2}
		2^{-\frac{q}{2}}\|\,\Dd_q \delta \uu\,\|_{L^\frac{2}{1-\ee}}\\
		\,&\lesssim\,
		\sum_{q=-1}^\infty
		\sum_{|j-q|\leq 5}
		\|\,\sdm{j} (\nabla \ddd_1\otimes S_N \ddd_1)\,\|_{L^\frac{2}{\ee}}
		2^{-\frac{j}{2}}\|\,\Dd_j (\ddd_1 \otimes \ddd_1:\delta A )\,\|_{L^2}
		2^{-\frac{q}{2}}\|\,\Dd_q \delta \uu\,\|_{L^\frac{2}{1-\ee}}
		\\
		&\quad \ +\,
		\sum_{q=-1}^\infty
		\sum_{|j-q|\leq 5}
		\|\,\sdm{j}(\nabla \ddd_1\otimes (\Id-S_N) \ddd_1)\,\|_{L^\frac{2}{\ee}}
		2^{-\frac{j}{2}}\|\,\Dd_j (\ddd_1 \otimes \ddd_1: \delta A) \|_{L^2}
		2^{-\frac{q}{2}}\|\,\Dd_q \delta \uu\,\|_{L^\frac{2}{1-\ee}}\\
		\,&=:\,
		\mathcal{K}_1^{(1)}\,+\,\mathcal{K}_1^{(2)},
	\end{align*}
where we have separated the low frequencies of $S_N\ddd_1$ in $\mathcal{K}_1^{(1)}$ and the high frequencies 
$(\Id-S_N) \ddd_1$ in $\mathcal{K}_1^{(2)}$. For the estimate of $\mathcal{K}_1^{(1)}$, we recall the inequalities  $\| S_N \ddd_1 \|_{L^\infty} \leq C\sqrt{N}\| \ddd_1 \|_{H^1}$ in Lemma \ref{lemma:SN-infty} and $\| \nabla \ddd_1 \|_{L^{2/\ee}} \leq  C\| \nabla \ddd_1 \|_{H^{1-\ee}}/\sqrt{\ee}$ of Lemma \ref{lemma:eps}. Hence, it follows that 
\begin{equation}\label{est-K11-final-section}
\begin{aligned}
	|\mathcal{K}_1^{(1)}|
	\,&\lesssim\,
	\sum_{q=-1}^\infty
	\sum_{|j-q|\leq 5}
	\|\nabla \ddd_1	\|_{L^\frac{2}{\ee}}
	\|S_N \ddd_1	\|_{L^\infty}
	2^{-\frac{j}{2}}
	\|\Dd_j ( \ddd_1 \otimes \ddd_1:\delta A )\|_{L^2}
	2^{-\frac{q}{2}(1-\ee)}\|\Dd_q \delta \uu\|_{L^2}^{1-\ee}
	2^{-\frac{q}{2}\ee}\|\Dd_q\nabla \delta \uu\|_{L^2}^{\ee}\\
	&\lesssim\,
	\sqrt{\frac{N}{\ee}}
	\|\nabla \ddd_1\|_{L^2}^{\ee}
	\|\nabla \ddd_1\|_{H^1}^{1-\ee}
	\|\ddd_1\|_{H^1}
	\|\delta \uu\|_{H^{-\frac{1}{2}}}^{1-\ee}
	\|\nabla \delta \uu\|_{H^{-\frac{1}{2}}}^{\ee}
	\|\ddd_1 \otimes \ddd_1:\delta A\|_{H^{-\frac{1}{2}}}.
\end{aligned}
\end{equation}
Applying Lemma \ref{lemma:juve16} and the Cauchy-Schwarz inequality, we infer from  \eqref{est-K11-final-section} that
\begin{align*}
	|\mathcal{K}_1^{(1)} |
	& \lesssim \, \sqrt{\frac{N}{\ee}}
	 \|\ddd_1\|_{H^1}^{1+\ee}
	\|\nabla \ddd_1\|_{H^1}^{1-\ee}
	\|\delta \uu\|_{H^{-\frac{1}{2}}}^{1-\ee}
	\|\nabla \delta \uu\|_{H^{-\frac{1}{2}}}^{\ee}
	\Big(
			\sum_{q=-1}^\infty
			2^{-q}
			\int_{\TT^2}|\sdm{q} (\ddd_1 \otimes \ddd_1): \Dd_q \delta A |^2\,\dd x
			\Big)^\frac{1}{2} \\
	&\quad \,+\, \left(\frac{N}{\ee}\right)
	\|\ddd_1\|_{H^1}^{2(1+\ee)}
	\|\ddd_1 \|_{H^2}^{2(1-\ee)}
	\|\delta \uu\|_{H^{-\frac{1}{2}}}^{2(1-\ee)}
	\|\nabla \delta \uu\|_{H^{-\frac{1}{2}}}^{2\ee}
	\,+\,
		\|\ddd_1\|_{H^1}^2 
		\|\ddd_1\|_{H^2}^2
		\|\delta \uu\|_{H^{-\frac{1}{2}}}^2
		N
	    \\ 
	 &\quad 
	  \,+\,
	(\|\ddd_1\|_{H^1}^4+ \|\uu_1\|_{L^2}^4 + \|\uu_2\|_{L^2}^4) 
	(\| \ddd_1\|_{H^2}^2 + \|\nabla \uu_1\|_{L^2}^2 + \|\nabla \uu_2\|_{L^2}^2)
	2^{-N}.
\end{align*}
From Young's inequality, we further get 
\begin{equation}\label{juve18}
\begin{aligned}
	|\mathcal{K}_1^{(1)}|
	\,\leq\,
	&C \big( \eta^{-\frac{\ee}{1-\ee}} +
	\eta^{-\frac{1+\ee}{1-\ee}} \big)
	\left(\frac{N}{\ee}\right)^{\frac{1}{1-\ee}}
	\|\ddd_1\|_{H^1}^{2\frac{1+\ee}{1-\ee}}
	\|\ddd_1\|_{H^2}^{2}
	\|\delta \uu\|_{H^{-\frac{1}{2}}}^2
	\\
	& \,+\,
	\eta \nu 
	\|\nabla \delta \uu\|_{H^{-\frac{1}{2}}}^2
	\,+\,
	\eta \sum_{q=-1}^\infty
			2^{-q}
			\int_{\TT^2}|\sdm{q} (\ddd_1 \otimes \ddd_1): \Dd_q \delta A |^2\,\dd x +\,
	C\|\ddd_1\|_{H^1}^2
	\|\ddd_1\|_{H^2}^2
	\|\delta \uu\|_{H^{-\frac{1}{2}}}^2
	N\\
	&\,+\, C
	(\|\ddd_1\|_{H^1}^4+ \|\uu_1\|_{L^2}^4 + \|\uu_2\|_{L^2}^4) 
	(\| \ddd_1\|_{H^2}^2 + \|\nabla \uu_1\|_{L^2}^2 + \|\nabla \uu_2\|_{L^2}^2)
	2^{-N}.
\end{aligned}
\end{equation}
Next, we treat $\mathcal K_1^{(2)}$. Recalling the estimates  
$$\|\sdm{j} (\nabla \ddd_1\otimes (\Id-S_N) \ddd_1)\|_{L^\frac{2}{\ee}}\leq 
C2^{(1-\ee)j}\| \sdm{j} (\nabla \ddd_1\otimes (\Id-S_N) \ddd_1)\|_{L^2},\quad \| \Dd_q \delta \uu\|_{L^{\frac{2}{1-\ee}}}\leq C2^{q\ee}\| \Dd_q \delta \uu\|_{L^{2}},$$ 
thanks to the Bernstein inequalities in Lemma \ref{prop:Bernstein}, we deduce that  
	\begin{align*}
		|\mathcal K_1^{(2)}|
		\,&=\,
		\sum_{q=-1}^\infty
		\sum_{|j-q|\leq 5}
		\|\,\sdm{j} (\nabla \ddd_1\otimes (\Id-S_N) \ddd_1)\,\|_{L^\frac{2}{\ee}}
		2^{-\frac{j}{2}}\|\,\Dd_j (\ddd_1 \otimes \ddd_1:\delta A )\,\|_{L^2}
		2^{-\frac{q}{2}}\|\,\Dd_q \delta \uu\,\|_{L^\frac{2}{1-\ee}}
		\\
		&\lesssim\,
		\sum_{q=-1}^\infty
		\sum_{|j-q|\leq 5}
		2^{(1-\ee)j}
		\|\sdm{j} (\nabla \ddd_1\otimes (\Id-S_N) \ddd_1)\|_{L^2}
		2^{-\frac{j}{2}}\|\Dd_j (\ddd_1 \otimes \ddd_1:\delta A)\|_{L^2}
		2^{-\frac{q}{2}}2^{q\ee}
		\|\Dd_q \delta \uu\|_{L^2}
		\\
		&\lesssim\,
		\sum_{q=-1}^\infty
		\sum_{|j-q|\leq 5}
		2^{\frac{3}{2}(j-q)}2^{(q-j)\ee}
		2^{-\frac{j}{2}}
		\|\,\sdm{j}  (\nabla \ddd_1\otimes (\Id-S_N) \ddd_1)\,\|_{L^2}
		2^{-\frac{j}{2}}\|\,\Dd_j (\ddd_1 \otimes \ddd_1:\delta A)\,\|_{L^2}
		\|\,\Dd_q \nabla  \delta \uu\,\|_{L^2}
		\\
		&\lesssim\,
		\|\nabla \ddd_1\otimes (\Id-S_N) \ddd_1\|_{H^{-\frac{1}{2}}}
		\|\ddd_1 \otimes \ddd_1:\delta A\|_{H^{-\frac{1}{2}}}
		(\|\nabla \uu_1\|_{L^2}+ \|\nabla \uu_2\|_{L^2})
		\\
		&\lesssim\,
		\|\nabla \ddd_1\|_{L^2}
		\|(\Id-S_N) \ddd_1\|_{H^{\frac{1}{2}}}
		\|\,\ddd_1 \otimes \ddd_1:\delta A\,\|_{H^{-\frac{1}{2}}}
		(\|\nabla \uu_1\|_{L^2}+ \|\nabla \uu_2\|_{L^2})
		\\
		&\lesssim\,
		\|\ddd_1\|_{H^1}^2
		\|\ddd_1 \otimes \ddd_1:\delta A\|_{H^{-\frac{1}{2}}}
		(\|\nabla \uu_1\|_{L^2}+ \|\nabla \uu_2\|_{L^2})
		2^{-\frac{N}{2}}.
	\end{align*}
Combining the above inequality together with Lemma \ref{lemma:juve16} and using Young's inequality, we finally gather that
\begin{align*}
	|\mathcal{K}_1^{(2)}|
	&\leq\,
			C\eta^{-\frac{\ee}{1-\ee}} \left(\frac{N}{\ee}\right)^\frac{1}{1-\ee}
			\|\ddd_1\|_{H^1}^{2\frac{1+\ee}{1-\ee}}
			\|\ddd_1\|_{H^2}^2
			\|\delta \uu \|_{H^{-\frac{1}{2}}}^2
			\,+\,
	\eta \sum_{q=-1}^\infty
			2^{-q}
			\int_{\TT^2}|\sdm{q} (\ddd_1 \otimes \ddd_1): \Dd_q \delta A |^2\,\dd x
			\\ 
			&\quad \,+\,
			\eta \nu 
			\|\nabla \delta \uu \|_{H^{-\frac{1}{2}}}^2
			\,+\,
			C\|\ddd_1\|_{H^1}^2
			\|\ddd_1\|_{H^2}^2
			\|\delta \uu\|_{H^{-\frac{1}{2}}}^2
			N
			\,+\,  C\eta^{-1}\|\ddd_1\|_{H^1}^4
		(\|\nabla \uu_1\|_{L^2}^2 + \|\nabla \uu_2\|_{L^2}^2)
		2^{-N}
			\\
			&\quad +\,
			C
	(\|\ddd_1\|_{H^1}^4+ \|\uu_1\|_{L^2}^4 + \|\uu_2\|_{L^2}^4) 
	(\| \ddd_1\|_{H^2}^2 + \|\nabla \uu_1\|_{L^2}^2 + \|\nabla \uu_2\|_{L^2}^2)
	2^{-N}.
\end{align*}
This last relation together with \eqref{juve18} lead to the following estimate of $\mathcal{K}_1$ in \eqref{last-term-of-second-term}:
\begin{equation}\label{juve28b}
\begin{aligned}
	|\mathcal{K}_1|&\leq\,
			C \big( \eta^{-\frac{\ee}{1-\ee}} +
	\eta^{-\frac{1+\ee}{1-\ee}} \big)
			\left(\frac{N}{\ee}\right)^\frac{\ee}{1-\ee}
			\|\ddd_1\|_{H^1}^{2\frac{1+\ee}{1-\ee}}
			\| \ddd_1\|_{H^2}^2
			\|\delta \uu \|_{H^{-\frac{1}{2}}}^2\left(\frac{N}{\ee}\right)\\
			&\quad +\,
			2\eta\nu 
			\|\,\nabla \delta \uu \|_{H^{-\frac{1}{2}}}^2
			\,+\,
			2\eta \sum_{q=-1}^\infty
			2^{-q}
			\int_{\TT^2}|\sdm{q} (\ddd_1 \otimes \ddd_1): \Dd_q \delta A |^2\,\dd x
			+\,
			C 
			\|\ddd_1\|_{H^1}^2
			\|\ddd_1\|_{H^2}^2
			\|\delta \uu\|_{H^{-\frac{1}{2}}}^2
			N\\ 
			&\quad 
			+\,
				C(1+\eta^{-1})
	(\|\ddd_1\|_{H^1}^4+ \|\uu_1\|_{L^2}^4 + \|\uu_2\|_{L^2}^4) 
	(\|\ddd_1\|_{H^2}^2 + \|\nabla \uu_1\|_{L^2}^2 + \|\nabla \uu_2\|_{L^2}^2)
	2^{-N},
\end{aligned}
\end{equation}
where we have used the fact $N\geq 0$. Indeed, we choose the parameters $N\in \mathbb{N}$ and $\ee \in (0,1/2)$ in terms of $\Phi(t)$ such that 
\begin{align} 
N(t) = \Big\lfloor \log_2\Big(1+\frac{1}{\Phi(t)}\Big)\Big\rfloor + 1\quad\text{and}\quad  \ee(t):=\frac{1}{2+2\ln(1+ N(t))}.
\label{NNeea}
\end{align} 
Therefore, it follows that   
\begin{equation}\label{juve28}
\begin{aligned}
	|\mathcal{K}_1|& \leq\,
			C\eta^{-3}	(\|\ddd_1\|_{H^1}^2+\|\ddd_1\|_{H^1}^6)
			\|\ddd_1\|_{H^2}^2
			\Phi(t)
			\Big(1+\ln\Big(1+\frac{1}{\Phi(t)}\Big)\Big)
			\Big(1+\ln\Big(1+\ln\Big(1+\frac{1}{\Phi(t)}\Big)\Big)\Big)\\
			&\quad 
			+\,
			2\eta\nu 
			\|\,\nabla \delta \uu \|_{H^{-\frac{1}{2}}}^2
			 +\,
			2\eta \sum_{q=-1}^\infty
			2^{-q}
			\int_{\TT^2}|\sdm{q} (\ddd_1 \otimes \ddd_1): \Dd_q \delta A |^2\,\dd x\\ 
			&\quad +\,
			C\|\ddd_1\|_{H^2}^2
			\|\ddd_1\|_{H^1}^2
			\Phi(t)
			\Big(1+\ln\Big(1+\frac{1}{\Phi(t)}\Big)\Big)\\
			&\quad +\,
			C(1+\eta^{-1})
	(\|\ddd_1\|_{H^1}^4+ \|\uu_1\|_{L^2}^4 + \|\uu_2\|_{L^2}^4) 
	(\| \ddd_1\|_{H^2}^2 + \|\nabla \uu_1\|_{L^2}^2 + \|\nabla \uu_2\|_{L^2}^2)
			\Phi(t)\\
	&\leq C(1+\eta^{-3}) (\|\ddd_1\|_{H^1}^2+\|\ddd_1\|_{H^1}^6+ \|\uu_1\|_{L^2}^4 + \|\uu_2\|_{L^2}^4) 
	(\|\ddd_1\|_{H^2}^2 + \|\nabla \uu_1\|_{L^2}^2 + \|\nabla \uu_2\|_{L^2}^2) \mu(\Phi(t))\\
	&\quad +\,
			2\eta\nu 
			\|\,\nabla \delta \uu \|_{H^{-\frac{1}{2}}}^2
			\, +\,
			2\eta \sum_{q=-1}^\infty
			2^{-q}
			\int_{\TT^2}|\sdm{q} (\ddd_1 \otimes \ddd_1): \Dd_q \delta A |^2\,\dd x,
\end{aligned}
\end{equation}
where we have used the fact $\eta\in (0,1)$ and $\mu$ is given by \eqref{muaa}. 

Coming back to our decomposition \eqref{last-term-of-second-term}, with an analogous procedure as for proving \eqref{juve28}, we also infer that the term  $\mathcal{K}_2$ satisfies
\begin{equation}\label{juve27}
\begin{aligned}
	|\mathcal{K}_2|&\leq C(1+\eta^{-3}) (\|\ddd_1\|_{H^1}^2+\|\ddd_1\|_{H^1}^6+ \|\uu_1\|_{L^2}^4 + \|\uu_2\|_{L^2}^4) 
	(\| \ddd_1\|_{H^2}^2 + \|\nabla \uu_1\|_{L^2}^2 + \|\nabla \uu_2\|_{L^2}^2)\mu(\Phi(t))\\
	&\quad +\,
			2\eta\nu 
			\|\,\nabla \delta \uu \|_{H^{-\frac{1}{2}}}^2
			\, +\,
			2\eta \sum_{q=-1}^\infty
			2^{-q}
			\int_{\TT^2}|\sdm{q} (\ddd_1 \otimes \ddd_1): \Dd_q \delta A |^2\,\dd x.
\end{aligned}
\end{equation}

Now we turn our attention to the last two terms $\mathcal{K}_3$ and $\mathcal{K}_4$ in \eqref{last-term-of-second-term}. We begin with $\mathcal K_4$, since its estimate is somehow similar to the ones we have encountered: 
	\begin{align*}
		|\mathcal{K}_4|
		\,&=\,
		\Big|
			\sum_{q=-1}^\infty
		\sum_{j= q - 5}^\infty 
		2^{-q}
		\int_{\TT^2}
		\Dd_q\big(\Dd_j(\ddd_1\otimes\ddd_1)\Sd_{j+2} (\ddd_1 \otimes \ddd_1:\delta A) \big)
		:\Dd_q \delta A \,\dd x
		\Big|\\
		&\lesssim\,	
		\sum_{q=-1}^\infty
		\sum_{j= q - 5}^\infty 
		2^{-\frac{q}{2}}
		\|\,
			\Dd_q\big(\Dd_j(\ddd_1\otimes\ddd_1)
			\Sd_{j+2} (\ddd_1 \otimes \ddd_1:\delta A )\big) 
		\|_{L^\frac{2}{1+\ee}}
		2^{-\frac{q}{2}}
		\|\,\Dd_q \delta A\,\|_{L^\frac{2}{1-\ee}}\\
		&\lesssim\,	
		\sum_{q=-1}^\infty
		\sum_{j=q - 5}^\infty 
		2^{-\frac{q}{2}}
		\|\,
			\Dd_j(\ddd_1\otimes\ddd_1)
			\Sd_{j+2} (\ddd_1 \otimes \ddd_1:\delta A ) 
		\|_{L^\frac{2}{1+\ee}}
		2^{-\frac{q}{2}} 
		\|\,\Dd_q  \nabla \delta \uu\,\|_{L^\frac{2}{1-\ee}}\\
		&\lesssim\,	
		\sum_{q=-1}^\infty
		\sum_{j= q - 5}^\infty 
		2^{-\frac{q}{2}}\|\,\Dd_j(\ddd_1\otimes\ddd_1)\,\|_{L^\frac 2\ee }
		\|\Sd_{j+2} (\ddd_1 \otimes \ddd_1:\delta A ) \|_{L^2}
		2^{-\frac{q}{2}} 2^q\|\,\Dd_q  \delta \uu\,\|_{L^\frac{2}{1-\ee}}
		\\
		&\lesssim\,	
		\sum_{q=-1}^\infty
		\sum_{j= q - 5}^\infty 
		2^{\frac{q-j}{2}}\|\,\Dd_j\nabla (\ddd_1\otimes\ddd_1)\,\|_{L^\frac 2\ee }
		2^{-\frac{q}{2}} \|\,\Dd_q  \delta \uu\,\|_{L^\frac{2}{1-\ee}}
		2^{-\frac{j}{2}}\|\, \Sd_{j+2} (\ddd_1 \otimes \ddd_1:\delta A ) \|_{L^2}.
\end{align*}
Hence, it follows that  
\begin{equation}\label{juve15}
\begin{aligned}
		|\mathcal K_4|
		&\lesssim\,	
	    \|\ddd_1 \otimes \ddd_1:\delta A \|_{H^{-\frac{1}{2}}}
		\Big(
		\sum_{q=-1}^\infty
		\sum_{j= q - 5}^\infty 
		2^{q-j}\|\Dd_j\nabla (\ddd_1\otimes\ddd_1)\|_{L^\frac{2}{\ee} }^2
		2^{-q} \|\Dd_q  \delta \uu\|_{L^\frac{2}{1-\ee}}^2
		\Big)^\frac{1}{2}
		\\
		&\lesssim\,	
		\| \ddd_1 \otimes \ddd_1:\delta A \|_{H^{-\frac{1}{2}}}
		\big( \mathcal{K}_4^{(1)} +\mathcal{K}_4^{(2)}\big),
	\end{aligned}
\end{equation}
where
\begin{align*}
   \mathcal{K}_4^{(1)} 
   &:= \Big(
		\sum_{q=-1}^\infty
		\sum_{j= q - 5}^\infty 
		2^{q-j}\|\,\Dd_j(\nabla \ddd_1\otimes S_N \ddd_1)\,\|_{L^\frac{2}{\ee}}^2
		2^{-q(1-\ee)} \|\,\Dd_q  \delta \uu\,\|_{L^2}^{2(1-\ee)}
		2^{-q\ee}\|\,\Dd_q \nabla \delta \uu\,\|_{L^2}^{2\ee}
		\Big)^\frac{1}{2},\\
  \mathcal{K}_4^{(2)} 
   &:= \Big(
			\sum_{q=-1}^\infty
			\sum_{j= q - 5}^\infty 
			2^{q-j}\|\,\Dd_j(\nabla \ddd_1\otimes (\Id\,-\,S_N )\ddd_1)\,\|_{L^\frac{2}{\ee}}^2
			2^{-q} \|\,\Dd_q  \delta \uu\,\|_{L^\frac{2}{1-\ee}}^2
			\Big)^\frac{1}{2}.
\end{align*}
Thanks to Young's inequality, the first term $\mathcal K_4^{(1)}$ can be estimated as follows
	\begin{align*}
		|\mathcal K_4^{(1)}|
		&\lesssim\,
		\|\nabla \ddd_1\otimes S_N \ddd_1\|_{L^\frac{2}{\ee}}
		\|\delta \uu\|_{H^{-\frac{1}{2}}}^{1-\ee}
		\|\nabla \delta \uu\|_{H^{-\frac{1}{2}}}^{\ee} \\
		&\lesssim\,
		\|\nabla \ddd_1\|_{L^\frac{2}{\ee}} \| S_N \ddd_1\|_{L^\infty}
		\|\delta \uu\|_{H^{-\frac{1}{2}}}^{1-\ee}
		\|\nabla \delta \uu\|_{H^{-\frac{1}{2}}}^{\ee}\\
		&\lesssim\,
		\sqrt{	\frac{N}{\ee}}
		\| \ddd_1\|_{H^1}^{1+\ee}
		\| \nabla \ddd_1\|_{H^1}^{1-\ee}
		\| \delta \uu \|_{H^{-\frac{1}{2}}}^{1-\ee}
		\| \nabla \delta \uu \|_{H^{-\frac{1}{2}}}^{\ee}.
	\end{align*}
On the other hand, for $\mathcal K_4^{(2)}$, we observe that
\begin{align*}
	|\mathcal K_4^{(2)}| 
	\,&
	\lesssim\,
	\Big(
	\sum_{q=-1}^\infty
	\sum_{j= q - 5}^\infty 
	2^{q-j} 2^{2j(1-\ee)} \|\Dd_j(\nabla \ddd_1\otimes (\Id\,-\,S_N )\ddd_1)\|_{L^2 }^2
	2^{-q}2^{2q\ee}\|\Dd_q  \delta \uu\|_{L^2}^2
	\Big)^\frac{1}{2}\\
	&\lesssim\,
	\Big(
	\sum_{q=-1}^\infty
	\sum_{j= q - 5}^\infty 
	2^{2(q-j)\ee} 
	2^{j} \|\Dd_j(\nabla \ddd_1\otimes (\Id\,-\,S_N )\ddd_1)\|_{L^2 }^2
	\|\Dd_q   \delta \uu\|_{L^2}^2
	\Big)^\frac{1}{2}\\
	&\lesssim\,
	\|\nabla \ddd_1\otimes (\Id\,-\,S_N )\ddd_1\|_{H^{\frac{1}{2}}}
	\|\,  \delta \uu\,\|_{L^2}\\
	&\lesssim\,
	\| \nabla \ddd_1\|_{H^\frac34} 
	\| (\Id\,-\,S_N )\ddd_1\|_{H^{\frac{3}{4}}}
	\|\, \delta \uu\,\|_{L^2}\\
	&\lesssim\,
	\| \nabla \ddd_1\|_{H^1} 
	\|\ddd_1\|_{H^1} 
	2^{-\frac{N}{4}} 
	(\|\uu_1\|_{L^2}+\|\uu_2\|_{L^2})\\
	&\lesssim\,
	(\|\ddd_1\|_{H^1}^2+ \|\uu_1\|_{L^2}^2 +\|\uu_2\|_{L^2}^2)
	\| \nabla \ddd_1\|_{H^1}
	2^{-\frac{N}{4}}.
\end{align*}
From the above estimates,  we obtain
	\begin{align*}
		|\mathcal{K}_4^{(1)}| \,+\,|\mathcal{K}_4^{(2)}|
		\,\lesssim\,
		\sqrt{	\frac{N}{\ee}}
		\|\, \ddd_1\,			\|_{H^1}^{1+\ee}
		\|\,\nabla \ddd_1		\|_{H^1}^{1-\ee}
		\|\,\delta \uu\,		\|_{H^{-\frac{1}{2}}}^{1-\ee}
		\|\,\nabla \delta \uu\,	\|_{H^{-\frac{1}{2}}}^{\ee}
		\,+\,
		(\|\ddd_1\|_{H^1}^2+ \|\uu_1\|_{L^2}^2 +\|\uu_2\|_{L^2}^2)
	\| \nabla \ddd_1\|_{H^1}
	2^{-\frac{N}{4}}.
	\end{align*}
Combining this inequality together with Lemma \ref{lemma:juve16} and \eqref{juve15}, applying Young's inequality, we finally achieve that
	\begin{align*}
		|\mathcal K_4 |
		\,&\leq\,
		C(\eta^{-\frac{\ee}{1-\ee}}+\eta^{-\frac{1+\ee}{1-\ee}})
		\left(\frac{N}{\ee}
		\right)^\frac{1}{1-\ee}
		\|\ddd_1\|_{H^1}^{2\frac{1+\ee}{1-\ee}}\|\ddd_1\|_{H^2}^{2}
		\|\delta \uu 	\|_{H^{-\frac{1}{2}}}^{2}
		\\
		&\quad +\,
			2\eta\nu 
			\|\,\nabla \delta \uu \|_{H^{-\frac{1}{2}}}^2
			\, +\,
			2\eta \sum_{q=-1}^\infty
			2^{-q}
			\int_{\TT^2}|\sdm{q} (\ddd_1 \otimes \ddd_1): \Dd_q \delta A |^2\,\dd x \,
		+\, C
		\| \ddd_1\|_{H^1}^2
		\|\ddd_1\|_{H^2}^2
		\|\delta \uu\|_{H^{-\frac{1}{2}}}^2
		N
		\\ 
		&\quad +\,
		C(1+\eta^{-1}) (\|\ddd_1\|_{H^1}^4+ \|\uu_1\|_{L^2}^4 + \|\uu_2\|_{L^2}^{4}) 
			(\|\ddd_1\|_{H^2}^2 + \|\nabla \uu_1\|_{L^2}^2 + \|\nabla \uu_2\|_{L^2}^2)
			2^{-\frac{N}{2}}.
	\end{align*}
Comparing with \eqref{juve28b}, we see that $\mathcal{K}_4$ can actually be estimated in the same way as for $\mathcal{K}_1$ (recall \eqref{juve28}), by taking \begin{align*} 
N(t) = 2\Big\lfloor \log_2\Big(1+\frac{1}{\Phi(t)}\Big)\Big\rfloor + 2\quad\text{and}\quad  \ee(t):=\frac{1}{2+2\ln(1+ N(t))}.
\end{align*} 
Thus, we obtain  
\begin{equation}\label{juve29}
	\begin{aligned}
		|\mathcal K_4 |
		&\leq C(1+\eta^{-3}) (\|\ddd_1\|_{H^1}^2+\|\ddd_1\|_{H^1}^6+ \|\uu_1\|_{L^2}^4 + \|\uu_2\|_{L^2}^4) 
	(\|\ddd_1\|_{H^2}^2 + \|\nabla \uu_1\|_{L^2}^2 + \|\nabla \uu_2\|_{L^2}^2) \mu(\Phi(t))\\
	&\qquad +\,
			2\eta\nu 
			\|\,\nabla \delta \uu \|_{H^{-\frac{1}{2}}}^2
			\, +\,
			2\eta \sum_{q=-1}^\infty
			2^{-q}
			\int_{\TT^2}|\sdm{q} (\ddd_1 \otimes \ddd_1): \Dd_q \delta A |^2\,\dd x.
	\end{aligned}
\end{equation}

At last, for $\mathcal K_3$, we make use of the Bony's decomposition \eqref{bony-decomp} and decompose it into
\begin{equation}\label{juve22}
	\begin{aligned}
		\mathcal K_3= 
		\sum_{q=-1}^\infty
		2^{-q}
		\int_{\TT^2}
		\Dd_q(\ddd_1 \otimes \ddd_1:\delta A) 
		(\sdm{q}(\ddd_1\otimes\ddd_1):\Dd_q \delta A)\,\dd x
		\,=\,
		\mathcal K_3^{(1)}
		\,+\,
		\mathcal K_3^{(2)}
		\,+\,
		\mathcal K_3^{(3)}
		\,+\,
		\mathcal K_3^{(4)},
	\end{aligned}
\end{equation}
where
\begin{align*}
		\mathcal K_3^{(1)}&:=
		\sum_{q=-1}^\infty
		\sum_{|j-q|\leq 5}
		2^{-q}
		\int_{\TT^2}
		([\Dd_q,\,\sdm{j}(\ddd_1 \otimes \ddd_1)]:\Dd_j\delta A) (\sdm{q}(\ddd_1\otimes\ddd_1):
		\Dd_q \delta A)\,\dd x,\\
		\mathcal K_3^{(2)}
		&:=
		\sum_{q=-1}^\infty
		\sum_{|j-q|\leq 5}
		2^{-q}
		\int_{\TT^2}
		((\sdm{j}-\sdm{q})(\ddd_1 \otimes \ddd_1):\Dd_j\delta A) 
		(\sdm{q}(\ddd_1\otimes\ddd_1):\Dd_q \delta A)\,\dd x,\\
		\mathcal K_3^{(3)}
		&:=
		\sum_{q=-1}^\infty
		2^{-q}
		\int_{\TT^2}
		\big|\,\sdm{q}(\ddd_1 \otimes \ddd_1):\Dd_q\delta A\,\big|^2\,\dd x,\\
		\mathcal K_3^{(4)}
		&:=
		\sum_{q=-1}^\infty
		\sum_{j=q-5}^\infty 
		2^{-q}
		\int_{\TT^2}
		\Dd_q(\Dd_j(\ddd_1 \otimes \ddd_1):S_{j+2}\delta A) 
		(\sdm{q}(\ddd_1\otimes\ddd_1):\Dd_q \delta A)\,\dd x.
\end{align*}
Recalling the minus sign in $\mathcal{T}_4$, we point out that $\mathcal{K}_3^{(3)}$ is exactly one of the dissipative terms appearing in $\mathfrak{D}(t)$ and cancels with the corresponding one in Proposition \ref{prop:last-ineq-main-thm}, part $\mathrm{(b)}$. Therefore, We shall just  analyse $\mathcal{K}_3^{(1)}$, $\mathcal{K}_3^{(2)}$ and $\mathcal{K}_3^{(4)}$, by introducing one more time two parameters $N\in \mathbb{N}$ and $\varepsilon \in (0,1/2)$ that will be defined later on, in relation to the functional $\Phi(t)$.

We first apply the commutator inequality in Lemma  \ref{prop:comm-est}, inferring that 
$$\|[\Dd_q,\,\sdm{j}(\ddd_1 \otimes \ddd_1)]:\Dd_j\delta A\|_{L^2} 
\leq 2^{-q}\|\sdm{j}\nabla (\ddd_1 \otimes \ddd_1)\|_{L^{\frac{2}{\ee}}}\|\Dd_j\nabla \delta \uu \|_{L^{\frac{2}{1-\ee}}},
$$ to gather 
\begin{align*}		
	|\mathcal K_3^{(1)}|
	&\lesssim\,
	\Big(
	\sum_{q=-1}^\infty
	\sum_{|j-q|\leq 5}
	2^{-q}
	\|[\Dd_q,\,\sdm{j}(\ddd_1 \otimes \ddd_1)]:\Dd_j\delta A\|_{L^2}^2
	\Big)^\frac{1}{2}
	\Big(
	\sum_{q=-1}^\infty
	2^{-q}
	\int_{\TT^2}|\sdm{q}(\ddd_1 \otimes \ddd_1):\Dd_q\delta A|^2\,\dd x
	\Big)^\frac{1}{2}\\
	\,&\lesssim\,
	\bigg(
	\sum_{q=-1}^\infty
	\sum_{|j-q|\leq 5}
	2^{2(j-q)}2^{-2j}
	\|\sdm{j}\nabla (\ddd_1 \otimes \ddd_1)\|_{L^\frac{2}{\ee}}^2
	\|\Dd_j\nabla \delta \uu				\|_{L^\frac{2}{1-\ee}}^2
	\bigg)^\frac{1}{2}
\Big(
	\sum_{q=-1}^\infty
	2^{-q}
	\int_{\TT^2}|\sdm{q}(\ddd_1 \otimes \ddd_1):\Dd_q\delta A|^2\,\dd x
	\Big)^\frac{1}{2}\\
	\,&\lesssim\,
	\bigg(
	\sum_{j=-1}^\infty
	2^{-2j}
	\|\sdm{j}((\nabla \ddd_1) \otimes \ddd_1)\|_{L^\frac{2}{\ee}}^2
	\|\Dd_j\nabla \delta \uu				\|_{L^\frac{2}{1-\ee}}^2
	\bigg)^\frac{1}{2}
	\Big(
	\sum_{q=-1}^\infty
	2^{-q}
	\int_{\TT^2}|\sdm{q}(\ddd_1 \otimes \ddd_1):\Dd_q\delta A|^2\,\dd x
	\Big)^\frac{1}{2}.
\end{align*}
Like before, we separately control the low frequencies and high frequencies of $\ddd_1 = S_N \ddd_1 + (\Id -S_N)\ddd_1$, through
\begin{equation*}
    |\mathcal K_3^{(1)}|\leq \mathcal{R}_1+\mathcal{R}_2,
\end{equation*}
where
\begin{align*}		
	\mathcal{R}_1&:=
	\Big(
	\sum_{j=-1}^\infty
	2^{-j}
	\|\sdm{j}
		( \nabla \ddd_1 \otimes S_N \ddd_1)\|_{L^\frac{2}{\ee}}^2
	\|\Dd_j\delta \uu\|_{L^\frac{2}{1-\ee}}^2
	\Big)^\frac{1}{2}
	\Big(
	\sum_{q=-1}^\infty
	2^{-q}
	\int_{\TT^2}|\sdm{q}(\ddd_1 \otimes \ddd_1):\Dd_q\delta A|^2\,\dd x
	\Big)^\frac{1}{2},\\
	\mathcal{R}_2&:=
	\Big(
	\sum_{j=-1}^\infty
	2^{-j}
	\|\sdm{j}( \nabla \ddd_1 \otimes (\Id-S_N) \ddd_1)\|_{L^\frac{2}{\ee}}^2
	\|\Dd_j\delta \uu\|_{L^\frac{2}{1-\ee}}^2
	\Big)^\frac{1}{2}	
	\Big(
	\sum_{q=-1}^\infty
	2^{-q}
	\int_{\TT^2}|\sdm{q}(\ddd_1 \otimes \ddd_1):\Dd_q\delta A|^2\,\dd x
	\Big)^\frac{1}{2}.
\end{align*}
The low frequencies of $S_N \ddd_1$ are dealt with the embedding $\| S_N \ddd_1 \|_{L^\infty} \leq C\sqrt{N}\| \ddd_1 \|_{H^1}$, thanks to Lemma \ref{lemma:SN-infty}. Through a standard interpolation of Lebesgue norms and the embedding $\| \nabla \ddd_1 \|_{L^{2/\ee}}\leq C\| \nabla \ddd_1 \|_{H^{1-\ee}}/\sqrt{\ee}$ of Lemma \ref{lemma:eps}, we gather 
\begin{align*}
	\mathcal{R}_1
	\,&\leq\,
	\Big(
	\sum_{j=-1}^\infty
	2^{-j}
	\|\nabla \ddd_1		\|_{L^\frac{2}{\ee}}^2 
	\|S_N \ddd_1 		\|_{L^\infty}^2 
	\|\Dd_j\delta \uu	\|_{L^\frac{2}{1-\ee}}^2
	\Big)^\frac{1}{2}
	\Big(
	\sum_{q=-1}^\infty
	2^{-q}
	\int_{\TT^2}|\sdm{q}(\ddd_1 \otimes \ddd_1):\Dd_q\delta A|^2\,\dd x
	\Big)^\frac{1}{2}\\
	&\lesssim\,
	\sqrt{\frac{N}{\ee}}
	\bigg(
	\sum_{j=-1}^\infty
	\|\nabla \ddd_1	\|_{L^2}^{2\ee}
	\|\nabla \ddd_1	\|_{H^1}^{2(1-\ee)} 
	\|\ddd_1 		\|_{H^1}^2 
	2^{-j(1-\ee)}
	\|\,\Dd_j\delta \uu\|_{L^2}^{2(1-\ee)}
	2^{-j\ee}
	\|\,\Dd_j\nabla \delta \uu\|_{L^2}^{2\ee}
	\bigg)^\frac{1}{2}\\
	&\qquad \times 	\Big(
	\sum_{q=-1}^\infty
	2^{-q}
	\int_{\TT^2}|\sdm{q}(\ddd_1 \otimes \ddd_1):\Dd_q\delta A|^2\,\dd x
	\Big)^\frac{1}{2}
	\\
	&\lesssim\,
	\sqrt{\frac{N}{\ee}}
	\|		 \ddd_1 		\|_{H^1}^{1+\ee} 
	\|	\nabla \ddd_1 		\|_{H^1}^{1-\ee} 
	\|\,\delta \uu			\|_{H^{-\frac{1}{2}}}^{1-\ee}
	\|\,\nabla \delta \uu	\|_{H^{-\frac{1}{2}}}^{\ee}
	\Big(
	\sum_{q=-1}^\infty
	2^{-q}
	\int_{\TT^2}|\sdm{q}(\ddd_1 \otimes \ddd_1):\Dd_q\delta A|^2\,\dd x
	\Big)^\frac{1}{2},
\end{align*}
which together with Young's inequality eventually implies that
\begin{equation}\label{juve20}
\begin{aligned} 
	\mathcal{R}_1
	&\leq
	C\eta^{-\frac{1+\ee}{1-\ee}}\left(
	\frac{N}{\ee}
	\right)^\frac{1}{1-\ee}
	\|\ddd_1\|_{H^1}
	^{2\frac{1+\ee}{1-\ee}}
	\| \nabla \ddd_1 	\|_{H^1}^2
	\|	\delta \uu		\|_{H^{-\frac{1}{2}}}^2
	\,+\,
	\eta\nu
	\|\nabla \delta \uu\|_{H^{-\frac{1}{2}}}^2
	\\
	& \qquad 
	+ \eta \sum_{q=-1}^\infty
	2^{-q}
	\int_{\TT^2}|\sdm{q}(\ddd_1 \otimes \ddd_1):\Dd_q\delta A|^2\,\dd x\\
	& \leq
	C(1+\eta^{-3})
	\left(
	\frac{N}{\ee}
	\right)^\frac{\ee }{1-\ee}
	\big(
	\|\ddd_1\|_{H^1}^2
	+ \|\ddd_1\|_{H^1}^6
	\big)
	\| \nabla \ddd_1 	\|_{H^1}^2
	\|	\delta \uu		\|_{H^{-\frac{1}{2}}}^2
	\left(\frac{N}{\ee}\right)
	\,+\,
	\eta\nu
	\|\nabla \delta \uu\|_{H^{-\frac{1}{2}}}^2\\
	&\qquad \,+\,
	\frac{\eta}{2}  \sum_{q=-1}^\infty
	2^{-q}
	\int_{\TT^2}|\sdm{q}(\ddd_1 \otimes \ddd_1):\Dd_q\delta A|^2\,\dd x,
	\end{aligned} 
\end{equation}
since $(1+\ee)/(1-\ee)\in(1,3)$ under later choice $\ee\in(0,1/2)$. Next, we treat $\mathcal{R}_2$. 
Thanks to the Bernstein inequality in Lemma \ref{prop:Bernstein}, it holds
\begin{equation*}
    \| \Dd_ j \delta \uu \|_{L^\frac{2}{1-\ee}}\leq C2^{j\ee} \| \Dd_j \delta \uu \|_{L^2}\quad \text{and}
    \quad 
    \|\sdm{j}( \nabla \ddd_1 \otimes (\Id-S_N) \ddd_1)\|_{L^\frac{2}{\ee}}
    \leq 
    C2^{j(1-\ee)}
    \|\sdm{j}( \nabla \ddd_1 \otimes (\Id-S_N) \ddd_1)\|_{L^2},
\end{equation*}
for certain constant $C>0$ that does not depend on $N$ and $\varepsilon$. Then we estimate $\mathcal{R}_2$ through
	\begin{align*}
		\mathcal R_2
		&\lesssim\,
		\Big(
		\sum_{j=-1}^\infty
			2^{-j}
			2^{2j(1-\ee)}
			\|\sdm{j}( \nabla \ddd_1 \otimes (\Id-S_N) \ddd_1)	\|_{L^2}^22^{2j\ee}
			\|\Dd_j\delta \uu									\|_{L^2}^2
		\Big)^\frac{1}{2}
		\Big(
	\sum_{q=-1}^\infty
	2^{-q}
	\int_{\TT^2}|\sdm{q}(\ddd_1 \otimes \ddd_1):\Dd_q\delta A|^2\,\dd x
	\Big)^\frac{1}{2}
		\\
		&\lesssim\,
		\Big(
		\sum_{j=-1}^\infty
		\|\nabla \ddd_1 			\|_{L^2}^2 
		\|(\Id-S_N) \ddd_1			\|_{L^\infty}^2
		2^{j}
		\|\Dd_j  \delta \uu	\|_{L^2}^2
		\Big)^\frac{1}{2}
		\Big(
	\sum_{q=-1}^\infty
	2^{-q}
	\int_{\TT^2}|\sdm{q}(\ddd_1 \otimes \ddd_1):\Dd_q\delta A|^2\,\dd x
	\Big)^\frac{1}{2}
		\\
		&\lesssim\,
		\|\nabla \ddd_1 			\|_{L^2}
		\| (\Id-S_N) \ddd_1			\|_{L^\infty}
		\|\delta \uu				\|_{H^\frac{1}{2}}
		\Big(
	\sum_{q=-1}^\infty
	2^{-q}
	\int_{\TT^2}|\sdm{q}(\ddd_1 \otimes \ddd_1):\Dd_q\delta A|^2\,\dd x
	\Big)^\frac{1}{2}\\
		&\lesssim 
		\| \ddd_1 			\|_{H^1}^\frac{3}{2}
		\| \nabla \ddd_1			\|_{H^1}^\frac{1}{2}
		(\|\uu_1\|_{L^2}+ \|\uu_2\|_{L^2})			^\frac{1}{2}
		(\|\nabla \uu_1\|_{L^2} + \|\nabla \uu_2	\|_{L^2})^\frac{1}{2}
		2^{-\frac{N}{2}}
		\Big(
	\sum_{q=-1}^\infty
	2^{-q}
	\int_{\TT^2}|\sdm{q}(\ddd_1 \otimes \ddd_1):\Dd_q\delta A|^2\,\dd x
	\Big)^\frac{1}{2},
\end{align*}
where we have used the estimate  $\|(\Id-S_N)\ddd_1 \|_{L^\infty} \leq C2^{-N/2} \| \nabla \ddd_1 \|_{H^{1/2}}$. Therefore, applying the Cauchy-Schwartz inequality, we gather that
\begin{equation}\label{juve21}
	\begin{aligned}
		\mathcal{R}_2
		& \leq\,
		C\eta^{-1}
		\| \ddd_1\|_{H^1}^3
		(\|\uu_1\|_{L^2}+ \|\uu_2\|_{L^2})	
		(\| \nabla \ddd_1			\|_{H^1}^2+ \|\nabla \uu_1\|_{L^2}^2 + \|\nabla \uu_2\|_{L^2}^2)
		2^{-N}\\
		&\quad +\, \frac{\eta}{2}  
	\sum_{q=-1}^\infty
	2^{-q}
	\int_{\TT^2}|\sdm{q}(\ddd_1 \otimes \ddd_1):\Dd_q\delta A|^2\,\dd x.
	\end{aligned}
\end{equation}
Combining \eqref{juve20} and \eqref{juve21}, we arrive at
	\begin{align*}
		|\mathcal K_3^{(1)}|
		& \leq
			C(1+\eta^{-3})
	\left(
	\frac{N}{\ee}
	\right)^\frac{\ee }{1-\ee}
	\big(
	\|\ddd_1\|_{H^1}^2
	+ \|\ddd_1\|_{H^1}^6
	\big)
	\| \nabla \ddd_1 	\|_{H^1}^2
	\|	\delta \uu		\|_{H^{-\frac{1}{2}}}^2
	\left(\frac{N}{\ee}\right)\\
		&\quad +
		C\eta^{-1}
		\| \ddd_1\|_{H^1}^3
		(\|\uu_1\|_{L^2}+ \|\uu_2\|_{L^2})	
		(\|  \ddd_1	\|_{H^2}^2+ \|\nabla \uu_1\|_{L^2}^2 + \|\nabla \uu_2\|_{L^2}^2)
		2^{-N} \\
		&\quad +\,\eta\nu
	\|\nabla \delta \uu\|_{H^{-\frac{1}{2}}}^2  \,+\,
	 \eta  \sum_{q=-1}^\infty
	2^{-q}
	\int_{\TT^2}|\sdm{q}(\ddd_1 \otimes \ddd_1):\Dd_q\delta A|^2\,\dd x.
	\end{align*}
Taking again $N=N(t)=\lfloor \log_2(1+1/\Phi(t))\rfloor +1$ and $\ee =\ee(t):= 1/(2+2\ln N)\in (0, 1/2)$, we  obtain that
\begin{equation}\label{juve23}
	\begin{aligned}
		|\mathcal K_3^{(1)}|
		&\leq
		C(1+\eta^{-3})
	\big(	\|\ddd_1\|_{H^1}^2
	+ \|\ddd_1\|_{H^1}^6 \big)
	\|\ddd_1 \|_{H^2}^2
		\mu(\Phi(t)) 
		\\
		&\quad +C\eta^{-1}
		\| \ddd_1\|_{H^1}^3
		(\|\uu_1\|_{L^2}+ \|\uu_2\|_{L^2})	
		(\| \ddd_1\|_{H^2}^2+ \|\nabla \uu_1\|_{L^2}^2 + \|\nabla \uu_2\|_{L^2}^2)
		\Phi(t)\\
		&\quad +\,\eta\nu
	\|\nabla \delta \uu\|_{H^{-\frac{1}{2}}}^2  \,+\,
	 \eta  \sum_{q=-1}^\infty
	2^{-q}
	\int_{\TT^2}|\sdm{q}(\ddd_1 \otimes \ddd_1):\Dd_q\delta A|^2\,\dd x.	\end{aligned}
\end{equation}
An analogous procedure leads to a similar estimate for $\mathcal K_3^{(2)}$ such that 
\begin{equation}\label{juve24}
	\begin{aligned}
		|\mathcal K_3^{(2)}|
	&\leq
		C(1+\eta^{-3})
	\big(	\|\ddd_1\|_{H^1}^2
	+ \|\ddd_1\|_{H^1}^6 \big)
	\| \ddd_1 \|_{H^2}^2
		\mu(\Phi(t)) \\
		&\quad +C\eta^{-1}
		\| \ddd_1\|_{H^1}^3
		(\|\uu_1\|_{L^2}+ \|\uu_2\|_{L^2})	
		(\| \ddd_1\|_{H^2}^2+ \|\nabla \uu_1\|_{L^2}^2 + \|\nabla \uu_2\|_{L^2}^2)
		\Phi(t)
		\\ 
		&\quad +\,\eta\nu
	\|\nabla \delta \uu\|_{H^{-\frac{1}{2}}}^2  \,+\,
	 \eta  \sum_{q=-1}^\infty
	2^{-q}
	\int_{\TT^2}|\sdm{q}(\ddd_1 \otimes \ddd_1):\Dd_q\delta A|^2\,\dd x.
	\end{aligned}
\end{equation}
It remains to control the $\mathcal K_3^{(4)}$. We see that 
	\begin{align*}
		|\mathcal K_3^{(4)}|
	    &=
		\Big|\sum_{q=-1}^\infty
		\sum_{j=q-5}^\infty 
		2^{-q}
		\int_{\TT^2}
		\Dd_q(\Dd_j(\ddd_1 \otimes \ddd_1):\Sd_{j+2}\delta A) (\sdm{q}(\ddd_1\otimes\ddd_1):
		\Dd_q \delta A)\,\dd x
		\Big|\\
		&\leq\,
		\sum_{q=-1}^\infty
		\sum_{j=q-5}^\infty 
		2^{-\frac{q}{2}}
		\|\Dd_q(\Dd_j(\ddd_1 \otimes \ddd_1):S_{j+2}\delta A)\|_{L^2}
		2^{-\frac{q}{2}}
		\|\sdm{q}(\ddd_1\otimes\ddd_1):\Dd_q \delta A\|_{L^2}\\
		&\lesssim\,
		\sum_{q=-1}^\infty
		\sum_{j= q-5}^\infty 
		2^{\frac{q}{2}}
		\|\Dd_q(\Dd_j(\ddd_1 \otimes \ddd_1):\Sd_{j+2}\delta A)\|_{L^1}
		2^{-\frac{q}{2}}
		\|\sdm{q}(\ddd_1\otimes\ddd_1):\Dd_q \delta A\|_{L^2}\\
		&\lesssim\,
		\Big(
		\sum_{q=-1}^\infty
		\sum_{j=q-5}^\infty
		2^{q-j}
		2^{4j}
		\|\Dd_j (\ddd_1 \otimes \ddd_1)\|_{L^2}^2
		2^{-j}
		\|S_{j+2}\delta \uu\|_{L^2}^2
		\Big)^\frac{1}{2}
		\Big(
	\sum_{q=-1}^\infty
	2^{-q}
	\int_{\TT^2}|\sdm{q}(\ddd_1 \otimes \ddd_1):\Dd_q\delta A|^2\,\dd x
	\Big)^\frac{1}{2}\\
		&\lesssim\,
		\Big(
		\sum_{q=-1}^\infty
		\sum_{j=q-5}^\infty
		\sum_{|\alpha|=2} 
		2^{q-j}
		\big(
			\|\,\Dd_j(\partial^{\alpha} \ddd_1 \otimes S_N \ddd_1)\|_{L^2}^2
			+
			\|\,\Dd_j(\partial^{\alpha} \ddd_1 \otimes (\Id - S_N) \ddd_1)\|_{L^2}^2
			\\
			&\qquad +
			\|\,\Dd_j(\nabla \ddd_1 \otimes \nabla \ddd_1)\|_{L^2}^2
		\big)
		2^{-j}
		\|\,\Sd_{j+2}\delta \uu \,\|_{L^2}^2
		\Big)^\frac{1}{2}
		\Big(
	\sum_{q=-1}^\infty
	2^{-q}
	\int_{\TT^2}|\sdm{q}(\ddd_1 \otimes \ddd_1):\Dd_q\delta A|^2\,\dd x
	\Big)^\frac{1}{2}.
	\end{align*}
Here, we denote the multi-index  $\alpha=(\alpha_1,\alpha_2)$, with $\alpha_1, \alpha_2\in \mathbb{N}$ and $|\alpha|=\alpha_1+\alpha_2=2$ such that
$\partial^\alpha=\partial^{\alpha_1}_{x_1}\partial^{\alpha_2}_{x_2}$. 
Thanks to Young's inequality, we then deduce that 
	\begin{align*}
		\Big(
		\sum_{q=-1}^\infty
		\sum_{j= q-5}^\infty
		2^{q-j}
		\|\Dd_j(\nabla \ddd_1 \otimes \nabla \ddd_1)\|_{L^2}^2
		2^{-j}
		\|S_{j+2}\delta \uu \|_{L^2}^2
		\Big)^\frac{1}{2}
		&\lesssim
		\|\nabla \ddd_1\|_{L^4}^2
		\|\delta \uu\|_{H^{-\frac{1}{2}}}\\
		&\lesssim
		\|\nabla \ddd_1	\|_{L^2}
		\|\nabla \ddd_1	\|_{H^1}
		\|\delta \uu	\|_{H^{-\frac{1}{2}}},
	\end{align*}
while
	\begin{align*}
		\Big(
		\sum_{q=-1}^\infty
		\sum_{j= q-5}^\infty 
		\sum_{|\alpha|=2}
		2^{q-j}
		\|\Dd_j(\partial^\alpha \ddd_1 \otimes S_N \ddd_1)\|_{L^2}^2
		2^{-j}
		\|S_{j+2}\delta \uu \|_{L^2}^2
		\Big)^\frac{1}{2}
		&\leq 
		C
		\|S_N \ddd_1	\|_{L^\infty}
		\|\ddd_1	\|_{H^2}
		\|\delta \uu 	\|_{H^{-\frac{1}{2}}}\\
		&\leq
		C\sqrt{N}
		\|\ddd_1\|_{H^1}
		\|\ddd_1\|_{H^2}
		\|\delta \uu\|_{H^{-\frac{1}{2}}}.
	\end{align*}
Besides, we have
	\begin{align*}
		& \Big(
		\sum_{q=-1}^\infty
		\sum_{j= q-5}^\infty
		\sum_{|\alpha|=2} 
		2^{q-j}
		\big(\| \Delta_j (\partial^\alpha \ddd_1 \otimes (\Id - S_N) \ddd_1)\|_{L^2}^2
		2^{-j}\big)
		\|S_{j+2}\delta \uu \,\|_{L^2}^2
		\Big)^\frac{1}{2}\\
		&\qquad \lesssim
		\sum_{|\alpha|=2} \|\partial^\alpha \ddd_1\otimes (\Id-S_N)\ddd_1\|_{H^{-\frac{1}{2}}}
		\|\delta \uu \|_{L^2}
		\\
		&\qquad \lesssim\,
		\| \ddd_1\|_{H^2} 
		\|(\Id-S_N)\ddd_1		\|_{H^{\frac{1}{2}}}
		\|\delta \uu \|_{L^2}\\
		&\qquad \lesssim
		\| \ddd_1\|_{H^2}
		\| \ddd_1\|_{H^1}2^{-\frac{N}{2}}
		\|	\delta \uu 					\|_{L^2}
		\\
		&\qquad \lesssim
		\|\ddd_1\|_{H^1}
		\|\ddd_1\|_{H^2}
		(\|\uu_1\|_{L^2}+ \|\uu_2\|_{L^2})
		2^{-\frac{N}{2}}.
	\end{align*}
Summarizing the above estimates, we deduce from Young's inequality that
	\begin{align*}
		|K^{(4)}_{3}|
		& \leq\,
		C \eta^{-1} 
		\| \ddd_1\|_{H^1}^2
		\| \ddd_1\|_{H^2}^2
		\|\delta \uu 			\|_{H^{-\frac{1}{2}}}^2
		(1+N)+
		C \eta^{-1} 
		\|\ddd_1\|_{H^1}^2
		\|\ddd_1\|_{H^2}^2
		(\|\uu_1\|_{L^2}^2+ \|\uu_2\|_{L^2}^2)
		2^{-N}\\
		&\quad 
		\,+\,
		2 \eta  \sum_{q=-1}^\infty
	2^{-q}
	\int_{\TT^2}|\sdm{q}(\ddd_1 \otimes \ddd_1):\Dd_q\delta A|^2\,\dd x.
	\end{align*}
Taking $N=\lfloor \log_2(1+1/\Phi(t))\rfloor +1$, we then deduce that
\begin{equation}\label{juve25}
	\begin{aligned}
		|\mathcal{K}^{(4)}_3|
		& \leq\,
		C \eta^{-1} 
		\| \ddd_1\|_{H^1}^2
		(1+\|\uu_1\|_{L^2}^2+ \|\uu_2\|_{L^2}^2)
		\| \ddd_1\|_{H^2}^2
		\Phi(t)
		\Big(1+\ln\Big(1+\frac{1}{\Phi(t)}\Big)
		\Big)\\
		&\quad 
		+
	2	\eta  \sum_{q=-1}^\infty
	2^{-q}
	\int_{\TT^2}|\sdm{q}(\ddd_1 \otimes \ddd_1):\Dd_q\delta A|^2\,\dd x.
	\end{aligned}
\end{equation}
Therefore, combining the estimates \eqref{juve23}, \eqref{juve24} and  \eqref{juve25}, applying Young's inequality, we finally deduce from the identity \eqref{juve22} that
\begin{equation}\label{juve26}
	\begin{aligned}
		|\mathcal{K}_3-\mathcal{K}_3^{(3)}|
		& \leq\,
		C(1+\eta^{-3})
	\big(\|\ddd_1\|_{H^1}^2
	+ \|\ddd_1\|_{H^1}^6 + 
	\|\uu_1\|_{L^2}^4 + \|\uu_2\|_{L^2}^4\big)
	(\|\ddd_1\|_{H^2}^2+ \|\nabla \uu_1\|_{L^2}^2 + \|\nabla \uu_2\|_{L^2}^2)\mu(\Phi(t)) \\
	&\qquad +\,2\eta\nu
	\|\nabla \delta \uu\|_{H^{-\frac{1}{2}}}^2  \,+\,
	4\eta  \sum_{q=-1}^\infty
	2^{-q}
	\int_{\TT^2}|\sdm{q}(\ddd_1 \otimes \ddd_1):\Dd_q\delta A|^2\,\dd x.
	\end{aligned}
\end{equation}
Recalling the identity \eqref{last-term-of-second-term} and the estimates \eqref{juve28}, \eqref{juve27}, \eqref{juve29}, \eqref{juve26}, we can conclude that
\begin{align*}
	& \Big|
	\langle 	(\ddd_1 \cdot(\delta A \ddd_1)
		\ddd_1 \otimes \ddd_1,\, \nabla  \delta \uu 
	\rangle_{H^{-\frac{1}{2}}}(t)
	-\sum_{q=-1}^\infty
		2^{-q}
		\int_{\TT^2}
		\big|\,\sdm{q}(\ddd_1 \otimes \ddd_1):\Dd_q\delta A\,\big|^2(t,x)\,\dd x
	\Big|
	\\
	&\quad \leq  
	C(1+\eta^{-3}) g_2(t)
	\mu(\Phi(t))
		+\, 8\eta\nu
	\|\nabla \delta \uu\|_{H^{-\frac{1}{2}}}^2  \,+\,
	10 \eta  \sum_{q=-1}^\infty
	2^{-q}
	\int_{\TT^2}|\sdm{q}(\ddd_1 \otimes \ddd_1):\Dd_q\delta A|^2\,\dd x,
\end{align*}
where the $C>0$ is independent of $\eta$ and the function $g_2(t)\in L^1(0,T)$ is defined by
\begin{equation*}
\begin{aligned}	
	g_2(t)=\big(\|\ddd_1(t)\|_{H^1}^2
	+ \|\ddd_1(t)\|_{H^1}^6 + 
	\|\uu_1(t)\|_{L^2}^4 + \|\uu_2(t)\|_{L^2}^4\big)
	(\|\ddd_1(t)\|_{H^2}^2+ \|\nabla \uu_1(t)\|_{L^2}^2 + \|\nabla \uu_2(t)\|_{L^2}^2)+1.
\end{aligned}
\end{equation*}
 This completes the proof for part (b) of Proposition \ref{prop:last-ineq-main-thm}.
\end{proof}

\appendix

\section{Proof of Lemma \ref{lemma:juve16}}
\setcounter{equation}{0}
In this appendix, we provide a proof for the technical Lemma \ref{lemma:juve16}.  
\begin{lemma}\label{lemma:juve16b}
	The following inequality holds for any positive integer $N\in\mathbb{N}$ and any $\ee\in (0, 1/2)$:
	\begin{equation}\label{ineq-d1d1deltaA-in-H12}
	\begin{aligned}
		\|\ddd_1 \otimes \ddd_1:\delta A \|_{H^{-\frac{1}{2}}}
		& \lesssim \Big(
			\sum_{q=-1}^\infty
			2^{-q}
			\int_{\TT^2}|\sdm{q} (\ddd_1 \otimes \ddd_1): \Dd_q \delta A |^2\,\dd x
			\Big)^\frac{1}{2}
			\\ 
			&\quad + \sqrt{\frac{N}{\ee}}
			\|\ddd_1				\|_{H^1}^{1+\ee}
			\|\nabla \ddd_1			\|_{H^1}^{1-\ee}
			\|\delta \uu 			\|_{H^{-\frac{1}{2}}}^{1-\ee}
			\|\nabla \delta \uu 	\|_{H^{-\frac{1}{2}}}^\ee
			+ 
			\|\ddd_1\|_{H^1}
			\|\ddd_1\|_{H^2}
			\|\delta \uu\|_{H^{-\frac{1}{2}}}
			\sqrt{N}
			\\
			&\quad 
			+ (\|\ddd_1\|_{H^1}^2+ \|\uu_1\|_{L^2}^2 + \|\uu_2\|_{L^2}^{2}) 
			(\|\ddd_1\|_{H^2}+ \|\nabla \uu_1\|_{L^2}+ \|\nabla \uu_2\|_{L^2})
			2^{-\frac{N}{2}}.
\end{aligned}
\end{equation}
\end{lemma}
\begin{proof}
Applying Bony's decomposition \eqref{bony-decomp} to the term under consideration, we get 
\begin{align*}
&	\|  \ddd_1 \otimes \ddd_1:\delta A \|_{H^{-\frac{1}{2}}}
	 =
	\bigg\|
	\Big(
	2^{-\frac{q}{2}}
	\| \Dd_q(\ddd_1 \otimes \ddd_1:\delta A) \|_{L^2}
	\Big)_{q\in\tilde \NN}
	\bigg\|_{\ell^2(\tilde \NN)}
		\,\leq \,
	\mathcal{J}_1
	\,+\,
	\mathcal{J}_2
	\,+\,
	\mathcal{J}_3
	\,+\,
	\mathcal{J}_4,
\end{align*}
where $\tilde \NN :=\NN\cup \{-1\}$, and 
\begin{align*}
\mathcal{J}_1
	&:=
	\bigg\|
	\Big(
	2^{-\frac{q}{2}}
	\sum_{|j-q|\leq 5}
	\| [\Dd_q,\, \sdm{j} (\ddd_1 \otimes \ddd_1)]:\Dd_q \delta A \|_{L^2}
	\Big)_{q\in\tilde \NN}
	\bigg\|_{\ell^2(\tilde \NN)},\\
\mathcal{J}_2	
    &:=
	\bigg\|
	\Big(
	2^{-\frac{q}{2}}
	\sum_{|j-q|\leq 5}
	\| ( \sdm{j}-\sdm{q}) (\ddd_1 \otimes \ddd_1):\Dd_q \Dd_j\delta A \|_{L^2}
	\Big)_{q\in\tilde \NN}
	\bigg\|_{\ell^2(\tilde \NN)},
	\\
	\mathcal{J}_3
	& :=
	\bigg\|
	\Big(
	2^{-\frac{q}{2}}
	\|  \sdm{q} (\ddd_1 \otimes \ddd_1): \Dd_q \delta A \|_{L^2}
	\Big)_{q\in\tilde \NN}
	\bigg\|_{\ell^2(\tilde \NN)},
	\\
	\mathcal{J}_4
	& :=
	\bigg\|
	\Big(
	2^{-\frac{q}{2}}
	\sum_{j= q- 5}^\infty 
	\| \Dd_q(\Dd_j (\ddd_1 \otimes \ddd_1):S_{j+2}\delta A) \|_{L^2}
	\Big)_{q\in\tilde \NN}
	\bigg\|_{\ell^2(\tilde \NN)}.
\end{align*}
 We first note that $\mathcal{J}_3$ corresponds to the first term on the right-hand side of the  inequality \eqref{ineq-d1d1deltaA-in-H12}. Next, we deal with $\mathcal{J}_1$ by using the commutator estimate in Lemma \ref{prop:comm-est}:
	\begin{align*}
		\mathcal{J}_1
		& \lesssim\,
		\Big\|2^{-\frac{3}{2}q}
		\Big(\sum_{|j-q|\leq 5}
		\|\,\sdm{j}\nabla  (\ddd_1 \otimes \ddd_1)\|_{L^\frac{2}{\ee}}
		\|\Dd_j  \delta \nabla  \uu \|_{L^\frac{2}{1-\ee}}
		\Big)_{q\in\tilde \NN}
		\Big\|_{\ell^2(\tilde \NN)}\\
		&\lesssim\,
		\Big\|
		\Big(
		\|\,\sdm{j} (\nabla  \ddd_1 \otimes \ddd_1)\|_{L^\frac{2}{\ee}}
		2^{-\frac j 2}
		\|\Dd_j \delta \uu \|_{L^\frac{2}{1-\ee}}
		\Big)_{j\in\tilde \NN}
		\Big\|_{\ell^2(\tilde \NN)}\\
		& \lesssim\,
		\Big\|
		\Big(
		\|\sdm{j} (\nabla  \ddd_1 \otimes S_N \ddd_1)\|_{L^\frac{2}{\ee}}
		2^{-\frac j 2}
		\|\Dd_j \delta \uu \|_{L^\frac{2}{1-\ee}}
		\Big)_{j\in\tilde \NN}
		\Big\|_{\ell^2(\tilde \NN)}
		\\ 
		&\qquad +
		\Big\|
		\Big(
		\|\sdm{j} (\nabla  \ddd_1 \otimes (\Id -S_N)\ddd_1)\|_{L^\frac{2}{\ee}}
		2^{-\frac j 2}
		\|\Dd_j \delta \uu \|_{L^\frac{2}{1-\ee}}
		\Big)_{j\in\tilde \NN}
		\Big\|_{\ell^2(\tilde \NN)}.
	\end{align*}
As a consequence, it holds 
	\begin{align*}
		\mathcal{J}_1
		&\lesssim\,
		\Big\|
		\Big(
		\|\nabla  \ddd_1\|_{L^\frac{2}{\ee}} \|S_N \ddd_1\|_{L^\infty}
		2^{-\frac j 2}
		\|\Dd_j   \delta \uu \|_{L^2}	^{1-\ee}
		\|\Dd_j \nabla  \delta \uu \|_{L^2}^\ee
		\Big)_{j\in\tilde \NN}
		\Big\|_{\ell^2(\tilde \NN)}
		\\
		&\quad 
		+\,
		\Big\|
		\Big(
		2^{j(1-\ee)}\|\sdm{j}  (\nabla  \ddd_1 \otimes (\Id -S_N)\ddd_1)\|_{L^2}
		2^{-\frac j 2}
		2^{j \ee}
		\|\Dd_j \delta \uu \|_{L^2}
		\Big)_{j\in\tilde \NN}
		\Big\|_{\ell^2(\tilde \NN)},
	\end{align*}
which implies
	\begin{align*}
		\mathcal{J}_1
		\,&\,\lesssim\,
		\sqrt{\frac{N}{\ee}}
		\|\ddd_1				\|_{H^1}^{1+\ee}
		\|\nabla \ddd_1			\|_{H^1}^{1-\ee}
		\|\delta \uu 			\|_{H^{-\frac{1}{2}}}^{1-\ee}
		\|\nabla \delta \uu 	\|_{H^{-\frac{1}{2}}}^\ee
		\,+\,
		\|\nabla  \ddd_1\|_{L^2} 
		\| (\Id -S_N)\ddd_1			\|_{H^{\frac{1}{2}}}
		\|\nabla \delta \uu\|_{L^2}
		\\
		\,&\,\lesssim\,
		\sqrt{\frac{N}{\ee}}
		\|\ddd_1				\|_{H^1}^{1+\ee}
		\|\nabla \ddd_1			\|_{H^1}^{1-\ee}
		\|\delta \uu 			\|_{H^{-\frac{1}{2}}}^{1-\ee}
		\|\nabla \delta \uu 	\|_{H^{-\frac{1}{2}}}^\ee
		\,+\,
		\|\ddd_1\|_{H^1}^{2} 
		(\|\nabla \uu_1\|_{L^2} +\|\nabla \uu_2\|_{L^2})
		2^{-\frac{N}{2}}.	
	\end{align*}
The second term $\mathcal{J}_2$ can be estimated with a similar procedure such that 
	\begin{align*}
		\mathcal{J}_2
		&\lesssim \,
		\Big\|
		\Big(
		2^{-\frac{q}{2}}
		\sum_{|j-q|\leq 5}
		\| ( \sdm{j}-\sdm{q}) (\ddd_1 \otimes \ddd_1)\|_{L^\frac{2}{\ee}} 
		\|\Dd_j\Dd_q \delta A \|_{L^\frac{2}{1-\ee}} 
		\,
		\Big)_{q\in\tilde \NN}
		\Big\|_{\ell^2(\tilde \NN)}
		\\
		&\lesssim \,
		\Big\|
		\Big(
		2^{-\frac{q}{2}}
		\sum_{|j-q|\leq 5}
		\| ( \sdm{j}-\sdm{q}) \nabla (\ddd_1 \otimes \ddd_1)\|_{L^\frac{2}{\ee}} 
		\|\Dd_j\Dd_q \delta \uu  \|_{L^\frac{2}{1-\ee}} 
		\,
		\Big)_{q\in\tilde \NN}
		\Big\|_{\ell^2(\tilde \NN)}
		\\
		&\lesssim\,
		\Big\|
		\Big(
		\sum_{|j-q|\leq 5}
		\|(\sdm{j}-\sdm{q})(\nabla \ddd_1 \otimes S_N  \ddd_1)\|_{L^\frac{2}{\ee}}
		2^{-\frac{j}{2}}
		\|\Dd_j \Dd_q \delta \uu \|_{L^\frac{2}{1-\ee}}
		\Big)_{q\in\tilde \NN}
		\Big\|_{\ell^2(\tilde \NN)}\\
		&\quad +
		\Big\|
		\Big(
		\sum_{|j-q|\leq 5}
		\|(\sdm{j}-\sdm{q})(\nabla  \ddd_1 \otimes (\mathrm{Id}-S_N)\ddd_1)\|_{L^\frac{2}{\ee}}
		2^{-\frac j 2}
		\|\Dd_j \Dd_q \delta \uu \|_{L^\frac{2}{1-\ee}}
		\Big)_{q\in\tilde \NN}
		\Big\|_{\ell^2(\tilde \NN)}\\
		&\lesssim\,
		\sqrt{\frac{N}{\ee}}
		\|\ddd_1\|_{H^1}^{1+\ee}
		\| \nabla \ddd_1\|_{H^1}^{1-\ee}
		\|\delta \uu \|_{H^{-\frac{1}{2}}}^{1-\ee}
		\|\nabla \delta \uu \|_{H^{-\frac{1}{2}}}^\ee
		+
		\|\ddd_1\|_{H^1}^{2} 
		(\|\nabla \uu_1\|_{L^2} +\|\nabla \uu_2\|_{L^2})	
		2^{-\frac{N}{2}}.
	\end{align*}	
Finally, the last part $\mathcal{J}_4$ can be  bounded by
	\begin{align*}
		\mathcal{J}_4
		\,&\lesssim\,
		\Big\|
		\Big(
		2^\frac{q}{2}
		\sum_{j= q- 5}^\infty 
		\| \Dd_q(\Dd_j (\ddd_1 \otimes \ddd_1):\Sd_{j+2}\delta A)\|_{L^1}
		\Big)_{q\in\tilde \NN}
		\Big\|_{\ell^2(\tilde \NN)}\\
		\,&\lesssim\,
		\Big\|
		\Big(
		\sum_{j= q- 5}^\infty 
		2^\frac{q-j}{2}
		2^j
		\| \Dd_j (\ddd_1 \otimes \ddd_1)\,\|_{L^2}
		2^{-\frac{j}{2}}\|\,\Sd_{j+2}\nabla \delta \uu\|_{L^2}
		\Big)_{q\in\tilde \NN}
		\Big\|_{\ell^2(\tilde \NN)}\\
		\,&\lesssim\,
		\Big\|
		\Big(
		\sum_{j= q- 5}^\infty 
		2^\frac{q-j}{2}
		2^{2j}
		\| \Dd_j (\ddd_1 \otimes \ddd_1)\,\|_{L^2}
		2^{-\frac{j}{2}}\|\,\Sd_{j+2} \delta \uu\|_{L^2}
		\Big)_{q\in\tilde \NN}
		\Big\|_{\ell^2(\tilde \NN)},
	\end{align*}
thus
	\begin{align*}		
		\mathcal{J}_4
		\,&\lesssim\,
		\Big\|
		\Big(
		\sum_{j= q- 5}^\infty 
		2^\frac{q-j}{2}
		\Big(\sum_{|\alpha|=2}
			\| \Dd_j (\partial^\alpha  \ddd_1 \otimes \ddd_1)\,\|_{L^2}\,+\,
			\| \Dd_j (\nabla \ddd_1 \otimes \nabla \ddd_1)\,\|_{L^2}
		\Big)	
		2^{-\frac{j}{2}}\|\,\Sd_{j+2} \delta \uu\|_{L^2}
		\Big)_{q\in\tilde \NN}
		\Big\|_{\ell^2(\tilde \NN)}\\
		\,&\lesssim\,
		\Big\|
		\Big(
		\sum_{j= q- 5}^\infty 
		2^\frac{q-j}{2}
		\Big( \sum_{|\alpha|=2}
			\| \Dd_j (\partial^\alpha  \ddd_1 \otimes S_N \ddd_1)\,\|_{L^2}\,+\, \sum_{|\alpha|=2}
			\| \Dd_j (\partial^\alpha \ddd_1 \otimes (\Id - S_N )\ddd_1)\,\|_{L^2}\\
			&\hspace{3cm}
			+\,
			\| \Dd_j (\nabla \ddd_1 \otimes \nabla \ddd_1)\|_{L^2}
		\Big)	
		2^{-\frac{j}{2}}\|\,\Sd_{j+1} \delta \uu\|_{L^2}
		\Big)_{q\in\tilde \NN}
		\Big\|_{\ell^2(\tilde \NN)}
		\\
		&\lesssim\,
		\|\ddd_1\|_{H^2}
		\|S_N \ddd_1\|_{L^\infty}
		\|\delta \uu\|_{H^{-\frac{1}{2}}}
		\,+\,
		\|\ddd_1\|_{H^2}
		\|(\Id-S_N)\ddd_1\|_{H^{\frac{1}{2}}}
		\|\delta \uu\|_{L^2}
		\,+\,
		\|\ddd_1\|_{W^{1,4}}^2
		\|\delta \uu\|_{H^{-\frac{1}{2}}}
		\\
		&\lesssim\,
		\| \ddd_1\|_{H^1}
		\|\ddd_1\|_{H^2}
		\|\delta \uu\|_{H^{-\frac{1}{2}}}
		\big(1+ \sqrt{N} \big)
		\,+\,
		\|\ddd_1\|_{H^2}
		\|\ddd_1\|_{H^1}
		(\|\uu_1\|_{L^2}+ \|\uu_2\|_{L^2}) 
		2^{-\frac{N}{2}}.
	\end{align*}
Here, we again use the notation of multi-index  $\alpha=(\alpha_1,\alpha_2)$, with $\alpha_1, \alpha_2\in \mathbb{N}$ and $|\alpha|=\alpha_1+\alpha_2=2$ such that
$\partial^\alpha=\partial^{\alpha_1}_{x_1}\partial^{\alpha_2}_{x_2}$. 

Collecting the above estimates and noticing that $N\geq 1$, we arrive at the conclusion of this lemma.
\end{proof}

\section{Derivation of the integral identity \eqref{energy-identity}}
\label{sec:integral-identity}
\setcounter{equation}{0}

In what follows, we present a derivation of the integral identity \eqref{energy-identity} for the difference of two given weak solutions. To this end, we first note that from the definition of weak solution (i.e., Definition \ref{def:weak-sol-energy-space}), the difference $\delta \uu = \uu_1-\uu_2$ belongs to $L^\infty(0,T; L^2(\TT^2))$. Thus, any Fourier coefficient $\delta \uu_n$ belongs to $L^\infty(0,T)$ and therefore also to $L^2(0,T)$. Hence, we invoke the equation for $\delta \uu$ given by  \eqref{main_system2}, namely,
\begin{equation*}
    \partial_t\delta \uu + \delta \uu\cdot \nabla \uu_1+\uu_2\cdot \nabla \delta \uu -\nu\Delta\delta \uu+\nabla \delta  \pre   =
			-\Div\, 
			\big(\nabla \ddd_1\odot\nabla \delta  \ddd\big)
			-\,\Div\, 
			\big(\nabla \delta \ddd\odot\nabla \ddd_2\big)
			+
			\Div\,\delta \bm{\sigma},
\end{equation*}
which is well-defined in a distributional sense in $\mathcal{D}'((0,T)\times \mathbb{T}^2)$. Applying a general test function of the form $\beta(t) e^{\mathrm{i}n\cdot x}$ to the equation, with $\beta(t) \in \mathcal{D}(0,T)$, we see that any Fourier coefficient $\delta \uu_n(t) :=\frac{1}{(2\pi)^2}\int_{\mathbb{T}^2} \delta \uu (t,x)e^{-\mathrm{i}n\cdot x}\dd x$, for a.e. $t\in (0,T)$, is a distributional solution of the following ordinary differential equation:
\begin{align}\label{appdx:eq-deltau}
	\partial_t \delta \uu_n
	=
	-\nu |n|^2\delta \uu_n
	+
	\mathrm{i} n\cdot 
	\Big[
	-  (\delta \uu\otimes\uu_1)_n 
	-  (\uu_2\otimes \delta \uu)_n 
	- (\delta \pre)_n \mathbb{I}
	-  (\nabla \delta \ddd \odot \nabla \ddd_1)_n 
	-  (\nabla \ddd_2 \odot \nabla \delta \ddd)_n
	+  \delta \bm{\sigma}_n
	\Big] 
\end{align}
in $\mathcal{D}'(0,T)$. A further analysis shows that the right-hand side of the above identity belongs to $L^2(0,T)$.
Indeed, recalling the Sobolev embedding $L^{\frac{2}{1+\zeta}}(\TT^2)\hookrightarrow H^{-\zeta}(\mathbb{T}^2)$ with $\zeta\in (0,1)$, we remark that $\delta \bm{\sigma}$ is in $L^2(0,T; H^{-\zeta}(\TT^2))$. Therefore, it follows that $(1+|n|)^{-\zeta} \delta \bm{\sigma}_n$ belongs to $L^2(0,T)$, which implies $\mathrm{ i} n\cdot \delta \bm{\sigma}_n \in L^2(0,T)$ for any $n\in\mathbb{N}^2$. 
In a similar manner, since $\uu_1$, $\uu_2$, $\nabla \ddd_1$ and $\nabla \ddd_2$ belong to $L^4((0,T)\times \mathbb{T}^2)$, then $\delta \uu \otimes \uu_1$, $\uu_2 \otimes \delta \uu$, $\nabla \ddd_1\odot \nabla \delta \ddd$ and $\nabla \delta \ddd\odot \nabla \delta \ddd$ are in $L^2((0,T)\times \mathbb{T}^2)$. This yields that their Fourier coefficients belong to $L^2(0,T)$.

The above observation implies that any coefficient $\delta \uu_n$ belongs to $W^{1,2}(0,T)$ and the identity \eqref{appdx:eq-deltau} should be understood in the corresponding weak sense. Additionally, it holds that 
$|\delta \uu_n(t)|^2 \in W^{1,1}(0,T)\hookrightarrow \mathcal{C}([0,T])$ and after integrating with respect to time on $(0,t)$, we have the following identity: 
\begin{align*}
	\frac{1}{2}|\delta \uu_n(t)|^2 + \nu \int_0^t|n|^2|\delta \uu_n(s)|^2\, \dd s 
	&=  
	\int_0^t \Big( - \mathrm{i} n\cdot (\delta \uu\otimes\uu_1)_n 
	- \mathrm{i} n\cdot (\uu_2\otimes \delta \uu)_n
	- \mathrm{i} n\cdot (\delta \pre)_n \mathbb{I}
	-\mathrm{i} n\cdot (\nabla \delta \ddd \odot \nabla \ddd_1)_n \\ 
	&\qquad\quad\ \  -\mathrm{i} n\cdot (\nabla \ddd_2 \odot \nabla \delta \ddd)_n 
	+\mathrm{i} n\cdot \delta \bm{\sigma}_n\Big)(s)\cdot \delta \uu_n(s)\,\dd s.
\end{align*}
Recalling that $\delta \uu \in L^\infty(0,T; L^2(\TT^2))\cap H^{1}(0,T; W^{-1,\frac{2}{1+\zeta}}(\TT^2))$, we can multiply the identity by $2^{-q}\varphi_q(n)$, and take the sum over $n\in \ZZ^2$ with $|n|\in [3\cdot 2^{q-2}, 2^{q+3}/3]$, to gather that
\begin{align*}
&	\frac{1}{2} \big(2^{-q}\| \Dd_q \delta \uu (t) \|_{L^2}^2\big) 
	+ \nu \int_0^t 2^{-q}
	\| \nabla \Dd_q \delta \uu (s) \|_{L^2}^2\,\dd s \\
	& \quad  =
	\int_0^t 2^{-q} \int_{\TT^2} \Dd_q
		\Big( - \Div\,(\delta \uu\otimes\uu_1) 
		- \Div\,(\uu_2\otimes \delta \uu) 
		 - \nabla \delta \pre 
		- \Div\, (\nabla \delta \ddd \odot \nabla \ddd_1) \\
	&\qquad\qquad\qquad \quad    -\Div(\nabla \ddd_2 \odot \nabla \delta \ddd)
	+\Div\, \delta \bm{\sigma}\Big)(s,x)\cdot \Dd_q\delta \uu(s,x)\,\dd x\,\dd s\\
	& \quad  =
	\int_0^t 2^{-q} \int_{\TT^2} \Dd_q
		\Big( - \Div\,(\delta \uu\otimes\uu_1) 
		- \Div\,(\uu_2\otimes \delta \uu) 
		- \Div\, (\nabla \delta \ddd \odot \nabla \ddd_1) \\
	&\qquad\qquad\qquad \quad    -\Div(\nabla \ddd_2 \odot \nabla \delta \ddd)
	+\Div\, \delta \bm{\sigma}\Big)(s,x)\cdot \Dd_q\delta \uu(s,x)\,\dd x\,\dd s,
\end{align*}
where we have used the fact $\Div \delta \uu =0$. 
Taking the sum over  $q\in\tilde{\mathbb{N}}:=\mathbb{N}\cup \{-1\}$, we obtain 
\begin{equation}\label{appx:deltau-ser}
\begin{aligned}
	& \frac{1}{2} \| \delta \uu (t) \|_{H^{-\frac{1}{2}}}^2 + \nu 
	\sum_{q=-1}^\infty 	\int_0^t 
	2^{-q} 	\| \nabla \Dd_q \delta \uu (s) \|_{L^2}^2\,\dd s  \\
   &\quad  =
	\sum_{q=-1}^\infty
	\int_0^t  2^{-q}
	    \int_{\TT^2} \Dd_q
		\Big( - \Div\,(\delta \uu\otimes\uu_1) -\Div\,(\uu_2\otimes \delta \uu)  -\Div\, (\nabla \delta \ddd \odot \nabla \ddd_1) \\
	&\qquad\qquad\qquad \qquad   -\Div(\nabla \ddd_2 \odot \nabla \delta \ddd)
	+\Div\, \delta  \bm{\sigma}\Big)(s,x)\cdot \Dd_q\delta \uu(s,x)\,\dd x\,\dd s.
\end{aligned}
\end{equation}
Next, we observe that 
\begin{align*}
		\sum_{q=-1}^{\infty}
		&
		\int_0^t
		\Big|
		2^{-q}
		\int_{\TT^2}\Dd_q
		\Big(  \Div\,(\delta \uu\otimes\uu_1) +
		\Div\,(\uu_2\otimes \delta \uu)\Big)(s,x)\cdot \Dd_q\delta \uu(s,x)\,\dd x
		\Big|\,\dd s\\
		&\lesssim
		\sum_{q=-1}^{\infty} 
		\int_0^t
		2^{-q}
		\Big(2^q \|	\Dd_q(	\delta \uu(s)\otimes\uu_1(s)) \|_{L^2} 
		+ 2^q \| \Dd_q(	\uu_2(s)\otimes \delta \uu (s)) \|_{L^2} \Big)  
		\| \Dd_q\delta \uu(s)  \|_{L^2}
		\,\dd s\\
		&\lesssim 
		\big(
		\| \delta \uu \otimes \uu_1  \|_{L^2((0,t)\times \mathbb{T}^2)} 					+ 
		\| \uu_2 \otimes \delta \uu  \|_{L^2((0,t)\times \mathbb{T}^2)}	
		\big) 
		\| \delta \uu  \|_{L^2((0,t)\times \mathbb{T}^2)}\\
		&\lesssim 
		\big(\|\uu_1 \|_{L^4((0,t)\times \mathbb{T}^2)}^2
		+\|\uu_2 \|_{L^4((0,t)\times \mathbb{T}^2)}^2 \big) 
		\big(\| \uu_1 \|_{L^2((0,t)\times \mathbb{T}^2)}
		+ \| \uu_2 \|_{L^2((0,t)\times \mathbb{T}^2)})
		\\
		&\lesssim 
		\big(\| \uu_1 \|_{L^\infty(0,t; L^2(\mathbb{T}^2))} \| \nabla \uu_1 \|_{L^2(0,t; L^2(\mathbb{T}^2))}
		+
		\| \uu_2 \|_{L^\infty(0,t; L^2(\mathbb{T}^2))} \| \nabla \uu_2 \|_{L^2(0,t; L^2(\mathbb{T}^2))}\big)\\
		&\qquad \times
			\big(\| \uu_1 \|_{L^2((0,t)\times \mathbb{T}^2)}
		+ \| \uu_2 \|_{L^2((0,t)\times \mathbb{T}^2)})\\
		&\lesssim 
		t^\frac12 \big(\| \uu_1 \|_{L^\infty(0,t; L^2(\mathbb{T}^2))}^2 +\| \uu_2 \|_{L^\infty(0,t; L^2(\mathbb{T}^2))}^2\big)
		\big(\| \nabla \uu_1 \|_{L^2(0,t; L^2(\mathbb{T}^2))} +
		\| \nabla \uu_2 \|_{L^2(0,t; L^2(\mathbb{T}^2))}\big) <\infty.
\end{align*}
In a similar manner, it holds
\begin{align*}
	\sum_{q=-1}^{\infty} &
	\int_0^t
	\Big|
	    2^{-q}
		 \int_{\TT^2}\Dd_q
		\Big(  \Div\,(\nabla \delta \ddd \odot \nabla \ddd_1) +
		\Div\,(\nabla \ddd_2 \odot \nabla \delta \ddd)\Big)(s,x)\cdot \Dd_q\delta \uu(s,x)\,\dd x
	\Big|\,
	\dd s
	 \\
	&\lesssim
	t^\frac12
	\big(\| \nabla \ddd_1\|_{L^\infty(0,t;L^2(\mathbb{T}^2))} +\|\nabla \ddd_2\|_{L^\infty(0,t;L^2(\mathbb{T}^2))})
	\big(\| \uu_1 \|_{L^\infty(0,t; L^2(\mathbb{T}^2))} +\| \uu_2 \|_{L^\infty(0,t; L^2(\mathbb{T}^2))}\big)\\
	&\qquad \times 
	\big(\| \nabla \ddd_1\|_{L^2(0,t; H^1(\mathbb{T}^2))} + \|\nabla \ddd_2\|_{L^2(0,t; H^1(\mathbb{T}^2))} \big)
	<\infty.
\end{align*}
Finally, since $ \delta \bm{\sigma} \in L^2(0,T; H^{-\zeta}(\TT^2))$ with $\zeta\in (0,1)$, we infer that 
\begin{align*}
	\sum_{q=-1}^\infty 
	\int_0^t 2^{-q}
	\Big|
	\int_{\TT^2}\Dd_q
	\delta \bm{\sigma}(s,x) :\Dd_q\nabla \delta \uu(s,x)\,\dd x
	\Big|\,	\dd s
	&\leq 
	\sum_{q=-1}^\infty 
	2^{-q}
	\|	\Dd_q\delta  \bm{\sigma}		\|_{L^2((0,t)\times \mathbb{T}^2)}
	\| 	\Dd_q\nabla \delta \uu		\|_{L^2((0,t)\times \mathbb{T}^2)}\\
	&\leq 
	\| \delta  \bm{\sigma}(s) 				\|_{L^2(0,t;H^{-1}(\mathbb{T}^2))}
	\| \nabla \delta \uu(s) 				\|_{L^2((0,t)\times \mathbb{T}^2)}
	<\infty.
\end{align*}
Hence, we are able to apply the dominated convergence theorem for series to the identity \eqref{appx:deltau-ser} and bring the sum inside the sign of the integral such that 
\begin{align*}
	\frac12 \| \delta \uu (t) \|_{H^{-\frac{1}{2}}}^2  + \nu  
	\int_0^t \| \nabla \delta \uu (s) \|_{H^{-\frac{1}{2}}}^2\dd s
	&=  \int_0^t \Big( \langle \delta \uu \otimes \uu_1 + \uu_2\otimes \delta \uu,\,\nabla \delta 
	\uu \rangle_{H^{-\frac{1}{2}}} \\
	&\qquad \quad 
	+ \langle \nabla \delta \ddd \odot \nabla \ddd_1 
	+ \nabla \ddd_2 \odot \nabla \delta \ddd,\,\nabla \delta \uu 
	\rangle_{H^{-\frac{1}{2}}}-
	\langle \delta \bm{\sigma},\,\nabla \delta \uu\rangle_{H^{-\frac{1}{2}}}
	\Big)(s)\,\dd s.
\end{align*}
With an analogous procedure, we infer that  the following identities for $\delta \ddd$ holds: 
\begin{align*}
 	\frac12 \| \delta \ddd(t) \|_{L^2}^2 
 	+ 
 	\int_0^t \| \nabla \ddd(s) \|_{L^2}^2 \dd s
 	& = 
	\int_0^t 
	\Big( 
		- \int_{\mathbb{T}^2} (\uu_1\cdot \nabla \delta \ddd)\cdot \delta\ddd  \,\dd x
		- \int_{\mathbb{T}^2} 
		(\delta \uu \cdot \nabla \ddd_2)\cdot \delta\ddd \,\dd x 
		+ \frac32  \int_{\TT^2}
			((\nabla \delta \uu)\ddd_1)\cdot \delta\ddd  \,\dd x\\
		&\qquad \quad\ \  
		+ \frac32 \int_{\TT^2}( (\nabla  \uu_2)\delta \ddd) \cdot \delta \ddd\,\dd x 
		+\frac12  \int_{\TT^2}
		((\nabla^{\mathrm{tr}} \delta \uu)\ddd_1)\cdot \delta\ddd  \,\dd x\\
		&\qquad \quad\ \  
		+ \frac12 \int_{\TT^2}( (\nabla^{\mathrm{tr}}  \uu_2)\delta \ddd) \cdot \delta \ddd\,\dd x
		-\int_{\mathbb{T}^2} \delta \nabla_\ddd W(\ddd)\cdot \delta\ddd \,\dd x	\Big)(s)\,\dd s,
\end{align*}	
\begin{align*}
 	\frac12 \|\nabla  \delta \ddd(t) \|_{H^{-\frac{1}{2}}}^2 
 	+ 
 	\int_0^t \| \Delta \ddd(s) \|_{H^{-\frac{1}{2}}}^2 \dd s
 	& = 
	\int_0^t 
	\Big( 
	    \langle  \uu_1 \cdot \nabla \delta \ddd, \Delta \delta \ddd \rangle_{H^{-\frac{1}{2}}}
	    + \langle \delta \uu \cdot \nabla \ddd_2, \Delta \delta \ddd \rangle_{H^{-\frac{1}{2}}}
		- \frac32 \langle (\nabla \delta \uu)\ddd_1,\, \Delta \delta \ddd\rangle_{H^{-\frac{1}{2}}} \\
		&\qquad \quad\ \  
		- \frac32 \langle  (\nabla  \uu_2)\delta \ddd,\, \Delta \delta \ddd\rangle_{H^{-\frac{1}{2}}} 
		- \frac12 \langle (\nabla^{\mathrm{tr}} \delta \uu)\ddd_1,\, \Delta \delta \ddd\rangle_{H^{-\frac{1}{2}}}  \\
		&\qquad \quad\ \  
		- \frac12 \langle  (\nabla^{\mathrm{tr}}  \uu_2)\delta \ddd,\, \Delta \delta \ddd\rangle_{H^{-\frac{1}{2}}}
		+\langle  \delta \nabla_\ddd W(\ddd), \Delta \delta \ddd \rangle_{H^{-\frac{1}{2}}}
	\Big)(s)\,\dd s.
\end{align*}	
Adding the above three identities together, we arrive at the conclusion \eqref{energy-identity}.

\section{A Sobolev embedding theorem}\label{Appx:proof-lemma-eps}
\setcounter{equation}{0}
In this part, we aim to prove Lemma \ref{lemma:eps}, which addresses suitable Sobolev embedding results with the corresponding constants. For the sake of convenience for the readers, we recall the statement here:
\begin{lemma}\label{lemma:esp-in-the-appx}
    Let $s\in [0,d/2)$. There exists a constant $C_d>0$, depending only on $d$, such that
	\begin{equation*}
		\| f \|_{L^p}\leq C_d \sqrt{p} \| f \|_{H^s},\quad \text{with }p = \frac{2d}{d-2s},\quad \forall\, f\in H^s(\TT^d).
	\end{equation*}
\end{lemma}
\noindent 
In order to prove the above lemma, we need  some preparations. First, we re-introduce the radial smooth function $\chi \in \mathfrak{D}(B(0,4/3))\subset \mathfrak{D}(\mathbb{R}^d)$ of Section \ref{sec:Littlewood-Paley}, which is identically equal to $1$ on the ball $B(0,3/4)$.
\begin{definition}
	Let $\sigma<0$ be a negative real number. We denote by $B^\sigma=B^\sigma(\mathbb{T}^d)$ the Banach space 
	which consists of all distributions $f\in \mathfrak{D}'(\mathbb{T}^d)$ for which the following norm
	is finite:
	\begin{equation*}
		\| f \|_{B^{\sigma}}:= 
		\sup_{\lambda\geq 1} 
		\Big\{
		\lambda^{\sigma}
		\Big\| \sum_{n\in \mathbb{Z}^d} f_n 
		\chi (\lambda^{-1}n )e^{\mathrm{i} n\cdot x} 
		\Big\|_{L^\infty}
		\Big\}, 
		\qquad \text{where}\quad  f =  \sum_{n\in\mathbb{Z}^d} f_n e^{\mathrm{i} n\cdot x}.
	\end{equation*}
\end{definition}
\begin{remark}
	For each $s<0$, we remark that $B^{s}$ precisely corresponds with the Besov space $B_{\infty,\infty}^s$ in Definition \ref{def:Besov-spaces}, equipped with the norm 
	\begin{equation*}
		\| f \|_{B_{\infty, \infty}^s} = 
		\sup_{q\in \mathbb{N}\cup \{-1\}} 2^{qs}\| S_q f \|_{L^\infty}
	\end{equation*}
    (cf. also Lemma \ref{lemma:replacing-Dq-with-Sq}). In this section, we stick with the definition 
	of $B^s$ (and not $B_{\infty,\infty}^s$), since its norm will be somehow meaningful for our goal. 
\end{remark}

Next, we show the continuous embedding of $H^s(\mathbb{T}^d)$ in $B^{s-d/2}$  for $s\in [0,d/2)$.
\begin{lemma}\label{lemma:B-Hs-embedding}
	For any $s\in [0,d/2)$, the Sobolev space $H^s(\mathbb{T}^d)$ is continuously embedded in $B^{s-d/2}$ and 
	there exists a positive constant $C>0$, depending only on $d$ and $\chi$, such that 
	\begin{equation*}
		\| f \|_{B^{s-\frac{d}{2}}}\leq  
		\frac{C}{\sqrt{d/2-s}}\| f \|_{H^s}.
	\end{equation*}
\end{lemma}
\begin{proof}
	For any $\lambda\geq 1$ and any distribution $f = \sum_{n\in\mathbb{Z}^d} f_n e^{\mathrm{i}n\cdot x}$, we remark that
	\begin{align*}
		\big \| \sum_{n\in \mathbb{Z}^d} f_n 
		\chi (\lambda^{-1}\xi)e^{\mathrm{i} n\cdot x} \big\|_{L^\infty}
		& \leq 
		 \sum_{n\in \mathbb{Z}^d}
		 | \chi(\lambda^{-1} n)||f_n|
		 \leq 
		 \Big(
			 \sum_{n\in \mathbb{Z}^d}
			 \frac{| \chi(\lambda^{-1} n)|^2}{(1+|n|)^{2s}}
		 \Big)^\frac{1}{2}
		 \Big(
			 \sum_{n\in \mathbb{Z}^d}
			 (1+|n|)^{2s}|f_n|^2
		 \Big)^\frac{1}{2}\\
		 & =
		 (2\pi)^d\Big(
			 \sum_{n\in \mathbb{Z}^d}
			 \frac{| \chi(\lambda^{-1} n)|^2}{(1+|n|)^{2s}}
		 \Big)^\frac{1}{2}
		 \|  f \|_{H^s}.
	\end{align*}
	Next, we recast the remaining sum into integral form as follows. We first observe that
	for any $n\in \mathbb{Z}^d\setminus \{0\}$ and any $x\in n+[0,1]^d$, one has that
	$|x|\leq |n|+\sqrt{d}\leq (1+\sqrt{d})|n|$, which implies
	\begin{equation*}
	\begin{aligned}
		\frac{1}{(1+|n|)^{2s}}\leq \frac{1}{\;\;|n|^{2s}}\leq \frac{(1+\sqrt{d})^{2s}}{|x|^{2s}}
		\leq 
		\frac{(1+\sqrt{d})^{d}}{|x|^{2s}}.
	\end{aligned}
	\end{equation*}	 
	We are now in the position to handle the following series
	\begin{align*}
		\sum_{n\in \mathbb{Z}^d}
			 \frac{| \chi(\lambda^{-1} n)|^2}{(1+|n|)^{2s}}
		&=
		1+
		\sum_{n\in \mathbb{Z}^d\setminus\{0\}}
		\int_{n+[0,1]^d}  	 
		\frac{| \chi(\lambda^{-1} n)|^2}{(1+|n|)^{2s}}\,\dd x\\
		&\leq 
		1+(1+\sqrt{d})^{d}
		\sum_{n\in \mathbb{Z}^d\setminus\{0\}}
		\int_{n+[0,1]^d}  	 
		\frac{\mathbf{1}_{B(0,4/3)}(\lambda^{-1} n)}{|x|^{2s}}\,\dd x
		\\
		&\leq 
		1+(1+\sqrt{d})^{d}
		\sum_{n\in \mathbb{Z}^d}
		\int_{n+[0,1]^d}  	 
		\frac{\mathbf{1}_{B(0,4/3)}(\lambda^{-1} x/(1+\sqrt{d}))}{|x|^{2s}}\,\dd x
		\\
		&\leq 
		1+(1+\sqrt{d})^{d}
		\int_{\mathbb{R}^d}  	 
		\frac{\mathbf{1}_{B(0,4/3)}(\lambda^{-1} x/(1+\sqrt{d}))}{|x|^{2s}}\,\dd x
		\\
		&\leq 
		1+(1+\sqrt{d})^{d}
		\lambda^{d-2s}
		\int_{\mathbb{R}^d}  	 
		\frac{\mathbf{1}_{B(0,4/3)}(y/(1+\sqrt{d}))}{|y|^{2s}}\,\dd y\\
		&\leq 
		1+C_d
		\lambda^{d-2s}
		\int_0^{(1+\sqrt{d})8/3}
		\rho^{d-1-2s}\,\dd \rho\\
		&\leq 
		1+\frac{C_d}{d-2s}\lambda^{d-2s}\Big(\frac{8(1+\sqrt{d})}{3}\Big)^{d-2s}
		\\
		&\leq 
		1+\frac{C_d}{d-2s}\lambda^{d-2s}\Big(\frac{8(1+\sqrt{d})}{3}\Big)^{d}.
	\end{align*}
	The last inequality, together with the condition $\lambda\geq 1$ yields the conclusion of this lemma.
\end{proof}

Then we state an interpolation inequality for the $L^p$-norm of functions in terms of norms in $B^{s-\frac{d}{2}}$ and $H^s$. More precisely, 
\begin{lemma}\label{lemma:Lp-HsB-embedding}
	Let $s\in (0,d/2)$ and denote $p = 2d/(d-2s)>2$. Then the Sobolev space $H^s(\mathbb{T}^d)= H^s(\mathbb{T}^d)\cap B^{s-d/2}$ is embedded in $L^p(\mathbb{T}^d)$ and there exists a constant $C>0$ depending only on $d$ and $\chi$ such that
	\begin{equation*}
		\| f \|_{L^p} \leq 
		\frac{C}{(p-2)^\frac{1}{p}}
		\| f \|_{B^{s-\frac{d}{2}}}^{1-\frac{2}{p}}
		\| f \|_{H^s}^\frac{2}{p}.
	\end{equation*}
\end{lemma}
\begin{proof}
	Without loss of generality, we can assume that $\| f \|_{B^{s-d/2}}=1$. For any $A\geq 1$, we introduce 
	the following frequencies decomposition of a general distribution $f$: 
	\begin{equation*}
		f = f_{l,A} + f_{h,A},\qquad f_{l,A} := \sum_{n\in\mathbb{Z}^d}\chi(A^{-1}n)f_ ne^{\mathrm{i} n\cdot x},
	\end{equation*}
	where $f_{l,A}$ treats the low frequencies of $f$, while $f_{h,A}$ addresses the high frequencies of $f$.
	The triangular inequality implies that
	\begin{equation*}
		\{ x\in \mathbb{T}^d\,|\, |f(x)|>\lambda\} \subset 
		\Big\{x \in \mathbb{T}^d\,|\, |f_{l,A}(x)|>\frac{\lambda}{2}\Big\}
		\cup 
		\Big\{x \in \mathbb{T}^d\,|\, |f_{h,A}(x)|>\frac{\lambda}{2}\Big\}.
	\end{equation*}
	By the definition of $B^{s-d/2}$ we have that $\| f_{l,A} \|_{L^\infty}\leq A^{d/2-s}\| f \|_{B^{s-d/2}}
	= A^{d/2-s}$. Then for any $A\geq 1$, we denote by $\lambda =\lambda(A) := 2 A^{d/2-s}$, with inverse 
	$A = A(\lambda) =(\lambda/2)^{p/d}$. From this we deduce that 
	\begin{equation*}		 
		\| f_{l,A(\lambda)} \|_{L^\infty}\leq  A(\lambda)^{\frac{d}{2}-s} = \frac{\lambda}{2} \quad 
		\Longrightarrow\quad  m\big(\{x\in \mathbb{T}^d\,|\, |f_{l,A(\lambda)}(x)|>\lambda/2\}\big)=0.
	\end{equation*}
	Furthermore, we remark that, for all $f\in L^p(\mathbb{T}^d)$, it holds
	\begin{align*}
		 \| f \|_{L^p}^p 
		 &= \int_{\mathbb{T}^d} |f(x)|^p \,\dd x = 
		 \int_{\big\{x\in\mathbb{T}^d\,|\,|f(x)|<2\big\}}|f(x)|^2|f(x)|^{p-2} dx+
		 \int_{\big\{x\in\mathbb{T}^d\,|\,|f(x)|\geq 2\big\}}|f(x)|^p\, \dd x
		 \\		
		 &= 
		 (2\pi)^d2^{p-2}\| f \|_{L^2}^2+
		 \int_{\mathbb{T}^d} \int_{2}^{|f(x)|}p\lambda^{p-1} \,\dd\lambda \,\dd x\\ 
		 &=
		 (2\pi)^d2^{p-2}\| f \|_{H^s}^2+
		 \int_2^\infty p\lambda^{p-1}m(\{x\in\mathbb{T}^d\,|\,|f(x)|>\lambda\})\,\dd\lambda\\
		 &\leq 
		  (2\pi)^d2^{p-2}\| f \|_{H^s}^2+
		 \int_2^\infty p\lambda^{p-1}m\Big(\Big\{x \in \mathbb{T}^d\,|\, |f_{l,A(\lambda)}(x)|
		 >\frac{\lambda}{2}\Big\}
		\bigcup 
		\Big\{x \in \mathbb{T}^d\,|\, |f_{h,A(\lambda)}(x)|>\frac{\lambda}{2}\Big\}\}\Big)\,\dd \lambda\\
		&\leq 
		(2\pi)^d2^{p-2}\| f \|_{H^s}^2+
		\int_2^\infty p\lambda^{p-1}
		\underbrace{\bigg(
		m\Big(\Big\{x \in \mathbb{T}^d\,|\, |f_{l,A(\lambda)}(x)|>\frac{\lambda}{2}\Big\}\Big)
		}_{=0}+
			m\Big(
		\Big\{x \in \mathbb{T}^d\,|\, |f_{h,A(\lambda)}(x)|>\frac{\lambda}{2}\Big\}\Big)
		\bigg)\,\dd\lambda\\
		&\leq 
		(2\pi)^d2^{p-2}\| f \|_{H^s}^2+
		\int_2^\infty p\lambda^{p-1}
		m\Big(\Big\{x \in \mathbb{T}^d\,|\, |f_{h,A(\lambda)}(x)|>\frac{\lambda}{2}\Big\}\Big)\,\dd\lambda.
	\end{align*}
	Using the following fact 
	\begin{align*}
		m\Big(\Big\{ x\in\mathbb{T}^d\,|\, |f_{h,A(\lambda)}(x)|>\frac{\lambda}{2}\Big\}\Big)
		&= \int\limits_{\big\{ x\in\mathbb{T}^d\,|\, |f_{h,A(\lambda)}(x)|>\lambda/2\big\}}
		\frac{\lambda^2}{4}\frac{4}{\lambda^2}\,\dd x \\
		&\leq 
		\int\limits_{\big\{ x\in\mathbb{T}^d\,|\, |f_{h,A(\lambda)}(x)|>\lambda/2\big\}}
		|f_{h,A(\lambda)}(x)|^2 \frac{4}{\lambda^2}\,\dd x\\
		&\leq \frac{4}{\lambda^2}\| f_{h,A(\lambda)} \|_{L^2}^2,
	\end{align*}
	we get 
	\begin{align*}
		 \| f \|_{L^p}^p 
		 &\leq 
		  (2\pi)^d2^{p-2}\| f \|_{H^s}^2+
		  \int_2^\infty p\lambda^{p-1}\frac{4}{\lambda^2}\| f_{h,A(\lambda)} \|_{L^2}^2
		 \,\dd \lambda\\
		 &\leq 
		 (2\pi)^d2^{p-2}\| f \|_{H^s}^2+
		 \int_2^\infty 4p\lambda^{p-3}
		 \frac{1}{(2\pi)^d}\sum_{n\in\mathbb{Z}^d} (1-\chi(A(\lambda)^{-1}n))^2|f_n|^2\, \dd \lambda\\
		 &\leq 
		 (2\pi)^d2^{p-2}\| f \|_{H^s}^2+
		 \int_2^\infty 4p\lambda^{p-3}
		 \frac{1}{(2\pi)^d}\sum_{n\in\mathbb{Z}^d,\,|n|\geq \frac{3}{4}A(\lambda)}|f_n|^2\,\dd \lambda\\
		 &\leq 
		 (2\pi)^d2^{p-2}\| f \|_{H^s}^2+
		 \frac{4}{(2\pi)^d}
		 \sum_{n\in\mathbb{Z}^d}
		 |f_n|^2
		 \int_0^{2\left(\frac{4}{3}|n|\right)^\frac{d}{p}} p\lambda^{p-3}
		 \,\dd \lambda\\
		 &\leq 
		 (2\pi)^d2^{p-2}\| f \|_{H^s}^2+
		 \frac{4}{(2\pi)^d}\frac{p}{p-2}
		 \sum_{n\in\mathbb{Z}^d}2^{p-2}\left( \frac{4}{3}|n| \right)^{\frac{d(p-2)}{p}}		 
		 |f_n|^2\\
		 &\leq 
		 (2\pi)^d2^{p-2}\| f \|_{H^s}^2+
		 \frac{2^{p}}{(2\pi)^d}\frac{p}{p-2}
		 \left(\frac{4}{3}\right)^{2s}
		 \sum_{n\in\mathbb{Z}^d}|n|^{2s}		 
		 |f_n|^2\\
		 &\leq  
		 (2\pi)^d2^{p-2}\| f \|_{H^s}^2+
		 \frac{2^{p}}{(2\pi)^d}\frac{p}{p-2}
		 \left(\frac{4}{3}\right)^{2s}\| f \|_{H^s}^2.
	\end{align*}
	Since $2s/p= d/p-2d/p^2<d$, we finally gather that
	\begin{equation*}
		\| f \|_{L^p}\leq  
		\left(2(2\pi)^d
		+\frac{4^{d+1}}{3^{d}(p-2)^\frac{1}{p}}\right)
		\| f \|_{H^s}^{\frac{2}{p}},
	\end{equation*}
	which concludes the proof of the lemma, since $2(2\pi)^d\leq 4(2\pi)^d/(p-2)^{1/p}$ for $p>2$.
\end{proof}

Now we are in a position to prove Lemma \ref{lemma:esp-in-the-appx}. First, invoking Lemma \ref{lemma:B-Hs-embedding} and Lemma \ref{lemma:Lp-HsB-embedding}, we infer that there exists a positive constant $\tilde C$  depending only on $d$ and $\chi$, such that
	\begin{equation*}
		\| f \|_{L^p} \leq \tilde C \frac{1}{(p-2)^\frac{1}{p}}
		\left(\frac{1}{\sqrt{d/2-s}}\right)^{1-\frac{2}{p}}
		\| f \|_{H^s}
		= \tilde C 
		\frac{\sqrt{p}}{\sqrt{d}(p-2)^\frac{1}{p}p^\frac{1}{p}}
		\| f \|_{H^s}.
	\end{equation*}
	When treating high value of $p$, for instance, for $p\geq 4$ (i.e., $s\geq d/4$), we remark that 
	the above inequality reduces to
	\begin{equation*}
		\| f \|_{L^p} \leq C\sqrt{p} \| f \|_{H^s}.
	\end{equation*}
	On the other hand, for $s=0$, then $p=2$ and one simply has $\| f \|_{L^p} = \| f \|_{H^s}$. Hence, a standard interpolation easily leads to the conclusion 
	$ \| f \|_{L^p}\leq C\sqrt{p}\| f \|_{H^s}$, also for any $s\in (0,d/4)$, i.e., $p\in(2,4]$. The proof of Lemma \ref{lemma:esp-in-the-appx} is complete. 
	\hfill $\square$

\section*{Acknowledgments}
 This work was initiated during the first author's visit to School of Mathematical Sciences at Fudan University in 2017. The hospitality of the mathematical school is acknowledged. The second author was partially supported by NNSFC 12071084 and the Shanghai Center for Mathematical Sciences. The authors would like to express their appreciation to Prof. M. Paicu, for many useful inputs and valuable comments that have significantly contributed to the quality of this article.


\end{document}